\documentclass[11pt,reqno]{amsart}
\usepackage{amsmath}
\usepackage{amsfonts}
\usepackage{amssymb, amsthm, epsfig}
\usepackage{mathrsfs}
\usepackage{hyperref}
 \usepackage{yhmath}
\usepackage{cite}
\usepackage[utf8]{inputenc}
\usepackage[T1]{fontenc}
\usepackage{geometry}
\usepackage{graphicx}
\usepackage{verbatim}
\usepackage{color}
\usepackage{tikz}

\theoremstyle{plain}
\newtheorem{thm}{Theorem}[subsection]

\newtheorem{lem}[thm]{Lemma}
\newtheorem{prop}[thm]{Proposition}
\theoremstyle{definition}

\theoremstyle{remark}
\newtheorem{rem}{Remark}[section]

\numberwithin{equation}{section}

\allowdisplaybreaks

\DeclareSymbolFont{lettersA}{U}{pxmia}{m}{it}
\DeclareMathSymbol{\piup}{\mathord}{lettersA}{"19}

\makeatletter

\newcommand{\Rmnum}[1]{\expandafter\@slowromancap\romannumeral#1@}
\makeatother

\begin{document}
\title[]
{Oblique wave interactions in 2D steady supersonic  flows of Bethe-Zel'dovich-Thompson fluids}

\author{Geng Lai}

\address{Department of Mathematics, Shanghai University,
Shanghai, 200444, P.R. China;  Newtouch Center for Mathematics, Shanghai University, Shanghai, 20044, P.R. China.}
\email{\tt laigeng@shu.edu.cn}

\keywords{Bethe-Zel'dovich-Thompson fluid; supersonic ramp flow; oblique wave interaction; hodograph transformation; characteristic decomposition}
\date{\today}

\begin{abstract}
This paper studies steady supersonic flow in a 2D semi-infinite divergent duct.
We assume that the flow satisfies the slip boundary condition on the walls of the duct, and the state of the flow is given
at the inlet of the divergent duct. When the fluid is a polytropic ideal gas, the problem can be reduced to some   interactions of rarefaction simple waves, and the existence of a global classical solution inside the divergent duct can be established using the method of characteristics.
In this paper we assume that the fluid is a nonconvex Bethe-Zel'dovich-Thompson (BZT) fluid.
This type of fluid may significantly differ from polytropic ideal gases. For instance, physically admissible rarefaction shocks can occur.
Depending on the oncoming flow state and the flare angles of the divergent duct, thirteen distinct types of oblique wave interactions may occur, including oblique composite waves consisting of shocks and centered simple waves.
This paper systematically studies these oblique wave interactions and constructs global, piecewise smooth, supersonic solutions within the divergent duct using characteristic decomposition and hodograph transformation methods.
We also obtain the
detailed structures of these solutions in addition to their existence.
The results and methods of this paper are also applicable to some 2D Riemann problems for gases with nonconvex equations of state.
\end{abstract}

\maketitle
\tableofcontents

\section{Introduction}
The study of supersonic flows in ducts is fundamental in gasdynamics, as it provides essential insights into the behavior of high-speed flows and forms the basis for understanding complex aerodynamic phenomena.   Referring to Fig. \ref{Fig1.1}, a two-dimensional duct consists of a finite straight section and a semi-infinite divergent section, connected by two sharp corners $B$ and $D$.
A supersonic uniform stream flows through the straight section and enters the divergent section.
Assume that the flow is governed by the compressible Euler equations and the gas is an ideal gas.
Then at the corners, the flow is deflected by two centered simple waves with different types; see Courant $\&$ Friedrichs \cite{CF}. The two centered simple waves, denoted by ${\it f}_{\pm}^{1}$, interact inside the divergent duct, generating two new simple waves ${\it f}_{\pm}^{2}$  as a result. The simple waves ${\it f}_{\pm}^{2}$ will then reflect off the walls of the duct and generate two simple waves ${\it f}_{\pm}^{3}$ again.
The global existence of a continuous and piecewise smooth supersonic flow inside the divergent duct can then be established by analyzing the interactions and reflections of these rarefaction simple waves. There have been many studies on the global existence of piecewise smooth supersonic flows in divergent duct; we refer the reader to \cite{CQ,CF,Lai2,WX1,WX2} and the references therein.

\begin{figure}[htbp]
\begin{center}
\includegraphics[scale=0.38]{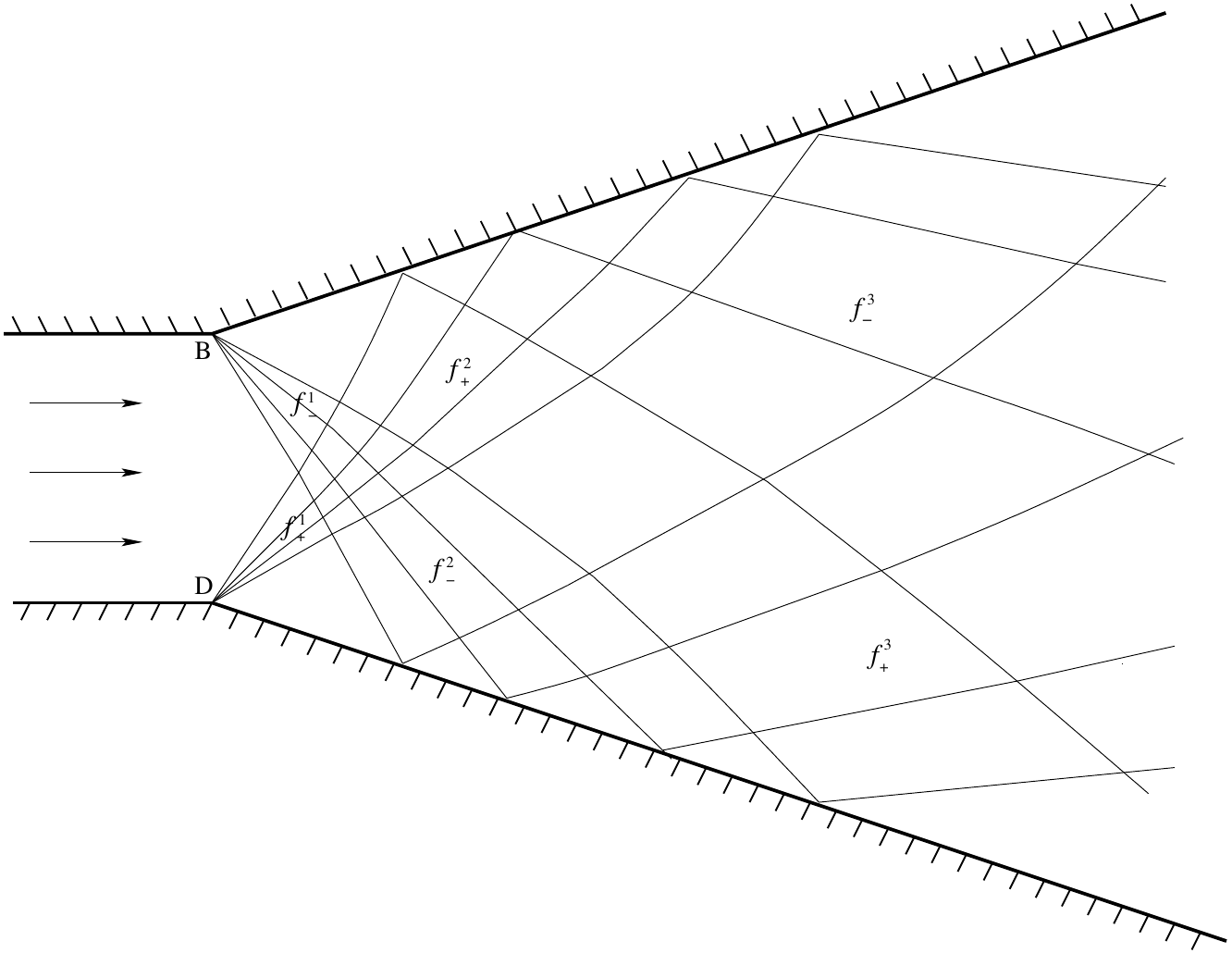}
\caption{\footnotesize Supersonic flow in a 2D semi-infinitely long divergent duct.}
\label{Fig1.1}
\end{center}
\end{figure}

In this paper,
we study the aforementioned supersonic nozzle flow problem for more general gases, such as gases with nonconvex isentropes.
The possibility that isentropes are nonconvex was first explored by Bethe \cite{Bethe} and Zel'dovich \cite{ZE} for the van der Waals equation of state,  and by Thompson and Lambrakis \cite{Thompson2} for more sophisticated equations of state. Thompson \cite{Thompson1}
introduced a non-dimensional thermodynamic quantity
$$\mathcal{G}=\frac{\tau^3}{2c^2}\Big(\frac{\partial^2 p}{\partial \tau^2}\Big)_{s}$$
called the fundamental derivative of gasdynamics to
delineate the dynamic behavior of compressible fluids, where $p$ is the pressure, $\tau$ is the specific volume, $s$ is the specific entropy and $c$ is the speed of sound.
Fluids with $\mathcal{G}<0$ within a certain thermodynamic region are now commonly referred to as Bethe-Zel'dovich-Thompson (BZT) fluids. 
Bethe-Zel'dovich-Thompson fluids may significantly differ from polytropic ideal gases. For example, physically admissible
rarefaction shocks may occur in BZT fluids; see \cite{BBKN,Cra,KNB,NSMGC,Thompson2,Weyl,ZGC} and the survey paper \cite{K}.


We assume that the flow in the 2D divergent nozzle is governed by the 2D steady potential flow equations, which can be written as
\begin{equation}\label{E1}
\left\{
  \begin{array}{ll}
   (\rho u)_x+(\rho v)_y=0,\\[4pt]
  \displaystyle  u_y-v_x=0
  \end{array}
\right.
\end{equation}
supplemented by the Bernoulli law
\begin{equation}\label{E2}
\frac{1}{2}(u^2+v^2)+h(\tau)=\mathcal{B},
\end{equation}
where $\mathcal{B}$ is a constant.
Here, $(u, v)$ is the flow velocity, $\rho=1/\tau$ is the density, $p=p(\tau)$ is the pressure, and $h(\tau)$ is the enthalpy which satisfies
\begin{equation}\label{5801}
h'(\tau)=\tau p'(\tau).
\end{equation}

In the paper, we assume that the equation of state $p=p(\tau)$ satisfies the following assumptions (see Figure  \ref{Fig1.2}):
\begin{description}
  \item[(A1)] $p'(\tau)<0$ for $\tau>0$;
  \item[(A2)] $p(\tau)\in C^2(0,+\infty)$;
  \item[(A3)] there exist $\tau_1^{i}$ and $\tau_2^{i}$ ($0<\tau_1^{i}<\tau_2^{i}$) such that $p''(\tau)>0$ for $ \tau\in(0, \tau_1^i)\cup(\tau_2^i, +\infty)$ and $ p''(\tau)<0$ for $\tau_1^i<\tau<\tau_2^i$;
  \item[(A4)] $\lim\limits_{\tau\rightarrow 0}p'(\tau)=-\infty$ and there exists a constant $\gamma>1$ such that $\lim\limits_{\tau\rightarrow+\infty}\tau^{\gamma+1}p'(\tau)$ exists.
\end{description}


\begin{figure}[htbp]
\begin{center}
\includegraphics[scale=0.5]{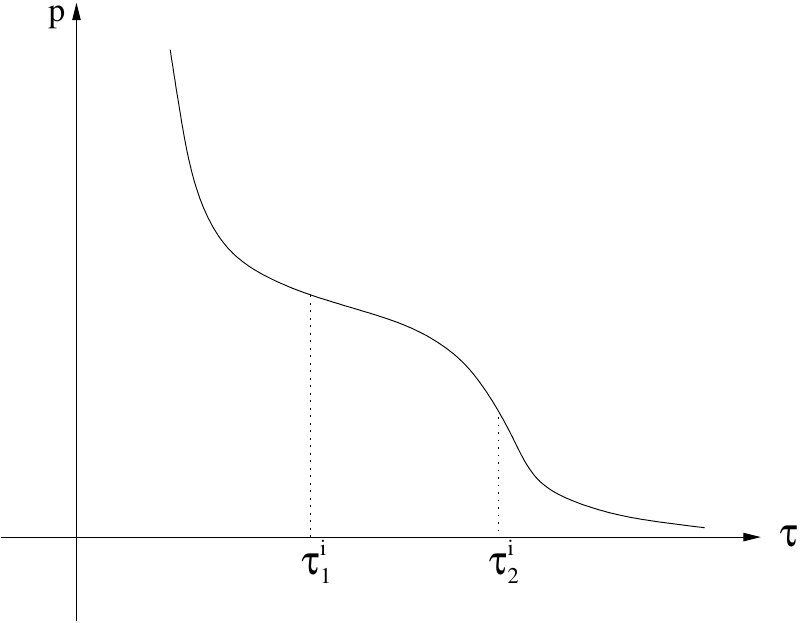}
\caption{\footnotesize The graph of the equation of state $p=p(\tau)$.}
\label{Fig1.2}
\end{center}
\end{figure}

%
Assume that
there is a supersonic flow, with a constant velocity $(u_0, 0)$ and specific volume $\tau_0$ through the straight section, arriving at the inlet of the divergent duct, where $u_0>c_0:=\tau_0\sqrt{-p'(\tau_0)}$.
Assume furthermore that the walls of the divergent duct can be represented by
$$
\mathcal{W}_{\pm}:\quad  y=\pm1+x\tan\theta_{\pm},~~ x>0,
$$
where
$0<\theta_{+}<\frac{\pi}{2}$ and $-\frac{\pi}{2}<\theta_{-}<0$ represent the flare angles of the divergent duct.
To understand the flow structure inside the divergent duct, we consider (\ref{E1}) with the following initial and boundary conditions:
\begin{equation}\label{42701}
\left\{
  \begin{array}{ll}
    (u, v)=(u_0, 0), & \hbox{ $x=0$, $-1<y<1$;} \\[2pt]
   \rho v=\rho u \tan\theta_{+}, & \hbox{$y=1+x\tan \theta_{+}$, $x>0$;} \\[2pt]
        \rho v=\rho u \tan\theta_{-}, & \hbox{$y=-1+x\tan\theta_{-}$, $x>0$.}
  \end{array}
\right.
\end{equation}
The problem (\ref{E1}), (\ref{42701}) is an initial-boundary value problem (IBVP).
We aim to construct a piecewise smooth solution to the IBVP (\ref{E1}), (\ref{42701}) in the domain $\Sigma=\{(x,y)\mid -1+x\tan\theta_{-}\leq y\leq 1+x\tan \theta_{+}, x>0\}$.

For the IBVP (\ref{E1}), (\ref{42701}), there exist two oblique waves issued from the corners $B(0, 1)$ and $D(0, -1)$, and the oblique waves can be constructed by solving a Riemann-type  problem.
It is well known that Riemann problems play a key role in understanding the wave structure of the compressible Euler equations.
Wendroff \cite{Wen1,Wen2} studied the 1D Riemann problem for the Euler equations with a nonconvex equation of state and constructed composite wave solutions that are composed of shock and simple waves.
Subsequently, Liu \cite{Liu1,Liu2} introduced an extended entropy condition and established the existence and uniqueness of the entropy solutions to the 1D Riemann problem.
We also refer the reader to the survey paper by Menikoff and Plohr
 \cite{MP} and  the  monograph by LeFloch and Thanh \cite{LeFloch2} for the 1D Riemann problem for gases with nonconvex equations of state.

Vimercati et al. \cite{VKG1} studied 2D steady supersonic flows of BZT fluids past compressible and rarefactive ramps.
They derived self-similar solutions to the supersonic ramp problem, including composite waves in which one or two centered simple waves are adjacent up to an oblique shock wave.
Recently, Li and Sheng \cite{Li-Sheng1,Li-Sheng2} also constructed self-similar solutions to the supersonic ramp problem for a van der Waals gas supplemented by the Maxwell equal area rule.
The study of the interactions of these oblique waves is important as it reveals fundamental mechanisms of nonlinear wave behavior in BZT fluids.
Vimercati et al. \cite{VKG2} studied oblique shock wave interactions in 2D steady flows of BZT fluids. Lai \cite{Lai6} studied
the interaction of oblique shock-fan composite waves in a 2D steady supersonic jet flow.
When the gas is a nonconvex BZT fluid, the problem (\ref{E1}), (\ref{42701}) involves several types of oblique wave interactions that have not been studied before.
From \cite{VKG1} one knows that the wave configuration at a rarefactive ramp may possibly be a fan wave (${\it f}$), fan-shock composite wave (${\it fs}$), fan-shock-fan composite wave (${\it fsf}$), shock wave (${\it s}$), or shock-fan composite wave (${\it sf}$).
Regarding the IBVP (\ref{E1}), (\ref{42701}) involving a non-convex BZT fluid,
several distinct types of oblique wave interactions can arise.
Depending on the oncoming flow state $(u_0, 0, \tau_0)$ and the flare angles $\theta_{\pm}$ of the divergent duct, we identify twelve distinct new types of oblique wave interactions, in addition to the known interactions of fan waves; see Fig. \ref{Fig1.3}.

\begin{figure}[htbp]
\begin{center}
\includegraphics[scale=0.38]{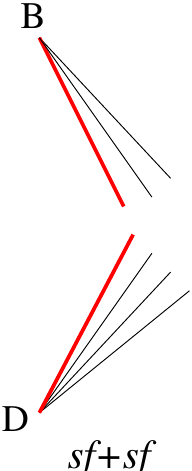}\qquad \quad\includegraphics[scale=0.38]{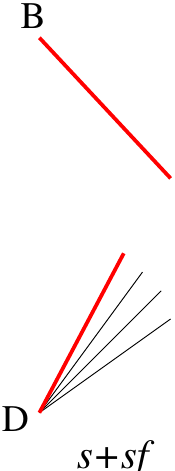}\qquad \quad \includegraphics[scale=0.38]{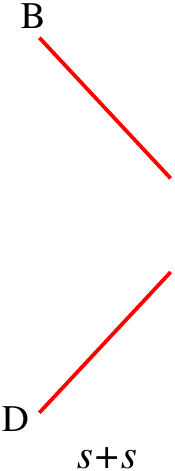}\qquad \quad \includegraphics[scale=0.38]{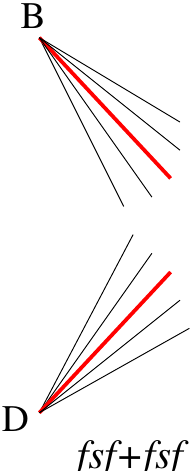}\qquad \quad\includegraphics[scale=0.38]{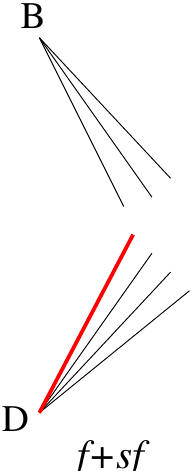}\qquad \quad \includegraphics[scale=0.38]{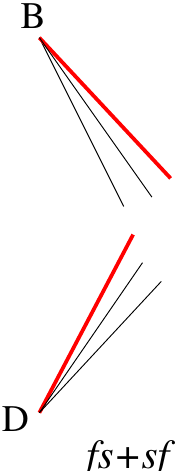}\\ \vskip 12pt
\includegraphics[scale=0.38]{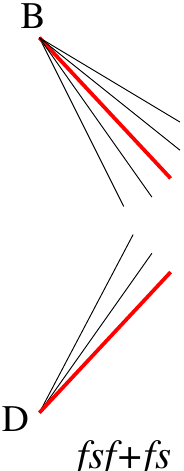}\qquad\quad \includegraphics[scale=0.38]{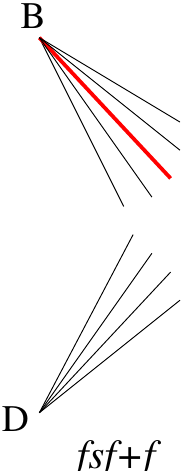}\qquad \quad \includegraphics[scale=0.38]{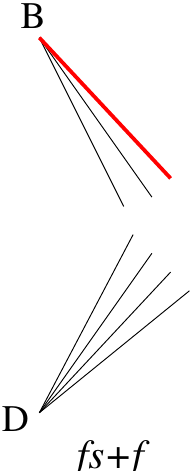}\qquad \quad \includegraphics[scale=0.38]{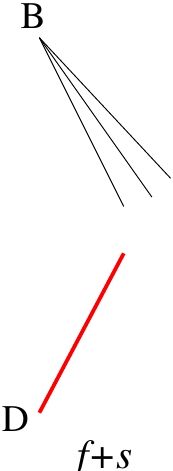}\qquad\quad\includegraphics[scale=0.38]{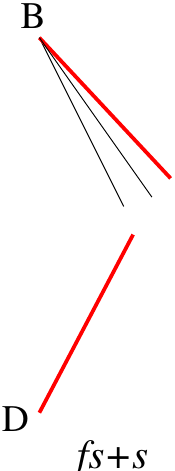}\qquad \quad \includegraphics[scale=0.38]{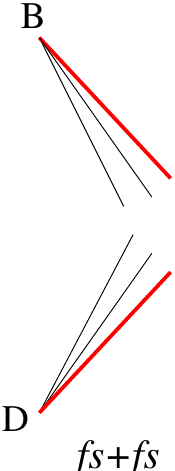}
\caption{\footnotesize Twelve additional types of oblique wave interactions in the divergent duct; where the red lines represent oblique shocks and the black lines represent centered simple waves.}
\label{Fig1.3}
\end{center}
\end{figure}

We aim to systematically study these oblique wave interactions and construct global piecewise smooth supersonic solutions within the divergent duct\footnote{Guardone and Vimercati \cite{GV} studied steady flows in a 2D converging-diverging nozzle for BZT fluids. They constructed exact solutions by assuming that the flows are quasi-one-dimensional.}.
To this aim, we solve some boundary value problems for the system (\ref{E1}), including Goursat problems, discontinuous Goursat problem, singular Cauchy problems, and shock free boundary problems.
In contrast to ideal gases, the flow downstream (or upstream) of a shock of a BZT fluid may be sonic in terms of the flow velocity relative to the shock front.
Moreover, in certain oblique wave interactions, the types of shocks are not known in advance. This results in a fact that the formulation of the boundary conditions on these shocks is a priori unknown.
By calculating the curvatures of the shock waves and using the Liu's extended entropy condition, we establish some a priori estimates for the type of the shocks.
We also find that sonic shocks,  which are envelopes of one family of wave characteristics, appear in certain oblique wave interactions.
The rest of the paper is organized as follows.

In Section 2 we analyze the structure of the system (\ref{E1}).
In Section 2.1 we introduce the characteristics, including characteristic lines, characteristic directions, characteristic angles and characteristic equations of
the system (\ref{E1}).
In Section 2.2 we derive several groups of characteristic decompositions for (\ref{E1}). These characteristic decompositions will be used to control the bounds of the derivatives of the solution.
In Section 2.3 we introduce a hodograph transformation based on the Riemann invariants.
In Section 2.4 we make some a priori assumptions on the hyperbolicity and the characteristic angles of the system (\ref{E1}).

In Section 3 we consider (\ref{E1}) with the following initial and boundary conditions:
\begin{equation}\label{112201a}
\left\{
  \begin{array}{ll}
    (u, v)=(u_0, 0), & \hbox{$x=0$, $y>0$;} \\[2pt]
    \rho v=\rho u \tan\theta_{-}, & \hbox{$y=-1+x\tan\theta_{-}$, $x>0$.}
  \end{array}
\right.
\end{equation}
The problem (\ref{E1}), (\ref{112201a}) is a Riemann-type initial and boundary value problem (RIBVP).
We use the classical phase plane analysis method to construct self-similar solutions to this RIBVP.
We also establish the uniqueness of the self-similar solution to this RIBVP, provided that $u_0$ is sufficiently large.
Thirteen distinct types of oblique wave interactions near the inlet of the divergent duct are also identified in this section.

In Section 4 we study several types of hyperbolic boundary value problems for (\ref{E1}), such as standard Goursat problems, discontinuous Goursat problems, simple wave problem, mixed boundary value problems, and singular Cauchy problems. We use the characteristic decomposition method and the hodograph transformation method to establish the global existence of classical solutions for these problems. The solutions constructed in this section will be used as building blocks of supersonic flows within the divergent duct.

In Section 5 we conduct a systematic study on the thirteen types of oblique wave interactions. The main results are stated as Theorems \ref{thm1}--\ref{thm13}, where we establish the global existence of piecewise smooth supersonic solutions within the divergent duct for each of the thirteen cases. We also obtain the
detailed structures of these solutions in addition to their existence.






\section{Systems of potential flow}

\subsection{Characteristic equation}
For smooth flow, system (\ref{E1}) can be written as the following matrix form:
\begin{equation}
\left(
 \begin{array}{cc}
c^{2}-u^{2} & -uv\\
  0 & -1 \\
  \end{array}
  \right)\left(
           \begin{array}{c}
             u \\
             v \\
           \end{array}
         \right)_{x}+\left(
                         \begin{array}{cc}
                          -uv &  c^{2}-v^{2}\\
                           1 & 0 \\
                         \end{array}
                       \right)\left(
                                \begin{array}{c}
                                  u \\
                                  v \\
                                \end{array}
                              \right)_{y}~=~\left(
                                               \begin{array}{c}
                                                 0 \\
                                                 0 \\
                                               \end{array}
                                             \right),
                                             \label{matrix1}
\end{equation}
where $c=\sqrt{-\tau^2p'(\tau)}$ is the sound speed.

The eigenvalues of system (\ref{matrix1}) are determined
by
\begin{equation}
(v-\lambda u)^{2}-c^{2}(1+\lambda^{2})=0,\label{characteristice}
\end{equation}
which yields
\begin{equation}
\lambda=\lambda_{\pm}=\frac{uv\pm
c\sqrt{u^{2}+v^{2}-c^{2}}}{u^{2}-c^{2}}.
\end{equation}
So, if and only if the flow is supersonic, i.e., $q>c$ where $q=\sqrt{u^2+v^2}$, system (\ref{matrix1}) is hyperbolic and has two families of wave characteristics
defined as the integral curves of
$\frac{{\rm d}y}{{\rm d}x}=\lambda_{\pm}$.

To characterize the wave's propagation direction, we introduce the concept of the characteristic direction.
The wave characteristic direction is defined as the tangent direction to the characteristic curve, which makes an acute angle $A$, termed the Mach angle, with the local flow velocity vector; see Fig. \ref{Fig2.1.1}.
By simple computation, we see that the $C_+$
characteristic direction forms with the direction of the flow velocity
 the angle $A$ from $(u, v)$ to $C_{+}$ in
the counterclockwise direction, and the $C_-$ characteristic direction forms with the
 direction of the flow velocity the angle $A$ from $(u, v)$ to $C_{-}$
in the clockwise direction.
 By computation, we also  have
\begin{equation}
c^{2}=q^{2}\sin^{2} A.\label{210cqo}
\end{equation}

\begin{figure}[htbp]
\begin{center}
\includegraphics[scale=0.55]{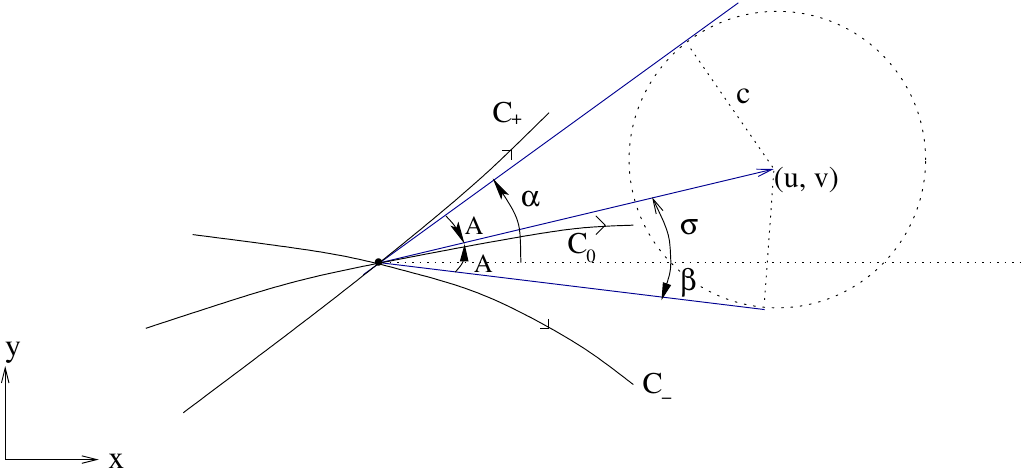}
\caption{ \footnotesize Characteristic curves, characteristic directions, and characteristic angles.}
\label{Fig2.1.1}
\end{center}
\end{figure}

Following \cite{CF} and \cite{Li3}, we introduce the concept of the characteristic angle. The $C_{+}$ ($C_{-}$, resp.) characteristic angle is defined as the counterclockwise angle from the positive $x$-axis to the $C_{+}$ ($C_{-}$, resp.) characteristic direction.
We denote by $\alpha$ and
 $\beta$ the $C_{+}$ and $C_{-}$ characteristic angle,  respectively,
 where $0\leq \alpha-\beta\leq \pi$. Let $\sigma$ be the counterclockwise angle from the positive $x$-axis to the direction of the flow velocity. From (\ref{characteristice}) and (\ref{210cqo}), we have
\begin{equation}
\alpha=\sigma+A,\quad \beta=\sigma-A,\quad\sigma=\frac{\alpha+\beta}{2},\quad A=\frac{\alpha-\beta}{2}, \label{tau}
\end{equation}
\begin{equation} \label{U}
u=q\cos\sigma, \quad v=q\sin\sigma,\quad u=c\frac{\cos\sigma}{\sin A},\quad \mbox{and}\quad v=c\frac{\sin\sigma}{\sin A}.
\end{equation}

Multiplying (\ref{matrix1}) on the left
by
$(1,\mp c\sqrt{u^{2}+v^{2}-c^{2}})$, we get
\begin{equation}
\left\{
  \begin{array}{ll}
  \displaystyle \bar{\partial}_{+}u+\lambda_{-}\bar{\partial}_{+}v
 =0,  \\[10pt]
    \displaystyle  \bar{\partial}_{-}u+\lambda_{+}\bar{\partial}_{-}v=0,
  \end{array}
\right.\label{form}
\end{equation}
where
\begin{equation}\label{41404}
\bar{\partial}_{+}:=\cos\alpha\partial_x+\sin\alpha\partial_y\quad \mbox{and} \quad \bar{\partial}_{-}:=\cos\beta\partial_x+\cos\beta\partial_y.
\end{equation}

From (\ref{U}) we have
\begin{equation}
\bar{\partial}_{\pm}u=\frac{\cos\sigma}{\sin A}\bar{\partial}_{\pm}c+\frac{c\cos\alpha\bar{\partial}_{\pm}\beta
-c\cos\beta\bar{\partial}_{\pm}\alpha}{2\sin^{2}A},\label{1}
\end{equation}
\begin{equation}
\bar{\partial}_{\pm}v=\frac{\sin\sigma}{\sin A}\bar{\partial}_{\pm}c
+\frac{c\sin\alpha\bar{\partial}_{\pm}\beta
-c\sin\beta\bar{\partial}_{\pm}\alpha}{2\sin^{2}A}.\label{2}
\end{equation}
Inserting (\ref{1}) and (\ref{2}) into (\ref{form}), we obtain
\begin{equation}\label{3}
\bar{\partial}_{+}c=\frac{c}{\sin2A}
(\bar{\partial}_{+}\alpha-\cos2A\bar{\partial}_{+}\beta)
\quad\mbox{and}\quad
\bar{\partial}_{-}c=\frac{c}{\sin2A}
(\cos2A\bar{\partial}_{-}\alpha-\bar{\partial}_{-}\beta).
\end{equation}

Differentiating the Bernoulli law (\ref{E2}) along the characteristic directions and
using (\ref{1}) and (\ref{2}), we get
\begin{equation}
\left(\frac{1}{\sin^{2}A}+\kappa\right)\bar{\partial}_{\pm}c=
\frac{c\cos A}{2\sin^{3}A}(\bar{\partial}_{\pm}\alpha-\bar{\partial}_{\pm}\beta),\label{bB1}
\end{equation}
where $$\kappa=\frac{-2p'(\tau)}{\tau p''(\tau)+2p'(\tau)}.$$

Inserting (\ref{bB1}) into (\ref{3}), we obtain
\begin{equation}\label{6}
\bar{\partial}_{+}\alpha=\Lambda\cos^{2}A\bar{\partial}_{+}\beta\quad \mbox{and}\quad \bar{\partial}_{-}\beta=\Lambda\cos^{2}A\bar{\partial}_{-}\alpha,
\end{equation}
respectively,
where
$$
\Lambda=~\varpi(\tau)-\tan^2A\quad \mbox{and}\quad \varpi(\tau)=-\frac{4p'(\tau)+\tau p''(\tau)}{\tau p''(\tau)}.
$$

Combining (\ref{3}) and (\ref{6}), we have
\begin{equation}
c\bar{\partial}_{+}\beta=-\frac{\tan A}{\mu}\bar{\partial}_{+}c=-\frac{p''(\tau)}{2c\rho^4}\tan A\bar{\partial}_{+}\rho,\label{7}
\end{equation}
\begin{equation}
c\bar{\partial}_{+}\alpha=-\left(\frac{1+\kappa}{2}\right)\Lambda\sin2 A\bar{\partial}_{+}c=-\frac{p''(\tau)}{4c\rho^4}\Lambda\sin2 A\bar{\partial}_{+}\rho,\label{10}
\end{equation}
\begin{equation}\label{82301}
c\bar{\partial}_{-}\alpha=\frac{\tan A}{\mu}\bar{\partial}_{-}c=\frac{p''(\tau)}{2c\rho^4}\tan A\bar{\partial}_{-}\rho,
\end{equation}
\begin{equation}
c\bar{\partial}_{-}\beta=\left(\frac{1+\kappa}{2}\right) \Lambda\sin2 A\bar{\partial}_{-}c=\frac{p''(\tau)}{4c\rho^4}\Lambda\sin2 A\bar{\partial}_{-}\rho.\label{8}
\end{equation}

From (\ref{1}), (\ref{2}), and (\ref{7})--(\ref{8}) we also have
\begin{equation}\label{11}
\bar{\partial}_{+}u=\kappa\sin\beta\bar{\partial}_{+}c=c\tau\sin\beta\bar{\partial}_{+}\rho,\quad\bar{\partial}_{-}u=-\kappa\sin\alpha\bar{\partial}_{-}c=-c\tau\sin\alpha\bar{\partial}_{-}\rho,
\end{equation}
\begin{equation}\label{72804}
\bar{\partial}_{+}v=-\kappa\cos\beta\bar{\partial}_{+}c=-c\tau\cos\beta\bar{\partial}_{+}\rho,
\quad\bar{\partial}_{-}v=\kappa\cos\alpha\bar{\partial}_{-}c=c\tau\cos\alpha\bar{\partial}_{-}\rho,
\end{equation}
\begin{equation}\label{12010a}
\bar{\partial}_{+}\sigma=-\frac{\tau^2\sin (2A)}{2}\bar{\partial}_{+}\rho, \quad
\bar{\partial}_{-}\sigma=\frac{\tau^2\sin (2A)}{2}\bar{\partial}_{-}\rho.
\end{equation}

\subsection{Characteristic decompositions}
We first give a  commutator relation derived by Li, Zhang and Zheng in \cite{Li2}.
\begin{prop}
(Commutator relation)  For the directional derivatives $\bar{\partial}_{\pm}$ defined in (\ref{41404}), we have
\begin{equation}
\bar{\partial}_{-} \bar{\partial}_{+}- \bar{\partial}_{+} \bar{\partial}_{-}=
\frac{1}{\sin2 A}\Big[(\cos2 A \bar{\partial}_{+}\beta- \bar{\partial}_{-}\alpha) \bar{\partial}_{-}-
(\bar{\partial}_{+}\beta-\cos2 A \bar{\partial}_{-}\alpha) \bar{\partial}_{+}\Big].
\label{comm}
\end{equation}
\end{prop}
\begin{proof}
From (\ref{41404}) we have
\begin{equation}\label{61513}
\partial_{x}=-\frac{\sin\beta\bar{\partial}_{+}-\sin\alpha\bar{\partial}_{-}}{\sin2A}\quad \mbox{and}\quad
\partial_{y}=\frac{\cos\beta\bar{\partial}_{+}-\cos\alpha\bar{\partial}_{-}}{\sin2A}.
\end{equation}
Then the commutator relation (\ref{comm}) can be obtained by a direct computation.
\end{proof}

\begin{prop}
For the variable $\rho$, we have the following characteristic decomposition
\begin{equation}\label{81104}
\left\{
  \begin{array}{ll}
   \displaystyle \bar{\partial}_{+}\bar{\partial}_{-}\rho=
    \frac{\tau^4p''(\tau)}{4c^2\cos^{2}A }\Big[(\bar{\partial}_{-}\rho)^2+(\mathcal{F}-1)\bar{\partial}_{-}\rho\bar{\partial}_{+}\rho\Big],\\[12pt]
  \displaystyle \bar{\partial}_{-}\bar{\partial}_{+}\rho=
    \frac{\tau^4p''(\tau)}{4c^2\cos^{2}A }\Big[(\bar{\partial}_{+}\rho)^2+(\mathcal{F}-1)\bar{\partial}_{-}\rho\bar{\partial}_{+}\rho\Big],
  \end{array}
\right.
\end{equation}
where
$$
\mathcal{F}=2\sin^2A-\frac{8p'(\tau)\cos^4A}{\tau p''(\tau)}.
$$
\end{prop}

\begin{proof}
It follows from (\ref{11}) and (\ref{comm}) that
\begin{equation}
\begin{array}{rcl}
&&\bar{\partial}_{+}\big[c\tau\sin\alpha\bar{\partial}_{-}\rho\big]+
\bar{\partial}_{-}\big[c\tau\sin\beta\bar{\partial}_{+}\rho\big]
\\[6pt]&=&-\displaystyle\frac{1}{\sin2 A}\big[(\bar{\partial}_{+}\beta-\cos2 A \bar{\partial}_{-}\alpha)c\tau\sin\beta \bar{\partial}_{+}\rho-( \bar{\partial}_{-}\alpha- \cos2 A\bar{\partial}_{+}\beta)c\tau\sin\alpha \bar{\partial}_{-}\rho\big].
\end{array}
\end{equation}
Hence, we have
\begin{equation}
\begin{array}{rcl}
&&\displaystyle(\sin\alpha+\sin\beta)\frac{1}{c\tau }\frac{d(c\tau)}{d\rho }\bar{\partial}_{+}\rho\bar{\partial}_{-}\rho+\sin\alpha\bar{\partial}_{+}\bar{\partial}_{-}\rho+
\sin\beta\bar{\partial}_{-}\bar{\partial}_{+}\rho\\[8pt]&=&-\displaystyle\frac{1}{\sin2 A}
\Big[(\sin\beta \bar{\partial}_{+}\beta-\sin\beta \cos2 A \bar{\partial}_{-}\alpha+\cos\beta\sin2 A\bar{\partial}_{-}\beta) \bar{\partial}_{+}\rho\\[4pt]&&\qquad\qquad\qquad\qquad\displaystyle-( \sin\alpha\bar{\partial}_{-}\alpha- \sin\alpha\cos2 A\bar{\partial}_{+}\beta-\cos\alpha\sin2 A\bar{\partial}_{+}\alpha) \bar{\partial}_{-}\rho\Big].
\end{array}\label{12}
\end{equation}

Applying the commutator relation (\ref{comm}) for $\rho$, we get
\begin{equation}
\bar{\partial}_{-} \bar{\partial}_{+}\rho- \bar{\partial}_{+} \bar{\partial}_{-}\rho=
\frac{1}{\sin2 A}\big[(\cos2 A \bar{\partial}_{+}\beta- \bar{\partial}_{-}\alpha) \bar{\partial}_{-}\rho-
(\bar{\partial}_{+}\beta-\cos2 A \bar{\partial}_{-}\alpha) \bar{\partial}_{+}\rho\big].\label{13}
\end{equation}
Inserting this into (\ref{12}), we get
\begin{equation}
\begin{array}{rcl}
&&\displaystyle(\sin\alpha+\sin\beta)\frac{1}{c\tau }\frac{d(c\tau)}{d\rho }\bar{\partial}_{+}\rho\bar{\partial}_{-}\rho+(\sin\alpha+\sin\beta)\bar{\partial}_{+}\bar{\partial}_{-}\rho\\[6pt]
&=&-\displaystyle\frac{1}{\sin2 A}
\Big[\cos\beta\sin2 A\bar{\partial}_{-}\beta\bar{\partial}_{+}\rho+( \sin\alpha+\sin\beta)\cos2 A\bar{\partial}_{+}\beta\bar{\partial}_{-}\rho\\[6pt]&&
\qquad\qquad\qquad\qquad-( \sin\alpha+\sin\beta)\bar{\partial}_{-}\alpha\bar{\partial}_{-}\rho+\cos\alpha\sin2 A\bar{\partial}_{+}
\alpha\bar{\partial}_{-}\rho\Big].
\end{array}\label{14}
\end{equation}
Thus, by (\ref{7})--(\ref{8}) we get
(\ref{81104}).
This completes the proof.
\end{proof}

From (\ref{81104}) we also have
\begin{equation}\label{53001}
\left\{
  \begin{array}{ll}
   \displaystyle \bar{\partial}_{+}\Big(\frac{\bar{\partial}_{-}\rho}{\rho}\Big)=
    \frac{\tau^3p''(\tau)}{4c^2\cos^{2}A }\bigg(\big(\frac{\bar{\partial}_{-}\rho}{\rho}\big)^2+(\mathcal{H}-1)\frac{\bar{\partial}_{-}\rho}{\rho}\frac{
    \bar{\partial}_{+}\rho}{\rho}\bigg),\\[12pt]
  \displaystyle \bar{\partial}_{-}\Big(\frac{\bar{\partial}_{+}\rho}{\rho}\Big)=
    \frac{\tau^3p''(\tau)}{4c^2\cos^{2}A }\bigg(\big(\frac{\bar{\partial}_{+}\rho}{\rho}\big)^2+(\mathcal{H}-1)\frac{\bar{\partial}_{-}\rho}{\rho}\frac{
    \bar{\partial}_{+}\rho}{\rho}\bigg),
  \end{array}
\right.
\end{equation}
where
$$
\mathcal{H}=2\sin^2A-\frac{4p'(\tau)\cos^2A\cos 2A}{\tau p''(\tau)}.
$$

The characteristic decomposition method was introduced by Li, Zhang and Zheng \cite{Li2} in investigating simple waves of the 2D compressible Euler equations.
One can also use the characteristic decompositions (\ref{81104}) and the characteristic equations (\ref{11})--(\ref{72804}) to prove the following classical theorem about simple waves of 2D steady Euler equations (see \cite{CF}):
\begin{thm}\label{thm0}
The flow in a region adjacent to a straight characteristic line is a simple wave or a constant state.
\end{thm}
Li, Zhang and Zheng \cite{Li2} applied the characteristic decomposition method to generalize this result to the pseudo-steady case.

\subsection{Hodograph transformation}
There are sonic shocks in certain oblique wave interaction regions, and the sonic shocks are envelopes of one out of the two families of wave characteristic curves. This results in a fact that the flow on the sonic side of a sonic shock is not $C^1$ smooth up to the shock boundary. In order to overcome the difficulty cased by the singularity on sonic shock boundaries, the author \cite{Lai3,Lai4} proposed a hodograph transformation.
In this part we shall introduce the hodograph transformation based on Riemann invariants.

\subsubsection{Riemann invariants}
The Riemann invariants of the system (\ref{matrix1}) are
$$
r=\sigma+\int^{q}_{u_0}\frac{\sqrt{q^{2}-c^{2}}}{qc}dq\quad \mbox{and}\quad
s=\sigma-\int^{q}_{u_0}\frac{\sqrt{q^{2}-c^{2}}}{qc}dq.
$$
In view of the Riemann invariants, we have
\begin{equation}\label{52801}
\left\{
  \begin{array}{ll}
    \bar{\partial}_{+}s=0,  \\[2pt]
     \bar{\partial}_{-}r=0.
  \end{array}
\right.
\end{equation}

It follows from
$$
\left(
  \begin{array}{cc}
    \displaystyle\frac{\partial r}{\partial \sigma} & \displaystyle\frac{\partial r}{\partial q}   \\[8pt]
    \displaystyle\frac{\partial s}{\partial \sigma}  & \displaystyle\frac{\partial s}{\partial q}
  \end{array}
\right)\left(
         \begin{array}{cc}
          \displaystyle \frac{\partial \sigma}{\partial r}  & \displaystyle\frac{\partial \sigma}{\partial s} \\[8pt]
           \displaystyle\frac{\partial q}{\partial r} & \displaystyle \frac{\partial q}{\partial s}
         \end{array}
       \right)=\left(
                 \begin{array}{cc}
                   1 & 0 \\[12pt]
                   0 & 1 \\
                 \end{array}
               \right)
$$
that
\begin{equation}
\frac{\partial\sigma}{\partial r}=\frac{\partial\sigma}{\partial s}=\frac{1}{2},\quad \frac{\partial q}{\partial
r}=\frac{q\sin A}{2\cos A}, \quad \mbox{and} \quad
\frac{\partial q}{\partial
s}=-\frac{q\sin A}{2\cos A}.\label{1201}
\end{equation}
Thus, by (\ref{U}) we have
\begin{equation}\label{102802a}
\frac{\partial u}{\partial r}=-\frac{q\sin\beta}{2\cos A}, \quad \frac{\partial u}{\partial s}=-\frac{q\sin\alpha}{2\cos A}, \quad \frac{\partial v}{\partial r}=\frac{q\cos\beta}{2\cos A}, \quad \mbox{and} \quad
\frac{\partial v}{\partial s}=\frac{q\cos\alpha}{2\cos A}.
\end{equation}
Combining these with the Bernoulli law (\ref{E2}), we have
\begin{equation}\label{102302d}
\frac{\partial \rho}{\partial s}=\frac{\rho }{\sin (2A)}\quad \mbox{and}\quad \frac{\partial \rho}{\partial r}=-\frac{\rho }{\sin (2A)}.
\end{equation}
From (\ref{102802a}) we also have
\begin{equation}\label{102803a}
\frac{\partial r}{\partial u}=\frac{\cos\alpha}{c}, \quad \frac{\partial r}{\partial v}=\frac{\sin\alpha}{c}, \quad \frac{\partial s}{\partial u}=-\frac{\cos\beta}{c}, \quad \mbox{and} \quad
\frac{\partial s}{\partial v}=-\frac{\sin\beta}{c}.
\end{equation}

From $\sin A=\frac{c}{q}$, we obtain
\begin{equation}\label{wq}
\cos A\frac{\partial A}{\partial q}=-\frac{c}{q^{2}}+\frac{1}{q}\frac{\partial c}{\partial q}.
\end{equation}
By the Bernoulli law (\ref{E2}), we have
\begin{equation}
\frac{{\rm d} c}{{\rm d} q}=-\frac{q}{c\kappa(\tau)}.\label{dc}
\end{equation}
Inserting (\ref{dc}) into (\ref{wq}), we get
\begin{equation}\label{102301}
\frac{{\rm d}  A}{{\rm d} q}=-\frac{q}{\sqrt{q^{2}-c^{2}}}\left(\frac{c}{q^{2}}+\frac{1}{c\kappa}\right)
=\frac{\sqrt{q^2-c^2}}{qc}+\frac{q}{\sqrt{q^2-c^2}}\cdot\frac{\tau p''(\tau)}{2cp'(\tau)}.
\end{equation}

By (\ref{1201}) and (\ref{102301}) we have
\begin{equation}\label{omegar}
\frac{\partial A}{\partial r}=-\frac{q^{2}c}{2(q^{2}-c^{2})}\left(\frac{c}{q^{2}}+\frac{1}{c\kappa}\right)\quad\mbox{and}\quad
\frac{\partial A}{\partial s}=\frac{q^{2}c}{2(q^{2}-c^{2})}\left(\frac{c}{q^{2}}+\frac{1}{c\kappa}\right)
\end{equation}
Hence,
\begin{equation}\label{alphar}
\frac{\partial\alpha}{\partial s}=\frac{\partial\beta}{\partial r}=\frac{1}{2}+\frac{q^{2}c}{2(q^{2}-c^{2})}\left(\frac{c}{q^{2}}+\frac{1}{c\kappa}\right)
=-\frac{q^2\tau p''(\tau)}{4(q^2-c^2)p'(\tau)}
 \end{equation}
and
 \begin{equation}\label{alphar1}
\frac{\partial\alpha}{\partial r}=\frac{\partial\beta}{\partial s}=\frac{1}{2}-\frac{q^{2}c}{2(q^{2}-c^{2})}\left(\frac{c}{q^{2}}+\frac{1}{c\kappa}\right)=1+\frac{q^2\tau p''(\tau)}{4(q^2-c^2)p'(\tau)}.
 \end{equation}

Applying the commutator relation (\ref{comm}) for the Riemann invariants,
we also have (see \cite{Lai6})
\begin{equation}
\left\{
 \begin{array}{ll}
 \displaystyle \bar{\partial}_{+}\bar{\partial}_{-}s=-\frac{\tau p''(\tau)}{4p'(\tau)\sin(2A)\cos^2A}\big(\bar{\partial}_{-}s
  -\cos(2A)\bar{\partial}_{+}r\big)\bar{\partial}_{-}s,  \\[12pt]
  \displaystyle \bar{\partial}_{-}\bar{\partial}_{+}r=\frac{\tau p''(\tau)}{4p'(\tau)\sin(2A)\cos^2A}\big(\bar{\partial}_{+}r
   -\cos(2A)\bar{\partial}_{-}s
\big)\bar{\partial}_{+}r.
  \end{array}
\right.\label{cdr}
\end{equation}

\subsubsection{Systems in $(r, s)$-plane}
Let the hodograph transformation be
$$
T:  (x, y)\rightarrow (r, s)
$$
for (\ref{52801}).
If the Jacobian
$$
j(r, s; x, y)=\frac{\partial(r, s)}{\partial(x, y)}=\frac{\partial r}{\partial{x}}\frac{\partial s}{\partial{y}}-\frac{\partial s}{\partial{x}}\frac{\partial r}{\partial{y}}\neq 0
$$
for a solution $(r, s)(x, y)$ of (\ref{52801}), we may consider $x$ and $y$ as functions of $r$ and $s$. From
\begin{equation}
\frac{\partial r}{\partial{x}}=j\frac{\partial y}{\partial s}, \quad \frac{\partial r}{\partial{y}}=-j\frac{\partial x}{\partial s}, \quad \frac{\partial s}{\partial{x}}=-j\frac{\partial y}{\partial r}, \quad \frac{\partial s}{\partial{y}}=j\frac{\partial x}{\partial r},
\end{equation}
we then see that $x(r, s)$ and $y(r, s)$ satisfy the equations
\begin{equation}\label{HT}
\left\{
  \begin{array}{ll}
\hat{\partial}_{+}y=\lambda_{+}\hat{\partial}_{+}x, \\[2pt]
 \hat{\partial}_{-}y=\lambda_{-}\hat{\partial}_{-}x,
  \end{array}
\right.
\end{equation}
where
\begin{equation}\label{122501a}
\hat{\partial}_{+}:=\frac{\partial}{\partial r}\quad \mbox{and}\quad \hat{\partial}_{-}:=-\frac{\partial}{\partial s}.
\end{equation}

From (\ref{HT}) we have the following characteristic decompositions on the $(r, s)$-plane:
\begin{equation}\label{6506}
\left\{
  \begin{array}{ll}
    (\lambda_{+}-\lambda_{-})\hat{\partial}_{-}\hat{\partial}_{+}x=\hat{\partial}_{+}\lambda_{-}\hat{\partial}_{-}x-
\hat{\partial}_{-}\lambda_{+}\hat{\partial}_{+}x,  \\[4pt]
     (\lambda_{+}-\lambda_{-})\hat{\partial}_{+}\hat{\partial}_{-}x=\hat{\partial}_{+}\lambda_{-}\hat{\partial}_{-}x-
\hat{\partial}_{-}\lambda_{+}\hat{\partial}_{+}x.
  \end{array}
\right.
\end{equation}
By (\ref{alphar}) we have
\begin{equation}\label{121902a}
\hat{\partial}_{+}\lambda_{-}=-\frac{q^2\tau p''(\tau)\sec^2\beta}{4(q^2-c^2)p'(\tau)}\quad \mbox{and}\quad
\hat{\partial}_{-}\lambda_{+}=\frac{q^2\tau p''(\tau)\sec^2\alpha}{4(q^2-c^2)p'(\tau)}.
\end{equation}

We compute
$$
\hat{\partial}_{\pm}\rho=\rho_x\hat{\partial}_{\pm}x+\rho_y\hat{\partial}_{\pm}y=
\frac{\bar{\partial}_{\pm}\rho\hat{\partial}_{\pm}x}{\cos(\sigma\pm A)}.
$$
Then by (\ref{102302d}) we obtain
\begin{equation}\label{53002}
\bar{\partial}_{\pm}\rho=\frac{\cos(\sigma\pm A)\hat{\partial}_{\pm}\rho}{\hat{\partial}_{\pm}x}=\mp\frac{\rho\cos(\sigma\pm A)}{\sin (2A)\hat{\partial}_{\pm}x}.
\end{equation}

\subsection{Hyperbolicity and invariant region}
Regarding the problem (\ref{E1}) and (\ref{42701}), the Bernoulli constant of the potential flow is $$\mathcal{B}=\frac{u_0^2}{2}+h(\tau_0).$$ Moreover, by the assumptions for $p=p(\tau)$ we see that when $q>u_0$ the variables $\tau$, $A$, and $c$ can be seen as functions of $q$. We denote these functions by
\begin{equation}\label{111901}
\tau=\hat{\tau}(q), \quad A=\hat{A}(q),\quad c=\hat{c}(q), \quad  q>u_0.
\end{equation}
We can also see $A$, $c$, $q$ as functions of $\tau$, denoted by
$$
q=\check{q}(\tau):=\sqrt{2\mathcal{B}-h(\tau)},\quad c=\check{c}(\tau):=\sqrt{-\tau^2p'(\tau)}, \quad A=\check{A}(\tau):=\arcsin\left(\frac{\check{c}(\tau)}{\check{q}(\tau)}\right), \quad  \tau>\tau_0.
$$

If $\tau_0>\tau_2^i$, then by $u_0>c_0$ we have $q>c$ for all $\tau>\tau_0$. However, if $\tau_0<\tau_2^i$ we can not deduce that $q>c$ for all $\tau>\tau_0$.
\begin{lem}
Assume that the equation of state $p=p(\tau)$ satisfies assumptions (A1)--(A4).
  Then when $u_0>\max\{c_0, \tau_2^i\sqrt{-p'(\tau_2^i)}\}$ there holds
$q>c$ for $\tau>\tau_0$.
\end{lem}
\begin{proof}
We first discuss the case of $\tau_0<\tau_1^i$.
From (\ref{dc}) we have
\begin{equation}\label{112101a}
\frac{{\rm d}(q-c)}{{\rm d}q}=\frac{c-q}{c}-\frac{q\tau p''(\tau)}{2c p'(\tau)}.
\end{equation}
Combining this with $q_0>c_0$, we have
 $q>c$ for $\tau_0<\tau<\tau_1^i$.

From (\ref{dc}) and $\frac{{\rm d}q}{{\rm d}\tau}>0$ we have $c\leq \tau_2^i\sqrt{-p'(\tau_2^i)}$ and $q>u_0$ for $\tau_1^i\leq\tau\leq\tau_2^i$. Then by $u_0>\max\{c_0, \tau_2^i\sqrt{-p'(\tau_2^i)}\}$ we have $q>c$ for $\tau_1^i\leq\tau\leq\tau_2^i$.
Combining (\ref{112101a}) and $(q-c)(\tau_2^i)>0$ we have $q>c$ for $\tau>\tau_2^i$.

The proof for the case of $\tau_0\geq\tau_1^i$ is similar. We omit the details. This completes the proof.
\end{proof}

\begin{figure}[htbp]
\begin{center}
\includegraphics[scale=0.55]{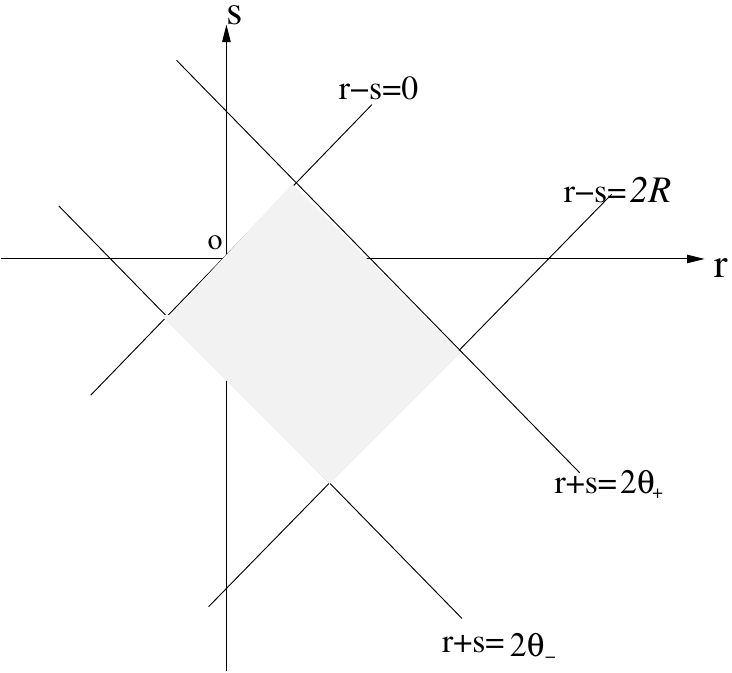}
\caption{\footnotesize Invariant region in the $(r, s)$-plane.}
\label{Fig2.2}
\end{center}
\end{figure}

We define the following constants that depend only on the oncoming flow state:
\begin{equation}\label{3506}
q_{\infty}=\sqrt{2\mathcal{B}-h_{\infty}}, \quad h_{\infty}=\lim\limits_{\tau\rightarrow +\infty}h(\tau), \quad \mbox{and}\quad \mathcal{R}=\int_{u_0}^{q_{\infty}}\frac{\sqrt{q^2-\hat{c}^2(q)}}{q\hat{c}(q)}{\rm d}q.
\end{equation}

In this paper,
we also make the following assumptions:
\begin{description}
  \item[(H1)] $q>c$, or equivalently, $0<A<\frac{\pi}{2}$ for $\tau\geq \tau_0$;
  \item[(H2)] $-\frac{\pi}{2}<\beta(r, s)<\alpha(r, s)<\frac{\pi}{2}$ for $(r, s)\in \Pi$,
where
\begin{equation}
\Pi:=\left\{(r, s)~\big|~ \theta_{-}\leq \frac{r+s}{2}\leq \theta_{+},~ 0\leq \frac{r-s}{2} <\mathcal{R} \right\};
\end{equation}
see Fig. \ref{Fig2.2}.
\end{description}
Obviously, when $u_0$ is sufficiently large assumptions {(H1)} and {(H2)} hold.
We will see that the region $\Pi$ is an invariant region of the solution of the IBVP  (\ref{E1}), (\ref{42701}).


\subsection{Notation}
We give some notations of this paper.

We will use various curves to describe the flow structure within the divergent duct. In order to distinguish the different meanings of these curves, we shall give some notations; see Fig. \ref{Fig2.3}.
Let $E$ and $F$ be two points on the $(x,y)$-plane.
\begin{itemize}
  \item $C_{+}^{E}$ ($C_{-}^{E}$, resp.): a forward $C_{+}$ ($C_{-}$, resp.) characteristic line issued from point $E$;
  \item $\widetilde{EF}$: a $C_{+}$ (or $C_{-}$) characteristic curve segment from points $E$ to $F$ along the characteristic direction;
  \item $\overline{EF}$: a straight $C_{+}$ (or $C_{-}$) characteristic line segment from points $E$ to $F$;
\item $\wideparen{EF}$: a shock curve connecting points $E$ and $F$.
\end{itemize}
We need to mention that $C_{+}^{E}$ (or $C_{-}^{E}$) could be a finite curve or a semi-infinite curve.
\vskip 4pt

We use $\Omega_i$ ($\overline{\Omega}_{i}$), $i=1, 2, 3, \cdot\cdot\cdot$ to denote open (closed, resp.) domains in the $(x, y)$-plane.

\begin{figure}[htbp]
\begin{center}
\includegraphics[scale=0.4]{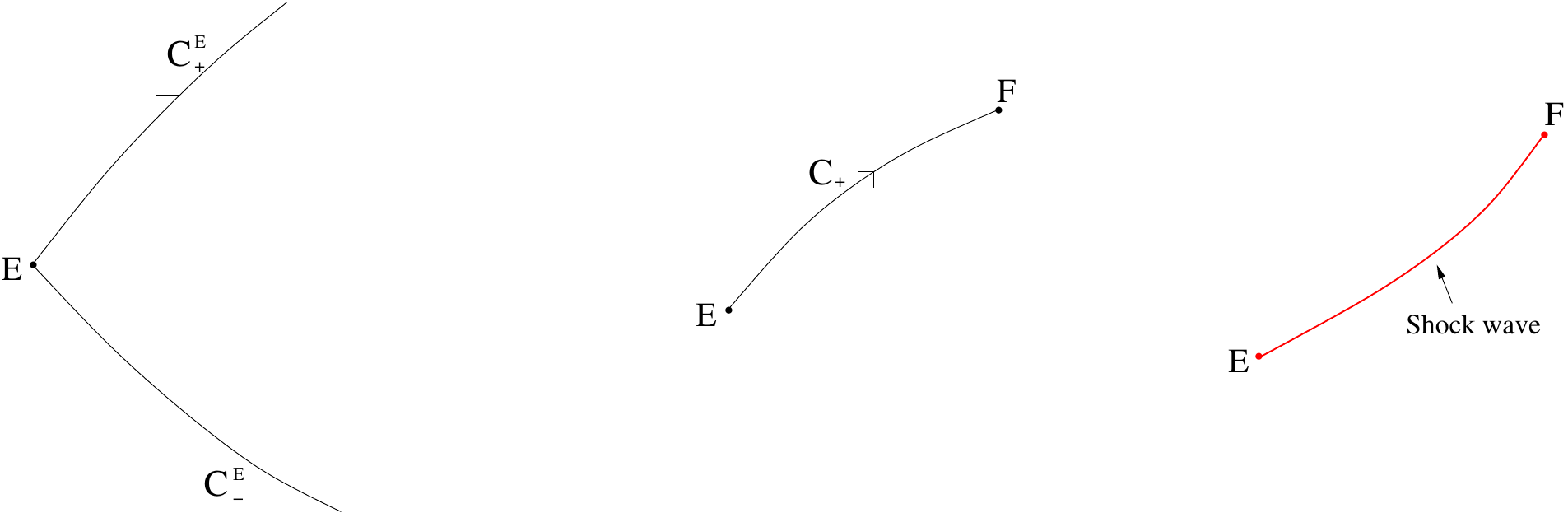}
\caption{\footnotesize Characteristic curves $C_{\pm}^{E}$ and $\widetilde{EF}$, and shock wave $\wideparen{EF}$.}
\label{Fig2.3}
\end{center}
\end{figure}

\vskip 4pt

Let ${\it S}$ be an oblique shock wave in the $(x, y)$-plane. We give the following two notations:
\begin{itemize}
  \item $\bar{\partial}_{_\Gamma}$: the derivation along ${\it S}$;
  \item $\phi$: the inclination angle of ${\it S}$.
\end{itemize}
Assume that ${\it S}$ can be represented by $y=y_s(x)$, $x\in (x', x'')$. Then,
\begin{equation}\label{20101}
\phi:=\arctan \big(y_s'(x)\big), \quad \bar{\partial}_{_\Gamma}u_{_{f/b}}:=\cos\phi\frac{{\rm d}u_{_{f/b}}(x, y_s(x))}{{\rm d}x},\quad \bar{\partial}_{_\Gamma}v_{_{f/b}}:=\cos\phi\frac{{\rm d}v_{_{f/b}}(x, y_s(x))}{{\rm d}x}
\end{equation}
along ${\it S}$,
where $(u_{_{f/b}}, v_{_{f/b}})$ denote the flow upstream/downstream of the shock ${\it S}$, respectively.
If the flow upstream/downstream of ${\it S}$ is $C^1$ smooth up to ${\it S}$, then we also have
\begin{equation}\label{110401}
\bar{\partial}_{_\Gamma}(u_{_{f/b}}, v_{_{f/b}})=\cos\phi\partial_{x}(u_{_{f/b}}, v_{_{f/b}})+\sin\phi\partial_{y}(u_{_{f/b}}, v_{_{f/b}})\quad\mbox{along}~{\it S}.
\end{equation}
However, if the upstream/downstream  flow is not $C^1$ smooth up to ${\it S}$ then (\ref{110401}) is not hold along ${\it S}$.

\section{Oblique waves around the sharp corner}

\subsection{Oblique shocks}
According to assumptions (A1)--(A4),
there exist $\tau_{1}^a$ and $\tau_{2}^a$ where $\tau_1^a<\tau_{1}^{i}<\tau_{2}^{i}<\tau_{2}^a$ such that
\begin{equation}\label{5901}
p'(\tau_{1}^a)=p'(\tau_{2}^{i})\quad \mbox{and}\quad p'(\tau_{1}^{i})=p'(\tau_{2}^a);
 \end{equation}
see Fig. \ref{Fig3.1}.
\begin{figure}[htbp]
\begin{center}
\includegraphics[scale=0.5]{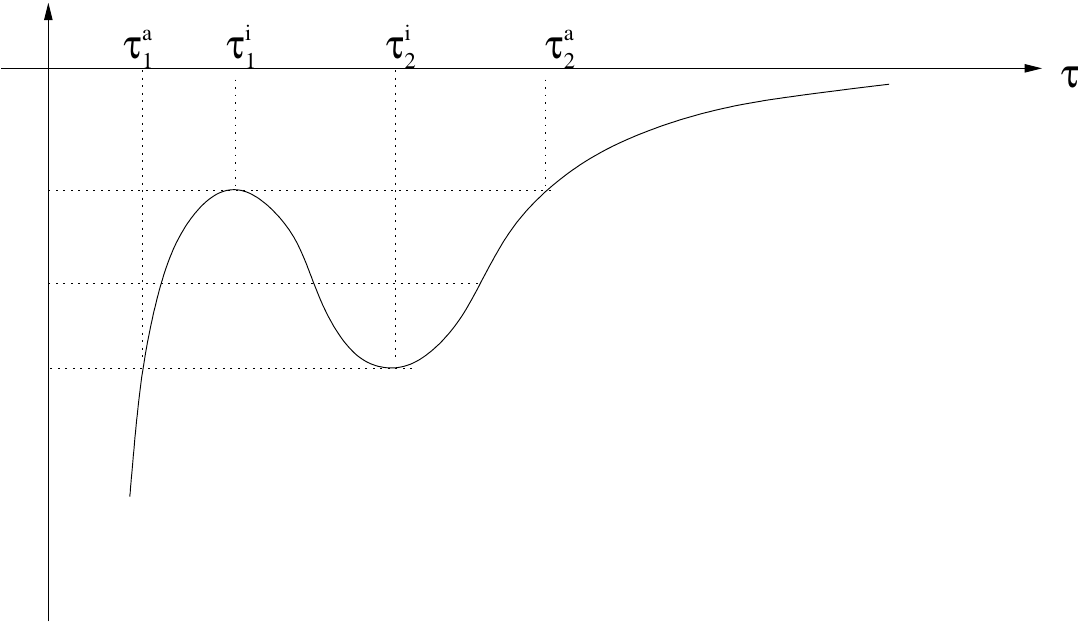}
\caption{\footnotesize The graph of $p'(\tau)$.}
\label{Fig3.1}
\end{center}
\end{figure}


We now discuss oblique shocks of (\ref{E1}).
Assume that there is an oblique shock front and the states on its both sides are constant.
Let $\phi$ be the inclination angle of this shock front.
We use subscripts `$f$' and `$b$' to denote the states on the front and back sides of the
shock; see Fig. \ref{Fig3.2}.
We denote by $N$ ($N>0$) and $L$
the components of $(u, v)$ normal and tangential to the shock direction, respectively.
Then
the Rankine-Hugoniot conditions of the oblique shock wave can be written in the following form:
\begin{equation}\label{RH}
\left\{
  \begin{array}{ll}
    \rho_f N_f=\rho_b N_b=m, \\[4pt]
    L_{f}=L_b,\\[4pt]
   N_f^2+2h(\tau_f)= N_b^2+2h(\tau_b).
  \end{array}
\right.
\end{equation}

Since the gas considered here is nonconvex, the shock is also required to satisfy the
Liu's extended entropy condition (Liu \cite{Liu1}):
\begin{equation}\label{43001}
-\frac{2h(\tau_f)-2h(\tau_b)}{\tau_f^2-\tau_b^2}\geq -\frac{2h(\tau_f)-2h(\tau)}{\tau_f^2-\tau^2}\quad \mbox{for~ all}~ \tau\in\big(\min\{\tau_f, \tau_b\}, \max\{\tau_f, \tau_b\}\big).
\end{equation}
 By (\ref{43001}) we can prove that along the shock,
\begin{equation}\label{20102}
N_f\geq c_f\quad \mbox{and}\quad N_b\leq c_b.
\end{equation}
We also assume that the flows upstream and downstream of the shock are supersonic.
Then by (\ref{20102}) we have that along the shock,
$$
\beta_b\leq \phi\leq \beta_f\quad \mbox{or}\quad \alpha_f\leq \phi\leq \alpha_b;
$$
see Fig. \ref{Fig3.2}.
 As in Lax \cite{Lax},  we call the oblique shock 1-shock (2-shock, resp.) if $\beta_b\leq \phi\leq \beta_f$
 ($\alpha_f\leq \phi\leq \alpha_b$, resp.); see Fig. \ref{Fig3.2}.
In this paper, we confine ourselves to rarefaction shocks, i.e., the shocks which satisfy $\tau_f<\tau_b$.
We also have
\begin{equation}\label{111306}
\left\{
  \begin{array}{ll}
    N=-u\sin\phi+v\cos\phi\quad \mbox{and}\quad L=u\cos\phi+v\sin\phi, & \hbox{1-shock;} \\[4pt]
    N=u\sin\phi-v\cos\phi\quad \mbox{and}\quad L=u\cos\phi+v\sin\phi, & \hbox{2-shock.}
  \end{array}
\right.
\end{equation}

\begin{figure}[htbp]
\begin{center}
\includegraphics[scale=0.6]{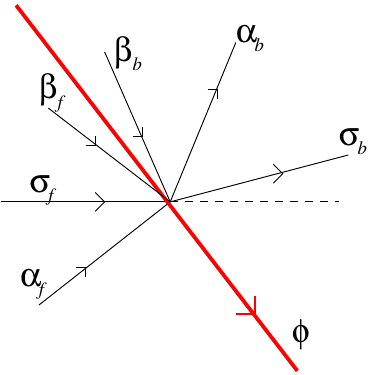}\qquad\qquad\qquad \qquad\qquad \includegraphics[scale=0.6]{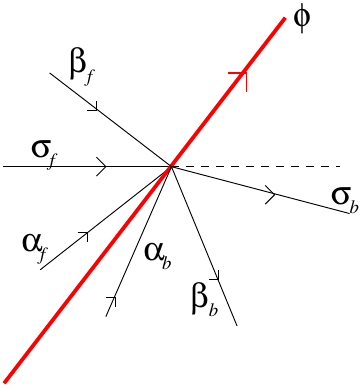}
\caption{\footnotesize Left: 1-shock; right: 2-shock.}
\label{Fig3.2}
\end{center}
\end{figure}

\begin{prop}\label{pro3.1.1}
For rarefaction 1-shocks,
there hold $A_f\leq \sigma_f-\phi<\sigma_b-\phi\leq A_b$;
for rarefaction 2-shocks,
there hold $A_f\leq\phi-\sigma_f<\phi-\sigma_b\leq A_b$.
\end{prop}
\begin{proof}
For rarefaction 1-shocks,
we have $L=q\cos(\sigma-\phi)$.
So, by the second relation of (\ref{RH}) we have
\begin{equation}\label{111311}
q_f\cos(\sigma_f-\phi)=q_b\cos(\sigma_b-\phi).
\end{equation}
By the definition of 1-shocks, we have
\begin{equation}\label{111312}
\sigma_b-\phi\leq \sigma_b-\beta_b=A_b<\frac{\pi}{2}\quad \mbox{and} \quad \sigma_{f}-\phi\geq \sigma_f-\beta_f=A_f.
\end{equation}
Since $\tau_f<\tau_b$, by the Bernoulli law (\ref{E2}) we have $q_b>q_f$. Combining this with (\ref{111311}) and (\ref{111312}) we have $A_f\leq \sigma_f-\phi<\sigma_b-\phi\leq A_b$.

The proof for the other is similar, we omit the details.
\end{proof}

In what follows, we are going to cite some propositions. The proofs of these propositions can be found in Lai \cite{Lai4}.
\begin{prop}\label{5602}
({\bf Double-sonic})
There exists a unique pair $\tau_1^e$ and $\tau_2^e$, where $1<\tau_1^e<\tau_2^e$, such that
\begin{equation}\label{42301}
p'(\tau_1^e)=p'(\tau_2^e)=\frac{2h(\tau_1^e)-2h(\tau_2^e)}{(\tau_1^e)^2-(\tau_2^e)^2},
\end{equation}
and
\begin{equation}\label{42301aa}
-\frac{2h(\tau_1^e)-2h(\tau_2^e)}
{(\tau_1^e)^2-(\tau_2^e)^2}>-\frac{2h(\tau_1^e)-2h(\tau)}{(\tau_1^e)^2-\tau^2}\quad \mbox{for all}~ \tau\in(\tau_1^e, \tau_2^e).
\end{equation}
Moreover, $\tau_1^e\in (\tau_1^a, \tau_1^i)$ and $\tau_2^e\in (\tau_2^i, \tau_2^a)$.
\end{prop}


\begin{prop}\label{51003}
({\bf Post-sonic}) For any $\tau_{f}\in(\tau_1^e, \tau_2^i)$, there exists one and only one $\psi_{po}(\tau_{f})\in(\tau,+\infty)$ such that
\begin{equation}\label{43003}
p'(\psi_{po}(\tau_f))=\frac{2h(\psi_{po}(\tau_{f}))-2h(\tau_{f})}{\psi_{po}^2(\tau_{f})-\tau_{f}^2}<p'(\tau_{f}),
\end{equation}
and
\begin{equation}\label{43004}
-\frac{2h(\psi_{po}(\tau_{f}))-2h(\tau_{f})}{\psi_{po}^2(\tau_{f})-\tau_{f}^2}>-\frac{2h(\tau)-2h(\tau_{f})}{\tau^2-\tau_{f}^2}\quad \mbox{for all}~\tau\in (\tau_{f}, \psi_{po}(\tau_{f})).
\end{equation}
Moreover, the function $\psi_{po}(\tau_{f})\in C^1[\tau_1^e, \tau_2^i)$ and satisfies
\begin{equation}\label{5906}
 \psi_{po}'(\tau_{f})<0\quad \mbox{and}\quad \tau_2^i<\psi_{po}(\tau_{f})<\tau_2^e\quad \mbox{for}~  \tau_{f}\in (\tau_1^e, \tau_2^i); 
\end{equation}
$$
\psi_{po}(\tau_{f})\rightarrow \tau_2^e\quad \mbox{as}~  \tau_{f}\rightarrow\tau_1^e; \qquad
\psi_{po}(\tau_{f})\rightarrow \tau_2^i\quad \mbox{as}~ \tau_{f}\rightarrow\tau_2^i.
$$
\end{prop}

\begin{prop}\label{pre}
({\bf Pre-sonic})
For any $\tau_{f}\in (\tau_1^e, \tau_1^i)$, there exists one and only one $\psi_{pr}(\tau_{f})\in (\tau_{f}, +\infty)$ such that
\begin{equation}\label{5701}
p'(\psi_{pr}(\tau_{f}))<\frac{2h(\psi_{pr}(\tau_{f}))-2h(\tau_{f})}{\psi_{pr}^2(\tau_{f})-\tau_{f}^2}=p'(\tau_{f})
\end{equation}
and
\begin{equation}\label{43007}
-\frac{2h(\psi_{pr}(\tau_{f}))-2h(\tau_{f})}{\psi_{pr}^2(\tau_{f})-\tau_{f}^2}>-
\frac{2h(\tau_{f})-2h(\tau)}{\tau_{f}^2-\tau^2}\quad \mbox{for all}~ \tau\in (\tau_{f}, \psi_{pr}(\tau_{f})).
\end{equation}
Moreover, the function $\psi_{pr}(\tau_f)$ satisfies
\begin{equation}
\psi_{pr}'(\tau_{f})<0\quad \mbox{and}\quad
\psi_{pr}(\tau_{f})\in (\tau_1^i, \tau_2^e)\quad \mbox{for} ~\tau_{f}\in (\tau_1^e, \tau_1^i).
\end{equation}
\end{prop}

\begin{prop}\label{tran}
({\bf Transonic}) Assume $\tau_f\in (\tau_1^e, \tau_2^i)$. Then,
\begin{equation}\label{43012}
p'(\tau_b)<\frac{2h(\tau_b)-2h(\tau_f)}{\tau_b^2-\tau_f^2}<p'(\tau_f)\quad \mbox{for}~ \tau_b\in (\tau_*, \psi_{po}(\tau_f)),
\end{equation}
where $\tau_*=\left\{
                \begin{array}{ll}
                  \psi_{pr}(\tau_f), & \hbox{$\tau_f\leq \tau_1^i$;} \\
                  \tau_f, & \hbox{$\tau_f>\tau_1^i$.}
                \end{array}
              \right.
$ Moreover,
\begin{equation}\label{43013}
-\frac{2h(\tau_b)-2h(\tau_f)}{\tau_b^2-\tau_f^2}>-\frac{2h(\tau)-2h(\tau_f)}{\tau^2-\tau_f^2}\quad \mbox{for all}~ \tau\in (\tau_f, \tau_b).
\end{equation}
\end{prop}

From the first and third relations of (\ref{RH}), we have
\begin{equation}\label{112203a}
m^2=-\frac{2h(\tau_f)-2h(\tau_b)}{\tau_f^2-\tau_b^2}.
\end{equation}
We have the following classifications about the oblique shocks for BZT fluids.
\begin{itemize}
  \item When $\tau_f=\tau_1^e$ and $\tau_b=\tau_2^e$, they are referred to as double-sonic shocks. For double-sonic shocks, one has
$$
N_f=c_f:=\sqrt{-\tau_f^2p'(\tau_f)} \quad \mbox{and}\quad N_b=c_b:=\sqrt{-\tau_b^2p'(\tau_b)}.
$$
  \item When $\tau_f\in (\tau_1^e, \tau_2^i)$ and $\tau_b=\psi_{po}(\tau_f)$, they are referred to as post-sonic shocks. For oblique post-sonic shocks, one has
$$
N_f>c_f\quad \mbox{and}\quad N_b=c_b.
$$
\item When $\tau_f\in (\tau_1^e, \tau_1^i)$ and $\tau_b=\psi_{pr}(\tau_f)$, they are referred to as pre-sonic shocks. For oblique pre-sonic shocks, one has
$$
N_f=c_f\quad \mbox{and}\quad N_b<c_b.
$$
\item For oblique transonic shocks \footnote{In this paper, the term ``transonic" refers to the flow velocity relative to the shock front.},  one has
$$
N_f>c_f\quad \mbox{and}\quad N_b<c_b.
$$
\end{itemize}

\subsection{Centered simple waves}
In this part we study centered simple waves with $B$ (or $D$) as the center point.
For convenience, we define the self-similar variables
\begin{equation}\label{112204a}
\xi=\arctan\left(\frac{y+1}{x}\right)\quad \mbox{and}\quad \eta=\arctan\left(\frac{y-1}{x}\right)
\end{equation}
 at $D$ and $B$, respectively.

\begin{lem}\label{lem3.5}
Assume $(\bar{q}(\xi), \bar{\tau}(\xi), \bar{\sigma}(\xi))$ satisfies the equations
\begin{equation}\label{102302}
\left\{
  \begin{array}{ll}
    \bar{q}'(\xi)\cos(\bar{A}(\xi))+\bar{q}(\xi)\bar{\sigma}'(\xi)\sin(\bar{A}(\xi))=0, \\[4pt]
    \bar{q}(\xi)\bar{q}'(\xi)+\bar{\tau}(\xi)p'(\bar{\tau}(\xi))\bar{\tau}'(\xi)=0,  \\[4pt]
    \bar{\sigma}(\xi)+\bar{A}(\xi)=\xi
  \end{array}
\right.
\end{equation}
 for $\xi\in (\xi', \xi'')$,
where $\bar{q}(\xi)>\bar{c}(\xi)>0$,
\begin{equation}\label{102303}
\bar{A}(\xi)=\arcsin\left(\frac{\bar{c}(\xi)}{\bar{q}(\xi)}\right)\quad \mbox{and}\quad \bar{c}(\xi)=\bar{\tau}(\xi)\sqrt{-p'(\bar{\tau}(\xi))}.
\end{equation}
Let
\begin{equation}\label{102306}
\bar{u}(\xi)=\bar{q}(\xi)\cos(\bar{\sigma}(\xi))\quad \mbox{and}\quad \bar{v}(\xi)=\bar{q}(\xi)\sin(\bar{\sigma}(\xi)).
\end{equation}
Then the function
\begin{equation}
(u, v)=(\bar{u}, \bar{v})(\xi), \quad   \xi'<\xi<\xi''
\end{equation}
is a centered simple wave solution with straight $C_{+}$ characteristic lines of  (\ref{E1}) on the fan-shaped domain
\begin{equation}\label{102501}
\Xi_{+}=\big\{(x, y)~\big|~(x, y)=(r\cos\xi, r\sin\xi-1),~ r>0,
~\xi\in(\xi', \xi'')
\big\}.
\end{equation}
\end{lem}
\begin{proof}
Firstly, by the second equation (\ref{102302}) we know that the Bernoulli law (\ref{E2}) holds.
From the third equation of (\ref{102302}) we have
\begin{equation}\label{82104}
\big(\bar{u}(\xi), \bar{v}(\xi)\big)\cdot(\sin\xi, -\cos\xi)=\bar{c}(\xi).
\end{equation}
So, for any fixed $\xi\in (\xi', \xi'')$, the ray $x=r\cos\xi$, $y=r\sin\xi-1$ ($r>0$)
is tangent to the sonic circle $(x-\bar{u}(\xi))^2+(y-\bar{v}(\xi))^2=\bar{c}^2(\xi)$.
Therefore, by the result of Section 2 we see that this ray
 is
a straight $C_{+}$ characteristic line. Hence, we have $$\alpha=\xi, \quad \bar{\partial}_{+}u=0,\quad \bar{\partial}_{+}v=0,\quad\mbox{and} \quad\bar{\partial}_{+}\tau=0\quad \mbox{in}~\Xi_{+}.$$ Consequently,
\begin{equation}\label{82101}
\bar{\partial}_{+}u+\lambda_{-}\bar{\partial}_{+}v=0\quad\mbox{in}~ \Xi_{+}.
\end{equation}

By the first equation of (\ref{102302}) and $\alpha=\xi$, we have
\begin{equation}\label{82102}
\begin{aligned}
\bar{\partial}_{-}u+\lambda_{+}\bar{\partial}_{-}v
&=\big( \bar{u}'(\xi)+\tan\xi  \bar{v}'(\xi)\big)\bar{\partial}_{-}\xi\\&=\big(\bar{q}'(\xi)\cos(\bar{A}(\xi))
+\bar{q}(\xi)\bar{\sigma}'(\xi)\sin(\bar{A}(\xi))\big)\frac{\bar{\partial}_{-}\xi}{\cos\xi}=0 \quad\mbox{in}~ \Xi_{+}.
\end{aligned}
\end{equation}
This completes the proof of this lemma.
\end{proof}

The system (\ref{102302}) can by (\ref{102303}) be written as
\begin{equation}\label{102304}
\left\{
  \begin{array}{ll}
    \displaystyle\bar{q}'=\frac{2\bar{c}p'(\bar{\tau})\cos \bar{A}}{\bar{\tau}p''(\bar{\tau})},\\[10pt]
     \displaystyle\bar{\tau}'=-\frac{2\bar{q}\bar{c}\cos \bar{A}}{\bar{\tau}^2p''(\bar{\tau})},\\[10pt]
 \displaystyle\bar{\sigma}'=-\frac{2 p'(\bar{\tau})\cos^2 \bar{A}}{\bar{\tau}p''(\bar{\tau})}.
  \end{array}
\right.
\end{equation}
From the second equation of (\ref{102304}) we also have
\begin{equation}\label{102304d}
\bar{c}'=\frac{(2p'(\bar{\tau})+\bar{\tau}p''(\bar{\tau}))\bar{q}\cos \bar{A}}{\bar{\tau}p''(\bar{\tau})}.
\end{equation}

\begin{lem}
Assume $(\tilde{q}(\eta), \tilde{\tau}(\eta), \tilde{\sigma}(\eta))$ satisfies the equations
\begin{equation}\label{102302a}
\left\{
  \begin{array}{ll}
    \tilde{q}'(\eta)\cos(\tilde{A}(\eta))-\tilde{q}(\eta)\tilde{\sigma}'(\eta)\sin(\tilde{A}(\eta))=0, \\[4pt]
    \tilde{q}(\eta)\tilde{q}'(\eta)+\tilde{\tau}(\eta)p'(\tilde{\tau}(\eta))\tilde{\tau}'(\eta)=0,  \\[4pt]
    \tilde{\sigma}(\eta)-\tilde{A}(\eta)=\eta
  \end{array}
\right.
\end{equation}
 for $\eta\in (\eta', \eta'')$,
where $\tilde{q}(\eta)>\tilde{c}(\eta)>0$,
\begin{equation}\label{102303a}
\tilde{A}(\eta)=\arcsin\Big(\frac{\tilde{c}(\eta)}{\tilde{q}(\eta)}\Big)\quad \mbox{and}\quad \tilde{c}(\eta)=\tilde{\tau}(\eta)\sqrt{-p'(\tilde{\tau}(\eta))}.
\end{equation}
Let
\begin{equation}\label{102306a}
\tilde{u}(\eta)=\tilde{q}(\eta)\cos(\tilde{\sigma}(\eta))\quad \mbox{and}\quad \tilde{v}(\eta)=\tilde{q}(\eta)\sin(\tilde{\sigma}(\eta)).
\end{equation}
Then the function
\begin{equation}
(u, v)=(\tilde{u}, \tilde{v})(\eta), \quad   \eta'<\eta<\eta''
\end{equation}
is a centered simple wave solution with straight $C_{-}$ characteristic lines of  (\ref{E1}) on the fan-shaped domain
\begin{equation}\label{102502}
\Xi_{-}=\big\{(x, y)~\big|~(x, y)=(r\cos\eta, r\sin\eta+1),~ r>0,
~\eta\in(\eta', \eta'')
\big\}.
\end{equation}
\end{lem}
\begin{proof}
The proof is similar to that of Lemma \ref{lem3.5}; we omit the details.
\end{proof}

The system (\ref{102302a}) can by (\ref{102303a}) be written as
\begin{equation}\label{102304a}
\left\{
  \begin{array}{ll}
    \displaystyle\tilde{q}'=-\frac{2\tilde{c}p'(\tilde{\tau})\cos \tilde{A}}{\tilde{\tau}p''(\tilde{\tau})},\\[10pt]
     \displaystyle\tilde{\tau}'=\frac{2\tilde{q}\tilde{c}\cos \tilde{A}}{\tilde{\tau}^2p''(\tilde{\tau})},\\[10pt]
 \displaystyle\tilde{\sigma}'=-\frac{2 p'(\tilde{\tau})\cos^2 \tilde{A}}{\tilde{\tau}p''(\tilde{\tau})}.
  \end{array}
\right.
\end{equation}
From the second equation of (\ref{102304a}) we also have
\begin{equation}
\tilde{c}'=-\frac{(2p'(\tilde{\tau})+\tilde{\tau}p''(\tilde{\tau}))\tilde{q}\cos \tilde{A}}{\tilde{\tau}p''(\tilde{\tau})}.
\end{equation}

We call the simple waves with straight $C_{+}$ ($C_{-}$, resp.) characteristic lines the 2-fan (1-fan, resp.) waves.

\subsection{Oblique waves around the corners} 
We now consider the RIBVP (\ref{E1}), (\ref{112201a}).
We shall show that for any fixed $\tau_0$, when $u_0$ is sufficiently large this problem  admits a unique  self-similar solution for any $\theta_{-}\in (-\frac{\pi}{2}, 0)$.
We divide the discussion into the following four cases: (1) $\tau_0>\tau_2^i$; (2) $\tau_1^{i}<\tau_0\leq\tau_2^{i}$; (3) $\tau_1^{e}<\tau_0\leq\tau_1^{i}$; (4) $\tau_0<\tau_1^{e}$.

\subsubsection{$\tau_0>\tau_2^i$}
We consider (\ref{102304}) with data
\begin{equation}\label{102305}
(\bar{q}, \bar{\tau}, \bar{\sigma})(\xi_0)=(u_0, \tau_0, 0),
\end{equation}
where $\xi_0=\arcsin(\frac{c_0}{u_0})$.
In view of assumption (H1), the initial value problem (\ref{102304}), (\ref{102305}) admits a supersonic solution for $\xi\in (-\mathcal{R}, \xi_0)$. Moreover, the solution satisfies
\begin{equation}\label{102503}
\begin{aligned}
&\quad \bar{q}'(\xi)<0, \quad  \bar{\sigma}'(\xi)>0, \quad \bar{\tau}'(\xi)<0,\quad \mbox{and}\quad 0<\bar{A}(\xi)<\frac{\pi}{2}\quad \mbox{for}\quad \xi\in (-\mathcal{R}, \xi_0];\\&\qquad\qquad\qquad\qquad \lim\limits_{\xi\rightarrow -\mathcal{R}}\bar{\tau}(\xi)=+\infty;\quad \lim\limits_{\xi\rightarrow -\mathcal{R}}\bar{q}(\xi)=q_{_{\infty}}.
\end{aligned}
\end{equation}
Thus the set of states which can be connected to $(u_0, 0)$ by a rarefactive 2-fan wave on the downstream must lie in the curve
$$
{\it F}_{2}(u_0, 0):\quad u=\bar{u}(\xi), \quad v=\bar{v}(\xi), \quad \xi\in (-\mathcal{R}, \xi_0),
$$
where $(\bar{u}, \bar{v})(\xi)$ is defined by (\ref{102306}).
We call this curve the 2-fan wave curve.


\begin{figure}[htbp]
\begin{center}
\includegraphics[scale=0.53]{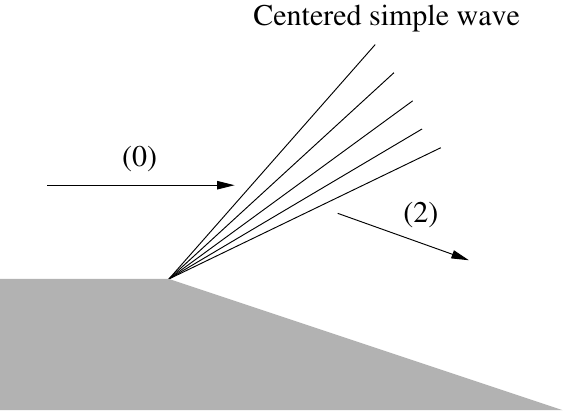}\qquad\qquad\qquad \includegraphics[scale=0.53]{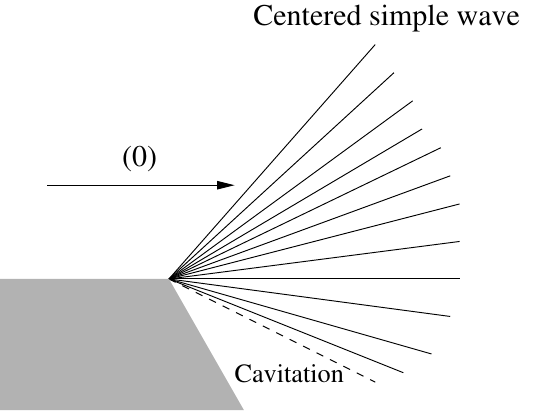}
\caption{\footnotesize Centered simple wave for $\tau_0>\tau_2^{i}$. }
\label{Fig3.3.1}
\end{center}
\end{figure}

  When
 $\theta_{-}>-\mathcal{R}$, there exists a $\xi_2\in (-\mathcal{R}, \xi_0)$ such that
$\bar{v}(\xi_2)=\bar{u}(\xi_2)\tan\theta_{-}$, and the RIBVP (\ref{E1}), (\ref{112201a}) admits a centered simple solution
with the following form:
$$
(u, v)=\left\{
               \begin{array}{ll}
                 (u_0, 0), & \hbox{$\xi>\xi_0$;} \\[2pt]
                 (\bar{u}, \bar{v})(\xi), & \hbox{$\xi_2<\xi<\xi_0$;} \\[2pt]
                 (u_{2}, v_{2}), & \hbox{$\theta_{-}<\xi<\xi_2$,}
               \end{array}
             \right.
$$
where  $(u_{2}, v_{2})=(\bar{u}, \bar{v})(\xi_2)$; see Fig. \ref{Fig3.3.1} (left). When
 $\theta_{-}\leq -\mathcal{R}$, the self-similar solution of the  RIBVP (\ref{E1}), (\ref{112201a}) has the form:
$$
(u, v)=\left\{
               \begin{array}{ll}
                 (u_0, 0), & \hbox{$\xi>\xi_0$;} \\[2pt]
                 (\bar{u}, \bar{v})(\xi), & \hbox{$-\mathcal{R}<\xi<\xi_0$;} \\[2pt]
                 \mbox{vacuum}, & \hbox{$\theta_{-}<\xi\leq -\mathcal{R}$;}
               \end{array}
             \right.
$$
 see Fig. \ref{Fig3.3.1} (right).

\subsubsection{$\tau_1^{i}<\tau_0\leq\tau_2^{i}$}
Since $\tau_1^{i}<\tau_0\leq\tau_2^{i}$, by
Proposition \ref{51003} one has
\begin{equation}\label{5601}
p'(\psi_{po}(\tau_0))=\frac{2h(\psi_{po}(\tau_0))-2h(\tau_0)}{\psi_{po}^2(\tau_0)-\tau_0^2}.
\end{equation}
We let $\tau_{po}=\psi_{po}(\tau_0)$, $\rho_{po}=\frac{1}{\tau_{po}}$ and $c_{po}=\sqrt{-\tau_{po}^2p'(\tau_{po})}$.
We shall show that an oblique shock or a shock-fan composite wave, in which a centered simple wave adjacent up to a shock, may occur.


Assume that there is an oblique 2-shock issued from the corner and the flow state on its front side is $(u_0, 0)$.
 We denote by $\phi$ the inclination angle of this oblique shock, and by $(u_b, v_b)$ the state on the back side of the shock. Then by (\ref{111306}) and (\ref{RH}) we have
\begin{equation}\label{102401}
u_b=N_b\sin\phi+L_b\cos\phi, \quad v_b=-N_b\cos\phi+L_b\sin\phi,
\end{equation}
where
$$
\phi=\arcsin \left(\frac{\tau_0}{u_0}\sqrt{-\frac{2h(\tau_b)-2h(\tau_0)}{\tau_b^2-\tau_0^2}}\right), \quad
N_b=\tau_b\sqrt{-\frac{2h(\tau_b)-2h(\tau_0)}{\tau_b^2-\tau_0^2}},\quad L_b=u_0\cos\phi
$$
and $\tau_b\in (\tau_0, \tau_{po}]$.
This implies that $u_b$, $v_b$, and $\phi$ can be seen as functions of $\tau_b\in (\tau_0, \tau_{po}]$.
Thus the set of states which can be connected to $(u_0, 0)$ by an oblique rarefactive 2-shock on the downstream must lie in the curve
$$
{\it S}_{2}(u_0, 0):\quad u=u_b(\tau_b), \quad v=v_b(\tau_b), \quad \tau_b\in (\tau_0, \tau_{po}].
$$
We call this curve the 2-shock curve.

\begin{lem}\label{120301}
The oblique 2-shock curve ${\it S}_{2}(u_0, 0)$ does not meet the $u$-axis.
\end{lem}
\begin{proof}
From $L_f=L_b$, $N_b=\frac{\tau_b N_f}{\tau_0}>N_f$, and $\lim\limits_{\tau_b\rightarrow\tau_0}\phi(\tau_b)=\arcsin(\frac{c_0}{u_0})$,
we have $v_b(\tau_b)<0$ as $\tau_b-\tau_0>0$ is sufficiently small.
Suppose that there is a $\tau_*\in (\tau_0, \tau_{po}]$ such that $v_b(\tau_*)=0$.
Then $\phi(\tau_*)=\frac{\pi}{2}$ and
$u_b(\tau_*)=\tau_*\sqrt{-\frac{2h(\tau_*)-2h(\tau_0)}{\tau_*^2-\tau_0^2}}$.
By assumption (H1) we have $\check{q}(\tau_*)=u_b(\tau_*)>\check{c}(\tau_*)$. While by Proposition \ref{tran} we have
$\tau_*\sqrt{-\frac{2h(\tau_*)-2h(\tau_0)}{\tau_*^2-\tau_0^2}}<\check{c}(\tau_*)$. This leads to a contradiction.
This completes the proof.
\end{proof}

Set
\begin{equation}\label{4901}
\left\{
  \begin{array}{ll}
  \displaystyle \phi_{po}=\arcsin \Big(\frac{\rho_{po}c_{po}}{\rho_0u_0}\Big), \\[4pt]
  u_{po}=c_{po}\sin\phi_{po}+u_0\cos^2\phi_{po}, \\[4pt]
v_{po}=u_0\cos\phi_{po}\sin\phi_{po}-c_{po}\cos\phi_{po}.
  \end{array}
\right.
\end{equation}
Then we have $(\phi, u_b, v_b)=(\phi_{po}, u_{po}, \tau_{po})$ for $\tau_{b}=\tau_{po}$.

We define
\begin{equation}\label{102408}
\sigma_{po}=\arctan\Big(\frac{v_{po}}{u_{po}}\Big).
\end{equation}
Then when
$\theta_{-}>\sigma_{po}$  there exists a $\tau_2\in (\tau_0, \tau_{po})$ such that
$\arctan(\frac{v_{2}}{u_{2}})=\theta_{-}$ where  $(u_{2}, v_{2})=(u_b(\tau_2), v_b(\tau_2))$,  and the RIBVP (\ref{E1}), (\ref{112201a}) admits a self-similar solution
with the following form:
$$
(u, v)=\left\{
               \begin{array}{ll}
                 (u_0, 0), & \hbox{$\phi_{2}<\xi<\frac{\pi}{2}$;} \\[2pt]
                 (u_{2}, v_{2}), & \hbox{$\theta_{-}<\xi<\phi_{2}$,}
               \end{array}
             \right.
$$
where $\phi_{2}=\phi(\tau_2)$;
see Fig. \ref{Fig3.3.2} (left).

\begin{figure}[htbp]
\begin{center}
\includegraphics[scale=0.46]{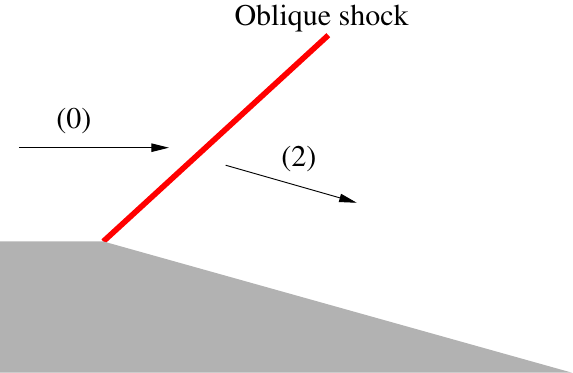}\qquad \qquad\qquad \includegraphics[scale=0.46]{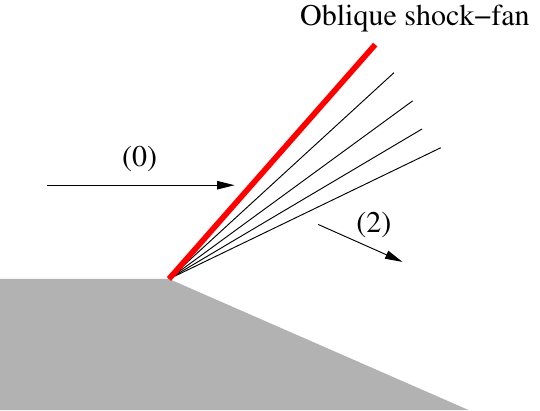}
\caption{\footnotesize Oblique wave at the corner $D$ for  $\tau_1^{i}<\tau_0\leq \tau_2^{i}$. }
\label{Fig3.3.2}
\end{center}
\end{figure}

We next consider oblique 2-shock-fan composite waves composed of a post-sonic 2-shock and a 2-fan wave.
The inclination angle of the post-sonic shock is $\phi_{po}$, and the states on the front and back sides of the shock are $(u_0, 0)$ and $(u_{po}, v_{po})$, respectively. In order to obtain the centered simple wave adjacent up to the oblique post-sonic shock, we consider (\ref{102304}) with data
\begin{equation}\label{102402}
(\bar{q}, \bar{\tau}, \bar{\sigma})(\phi_{po})=(q_{po}, \tau_{po}, \sigma_{po}),
\end{equation}
where $q_{po}=\sqrt{u_{po}^2+v_{po}^2}$.

Let
\begin{equation}\label{92601a}
\xi_v=\sigma_{po}-\int_{q_{po}}^{q_{\infty}}\frac{\sqrt{q^2-\hat{c}^2(q)}}{q\hat{c}(q)}{\rm d}q.
\end{equation}
Then the initial value problem (\ref{102304}), (\ref{102402}) admits a solution for $\xi\in (\xi_v, \phi_{po})$. Moreover, the solution satisfies
\begin{equation}\label{2201}
\begin{aligned}
&\quad \bar{q}'(\xi)<0, \quad  \bar{\sigma}'(\xi)>0, \quad \bar{\tau}'(\xi)<0,\quad \mbox{and}\quad 0<\bar{A}(\xi)<\frac{\pi}{2}\quad \mbox{for}\quad \xi\in (\xi_v, \phi_{po}];\\&\qquad\qquad\qquad\qquad \lim\limits_{\xi\rightarrow \xi_v}\bar{\tau}(\xi)=+\infty;\quad \lim\limits_{\xi\rightarrow \xi_v}\bar{q}(\xi)=q_{_{\infty}}.
\end{aligned}
\end{equation}
Thus the set of states which can be connected to $(u_0, 0)$ by an oblique 2-shock-fan composite wave on the downstream must lie in the curve
$$
{\it SF}_{2}(u_0, 0):\quad u=\bar{u}(\xi), \quad v=\bar{v}(\xi), \quad \xi\in (\xi_{v}, \xi_{po}),
$$
where $(\bar{u}, \bar{v})(\xi)$ is defined by (\ref{102306}).
We call this curve the  2-shock-fan composite wave curve.

\begin{figure}[htbp]
\begin{center}
\includegraphics[scale=0.38]{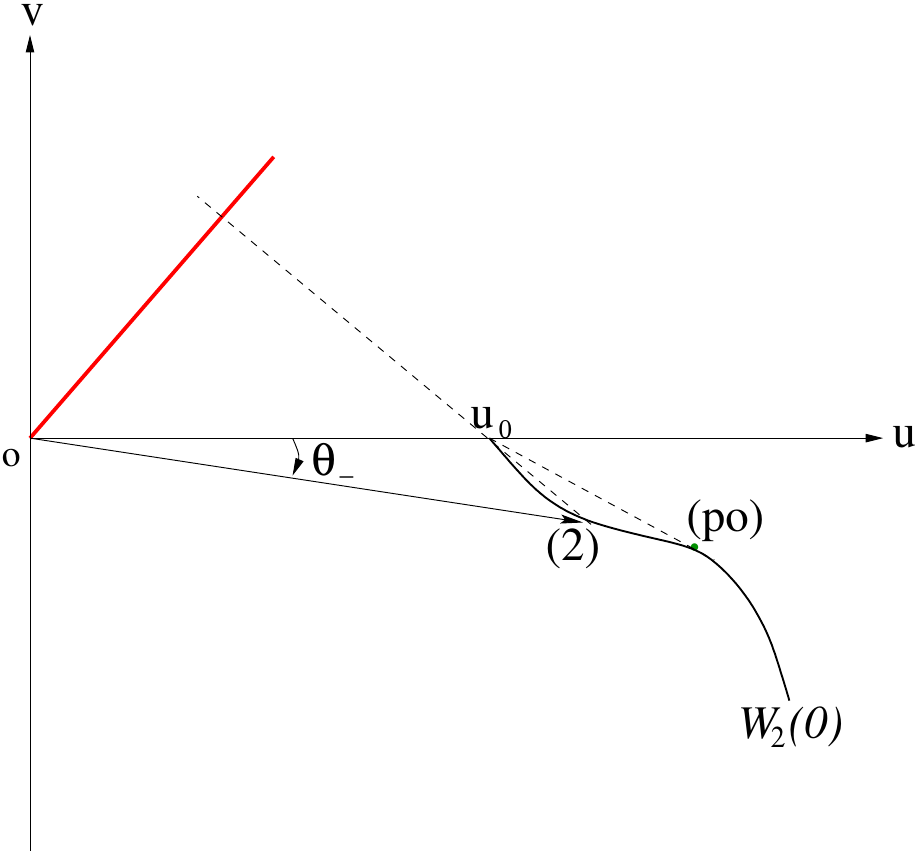}\qquad \qquad \includegraphics[scale=0.38]{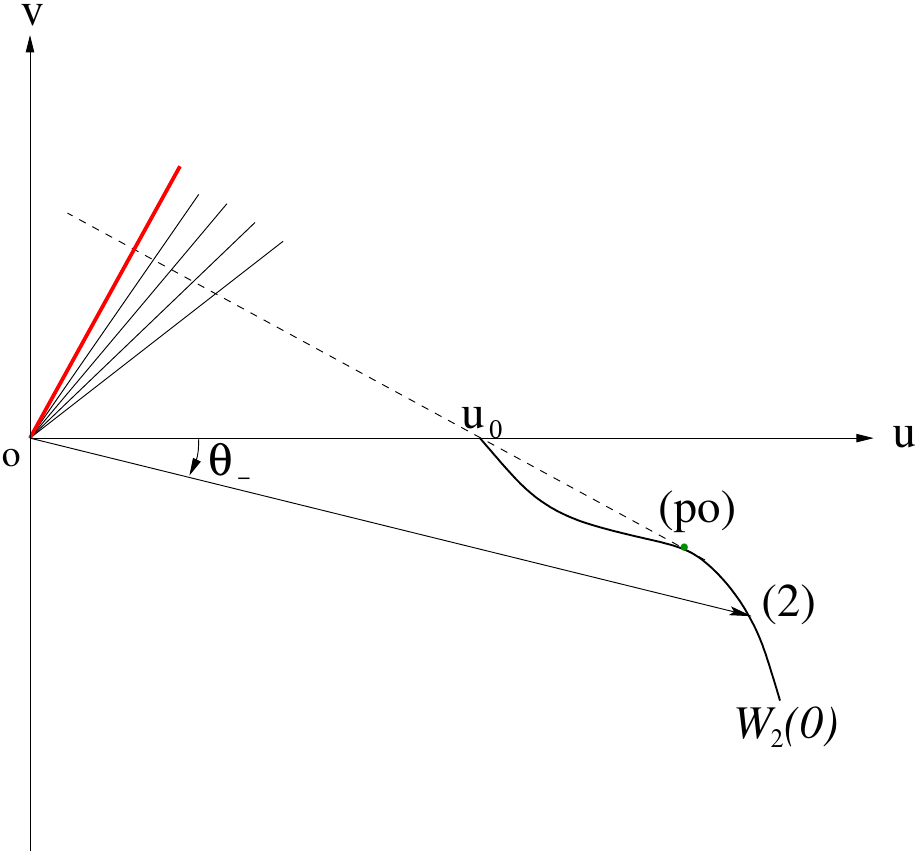}
\caption{\footnotesize The rarefactive 2-wave curve ${\it W}_{2}(u_0, 0)$, where $(0)$, $(po)$, and $(2)$ represent $(u_0, 0)$, $(u_{po}, v_{po})$, and $(u_2, v_2)$, respectively.}
\label{Fig3.3.3}
\end{center}
\end{figure}

 When
 $\xi_{v}<\theta_{-}<\sigma_{po}$, there exists a $\xi_2\in (\xi_v, \phi_{po})$ such that
$\bar{v}(\xi_2)=\bar{u}(\xi_2)\tan\theta_{-}$, and the RIBVP (\ref{E1}), (\ref{112201a}) admits an oblique fan-shock composite wave solution
with the following form:
\begin{equation}\label{8230a}
(u, v)=\left\{
               \begin{array}{ll}
                 (u_0, 0), & \hbox{$\phi_{po}<\xi<\frac{\pi}{2}$;} \\[2pt]
                 (\bar{u}, \bar{v})(\xi), & \hbox{$\xi_2<\xi<\phi_{po}$;} \\[2pt]
                 (u_{2}, v_{2}), & \hbox{$\theta_{-}<\xi<\xi_2$,}
               \end{array}
             \right.
\end{equation}
where $(u_{2}, v_{2})=(\bar{u}, \bar{v})(\xi_2)$;
see Fig. \ref{Fig3.3.2} (right).
 When $\theta_{-}<\xi_v$,  the RIBVP (\ref{E1}), (\ref{112201a}) admits an oblique fan-shock composite wave solution
with the following form:
\begin{equation}\label{101102}
(u, v)=\left\{
               \begin{array}{ll}
                 (u_0, 0), & \hbox{$\phi_{po}<\xi<\frac{\pi}{2}$;} \\[2pt]
                 (\bar{u}, \bar{v})(\xi), & \hbox{$\xi_v<\xi<\phi_{po}$;} \\[2pt]
                 \mbox{vacuum}, & \hbox{$\theta_{-}<\xi<\xi_v$.}
               \end{array}
             \right.
\end{equation}

Let ${\it W}_{2}(u_0, 0)={\it S}_{2}(u_0, 0)+{\it SF}_{2}(u_0, 0)$. We call the curve ${\it W}_{2}(u_0, 0)$ the rarefactive 2-wave curve; see Fig. \ref{Fig3.3.3}.
By the symmetry, we can define the rarefactive 1-wave curve ${\it W}_{1}(u_0, 0)$ of the state $(u_0, 0)$, i.e.,
${\it W}_{1}(u_0, 0)=\big\{(u, v)\mid  (u, -v)\in {\it W}_{2}(u_0, 0)\big\}$.

\begin{lem}\label{lem3.8}
For any fixed $\tau_0\in (\tau_1^i, \tau_2^i)$. When $u_0$ is sufficiently large, there holds
$$
\frac{{\rm d}\sigma}{{\rm d}\tau}<0\quad \mbox{along}~ {\it W}_{2}(u_0, 0),
$$
where $\sigma=\arctan(\frac{v}{u})$ and $\tau=\hat{\tau}(q)$.
\end{lem}
\begin{proof}
In view of (\ref{111306}) and (\ref{RH}),
we have that along ${\it S}_{2}(u_0, 0)$,
$$
\sigma=\phi-\arctan\Big(\frac{\tau\tan\phi}{\tau_0}\Big),
$$
where $\phi=\arcsin \left(\frac{\tau_0}{u_0}\sqrt{-\frac{2h(\tau)-2h(\tau_0)}{\tau^2-\tau_0^2}}\right)$.
So, we have that along ${\it S}_{2}(u_0, 0)$,
$$
\begin{aligned}
\frac{{\rm d}\sigma}{{\rm d}\tau}=\frac{(1-\tau\rho_0)(1-\tau\rho_0\tan^2\phi)}{1+\tau^2\rho_0^2\tan^2\phi}
\frac{{\rm d}\phi}{{\rm d}\tau}-\frac{\rho_0\tan\phi}{1+\tau^2\rho_0^2\tan^2\phi}.
\end{aligned}
$$
By (\ref{43012}) we have
\begin{equation}\label{122201a}
\cos\phi\frac{{\rm d}\phi}{{\rm d}\tau}=\frac{-\tau\tau_0}{u_0(\tau^2-\tau_0^2)}
\sqrt{\frac{\tau_0^2-\tau^2}{2h(\tau)-2h(\tau_0)}}\left(p'(\tau)-\frac{2h(\tau)-2h(\tau_0)}{\tau^2-\tau_0^2}\right)>0,\quad \tau_0<\tau<\tau_{po}.
\end{equation}
So, when $u_0$ is sufficiently large, e.g., $1-\tau_{po}\rho_0\tan^2\phi(\tau_{po})>0$,
we have
$\frac{{\rm d}\sigma}{{\rm d}\tau}<0$ along ${\it S}_{2}(u_0, 0)$.

From (\ref{2201}) we immediately have $\frac{{\rm d}\sigma}{{\rm d}\tau}<0$ along ${\it SF}_{2}(u_0, 0)$.
This completes the proof of the lemma.
\end{proof}

By (\ref{5601}) and (\ref{102401}) we have
$$
\begin{aligned}
&u_b'(\tau)\cos\phi+v_b'(\tau)\sin\phi\\~=~&\Big(\frac{{\rm d}L_b}{{\rm d}\phi}+u_b(\tau)\sin\phi-v_b(\tau)\cos\phi\Big)\frac{{\rm d}\phi}{{\rm d}\tau}=(N_b-N_f)\frac{{\rm d}\phi}{{\rm d}\tau}=0\quad \mbox{for}~\tau=\tau_{po}.
\end{aligned}
$$
Combining this with (\ref{82102}) we see that the wave curves ${\it S}_2(u_0, 0)$ and ${\it SF}_2(u_0, 0)$ are in at least first-order contact at the point $(u_{po}, v_{po})$.

\subsubsection{$\tau_1^{e}<\tau_0\leq \tau_1^{i}$}
We first solve (\ref{102304}) with data
\begin{equation}\label{102407}
(\bar{q}, \bar{\tau}, \bar{\sigma})(\xi_0)=(u_0, \tau_0, 0).
\end{equation}
where $\xi_0=\arcsin(\frac{c_0}{u_0})$.
 Since $\tau_1^e<\tau_0<\tau_1^i$, there exists a $\xi_i<\xi_0$ such that
 initial value problem (\ref{102304}), (\ref{102407}) admits a solution for $\xi\in [\xi_i, \xi_0)$. Moreover, the solution satisfies $\bar{\tau}(\xi_i)=\tau_1^i$, and
$\bar{q}'(\xi)<0$, $\bar{\sigma}'(\xi)>0$, and $\bar{\tau}'(\xi)<0$ for $\xi\in [\xi_i, \xi_0)$.
Let
$$
{\it F}_{2}(u_0, 0):\quad u=\bar{u}(\xi), \quad v=\bar{v}(\xi), \quad \xi\in [\xi_i, \xi_0],
$$
where $(\bar{u}, \bar{v})(\xi)$ is defined by (\ref{102306}).

Let
 \begin{equation}\label{81802}
\sigma_i=\arctan \bigg(\frac{\bar{v}(\xi_i)}{\bar{u}(\xi_i)}\bigg).
\end{equation}
When $\theta_{-}>\sigma_i$, there exists a unique $\xi_2 \in (\xi_i, \xi_0)$ such that $\theta_{-}=\arctan(\frac{\bar{v}(\xi_2)}{\bar{u}(\xi_2)})$,
and the IBVP (\ref{E1}), (\ref{112201a}) admits a centered simple wave solution
with the following form
\begin{equation}\label{100901}
(u, v)=\left\{
               \begin{array}{ll}
                 (u_0, 0), & \hbox{$\xi_0\leq \xi\leq \frac{\pi}{2}$;} \\[2pt]
                 (\bar{u}, \bar{v})(\xi), & \hbox{$\xi_2<\xi<\xi_0$;} \\[2pt]
                (u_2, v_2), & \hbox{$\theta_{-}<\xi<\xi_2$,}
               \end{array}
             \right.
\end{equation}
where $(u_2, v_2)= (\bar{u}, \bar{v})(\xi_2)$;
 see Fig. \ref{Fig3.3.4} (1).

\begin{figure}[htbp]
\begin{center}
\includegraphics[scale=0.45]{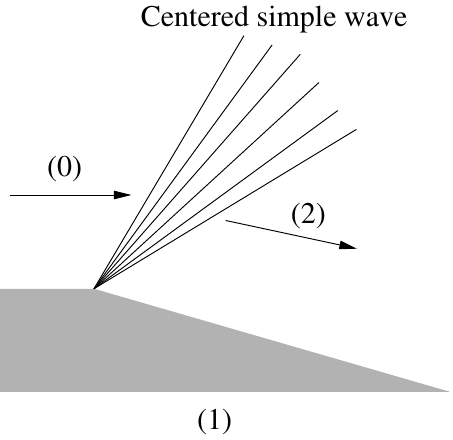} \quad\includegraphics[scale=0.42]{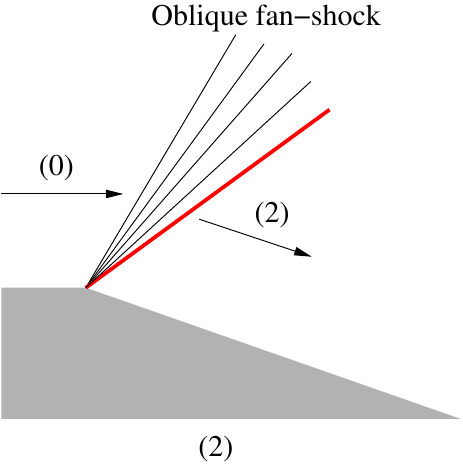}\quad
\includegraphics[scale=0.42]{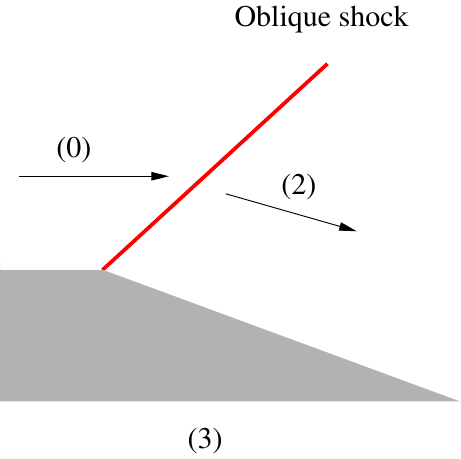}\quad \includegraphics[scale=0.42]{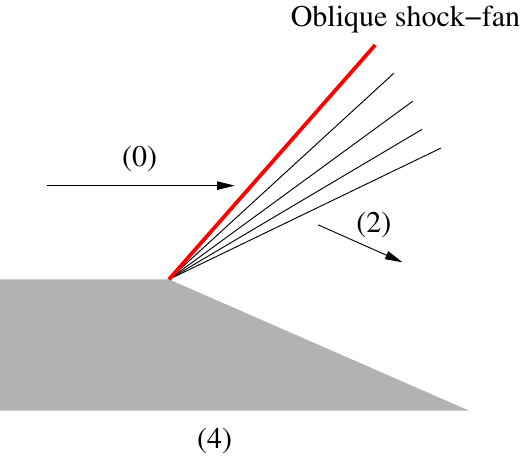}
\caption{\footnotesize Oblique wave at the corner $D$ for  $\tau_1^{e}<\tau_0\leq \tau_1^{i}$.}
\label{Fig3.3.4}
\end{center}
\end{figure}

 We next consider the possibility of the occurrence of an oblique fan-shock composite wave in which a centered simple wave is adjacent up to a pre-sonic shock.
We assume that the inclination angle of the pre-sonic shock is
$\phi\in (\xi_i, \xi_0]$.
Then the state on the front side of the pre-sonic shock is
$(\bar{u}, \bar{v})(\phi)$.
We denote by $(u_s, v_s)$ the  state on the back side of the pre-sonic shock. Then, we have
\begin{equation}\label{81803}
 u_{s}=N_{{s}}\sin\phi+L_{s}\cos\phi\quad \mbox{and}\quad v_{s}=-N_{s}\cos\phi+L_{s}\sin\phi,
\end{equation}
where
\begin{equation}\label{112902}
N_{s}=\frac{\psi_{pr}(\bar{\tau}(\phi))\bar{c}(\phi)}{\bar{\tau}(\phi)}\quad \mbox{and}\quad L_{s}=\sqrt{\bar{q}^2(\phi)-\bar{c}^2(\phi)}.
\end{equation}
Thus the set of states which can be connected to $(u_0, 0)$ by an oblique 2-fan-shock composite wave on the downstream must lie in the curve
$$
{\it FS}_{2}(u_0, 0):\quad u=u_s(\phi), \quad v=v_{s}(\phi), \quad \phi\in (\xi_{i}, \xi_{0}].
$$
We call this curve the 2-fan-shock composite wave curve.

Let
$$
\sigma_{_M}=\sup\limits_{\phi\in (\xi_i, \xi_0]}\arctan\bigg(\frac{v_{s}(\phi)}{u_{s}(\phi)}\bigg)\quad \mbox{and}\quad \sigma_{m}=\inf\limits_{\phi\in (\xi_i, \xi_0]}\arctan\bigg(\frac{v_{s}(\phi)}{u_{s}(\phi)}\bigg).
$$
When $\theta_{-}\in(\sigma_{m}, \sigma_{_M})$, there exists a $\phi_{pr}\in(\xi_i, \xi_0]$ such that $\theta_{-}=\arctan\big(\frac{v_{s}(\phi_{pr})}{u_{s}(\phi_{pr})}\big)$,
and the Riemann IBVP (\ref{E1}), (\ref{112201a}) admits a fan-shock composite wave solution
with the following form:
\begin{equation}\label{120702a}
(u, v)=\left\{
               \begin{array}{ll}
                 (u_0, 0), & \hbox{$\xi_0\leq \xi\leq \frac{\pi}{2}$;} \\[2pt]
                 (\bar{u}, \bar{v})(\xi), & \hbox{$\phi_{pr}<\xi<\xi_0$;} \\[2pt]
                (u_2, v_2), & \hbox{$\theta_{-}<\xi<{\phi}_{pr}$,}
               \end{array}
             \right.
\end{equation}
where $(u_2, v_2)=(u_{s}(\phi_{pr}), v_{s}(\phi_{pr}))$;
 see Fig. \ref{Fig3.3.4} (2).

Let $\sigma_*=\arctan\big(\frac{v_s(\xi_0)}{u_s(\xi_0)}\big)$ and let $\sigma_{po}$ be the same as that defined in (\ref{4901})--(\ref{102408}).
When $\sigma_{po}<\theta_{-}<\sigma_*$, the RIBVP (\ref{E1}), (\ref{112201a}) has an oblique shock wave solution; see Fig. \ref{Fig3.3.4} (3). When $\theta_{-}<\sigma_{po}$, the RIBVP (\ref{E1}), (\ref{112201a}) has an oblique shock-fan composite wave solution; see Fig. \ref{Fig3.3.4} (4).
As in Section 3.3.2, we can construct the 2-shock curve ${\it S}_{2}(u_0, 0)$ and the 2-shock-fan composite wave curve ${\it SF}_{2}(u_0, 0)$.
Then the rarefactive 2-wave curve of $(u_0, 0)$ for $\tau_1^e<\tau_0<\tau_1^i$ is ${\it W}_{2}(u_0, 0)={\it F}_{2}(u_0, 0)+{\it FS}_{2}(u_0, 0)+{\it S}_{2}(u_0, 0)+{\it SF}_{2}(u_0, 0)$; see Fig. \ref{Fig3.3.5}.

\begin{figure}[htbp]
\begin{center}
\includegraphics[scale=0.38]{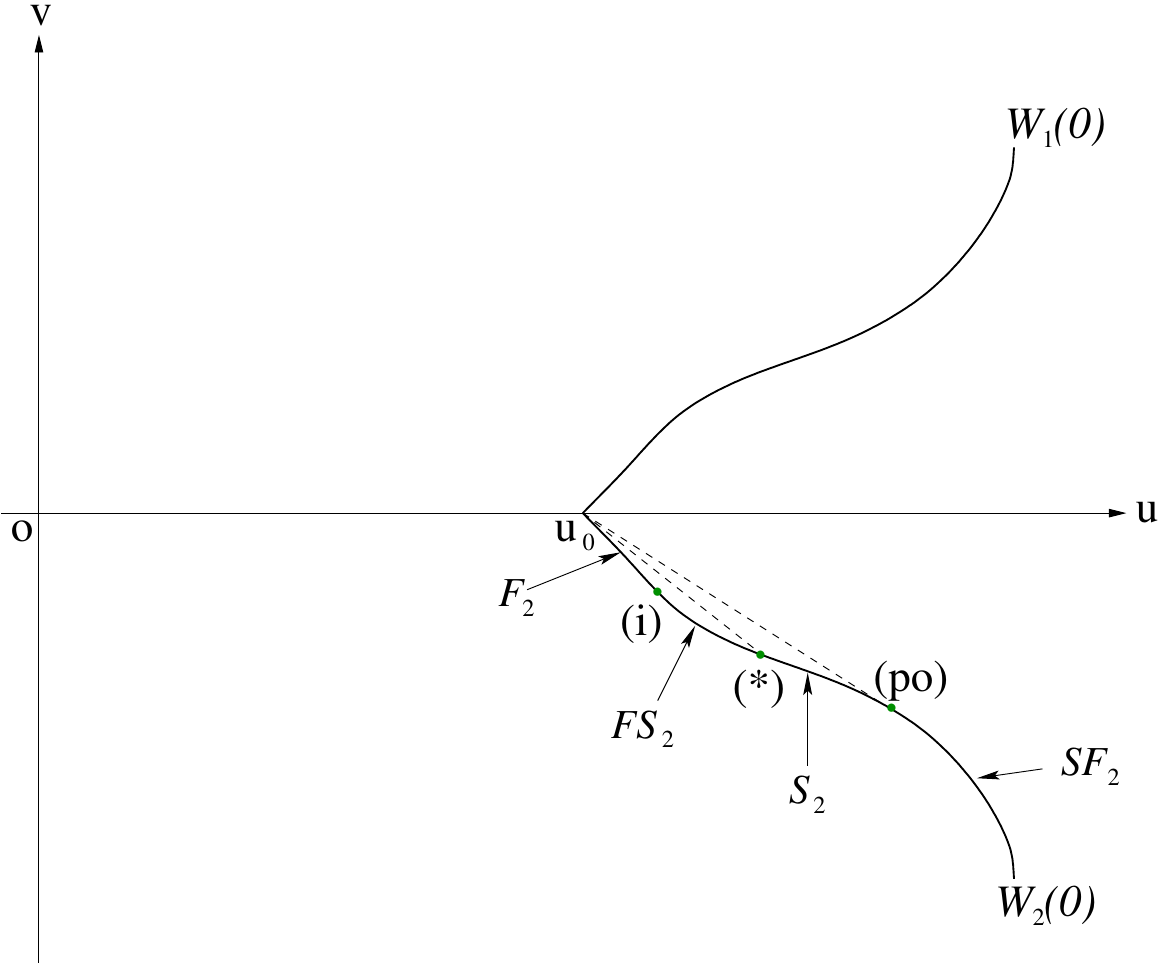}
\caption{\footnotesize The rarefactive wave curves ${\it W}_{1}(u_0, 0)$ and ${\it W}_{2}(u_0, 0)$ for $\tau_1^e<\tau_0<\tau_1^i$, in which $(i)$ and $(*)$ represent the points $(\bar{u}(\xi_i), \bar{v}(\xi_i))$
 and $(u_s(\xi_0), v_s(\xi_0))$.}
\label{Fig3.3.5}
\end{center}
\end{figure}

\begin{lem}\label{lem3.9}
For any fixed $\tau_0\in (\tau_1^e, \tau_1^i)$, when $u_0$ is sufficiently large there holds
$$
\frac{{\rm d}\sigma}{{\rm d}\tau}<0\quad \mbox{along}~ {\it W}_{2}(u_0, 0).
$$
\end{lem}

\begin{proof}
We first prove
\begin{equation}\label{112905}
\frac{{\rm d}\sigma}{{\rm d}\tau}<0\quad \mbox{along}~ {\it FS}_2(u_0, 0).
\end{equation}
We use (\ref{81803}) to compute
\begin{equation}\label{112901}
u_s(\phi)v_s'(\phi)-u_s'(\phi)v_s(\phi)=-L_s(\phi)N_s'(\phi)+N_s(\phi)L_s'(\phi)+N_s^2(\phi)+L_s^2(\phi).
\end{equation}
From (\ref{102304}), (\ref{102304d}), and (\ref{112902}) we have
\begin{equation}\label{112903}
L_s'(\phi)=-\bar{c}(\phi)
\end{equation}
and
\begin{equation}\label{112904}
N_s'(\phi)=\frac{\psi_{pr}'(\bar{\tau}(\phi))\bar{c}(\phi)\bar{\tau}'(\phi)}{\bar{\tau}(\phi)}
+\frac{\psi_{pr}(\bar{\tau}(\phi))}{\bar{\tau}(\phi)}L_s(\phi).
\end{equation}
Inserting (\ref{112903}) and  (\ref{112904}) into  (\ref{112901}), we get
$$
\begin{aligned}
&u_s(\phi)v_s'(\phi)-u_s'(\phi)v_s(\phi)\\=&-L_s(\phi)\frac{\psi_{pr}'(\bar{\tau}(\phi))\bar{c}(\phi)\bar{\tau}'(\phi)}{\bar{\tau}(\phi)}
-L_s^2(\phi)\frac{\psi_{pr}(\bar{\tau}(\phi))}{\bar{\tau}(\phi)}-\bar{c}(\phi)N_s(\phi)+N_s^2(\phi)+L_s^2(\phi)
\\=&-L_s(\phi)\frac{\psi_{pr}'(\bar{\tau}(\phi))\bar{c}(\phi)\bar{\tau}'(\phi)}{\bar{\tau}(\phi)}
+\bar{c}^2(\phi)\left(1-\frac{\psi_{pr}(\bar{\tau}(\phi))}{\bar{\tau}(\phi)}\right)
\left(\cot^2\bar{A}(\phi)-\frac{\psi_{pr}(\bar{\tau}(\phi))}{\bar{\tau}(\phi)}\right).
\end{aligned}
$$

When $u_0$ is sufficiently large, we have $\cot^2\bar{A}(\phi)>\frac{\psi_{pr}(\tau_0)}{\tau_0}$ for $\phi\in (\xi_i, \xi_0]$. In addition, by Proposition \ref{pre} and (\ref{102304}) we have $\psi_{pr}'(\bar{\tau}(\phi))<0$ and $\bar{\tau}'(\phi)<0$ for $\phi\in (\xi_i, \xi_0]$.
Then by (\ref{81803})
 we have that when $u_0$ is sufficiently large,
$$
\begin{aligned}
\frac{{\rm d}\sigma}{{\rm d}\tau}=~&\frac{u_s(\phi)v_s'(\phi)-u_s'(\phi)v_s(\phi)}{q_s^2(\phi)\bar{\tau}'(\phi)\psi_{pr}'(\bar{\tau}(\phi))}
\\[6pt]=~&\frac{1}{q_s^2(\phi)}\left[
\frac{-L_s(\phi)\bar{c}(\phi)}{\bar{\tau}(\phi)}+
\frac{\bar{c}^2(\phi)}{\bar{\tau}'(\phi)\psi_{pr}'(\bar{\tau}(\phi))}\left(1-\frac{\psi_{pr}(\bar{\tau}(\phi))}{\bar{\tau}(\phi)}\right)
\left(\cot^2\bar{A}(\phi)-\frac{\psi_{pr}(\bar{\tau}(\phi))}{\bar{\tau}(\phi)}\right)\right]
\\[6pt]<~&\frac{1}{q_s^2(\phi)}\left[
\frac{-L_s(\phi)\bar{c}(\phi)}{\bar{\tau}(\phi)}+
\frac{\bar{c}^2(\phi)}{\bar{\tau}'(\phi)\psi_{pr}'(\bar{\tau}(\phi))}\left(1-\frac{\psi_{pr}(\bar{\tau}(\phi))}{\bar{\tau}(\phi)}\right)
\left(\cot^2\bar{A}(\phi)-\frac{\psi_{pr}(\tau_0)}{\tau_0}\right)\right]\\[6pt]<~&0 \quad \mbox{along}~ {\it FS}_2(u_0, 0).
\end{aligned}
$$
This completes the proof for (\ref{112905}).

In view of (\ref{102304}), we have $\frac{{\rm d}\sigma}{{\rm d}\tau}<0$ along ${\it F}_{2}(u_0, 0)$ and ${\it SF}_{2}(u_0, 0)$. As in the proof of Lemma \ref{lem3.8}, we have that when $u_0$ is sufficiently large,  $\frac{{\rm d}\sigma}{{\rm d}\tau}<0$ along ${\it S}_{2}(u_0, 0)$.
This completes the proof of the lemma.
\end{proof}

By a direct computation we also have
$$
u_s'(\phi)\cos\phi+v_s'(\phi)\sin\phi=N_s(\phi)-\bar{c}(\phi)
=\bar{c}(\phi)\left(\frac{\psi_{pr}(\bar{\tau}(\phi))}{\bar{\tau}(\phi)}-1\right)
$$
for $\phi\in (\xi_i, \xi_0]$. Hence, we get $u_s'(\phi)\cos\phi+v_s'(\phi)\sin\phi=0$ for $\phi=\xi_i$. This implies that the wave curves ${\it F}_2(u_0, 0)$ and ${\it FS}_2(u_0, 0)$ are in at least first-order contact at the point $(\bar{u}(\xi_i), \bar{v}(\xi_i))$.

 By (\ref{102401}) and (\ref{81803}) we have that for $\phi=\xi_0$,
$$
u_s'(\phi)~=~
N_s'(\xi_0)\sin\xi_0+N_s(\xi_0)\cos\xi_0-2c_0\cos\xi_0,
$$
$$
v_s'(\phi)~=~
-N_s'(\xi_0)\cos\xi_0+N_s(\xi_0)\sin\xi_0+u_0(\cos^2\xi_0-\sin^2\xi_0),
$$
$$
u_b'(\phi)~=~
N_b'(\xi_0)\sin\xi_0+N_b(\xi_0)\cos\xi_0-2c_0\cos\xi_0,
$$
and
$$
v_b'(\phi)~=~
-N_s'(\xi_0)\cos\xi_0+N_s(\xi_0)\sin\xi_0+u_0(\cos^2\xi_0-\sin^2\xi_0).
$$

Since $$N_b(\xi_0)=\psi_{pr}(\tau_0)\sqrt{-\frac{2h(\psi_{pr}(\tau_0))-2h(\tau_0)}{\psi_{pr}^2(\tau_0)-\tau_0^2}}\quad
\mbox{and} \quad N_s(\xi_0)=\frac{\psi_{pr}(\tau_0)c_0}{\tau_0},$$ by (\ref{5701})
we have $N_b(\xi_0)=N_s(\xi_0)$.
We next prove $N_b'(\xi_0)=N_s'(\xi_0)$.

From (\ref{102401}) we have
\begin{equation}\label{122301}
N_b'(\xi_0)=\frac{{\rm d}}{{\rm d}\phi}\left(\frac{u_0\tau_b\sin\phi}{\tau_0}\right)\bigg|_{\phi=\xi_0}=
\frac{u_0\psi_{pr}(\tau_0)\cos\xi_0}{\tau_0}+\frac{c_0}{\tau_0}\frac{{\rm d}\tau_b}{{\rm d}\phi}\bigg|_{\phi=\xi_0}.
\end{equation}
From (\ref{112904}) we have
\begin{equation}\label{122302}
N_s'(\xi_0)=
\frac{u_0\psi_{pr}(\tau_0)\cos\xi_0}{\tau_0}+\frac{c_0}{\tau_0}\psi_{pr}'(\tau_0)\bar{\tau}'(\xi_0).
\end{equation}

As in (\ref{122201a}), we have
\begin{equation}\label{122202a}
\begin{aligned}
\frac{{\rm d}\tau_b}{{\rm d}\phi}\bigg|_{\phi=\xi_0}=~&\frac{u_0(\psi_{pr}^2(\tau_0)-\tau_0^2)\cos\xi_0}{-\tau_0\psi_{pr}(\tau_0)}
\sqrt{\frac{2h(\psi_{pr}(\tau_0))-2h(\tau_0)}{\tau_0^2-\psi_{pr}^2(\tau_0)}}\\ \qquad\qquad&\cdot\left(p'(\psi_{pr}(\tau_0))
-\frac{2h(\psi_{pr}(\tau_0))-2h(\tau_0)}{\psi_{pr}^2(\tau_0)-\tau_0^2}\right)^{-1}
\\ &=\frac{u_0c_0(\psi_{pr}^2(\tau_0)-\tau_0^2)\cos\xi_0}{-\tau_0^2\psi_{pr}(\tau_0)}
\cdot\left(p'(\psi_{pr}(\tau_0))
-\frac{2h(\psi_{pr}(\tau_0))-2h(\tau_0)}{\psi_{pr}^2(\tau_0)-\tau_0^2}\right)^{-1}.
\end{aligned}
\end{equation}

By the definition of the function $\psi_{pr}(\tau)$, we have
\begin{equation}\label{122303}
\psi_{pr}'(\tau_0)=\frac{p''(\tau_0)(\psi_{pr}^2(\tau_0)-\tau_0^2)}{2\psi_{pr}(\tau_0)}\cdot\left(p'(\psi_{pr}(\tau_0))
-\frac{2h(\psi_{pr}(\tau_0))-2h(\tau_0)}{\psi_{pr}^2(\tau_0)-\tau_0^2}\right)^{-1}.
\end{equation}
By the second equation of (\ref{102304}) we have
\begin{equation}\label{122304}
\bar{\tau}'(\xi_0)=-\frac{2u_0c_0\cos \xi_0}{\tau_0^2p''(\tau_0)}.
\end{equation}

Combining (\ref{122301})--(\ref{122304}) we have $N_s'(\xi_0)=N_b'(\xi_0)$.
Thus, $(u_s'(\xi_0), v_s'(\xi_0))=(u_b'(\xi_0), v_b'(\xi_0))$. This implies that
the wave curves
 ${\it S}_2(u_0, 0)$ and ${\it FS}_2(u_0, 0)$ are in at least first-order contact at the point $(u_s(\xi_0), v_s(\xi_0))$.

\subsubsection{$\tau_0<\tau_1^{e}$}
We first solve the initial value problem (\ref{102304}), (\ref{102407}).
 Since $\tau_0<\tau_1^e$, there exist $\xi_i$ and $\xi_{e}$, where  $\xi_i<\xi_{e}<\xi_0$, such that $\bar{\tau}(\xi_i)=\tau_1^i$ and $\bar{\tau}(\xi_{e})=\tau_1^e$.
Let $\sigma_i$ be the same as that defined in (\ref{81802}) and
$$
\bar{\sigma}_{_M}=\sup\limits_{\phi\in (\xi_i, \xi_{e}]}\arctan\bigg(\frac{v_{s}(\phi)}{u_{s}(\phi)}\bigg)\quad \mbox{and}\quad \bar{\sigma}_{m}=\inf\limits_{\phi\in (\xi_i, \xi_{e}]}\arctan\bigg(\frac{v_{s}(\phi)}{u_{s}(\phi)}\bigg),
$$
where the functions $u_s(\phi)$, $v_s(\phi)$, $\phi\in(\xi_i, \xi_e]$ are the same as that defined in (\ref{81803}).
As in Section 3.3.3, when $\theta_{-}>\sigma_i$ the RIBVP (\ref{E1}), (\ref{112201a}) admits a centered simple wave solution;  when $\theta_{-}\in(\bar{\sigma}_{m}, \bar{\sigma}_{_M})$ the RIBVP (\ref{E1}), (\ref{112201a}) admits a fan-shock composite wave solution; see Fig. \ref{Fig3.3.6} (1) and (2).

\begin{figure}[htbp]
\begin{center}
\includegraphics[scale=0.51]{2D7.pdf}  \qquad\qquad\includegraphics[scale=0.47]{2D6.pdf}\qquad\qquad
\includegraphics[scale=0.47]{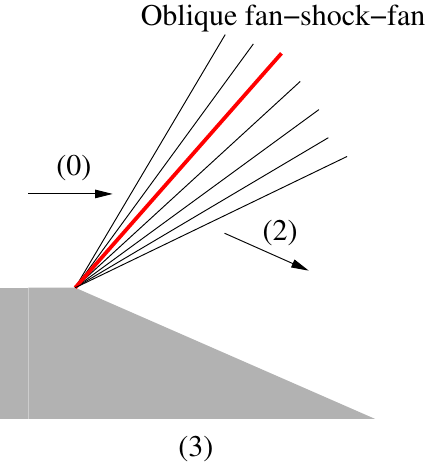}
\caption{\footnotesize Oblique wave at the corner $D$ for  $\tau_0\leq \tau_1^{e}$.}
\label{Fig3.3.6}
\end{center}
\end{figure}

We now consider the possibility of the occurrence of an oblique fan-shock-fan composite wave composted of a centered simple wave, a double-sonic shock, and another centered simple wave.
 The centered simple wave adjacent to the constant state $(0)$ can be represented by
 $$
 (u, v)=(\bar{u}, \bar{v})(\xi), \quad \xi_{e}\leq \xi\leq \xi_0,
 $$
 where $(\bar{u}, \bar{v})(\xi)$ is defined by (\ref{102306}).
 The inclination angle of the double-sonic shock is $\xi_{e}$, and the states on the front and back sides of the shock
 are $(\bar{u}, \bar{v})(\xi_{e})$ and $(u_s, v_s)(\xi_{e})$, respectively.

 In order to construct the other centered simple wave adjacent up to the double-sonic shock, we solve (\ref{102304}) with data
 \begin{equation}\label{81805}
 (\bar{q}, \bar{\tau}, \bar{\sigma})(\xi_{e})=(q_2^e, \tau_2^e, \sigma_2^e),
 \end{equation}
 where $q_2^e=\sqrt{u_s^2(\xi_{e})+v_s^2(\xi_{e})}$ and $\sigma_2^e=\arctan(\frac{v_s(\xi_{e})}{u_s(\xi_{e})})$.
Let
$$
\xi_v=\sigma_{2}^e-\int_{q_{2}^e}^{{q}_{\infty}}\frac{\sqrt{q^2-\hat{c}^2(q)}}{q\hat{c}(q)}{\rm d}q.
$$
Then the initial value problem (\ref{102304}), (\ref{81805}) admits a solution on $(\xi_v, \xi_{e})$. 

For convenience we use subscripts `$l$' and `$r$' to denote the solutions of the initial value problems (\ref{102304}, \ref{102407}) and (\ref{102304}, \ref{81805}), respectively.
When
 $\xi_v<\theta_{-}<\sigma_2^e$, there exists a $\xi_2\in (\xi_v, \xi_{e})$ such that
$\bar{v}_r(\xi_1)=\bar{u}_r(\xi_1)\tan\theta_w$, and  the IBVP (\ref{E1}), (\ref{112201a}) admits a fan-shock-fan composite wave solution
with the following form:
\begin{equation}\label{120401}
(u, v)=\left\{
               \begin{array}{ll}
                 (u_0, 0), & \hbox{$\xi_0<\xi<\frac{\pi}{2}$;} \\[2pt]
                 (\bar{u}_l, \bar{v}_l)(\xi), & \hbox{${\xi}_e<\xi\leq \xi_0$;} \\[2pt]
                  (\bar{u}_r, \bar{v}_r)(\xi), & \hbox{${\xi}_2\leq \xi<\xi_{e}$;} \\[2pt]
                 (u_2, v_2), & \hbox{$\theta_{-}<\xi<{\xi}_2$,}
               \end{array}
             \right.
\end{equation}
where $(u_2, v_2)=(\bar{u}_r, \bar{v}_r)(\xi_2)$;
see Fig. \ref{Fig3.3.6} (3).

\begin{figure}[htbp]
\begin{center}
\includegraphics[scale=0.42]{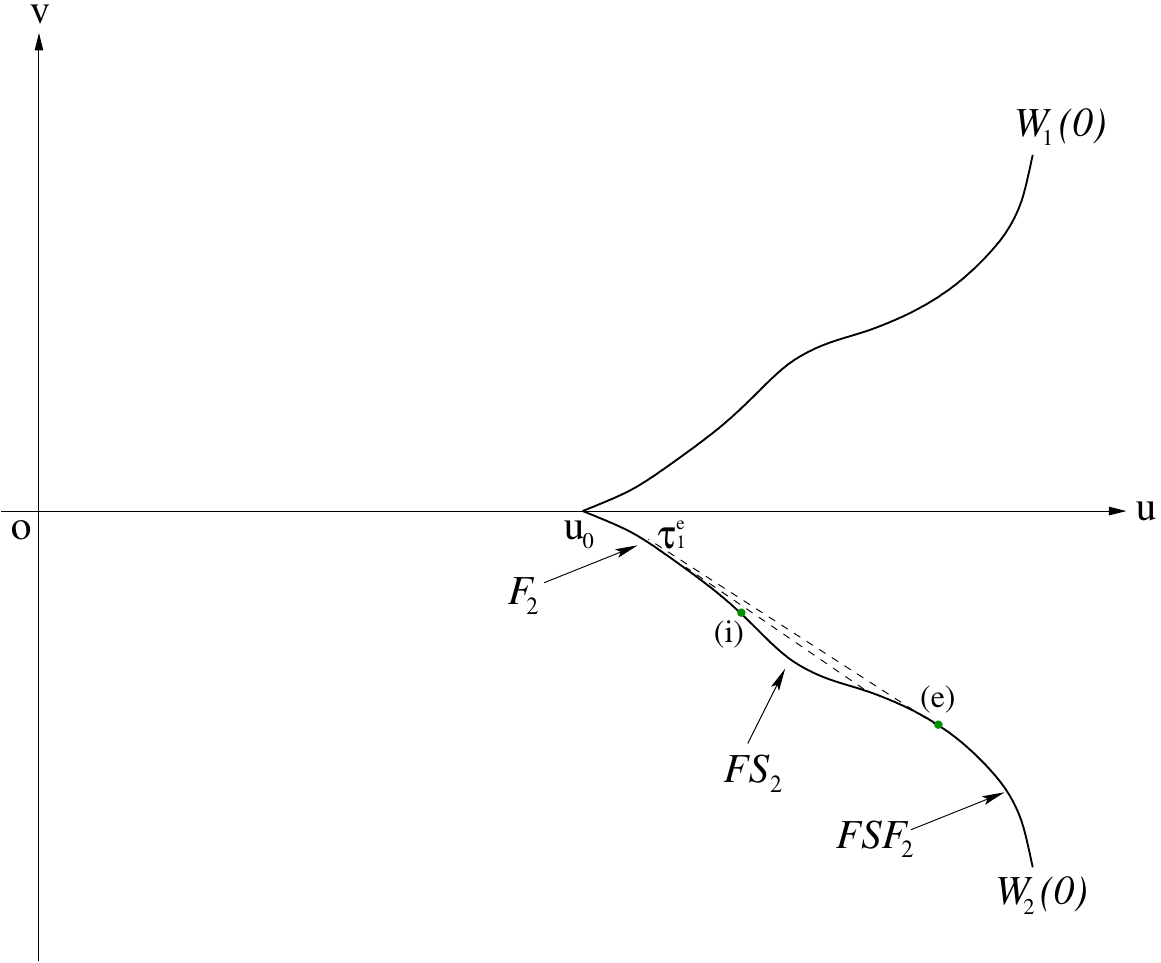}
\caption{\footnotesize  The  rarefactive wave curves ${\it W}_{1}(u_0, 0)$ and ${\it W}_{2}(u_0, 0)$ for $\tau_0<\tau_1^e$, in which $(i)$ and $(*)$ represent the points $(\bar{u}(\xi_i), \bar{v}(\xi_i))$
 and $(u_s(\xi_e), v_s(\xi_e))$.}
\label{Fig3.3.7}
\end{center}
\end{figure}

For $\tau_0<\tau_1^e$, we set
$$
{\it F}_{2}(u_0, 0):\quad u=\bar{u}_{l}(\xi), \quad v=\bar{v}_{l}(\xi), \quad \xi\in [\xi_i, \xi_0];
$$
$$
{\it FS}_{2}(u_0, 0):\quad u=u_s(\phi), \quad v=v_{s}(\phi), \quad \phi\in (\xi_{i}, \xi_{e});
$$
$$
{\it FSF}_{2}(u_0, 0):\quad u=\bar{u}_r(\xi), \quad v=\bar{v}_r(\xi), \quad \xi\in [\xi_v, \xi_e].
$$
Then the rarefactive 2-wave curve of $(u_0, 0)$ is ${\it W}_{2}(u_0, 0)={\it F}_{2}(u_0, 0)+{\it FS}_{2}(u_0, 0)+{\it FSF}_{2}(u_0, 0)$; see Fig. \ref{Fig3.3.7}.

\begin{lem}
For any fixed $\tau_0\in (0, \tau_1^e)$. When $u_0$ is sufficiently large, there holds
$$
\frac{{\rm d}\sigma}{{\rm d}\tau}<0\quad \mbox{along}~ {\it W}_{2}(u_0, 0).
$$
\end{lem}
\begin{proof}
The proof is similar to that of Lemma \ref{lem3.9}; we omit the details.
\end{proof}

\subsection{Classification of oblique wave interactions}
By the symmetry, we can construct the oblique wave around the corner $B$.
Depending on the oncoming flow state and the flare angles of the divergent duct, i.e., $(u_0, 0, \tau_0)$ and $\theta_{\pm}$,  thirteen distinct types of oblique wave interactions may occur near the inlet of the divergent duct; see Fig. \ref{Fig1.3}.

\section{Hyperbolic boundary value problems}
In this section we study some hyperbolic boundary value problems for (\ref{E1}) with the Bernoulli constant $\mathcal{B}=\frac{u_0^2}{2}+h(\tau_0)$. The solutions of these hyperbolic boundary value problems will be used as building blocks of the supersonic flows within the divergent duct.

\subsection{Standard Gourst problem}
Let $\widetilde{PE}$: $y=y_{+}(x)$ $(x_{_P}\leq x\leq x_{_E})$ and $\widetilde{PF}$: $y=y_{-}(x)$ $(x_{_P}\leq x\leq x_{_F})$ be two given smooth curves which satisfy
\begin{equation}\label{102307}
y_{+}(x_{_P})=y_{-}(x_{_P}), \quad y_{+}(x)>y_{-}(x)\quad \mbox{for}~ x_{_P}<x<\min\{x_{_E}, x_{_F}\};
\end{equation}
see Fig. \ref{Fig4.1}.
We give the boundary condition
\begin{equation}\label{6403a}
(u, v)=\left\{
               \begin{array}{ll}
                 (u_{+}, v_{+})(x,y), & \hbox{$(x,y)\in \widetilde{PE}$;} \\[2pt]
                 (u_{-}, v_{-})(x,y), & \hbox{$(x,y)\in \widetilde{PF}$,}
               \end{array}
             \right.
\end{equation}
where $(u_{+}, v_{+})\in C^1(\widetilde{PE})$ and $(u_{-}, v_{-})\in C^1(\widetilde{PF})$ such that
the following conditions hold:
\begin{equation}\label{6401a}
(u_{+}, v_{+})(P)= (u_{-}, v_{-})(P), \quad \tau(P)>\tau_2^i;
\end{equation}
\begin{equation}\label{6401}
(r, s)\in \Pi, \quad  s=\mbox{Const.},
 \quad \mbox{and}\quad\alpha=\arctan \big(y_{+}'(x)\big) \quad \mbox{on}~ \widetilde{PE};
\end{equation}
\begin{equation}\label{6402}
(r, s)\in \Pi, \quad r=\mbox{Const.}
\quad \mbox{and}\quad\beta=\arctan \big(y_{-}'(x)\big)   \quad \mbox{on} ~ \widetilde{PF};
\end{equation}
\begin{equation}\label{2301}
\frac{{\rm d}\rho(x, y_{+}(x))}{{\rm d}x}<0\quad \mbox{for}~ x\in [x_{_P}, x_{_E}]\quad \mbox{and}\quad
\frac{{\rm d}\rho(x, y_{-}(x))}{{\rm d}x}<0\quad \mbox{for}~ x\in [x_{_P}, x_{_F}].
\end{equation}
 We look for a solution of (\ref{E1}) in a sectorial domain between $\widetilde{PE}$ and $\widetilde{PF}$ such that it satisfies the boundary condition (\ref{6403a}).
From (\ref{6401}) and (\ref{6402}) we see that  $\widetilde{PE}$ is a $C_{+}$ characteristic curve  from $P(x_{_P}, y_{+}(x_{_P}))$ to $E(x_{_E}, y_{+}(x_{_E}))$, and $\widetilde{PF}$ is a  $C_{-}$ characteristic curve from $P$ to $F(x_{_F}, y_{-}(x_{_F}))$. 

\begin{figure}[htbp]
\begin{center}
\includegraphics[scale=0.42]{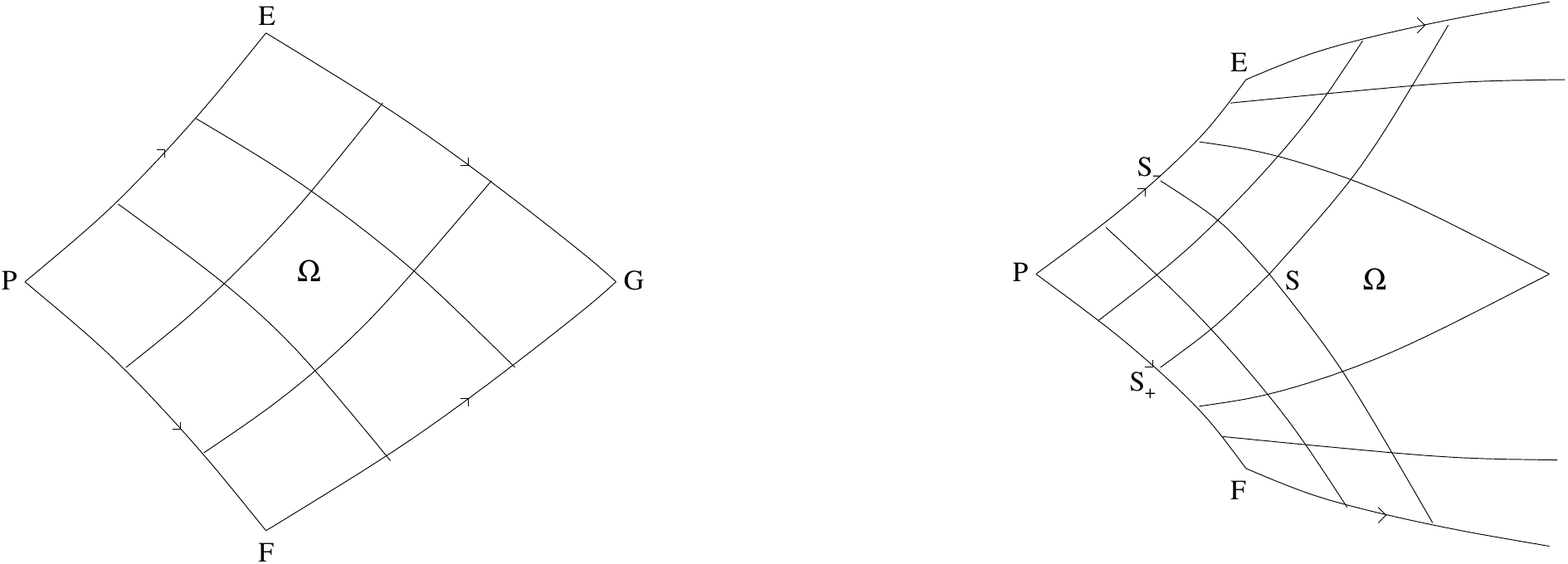}
\caption{\footnotesize A Goursat problem.  Left: $r(E)-s(F)<2\mathcal{R}$; right: $r(E)-s(F)\geq 2\mathcal{R}$.}
\label{Fig4.1}
\end{center}
\end{figure}

\begin{lem}\label{lem1}
Assume that conditions (\ref{6401a})--(\ref{2301}) hold.
Assume as well
$\rho(E)>0$ and $\rho(F)>0$.
Then the boundary value problem (\ref{E1}), (\ref{6403a}) admits a global $C^1$ solution on a curvilinear quadrilateral domain $\overline{\Omega}$ bounded by characteristic curves $\widetilde{PE}$, $\widetilde{PF}$, $C_{+}^{F}$, and $C_{-}^{E}$; see Fig. \ref{Fig4.1}. Moreover, the solution satisfies
\begin{equation}\label{113001a}
\rho>0, \quad \bar{\partial}_{\pm}\rho<0, \quad \sigma_y>0, \quad \mbox{and}\quad (r, s)\in \Pi\quad \mbox{in}~ \overline{\Omega}.
\end{equation}
\end{lem}

\begin{proof}
In view of $(u_{+}, v_{+})(P)= (u_{-}, v_{-})(P)$,
the boundary value problem (\ref{E1}), (\ref{6403a}) is a standard Goursat problem (SGP),  and the local existence of a classical solution can be obtained by the method of characteristics; see Li and Yu \cite{Li-Yu} (Chapter 2). In order to extend the local solution to a global one, we need to establish a uniform a priori estimate for the $C^1$ norm of the solution.

Let $\tau_*>\tau_2^i$ be a sufficiently large constant such that
\begin{equation}\label{61601}
0<A<\frac{\pi}{4}\quad \mbox{for}~ \tau>\tau_*.
\end{equation}

Let
$$\mathcal{M}_1:=2\min\left\{\min\limits_{x\in [x_{_P}, x_{_E}]}\cos\alpha(x, y_{+})\frac{{\rm d}\rho(x, y_{+})}{{\rm d}x}, ~ \min\limits_{x\in [x_{_P}, x_{_F}]}\cos\beta(x, y_{-})\frac{{\rm d}\rho(x, y_{-})}{{\rm d}x}\right\}$$
and
$$
\mathcal{M}_2:=2\min\left\{\min\limits_{x\in [x_{_P}, x_{_E}]}\frac{\cos\alpha(x, y_{+})}{\rho(x, y_{+})}\cdot\frac{{\rm d}\rho(x, y_{+})}{{\rm d}x}, ~ \min\limits_{x\in [x_{_P}, x_{_F}]}\frac{\cos\beta(x, y_{-})}{\rho(x, y_{-})}\cdot\frac{{\rm d}\rho(x, y_{-})}{{\rm d}x},~\tau_*\mathcal{M}_1\right\}.$$
Then by (\ref{6401})--(\ref{2301}) we have
\begin{equation}\label{6405}
\left\{
  \begin{array}{ll}
    \mathcal{M}_1<\bar\partial_{+}\rho<0\quad\mbox{and}\quad\frac{\bar\partial_{+}\rho}{\rho}> \mathcal{M}_2 & \hbox{along~ $\widetilde{PE}$;} \\[6pt]
    \mathcal{M}_1<\bar\partial_{-}\rho<0\quad\mbox{and}\quad \frac{\bar\partial_{-}\rho}{\rho}>\mathcal{M}_2 & \hbox{along~ $\widetilde{PF}$.}
  \end{array}
\right.
\end{equation}

Using the first equation of (\ref{81104})  and recalling (\ref{6405}), we have
\begin{equation}\label{61602}
 \mathcal{M}_1<\bar{\partial}_{-}\rho<0\quad \mbox{along}~ \widetilde{PE}.
\end{equation}

Suppose there is a point on $\widetilde{PE}$ such that $\frac{\bar\partial_{-}\rho}{\rho}=\mathcal{M}_2$ at this point. Then by (\ref{61602}) the definition of $\mathcal{M}_2$ we have $\tau>\tau_*$ at this point. Furthermore, by the first equation of (\ref{53001}) we have that at this point,
\begin{equation}\label{61603AA}
 \bar{\partial}_{+}\Big(\frac{\bar{\partial}_{-}\rho}{\rho}\Big)=
    \frac{\tau^3p''(\tau)}{4c^2\cos^{2}A }\bigg(\big(\frac{\bar{\partial}_{-}\rho}{\rho}\big)^2+(\mathcal{H}-1)\frac{\bar{\partial}_{-}\rho}{\rho}\frac{
    \bar{\partial}_{+}\rho}{\rho}\bigg)>
     \frac{\tau^3p''(\tau)\mathcal{H}}{4c^2\cos^{2}A }\frac{\bar{\partial}_{-}\rho}{\rho}\frac{
    \bar{\partial}_{+}\rho}{\rho}>0.
\end{equation}
Therefore, by $(\frac{\bar\partial_{-}\rho}{\rho})(P)>\mathcal{M}_2$ and the argument of continuity we have
\begin{equation}\label{61603}
 \frac{\bar{\partial}_{-}\rho}{\rho}>\mathcal{M}_2\quad \mbox{on}~\widetilde{PE}.
\end{equation}
Similarly, we have
\begin{equation}\label{6406}
    \mathcal{M}_1<\bar\partial_{+}\rho<0\quad\mbox{and}\quad  \frac{\bar{\partial}_{-}\rho}{\rho}>\mathcal{M}_2\quad \mbox{on}~ \widetilde{PF}.
\end{equation}

In what follows, we are going to prove that the solution of the Goursat problem satisfies
\begin{equation}\label{6404}
0<A<\frac{\pi}{2}, \quad
0<\rho<\rho(P), \quad \mathcal{M}_1<\bar\partial_{\pm}\rho<0,\quad\mbox{and}\quad\frac{\bar\partial_{\pm}\rho}{\rho}>\mathcal{M}_2.
\end{equation}

Firstly, by (\ref{102301}) we have
$$
0<A<\frac{\pi}{2}\quad \mbox{for}~ 0<\rho\leq\rho(P).
$$

Let $S$ be an arbitrary point in the determinate region of the standard Goursat problem, the backward $C_{+}$ and $C_{-}$ characteristic curves issued from $S$ intersect $\widetilde{PF}$ and $\widetilde{PE}$ at some points $S_{+}$ and $S_{-}$, respectively. Let $\overline{\Omega}_{_S}$ be a closed domain bounded by characteristic curves $\widetilde{PS_{+}}$, $\widetilde{PS_{-}}$, $\widetilde{S_{-}S}$, and $\widetilde{S_{+}S}$. We assert that if the inequalities in (\ref{6404}) hold for every point in $\overline{\Omega}_{_S}\setminus\{S\}$, then they also hold at $S$.

Firstly, by the assumption that $\frac{\bar{\partial}_{\pm}\rho}{\rho}>\mathcal{M}_2$  in $\overline{\Omega}_{_S}\setminus\{S\}$ we have $\rho(S)>0$. Integrating (\ref{81104}) along $\widetilde{S_{\pm}S}$ from $S_{\pm}$ to $S$ we have $\bar\partial_{\pm}\rho<0$ at $S$.

We next proof the rest estimates by contradiction.
Suppose $(\bar\partial_{-}\rho)(S)=\mathcal{M}_1$. Then by the assumption of the assertion we have $\bar\partial_{+}\bar\partial_{-}\rho\leq 0$ at $S$. While, by the first decomposition of (\ref{81104}) we have
$$
\bar{\partial}_{+}\bar{\partial}_{-}\rho=
    \frac{\tau^4p''(\tau)}{4c^2\cos^{2}A }\Big[(\bar{\partial}_{-}\rho)^2+(\mathcal{F}-1)\bar{\partial}_{-}\rho\bar{\partial}_{+}\rho\Big]\geq
\frac{\tau^4p''(\tau)\mathcal{F}}{4c^2\cos^{2}A }\bar{\partial}_{-}\rho\bar{\partial}_{+}\rho>0\quad \mbox{at}~ S.
$$
This leads to a contradiction. So, we have $(\bar\partial_{-}\rho)(S)>\mathcal{M}_1$. Similarly, we have $(\bar\partial_{+}\rho)(S)>\mathcal{M}_1$.

Suppose $\big(\frac{\bar\partial_{-}\rho}{\rho}\big)(S)=\mathcal{M}_1$.
Then by the assumption of the assertion we have $\bar{\partial}_{+}\big(\frac{\bar\partial_{-}\rho}{\rho}\big)\leq 0$ at the point $S$.
While, as in (\ref{61603AA}) we have $\bar{\partial}_{+}\big(\frac{\bar\partial_{-}\rho}{\rho}\big)> 0$ at the point $S$.
This leads to a contradiction.
Thus, we have
$\big(\frac{\bar\partial_{-}\rho}{\rho}\big)(S)>\mathcal{M}_1$. Similarly, we have $\big(\frac{\bar\partial_{+}\rho}{\rho}\big)(S)>\mathcal{M}_1$.
This completes the proof of the assertion.
Therefore, by the argument of continuity we obtain that the solution satisfies (\ref{6406}).
Combining this with (\ref{11}), (\ref{72804}), and (\ref{61513}) we can get a  uniform a priori $C^1$ norm estimate for the solution. Hence the global existence of a classical  solution follows routinely from Li \cite{LiT}.

By (\ref{102302d}) we have
\begin{equation}\label{122701}
\bar{\partial}_{+}r=-\frac{\sin(2A)}{\rho}\bar{\partial}_{+}\rho>0\quad \mbox{and}\quad
\bar{\partial}_{-}s=\frac{\sin(2A)}{\rho}\bar{\partial}_{-}\rho<0\quad\mbox{in}~\overline{\Omega}.
\end{equation}
Combining this with $(r, s)\in \Pi$ on $\widetilde{PE}\cup\widetilde{PF}$, we have $(r, s)\in \Pi$ in $\overline{\Omega}$.

By (\ref{12010a}), (\ref{61513}), and assumption (H2) we have
$$
\sigma_y=-\frac{\tau^2}{2}\big(\cos\beta\bar{\partial}_{+}\rho+\cos\alpha\bar{\partial}_{-}\rho\big)>0\quad\mbox{in}~ \overline{\Omega}.
$$
This completes the proof of the lemma.
\end{proof}

\begin{lem}\label{lem2}
The characteristic curves $C_{-}^{E}$ and $C_{+}^{F}$ intersect at a point if and only if $r(E)-s(F)<2\mathcal{R}$.
\end{lem}
\begin{proof}
In view of (\ref{122701}), the function $(r, s)(u_{+}(x,y), v_{+}(x,y))$, $(x,y)\in \widetilde{PE}$ has an inverse function, denoted by
$$(x,y)=(x_{1}, y_{1})(r, s), \quad (s, r)\in \big\{(s, r)\mid s=s(P), r(P)<r<r(E)\big\},$$ and the function $(r, s)(u_{-}(x,y), v_{-}(x,y))$, $(x,y)\in \widetilde{PF}$ has an inverse function, denoted by
$$(x,y)=(x_{2}, y_{2})(r, s), \quad (s, r)\in \big\{(s, r)\mid r=r(P), s(F)<r<s(P)\big\}.$$

We then  consider (\ref{HT}) with the boundary condition
\begin{equation}\label{6505}
(x, y)=\left\{
  \begin{array}{ll}
(x_{1}, y_{1})(r, s), & \hbox{$s=s(P)$, $r(P)<r<r(E)$;} \\[2pt]
   (x_{2}, y_{2})(r, s), & \hbox{$r=r(P)$, $s(F)<s<s(P)$.}
  \end{array}
\right.
\end{equation}
By (\ref{6506}) we know that if  $r(E)-s(F)<2\mathcal{R}$ then
$\rho>0$ in the  rectangular region $\Delta=\{(r, s)\mid r(P)\leq r\leq r(E), s(F)\leq s\leq s(P)\}$.
Then the linear problem (\ref{HT}), (\ref{6505}) admits a bounded solution in  $\Delta$.
This implies that the determinate region of the Goursat problem (\ref{E1}), (\ref{6403a})  is a bounded region.

We next consider the case of
$r(E)-s(F)\geq 2\mathcal{R}$. Suppose  $C_{-}^{E}$ and $C_{+}^{F}$ intersect at a point. Then
by Lemma \ref{lem1} we have $\rho>0$ at this point.
While, by (\ref{52801}) we have
$(r, s)=(r(E), s(F))$, and hence $\rho\leq 0$ at this point. This leads to a contradiction.
\end{proof}


Using the characteristic decomposition method and Theorem \ref{thm0} we have the following simple wave solutions about the SGP (\ref{E1}), (\ref{6403a}).

\begin{lem}\label{lem3}
Assume that conditions (\ref{6401a})--(\ref{6402}) hold.
 Assume as well
$$
\frac{{\rm d}\rho(x, y_{+}(x))}{{\rm d}x}=0\quad \mbox{for}~ x\in [x_{_P}, x_{_E}],\quad
\frac{{\rm d}\rho(x, y_{-}(x))}{{\rm d}x}<0\quad \mbox{for}~ x\in [x_{_P}, x_{_F}],
$$
 and $\rho(F)>0$. Then the SGP (\ref{E1}), (\ref{6403a}) admits a simple wave solution with straight $C_{+}$ characteristic lines on a  curvilinear quadrilateral domain $\overline{\Omega}$ bounded by characteristic lines $\overline{PE}$, $\widetilde{PF}$, $\overline{FG}$, and $\widetilde{EG}$; see Fig. \ref{Fig4.2} (left). Moreover, the solution satisfies
$\bar{\partial}_{-}\rho<0$ in $\overline{\Omega}$.
\end{lem}

\begin{figure}[htbp]
\begin{center}
\includegraphics[scale=0.42]{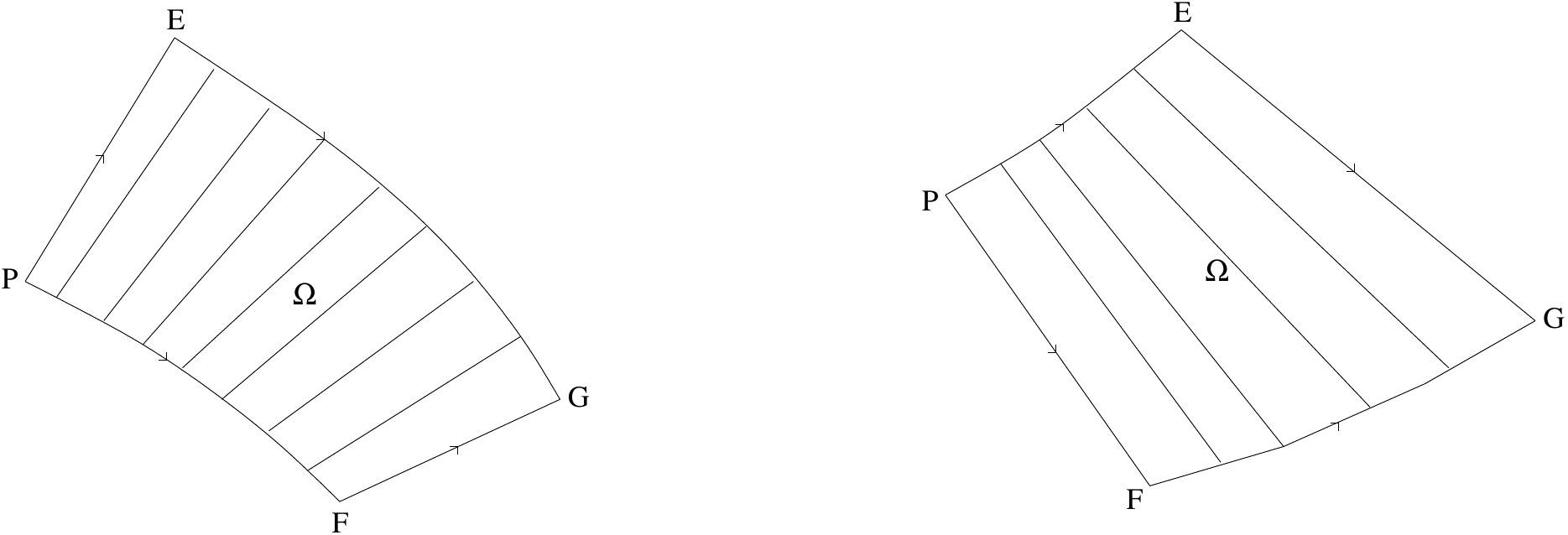}
\caption{\footnotesize Simple wave solution of the standard Goursat problem. }
\label{Fig4.2}
\end{center}
\end{figure}

\begin{lem}\label{lem4}
Assume that conditions (\ref{6401a})--(\ref{6402}) hold.
 Assume as well
$$
\frac{{\rm d}\rho(x, y_{+}(x))}{{\rm d}x}<0\quad \mbox{for}~ x\in [x_{_P}, x_{_E}],\quad
\frac{{\rm d}\rho(x, y_{-}(x))}{{\rm d}x}=0\quad \mbox{for}~ x\in [x_{_P}, x_{_F}],
$$
 and $\rho(E)>0$. Then the SGP (\ref{E1}), (\ref{6403a}) admits a simple wave solution with straight $C_{-}$ characteristic lines on a  curvilinear quadrilateral domain $\overline{\Omega}$ bounded by characteristic lines $\overline{PF}$, $\widetilde{PE}$, $\overline{EG}$, and $\widetilde{FG}$; see Fig. \ref{Fig4.2} (right). Moreover, the solution satisfies
$\bar{\partial}_{+}\rho<0$ in $\overline{\Omega}$.
\end{lem}

\subsection{Discontinuous Goursat problem}
Let $\widetilde{PI}$: $y=y_{+}(x)$ $(x_{_P}\leq x\leq x_{_I})$ and $\widetilde{PJ}$: $y=y_{-}(x)$ $(x_{_P}\leq x\leq x_{_J})$ be two given smooth curves which satisfy
$$
y_{+}(x_{_P})=y_{-}(x_{_P}), \quad y_{+}(x)>y_{-}(x)\quad \mbox{for}~ x_{_P}<x<\min\{x_{_I}, x_{_J}\};
$$
 see Fig. \ref{Fig4.3}.
We give the boundary condition
\begin{equation}\label{6403}
(u, v)=\left\{
               \begin{array}{ll}
                 (u_{+}, v_{+})(x,y), & \hbox{$(x,y)\in \widetilde{PI}$;} \\[2pt]
                 (u_{-}, v_{-})(x,y), & \hbox{$(x,y)\in \widetilde{PJ}$,}
               \end{array}
             \right.
\end{equation}
where $(u_{+}, v_{+})\in C^1(\widetilde{PI})$ and $(u_{-}, v_{-})\in C^1(\widetilde{PJ})$ such that the following conditions hold: $s\equiv\mbox{Const.}$
and $\alpha=\arctan \big(y_{+}'(x)\big)$ on $\widetilde{PI}$;
$r\equiv\mbox{Const.}$ and
$\beta=\arctan\big(y_{-}'(x)\big)$ on $\widetilde{PJ}$.
We also assume that the boundary data (\ref{6403}) satisfy the following conditions:
\begin{equation}\label{6507a}
\tau(u_{+}(P), v_{+}(P))>\tau_2^i,\quad  \tau(u_{-}(P), v_{-}(P))>\tau_2^i;
\end{equation}
\begin{equation}\label{6507}
r(u_{+}(P), v_{+}(P))>r(u_{-}(P), v_{-}(P)),\quad  s(u_{+}(P), v_{+}(P))>s(u_{-}(P), v_{-}(P));
\end{equation}
\begin{equation}\label{102309a}
r(u_{+}(P), v_{+}(P))-s(u_{-}(P), v_{-}(P))<2\mathcal{R};
\end{equation}
\begin{equation}
(r, s)\in \Pi\quad \mbox{on}~  \widetilde{PI}\cup\widetilde{PJ};
\end{equation}
\begin{equation}\label{102308a}
\frac{{\rm d}s(x, y_{-}(x))}{{\rm d}x}<0\quad\mbox{for}~ x\in [x_{_P}, x_{_J}];
\end{equation}
\begin{equation}\label{102308}
\frac{{\rm d}r(x, y_{+}(x))}{{\rm d}x}>0\quad\mbox{for}~ x\in [x_{_P}, x_{_I}].
\end{equation}
\begin{figure}[htbp]
\begin{center}
\includegraphics[scale=0.47]{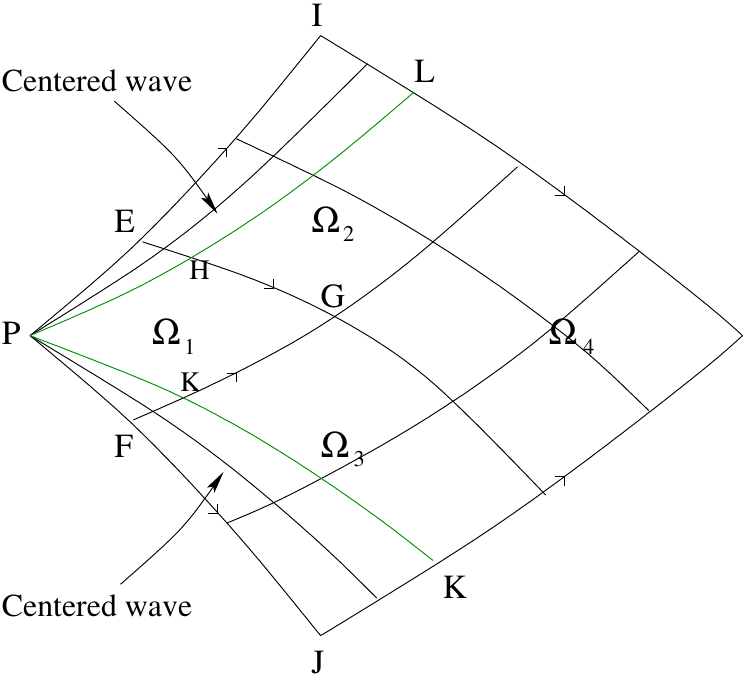}\qquad\qquad\quad
\includegraphics[scale=0.47]{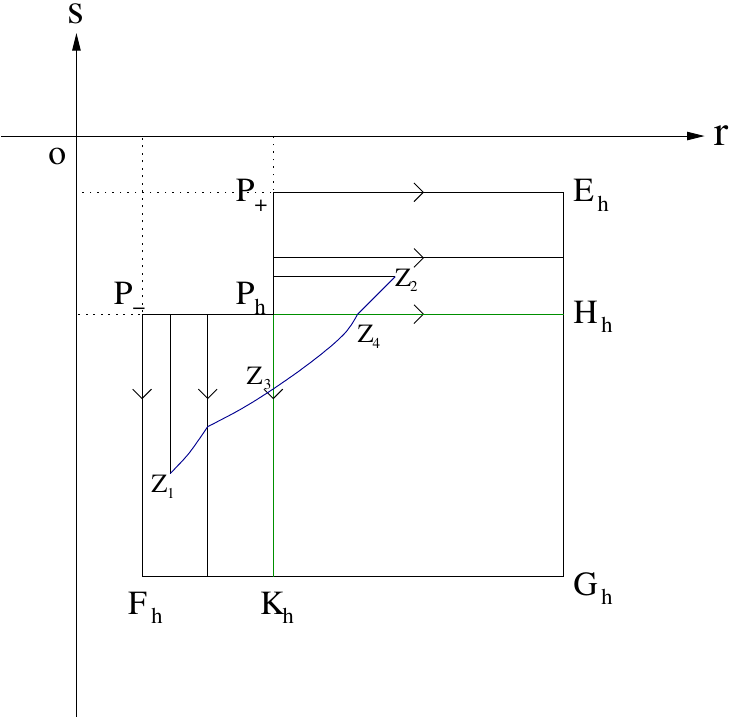}
\caption{\footnotesize A discontinuous Goursat problem. Left: physical plane; right: hodograph plane. Here, the green lines represent the weakly discontinuous lines.}
\label{Fig4.3}
\end{center}
\end{figure}

\begin{lem}\label{lem8}
Under assumptions (\ref{6507a})--(\ref{102308}), the  boundary value problem (\ref{E1}), (\ref{6403}) admits a global continuous and piecewise $C^1$ solution on a curvilinear quadrilateral domain $\Omega$ bounded by characteristic curves $\widetilde{PI}$, $\widetilde{PJ}$, $C_{+}^{J}$, and $C_{-}^{I}$; see Fig. \ref{Fig4.3} (left). Moreover, the solution satisfies $\rho>0$, $\bar{\partial}_{\pm}\rho<0$, and $(r, s)\in \Pi$ in $\overline{\Omega}\setminus\{P\}$.
\end{lem}

In view of (\ref{6507}), the boundary value problem (\ref{E1}), (\ref{6403}) is a discontinuous Goursat problem (DGP). So the method for proving Lemma \ref{lem1} does not work here. In what follows, we are going to use the hodograph transformation method to solve this DGP.

By assumption (\ref{102309a}) there exist a point $E$ on $\widetilde{PI}$ and a point $F$ on $\widetilde{PJ}$ such that \begin{equation}\label{102309}
r(u_{+}(E), v_{+}(E))-s(u_{-}(F), v_{-}(F))<2\mathcal{R}.
\end{equation}
For convenience, we define the constants
$$
\begin{aligned}
&r_{+}:=r(u_{+}(P), v_{+}(P)), \quad r_{-}:=r(u_{-}(P), v_{-}(P)), \\&
s_{+}:=s(u_{+}(P), v_{+}(P)), \quad s_{-}:=s(u_{-}(P), v_{-}(P)).
\end{aligned}
$$

See Fig. \ref{Fig4.3} (right).
Let
$$
\overline{P_{+}E_{h}}=\big\{(r, s)\mid s=s_{+}, r_{+}\leq r\leq r(E)\big\},
\quad
\overline{P_{-}F_{h}}=\big\{(r, s)\mid r=r_{-}, s(F)<r<s_{-}\big\},
$$
$$
\overline{P_{+}P_{h}}=\big\{(r, s)\mid r=r_{+}, s_{-}<s<s_{+}\big\},
\quad
\overline{P_{-}P_{h}}=\big\{(r, s)\mid s=s_{-}, r_{-}<r<r_{+}\big\},
$$
$$
\overline{E_{h}G_{h}}=\big\{(r, s)\mid r=r(E), s(F)<s<s(E)\big\},
\quad
\overline{F_{h}G_{h}}=\big\{(r, s)\mid s=s(F), r(F)<r<r(E)\big\}.
$$
We let $\Delta$ be an open domain bounded by $\overline{P_{+}E_{h}}$, $\overline{P_{+}P_{h}}$,
$\overline{P_{-}P_{h}}$, $\overline{P_{-}F_{h}}$, $\overline{E_{h}G_{h}}$ and $\overline{F_{h}G_{h}}$.
In view of (\ref{102308a}) and (\ref{102308}),
the function $(r, s)(u_{+}(x,y), v_{+}(x,y))$, $(x,y)\in \widetilde{PE}$ has an inverse function, denoted by
$$(x,y)=(x^{+}, y^{+})(r, s), \quad (r, s)\in \overline{P_{+}E_{h}},$$ and the function $(r, s)(u_{-}(x,y), v_{-}(x,y))$, $(x,y)\in \widetilde{PF}$ has an inverse function, denoted by
$$(x,y)=(x^{-}, y^{-})(r, s), \quad (r, s)\in \overline{P_{-}F_{h}}.$$

We  consider
 (\ref{HT}) with the boundary condition
\begin{equation}\label{6508}
(x, y)=\left\{
  \begin{array}{ll}
(x^{+}, y^{+})(r, s), & \hbox{$(r, s)\in \overline{P_{+}E_{h}}$;} \\[4pt]
   (x^{-}, y^{-})(r, s), & \hbox{$(r, s)\in \overline{P_{-}F_{h}}$;}\\[4pt]
   (x_{_P}, y_{_P}), & \hbox{$(r, s)\in \overline{P_{+}P_{h}}$;}\\[4pt]
   (x_{_P}, y_{_P}), & \hbox{$(r, s)\in \overline{P_{-}P_{h}}$.}
  \end{array}
\right.
\end{equation}
The linear problem (\ref{HT}), (\ref{6508}) admits a continuous and piecewise smooth solution in $\Delta$, which we denote as $(x, y)=(\hat{x}, \hat{y})(r, s)$ for convenience.
The solution has two weakly discontinuous lines $\overline{P_hK_h}$ and $\overline{P_hH_h}$ in $\Delta$, where
$\overline{P_hK_h}=\{(r,s)\mid r=r_{+}, s(F)<s<s_{-}\}$, $\overline{P_hH_h}=\{(r,s)\mid r_{+}<r<r(E), s=s_{-}\}$.
The derivatives $\hat{x}_r(r, s)$ and $\hat{y}_r(r, s)$ are discontinuous on $\overline{P_hK_h}$,  and $\hat{x}_s(r, s)$ and $\hat{y}_s(r, s)$ are discontinuous on $\overline{P_hH_h}$.

\begin{lem}\label{lem4.7}
The mapping $(x, y)=(\hat{x}, \hat{y})(r, s)$, $(r,s)\in\Delta$ has an inverse, denoted by
 $$
 (r, s)=(\hat{r}, \hat{s})(x, y), \quad (x,y)\in \Omega_1,
 $$
 where
 $\Omega_1=\big\{(x, y)\mid (x, y)=(\hat{x}, \hat{y})(r, s), (r, s)\in \Delta\big\}$.
\end{lem}
\begin{proof}
It suffices to prove that for any two points $(r_1, s_1)$ and $(r_2, s_2)$ in $\Delta$, $(\hat{x}, \hat{y})(r_1, s_1)\neq(\hat{x}, \hat{y})(r_2, s_2)$.
The proof proceeds in two steps.

\vskip 4pt
\noindent
{\it Step 1.}
By (\ref{121902a}) and (\ref{6507a}) we have
\begin{equation}\label{61702}
\hat{\partial}_{+}\lambda_{-}>0\quad \mbox{and}\quad
\hat{\partial}_{-}\lambda_{+}<0\quad\mbox{in}~ \Delta.
\end{equation}

From (\ref{6401}), (\ref{6402}), (\ref{102308a}), and (\ref{102308})
we have
\begin{equation}\label{61703}
\hat{\partial}_{+}x>0\quad \mbox{on}~ \overline{P_{+}E_{h}}; \quad
\hat{\partial}_{-}x>0\quad \mbox{on}~ \overline{P_{-}F_{h}}.
\end{equation}
By the boundary condition (\ref{6508}) we have
\begin{equation}
\hat{\partial}_{+}x=0\quad \mbox{on}~ \overline{P_{-}P_{h}}; \quad
\hat{\partial}_{-}x=0\quad \mbox{on}~ \overline{P_{+}P_{h}}.
\end{equation}
Then by the characteristic decompositions
(\ref{6506}) we have
\begin{equation}\label{102310}
\hat{\partial}_{+}x>0\quad \mbox{and} \quad \hat{\partial}_{-}x>0\quad\mbox{in}~ \Delta.
\end{equation}
\vskip 4pt
\noindent
{\it Step 2.}
Let $Z_1(r_1, s_1)$ and $Z_2(r_2, s_2)$ be two arbitrary points in $\Delta$. If $r_1=r_2$ or $s_1=s_2$, then by (\ref{102310}) we have $x(Z_1)\neq x(Z_2)$. So, we only need to discuss the following two cases: (1) $r_1<r_2$, $s_1>s_2$; (2) $r_1<r_2$, $s_1<s_2$.

If $r_1<r_2$ and $s_1>s_2$, then by (\ref{102310}) we have
$$
x(Z_2)=x(Z_1)+\int_{r_1}^{r_2} (\hat{\partial}_{+}x)(r, s_1){\rm d}r-\int_{s_1}^{s_2} (\hat{\partial}_{-}x)(r_2, s){\rm d}s>x(Z_1).
$$

We now discuss the case of $r_1<r_2$, $s_1<s_2$. Suppose $x(r_2, s_2)=x(r_1, s_1)$, then by (\ref{102310}) and (\ref{6508}) there exists a curve $\widehat{Z_1Z_2}$ in $\Delta$, which connects the points $Z_1$ and $Z_2$, such that $$
x(r, s)\equiv x(Z_1), \quad (r, s)\in \widehat{Z_1Z_2}.
$$
Moreover, by (\ref{102310}) we see that the curve $\widehat{Z_1Z_2}$ can be represented by a function $s=f(r)$, $r_1\leq r\leq r_2$ which satisfies
\begin{equation}
f'(r)=\frac{\hat{\partial}_{+}x(r, f(r))}{\hat{\partial}_{-}x(r, f(r))}, \quad  r\in(r_1, r_2)\setminus\{r_3, r_4\},
\end{equation}
where $r_3=r_{+}$, and $r_4$ is determined by $f(r_4)=s_{-}$ if it exists.

Using  (\ref{102310}) we have
$$
\begin{aligned}
\frac{{\rm d}y(r, f(r))}{{\rm d}r}&=y_{r}(r, f(r))+f'(r)y_{s}(r, f(r))
\\&=\Big(\tan\big(\alpha(r, f(r))\big)-\tan\big(\beta(r, f(r))\big)\Big)\hat{\partial}_{+}x(r, f(r))\\&>0\quad \mbox{for}~ r\in (r_1, r_2)\setminus\{r_3, r_4\}.
\end{aligned}
$$
This implies $y(Z_2)>y(Z_1)$. From this we immediately have $(x, y)(Z_2)\neq(x, y)(Z_1)$.
This completes the proof of the lemma.
\end{proof}

Let
$$
(\hat{u}, \hat{v})(x,y)=\big(u(\hat{r}(x,y),\hat{s}(x,y)), u(\hat{r}(x,y),\hat{s}(x,y))\big), \quad (x,y)\in\Omega_1.
$$
Obviously, the function $(u, v)=(\hat{u}, \hat{v})(x,y)$ satisfies (\ref{E1}) in $\Omega_1$.
Moreover, by (\ref{53002}) and (\ref{102310}) we have that the solution satisfies
\begin{equation}\label{102412}
(r, s)\in\Pi\quad \mbox{and} \quad
\bar{\partial}_{\pm}\rho<0\quad \mbox{in}~ \overline{\Omega}\setminus\{P\}.
\end{equation}
The solution has a multiple-valued singularity at the point $P$, i.e., the solution contains two centered waves with different types, which are centered at $P$. For the definition of centered waves of quasilinear hyperbolic systems and the related results, we refer the reader to Li and Yu (\cite{Li-Yu}, Chapter 5).

Let $\widetilde{EG}=\big\{(x, y)\mid(x, y)=(\hat{x}, \hat{y})(r, s), (r, s)\in \overline{E_{h}G_{h}}\big\}$ and
$\widetilde{FG}=\big\{(x, y)\mid (x, y)=(\hat{x}, \hat{y})(r, s), (r, s)\in \overline{F_{h}G_{h}}\big\}$. Then
 $\widetilde{EG}$ is a $C_{+}$ characteristic curve issued from $E$, $\widetilde{FG}$ is a $C_{-}$ characteristic curve issued from $F$, and the domain $\Omega_1$ is bounded by the characteristic curves $\widetilde{PE}$, $\widetilde{PF}$, $\widetilde{EG}$, and $\widetilde{FG}$.

By solving two SGPs for (\ref{E1}) we can obtain the flow  on a curvilinear quadrilateral domain $\overline{\Omega}_2$ bounded by characteristic curves $\widetilde{EI}$, $\widetilde{EG}$, $C_{+}^{G}$, and $C_{-}^{I}$.
By solving another two SGPs for (\ref{E1}) we can obtain the flow  on a curvilinear quadrilateral domain $\overline{\Omega}_3$ bounded by characteristic curves $\widetilde{FJ}$, $\widetilde{FG}$, $C_{-}^{G}$, and $C_{+}^{J}$.
Finally, we  solve a SGP for (\ref{E1}) with $C_{\pm}^{G}$ as the characteristic boundaries to obtain the flow in the reaming part of the determinate region $\overline{\Omega}$ of the DGP (\ref{E1}), (\ref{6403}).
The global existence of classical solutions to these SGPs can be obtained by Lemma \ref{lem1}; we omit the details.
This completes the proof of Lemma \ref{lem8}.

There are two weakly discontinuous lines $\widetilde{PK}$ and $\widetilde{PL}$ in $\Omega$, where
$\widetilde{PL}$ ($\widetilde{PK}$, resp.) is a $C_{+}$ ($C_{-}$, resp.) characteristic curve issued from $P$ with the inclination angle $\alpha(r_{+}, s_{-})$ ($\beta(r_{+}, s_{-})$, resp.) at $P$; see Fig. \ref{Fig4.3} (left). We need to mention that the values of $\bar{\partial}_{\pm}\rho$ on the weakly discontinuous lines are the limiting values of $\bar{\partial}_{\pm}\rho$ on both sides of them. Actually, $\bar{\partial}_{+}\rho$ ($\bar{\partial}_{-}\rho$, resp.) is continuous across $\widetilde{PL}$ ($\widetilde{PK}$, resp.), and $\bar{\partial}_{-}\rho$ ($\bar{\partial}_{+}\rho$, resp.) is discontinuous across $\widetilde{PL}$ ($\widetilde{PK}$, resp.)

\begin{figure}[htbp]
\begin{center}
\includegraphics[scale=0.43]{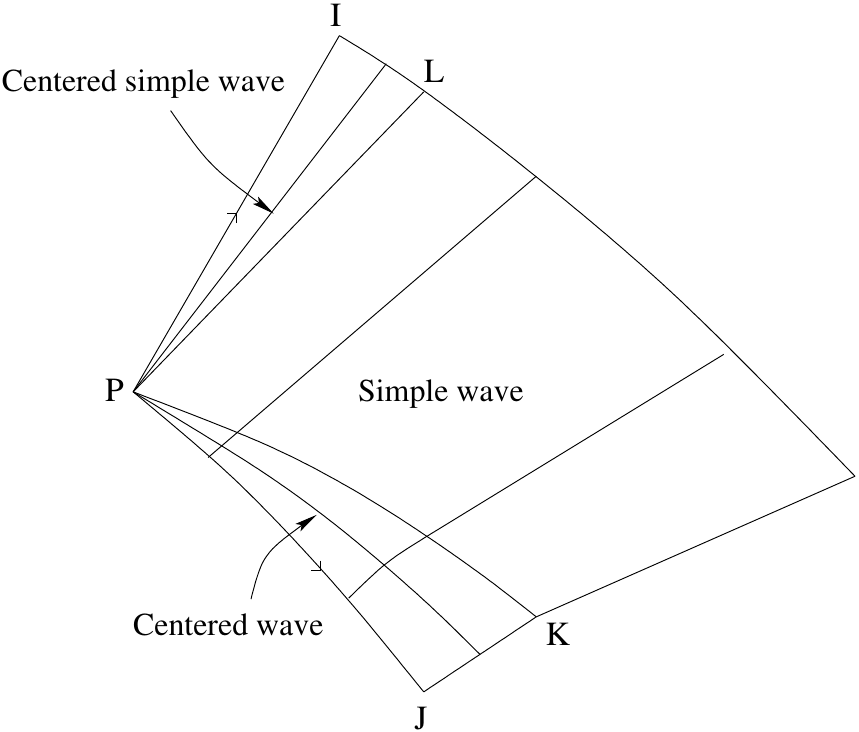}\qquad\qquad\qquad
\includegraphics[scale=0.45]{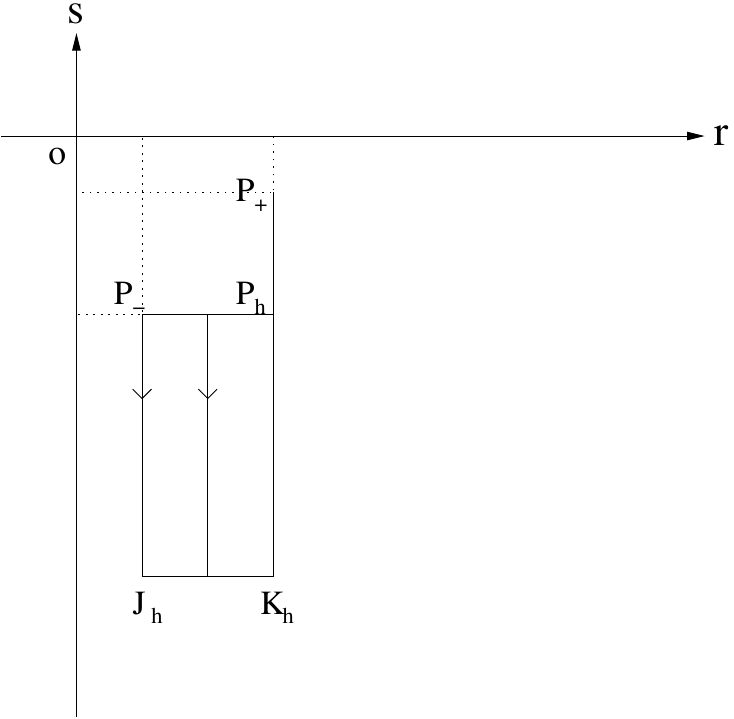}
\caption{\footnotesize A discontinuous Goursat problem with a straight characteristic boundary.}
\label{Fig4.4}
\end{center}
\end{figure}

\begin{rem}\label{rem4.1}
If the assumption (\ref{102308}) is replaced by
\begin{equation}\label{102308c}
\frac{{\rm d}r(x, y_{+}(x))}{{\rm d}x}=0\quad\mbox{for}~ x\in [x_{_P}, x_{_I}],
\end{equation}
we can also use the hodograph transformation method and the method of characteristics to construct a global solution.
Firstly, we construct a centered simple wave with straight $C_{+}$ characteristic lines issued from $P$ such that the images of the centered simple wave in the $(r, s)$-plane is $\overline{P_{+}P_{h}}$.
The centered simple wave region is bounded by characteristic lines $\overline{PI}$, $\overline{PL}$, and $\widetilde{IL}$, where $\overline{PL}$ is a straight $C_{+}$ characteristic line issued from $P$ and the state on it is $(r,s)=(r_{+}, s_{-})$.
We then use the hodograph transformation method to construct a centered wave on a curvilinear triangle region bounded by $\widetilde{PJ}$, $\widetilde{PK}$, and $\widetilde{JK}$. Finally, we solve a SGP as described in Lemma \ref{lem3} to obtain a simple
wave on a  curvilinear quadrilateral domain bounded by characteristic lines $\overline{PL}$, $\widetilde{PK}$, $C_{-}^{L}$, and $C_{+}^{K}$. See Fig. \ref{Fig4.4}.
\end{rem}

\begin{rem}
If assumptions (\ref{102308a}) and (\ref{102308}) are replaced by
$$
\frac{{\rm d}r(x, y_{+}(x))}{{\rm d}x}=0\quad\mbox{for}~ x\in [x_{_P}, x_{_I}]\quad \mbox{and}\quad
\frac{{\rm d}s(x, y_{-}(x))}{{\rm d}x}=0\quad\mbox{for}~ x\in [x_{_P}, x_{_J}],
$$
respectively.
Then the DGP (\ref{E1}), (\ref{6403}) is  a Riemann-type problem and admits a global solution consists of two centered simple waves with different types issued from the point $P$.
\end{rem}
%


\begin{rem}\label{rem4.3}
If the the assumption (\ref{102309a}) is not hold, i.e.,
$r_{+}-s_{-}>2\mathcal{R}$,
one can also use the hodograph transformation method to construct a global classical solution to the DGP (\ref{E1}), (\ref{6403}). The solution consists of two centered waves with different types issued from $P$ and a cavitation; see Fig. \ref{Fig4.5}.
The cavitation is a sectorial domain between two vacuum boundaries $V_1$: $y=y_{_P}+(x-x_{_P})\tan(r_{+}-\mathcal{R})$, $x>x_{_P}$ and
$V_2$: $y=y_{_P}+(x-x_{_P})\tan(s_{-}+\mathcal{R})$, $x>x_{_P}$.
One centered wave region is bounded by $\widetilde{PI}$, $C_{-}^{I}$, and
 $V_1$.
The other centered wave region is bounded by $\widetilde{PJ}$, $C_{+}^{J}$, and
 $V_2$.
\end{rem}

\begin{figure}[htbp]
\begin{center}
\includegraphics[scale=0.4]{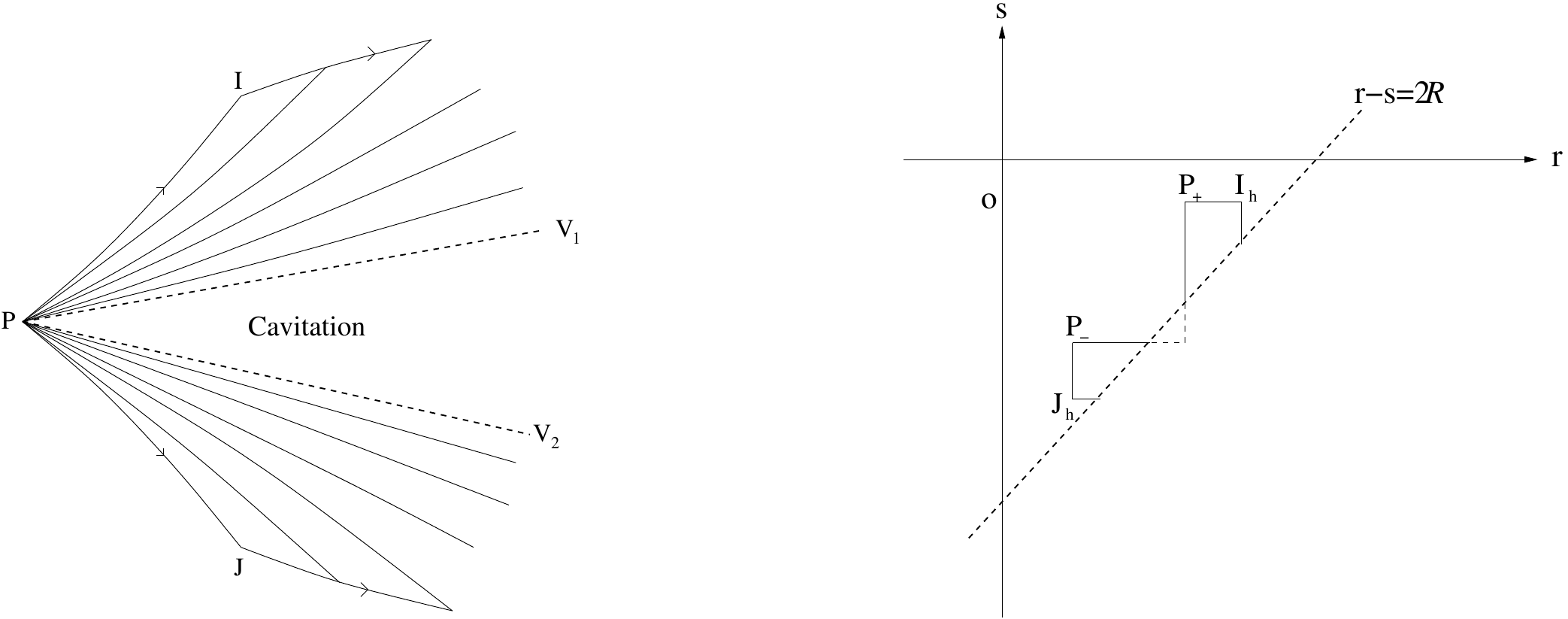}
\caption{\footnotesize  Formation of cavitation for $r_{+}-s_{-}>2\mathcal{R}$.}
\label{Fig4.5}
\end{center}
\end{figure}

\subsection{Mixed boundary value problem}
Let $\widetilde{PE}$: $y=y_{+}(x)$ $(x_{_P}\leq x\leq x_{_E})$ be a given smooth curve which satisfies $$y_{+}(x_{_P})=-1+x_{_P}\tan\theta_{-}, \quad y_{+}(x)>-1+x\tan\theta_{-}\quad  \mbox{for}~ x_{_P}<x\leq x_{_E};$$
see Fig. \ref{Fig4.6}.
 We give the following boundary conditions:
\begin{equation}\label{6503}
(u, v)=(u_{+}, v_{+})(x, y)\quad \mbox{for}~ (x,y)\in \widetilde{PE};
\end{equation}
\begin{equation}\label{6504}
  v(x,y)=u(x,y)\tan\theta_{-}\quad \mbox{for}~ y=-1+x\tan\theta_{-}, ~~x>x_{_P}.
\end{equation}
where $(u_{+}, v_{+})\in C^1(\widetilde{PE})$.
We also assume that condition (\ref{6401}) holds.
So, $\widetilde{PE}$ is a $C_{+}$ characteristic curve  from $P$ to $E$.



\begin{lem}\label{lem5}
Assume that condition (\ref{6401}) holds. Assume as well $\arctan\big(\frac{v_{+}(P)}{u_{+}(P)}\big)= \theta_{-}$, $\rho(E)>0$, and
$\frac{{\rm d}\rho(x, y_{+}(x))}{{\rm d}x}<0$ for $x\in [x_{_P}, x_{_E}]$.
Then  the problem (\ref{E1}), (\ref{6503})--(\ref{6504}) admits a global $C^1$ solution on a curvilinear triangle domain $\overline{\Omega}$ bounded by $\widetilde{PE}$, $\mathcal{W}_{-}$, and $C_{-}^{E}$; see Fig. \ref{Fig4.6} (left). Moreover, the solution satisfies $\rho>0$, $\bar{\partial}_{\pm}\rho<0$, and $(r, s)\in \Pi$ in $\overline{\Omega}$.
The forward $C_{-}$ characteristic curve $C_{-}^{E}$ intersects $\mathcal{W}_{-}$ at a point if and only if $r(E)-\theta_{-}<\mathcal{R}$.
\end{lem}
\begin{proof}
The problem (\ref{E1}), (\ref{6503}), and (\ref{6504}) is  a mixed characteristic boundary and slip boundary value problem  and can be actually seen as a half of a symmetrical SGP. Then the existence can be obtained by Lemma \ref{lem1}.
\end{proof}

\begin{figure}[htbp]
\begin{center}
\includegraphics[scale=0.5]{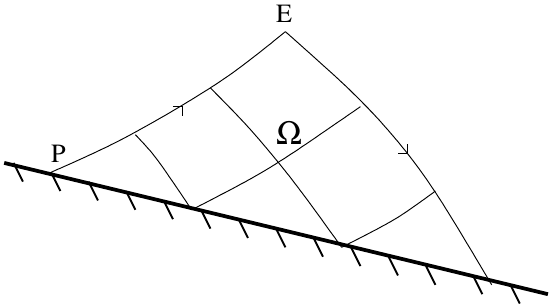}\qquad\includegraphics[scale=0.5]{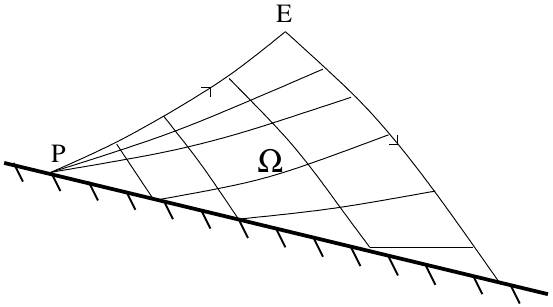}\qquad\includegraphics[scale=0.5]{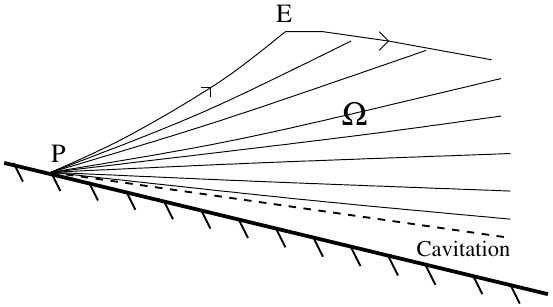}
\caption{\footnotesize Mixed characteristic boundary and slip boundary value problem.  }
\label{Fig4.6}
\end{center}
\end{figure}

\begin{lem}\label{lem5a}
Assume that condition (\ref{6401}) holds. Assume as well $\arctan\big(\frac{v_{+}(P)}{u_{+}(P)}\big)>\theta_{-}$, $\rho(E)>0$, and
$\frac{{\rm d}\rho(x, y_{+}(x))}{{\rm d}x}<0$ for $x\in [x_{_P}, x_{_E}]$.
\begin{itemize}
  \item If $r(u_{+}(P), v_{+}(P))<\mathcal{R}+\theta_{-}$, then the problem (\ref{E1}), (\ref{6503})--(\ref{6504}) admits a global continuous and piecewise smooth solution on a curvilinear triangle domain $\Omega$ bounded by $\widetilde{PE}$, $\mathcal{W}_{-}$, and $C_{-}^{E}$; see Fig. \ref{Fig4.6} (mid). Moreover, the solution satisfies $\rho>0$, $\bar{\partial}_{\pm}\rho<0$, and $(r, s)\in \Pi$ in ${\Omega}$.
  \item  If $r(u_{+}(P), v_{+}(P))\geq \mathcal{R}+\theta_{-}$, then the problem (\ref{E1}), (\ref{6503})--(\ref{6504})  admits a smooth solution on a curvilinear triangle domain $\Omega$ bounded by $\widetilde{PE}$, $C_{-}^{E}$, and a vacuum boundary $y=y_{+}(x_{_P})+(x-x_{_P})\tan (r(u_{+}(P), v_{+}(P))-\mathcal{R})$, $x>x_{_P}$; see Fig. \ref{Fig4.6} (right).
Moreover, the solution satisfies $\rho>0$, $\bar{\partial}_{\pm}\rho<0$, and $(r, s)\in \Pi$ in $\Omega$.
\end{itemize}
\end{lem}
\begin{proof}
The problem can actually be seen as a half of a symmetrical DGP. Then the existence can be obtained by Lemma \ref{lem4.7}.
\end{proof}
\subsection{Singular Cauchy problem and envelope of characteristic curves}

Let $\wideparen{FG}$: $y=y_{+}(x)$ $(x_{_F}\leq x\leq x_{_G})$ be a given smooth curve, $F=(x_{_F}, y_{+}(x_{_F}))$, and
$G=(x_{_G}, y_{+}(x_{_G}))$.
We give the data
\begin{equation}\label{61707}
(u, v)=(u_{+}, v_{+})(x, y)\quad \mbox{on}~ \wideparen{FG},
\end{equation}
where $(u_{+}, v_{+})\in C^1(\wideparen{FG})$.
We assume that the  data (\ref{61707}) satisfy the following conditions:
\begin{equation}\label{101401}
\alpha=\arctan \big(y_{+}'(x)\big) \quad \mbox{on}~ \wideparen{FG};
\end{equation}
\begin{equation}\label{61704a}
\tau(F)>\tau_2^i\quad \mbox{and} \quad
r(G)-s(F)<2\mathcal{R};
\end{equation}
\begin{equation}\label{61704}
\frac{{\rm d}}{{\rm d}x} r(x, y_{+}(x))>0\quad \mbox{and}\quad
\frac{{\rm d}}{{\rm d}x} s(x, y_{+}(x))>0,\quad x_{_F}\leq x\leq x_{_G};
\end{equation}
\begin{equation}\label{113005a}
(r, s)\in \Pi  \quad \mbox{on}~ \wideparen{FG}.
\end{equation}
We want to look for a solution of (\ref{E1}) on a curvilinear triangle domain ${\Omega}$ bounded by $\wideparen{FG}$ and characteristic curves $C_{+}^{F}$ and $C_{-}^{G}$ such that it satisfies (\ref{61707}); see Fig. \ref{Fig4.7} (left).
By (\ref{101401}) and (\ref{61704}) we know that the problem (\ref{E1}), (\ref{61707}) is not a standard Cauchy problem. We shall show that it is actually a singular Cauchy problem, since the solution would have singularity on $\wideparen{FG}$.

\begin{figure}[htbp]
\begin{center}
\includegraphics[scale=0.55]{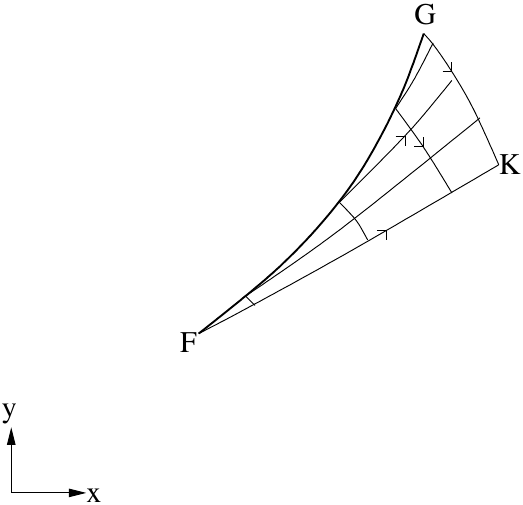}\qquad\qquad\qquad
\includegraphics[scale=0.42]{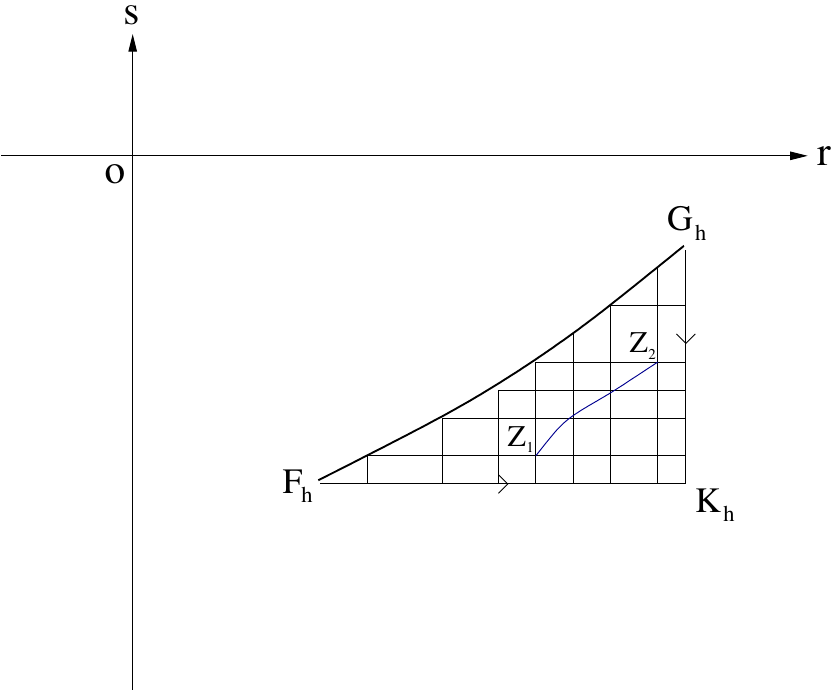}
\caption{\footnotesize A singular Cauchy problem. }
\label{Fig4.7}
\end{center}
\end{figure}

\begin{lem}\label{lem7}
Under assumptions (\ref{101401})--(\ref{113005a}). The problem (\ref{E1}), (\ref{61707}) admits a solution  $(u, v)(x,y)\in C^0(\overline{\Omega})\cap C^1(\overline{\Omega}\setminus\wideparen{FG})$, and this solution satisfies $\bar{\partial}_{\pm}\rho<0$ in $\overline{\Omega}\setminus\wideparen{FG}$. Moreover, the curve $\wideparen{FG}$ is an envelope of all the $C_{+}$ characteristic curves in $\overline{\Omega}$.
\end{lem}
\begin{proof}
The proof proceeds in two steps.
\vskip 4pt

\noindent
{\it Step 1.}
Let
$r_{+}(x,y)=r(u_{+}(x, y), v_{+}(x, y))$ and  $s_{+}(x, y)=s(u_{+}(x, y), v_{+}(x, y))$, $(x,y)\in \wideparen{FG}$.
Let
$$
\wideparen{F_{h}G_{h}}=\big\{(r, s)\mid (r, s)=(r_{+}(x, y), s_{+}(x, y)), (x,y)\in\wideparen{FG} \big\}.
$$
Then by (\ref{61704}) we know that the mapping
$(r, s)=(r_{+}(x, y), s_{+}(x,y))$, $(x,y)\in \wideparen{FG}$
has an inverse, denoted by
\begin{equation}\label{61706}
(x, y)=(x^{+}(r, s), y^{+}(r, s)), \quad (r, s)\in \wideparen{F_hG_h}.
\end{equation}

The curve $\wideparen{F_{h}G_{h}}$ can be represented by a function
$$
s=\varphi(r), \quad r(F)\leq r\leq r(G),
$$
where $r(F)=r(u_{+}(F), v_{+}(F))$ and $r(G)=r(u_{+}(G), v_{+}(G))$; see Fig. \ref{Fig4.7} (right).

Obviously, the Cauchy problem (\ref{HT}), (\ref{61706}) admits a classical solution on $\Delta$, where
 $$\Delta=\big\{(r, s)\mid s(F)\leq s\leq \varphi(r), r(F)\leq r\leq r(G)\big\}.$$

\noindent
{\it Step 2.}
 We next establish the global one-to-one inversion of the hodograph transformation.

 From (\ref{101401}) we have
$$
 \hat{\partial}_{+}y-\varphi'(r) \hat{\partial}_{-}y=\lambda_{+} ( \hat{\partial}_{+}x-\varphi'(r) \hat{\partial}_{-}x)\quad\mbox{on}~ \wideparen{F_hG_h}.
$$
 Combining this with (\ref{HT}) we get
 \begin{equation}\label{61711}
  \hat{\partial}_{-}x=0 \quad\mbox{on}~ \wideparen{F_hG_h}.
 \end{equation}
By (\ref{61704}) we have
\begin{equation}\label{113006a}
\hat{\partial}_{+}x-\varphi'(r) \hat{\partial}_{-}x>0\quad\mbox{on}~ \wideparen{F_hG_h}.
\end{equation}

Combining (\ref{61711}) and (\ref{113006a}) we get
 \begin{equation}\label{61712}
  \hat{\partial}_{+}x>0 \quad\mbox{on}~ \wideparen{F_hG_h}.
 \end{equation}
 Thus, by (\ref{61711}) and (\ref{61712}) we have
\begin{equation}\label{92503aa}
   \hat{\partial}_{-}x>0\quad \mbox{and}\quad   \hat{\partial}_{+}x>0\quad \mbox{in}~ \Delta_h\setminus \wideparen{F_hG_h}.
 \end{equation}

Let $Z_1(r_1, s_1)$ and $Z_2(r_2, s_2)$ be two arbitrary points in $\Delta$.
 If $r_1=r_2$ or $s_1=s_2$, then by (\ref{92503aa}) we have $x(r_1, s_1)\neq x(r_2, s_2)$. As in the proof of Lemma \ref{lem4.7}, we have $x(r_1, s_1)\neq x(r_2, s_2)$
for $r_1<r_2$ and $s_1>s_2$.
We now discuss the case of $r_1<r_2$, $s_1<s_2$. Suppose $x(r_2, s_2)=x(r_1, s_1)$, then by (\ref{92503aa}) and (\ref{113006a}) there exists a curve $\widehat{Z_1Z_2}$ in $\Delta$, which connects the points $Z_1$ and $Z_2$, such that $$
x(r, s)=x(r_1, s_1), \quad (r, s)\in \widehat{Z_1Z_2}.
$$
Then,
as in the proof of Lemma \ref{lem4.7} we have  $y(r_2, s_2)>y(r_1, s_1)$. So, $(x, y)(r_1, s_1)\neq (x,y)(r_2, s_2)$ for $r_1<r_2$ and $s_1<s_2$.
Thus, no two points from $(r, s)$ map to one point in the $(x, y)$-plane. So, the mapping $(x,y)=(x, y)(r, s)$, $(r, s)\in \Delta$ can be inversed back to the physical plane. This leads to the existence of a classical solution to the singular Cauchy problem.
From the one-to-one property, (\ref{101401}) and the definition of $C_{+}$ characteristic curves we can see that the curve $\wideparen{FG}$ is an envelope of all the $C_{+}$ characteristic curves in $\overline{\Omega}$.

Let $\overline{\Omega}=\{(x, y)\mid (x,y)=(x, y)(r, s), (r, s)\in \Delta\}$.
By (\ref{53002}) we also have that the solution satisfies
 $$
\bar{\partial}_{\pm}\rho<0\quad \mbox{in}~ \overline{\Omega}\setminus\wideparen{FG};\quad
 \bar{\partial}_{-}\rho=-\infty  \quad\mbox{on}~ \wideparen{FG}.
 $$
This completes the proof.
\end{proof}

\begin{rem}
If the assumption $r(G)-s(F)<2\mathcal{R}$ is not hold, we can still obtain a global classical solution to the singular Cauchy problem.  We can inset finite points $G_i$, $i=0, 1, 2, 3, \cdot\cdot\cdot n$ into $\widetilde{FG}$, where $G_0=F$ and $G_n=G$, such that  $r(G_{i})-s(G_{i-1})<2\mathcal{R}$, $i=1, 2, 3, \cdot\cdot\cdot n$. Through solving a finite number of singular Cauchy problems and Goursat problems, we can obtain a global solution to the original problem.
The determinate region of the singular Cauchy problem is an infinite region bounded by $\widetilde{FG}$, a forward $C_{+}$ characteristic curve issued from $F$, and a forward $C_{-}$ characteristic curve issued from $G$.
\end{rem}

\section{Oblique wave interactions within the divergent duct}

\subsection{Interaction of fan waves}
In this part we study the IBVP (\ref{E1}), (\ref{42701})
for $\tau_0>\tau_2^i$.
 The main result of this subsection is stated as the following theorem.
\begin{thm}\label{thm1}
Assume $\tau_0>\tau_2^i$ and $u_0>c_0$. Then the IBVP (\ref{E1}), (\ref{42701}) admits a global continuous and piecewise smooth solution in $\Sigma$. Moreover, if $\max\{\theta_{+}, -\theta_{-}\}<\mathcal{R}$ then there is no vacuum inside the duct.
\end{thm}

\begin{figure}[htbp]
\begin{center}
\includegraphics[scale=0.42]{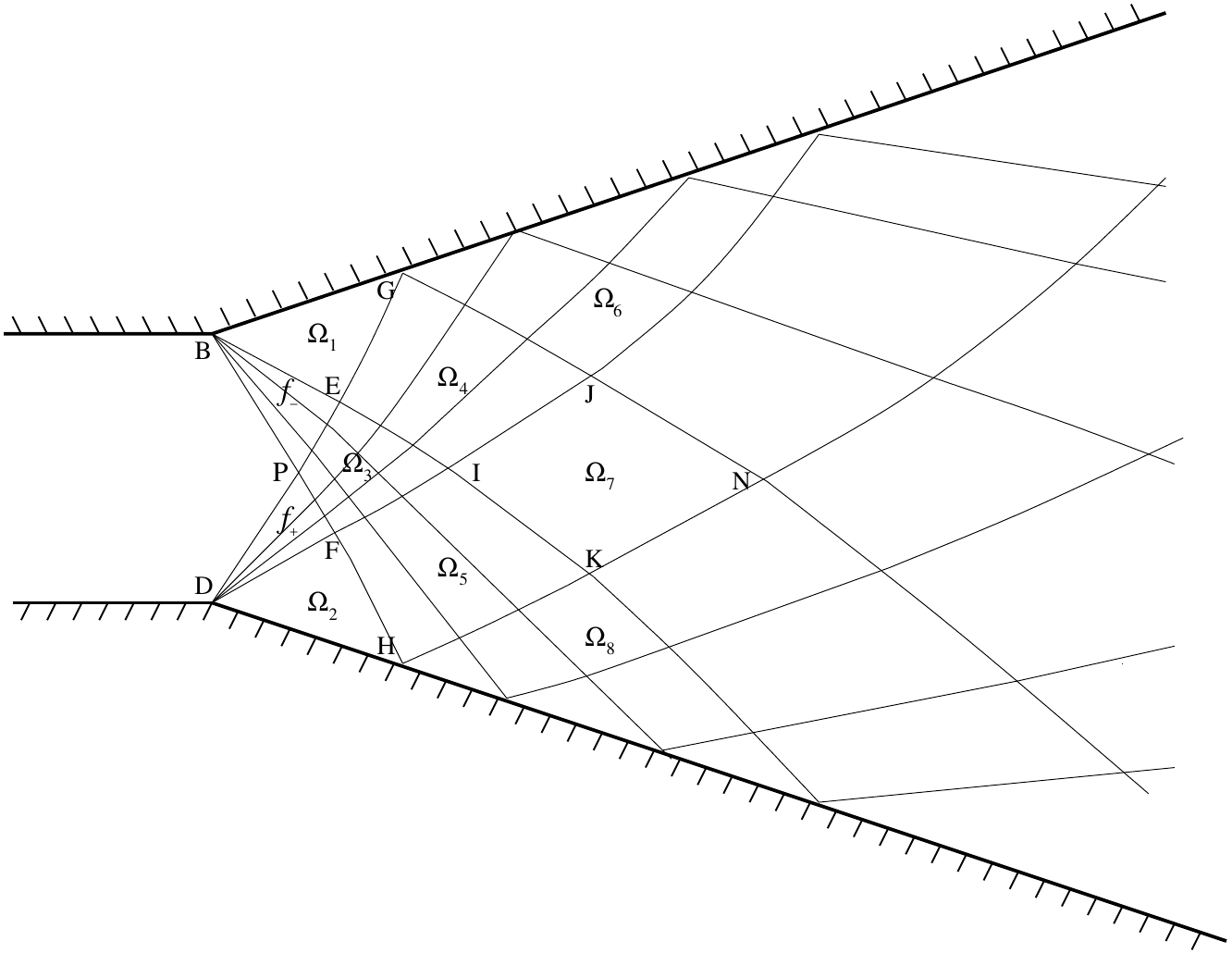}
\caption{\footnotesize Interactions and reflections of simple waves in the duct. }
\label{Fig5.1.1}
\end{center}
\end{figure}

When $\tau_0>\tau_2^i$ and $u_0>c_0$,
 there are two centered simple waves with different types issued from the corners $B$ and $D$; see Fig. \ref{Fig5.1.1}.
We first assume $\max\{\theta_{+}, -\theta_{-}\}<\mathcal{R}$.
Then the  flow near the corners  $B$ and $D$ can be represented by
$$
(u, v)=\left\{
               \begin{array}{ll}
                 (u_0, 0), & \hbox{$-\frac{\pi}{2}\leq \eta\leq \eta_0$;} \\[2pt]
                 (\tilde{u}, \tilde{v})(\eta), & \hbox{$\eta_0\leq \eta\leq \eta_2$;} \\[2pt]
                 (u_1, v_1), & \hbox{$\eta_1<\eta<\theta_{+}$,}
               \end{array}
             \right.
\quad \mbox{and}\quad
(u, v)=\left\{
               \begin{array}{ll}
                 (u_0, 0), & \hbox{$\xi_0\leq \xi\leq \frac{\pi}{2}$;} \\[2pt]
                 (\bar{u}, \bar{v})(\xi), & \hbox{$\xi_2\leq \xi\leq \xi_0$;} \\[2pt]
                 (u_2, v_2), & \hbox{$\theta_{-}<\xi<\xi_2$,}
               \end{array}
             \right.
$$
respectively.
Here, $\eta_0=-\xi_0=-\arctan(\frac{c_0}{u_0})$;
$(\bar{u}, \bar{v})(\xi)$ is determined by (\ref{102306}) and (\ref{102304}) with initial data $(\bar{q}, \bar{\tau}, \bar{\sigma})(\xi_0)=(u_0, \tau_0, 0)$;
$(\tilde{u}, \tilde{v})(\eta)$ is determined by (\ref{102306a}) and (\ref{102304a}) with initial data $(\tilde{q}, \tilde{\tau}, \tilde{\sigma})(\eta_0)=(u_0, \tau_0, 0)$;
$\eta_1$ is  determined by $\tilde{v}(\eta_1)=\tilde{u}(\eta_1)\tan\theta_{+}$;
$\xi_2$ is  determined by $\hat{v}(\xi_2)=\hat{u}(\xi_2)\tan\theta_{-}$;
 $(u_1, v_1)=(\tilde{u}, \tilde{v})(\eta_{1})$ and $(u_2, v_2)=(\bar{u}, \bar{v})(\xi_2)$.
For convenience, we set
\begin{equation}\label{def12}
q_k=\sqrt{u_k^2+v_k^2}, \quad\tau_k=\hat{\tau}(q_k), \quad c_k=\tau_k\sqrt{-p'(\tau_k)}, \quad \mbox{and}\quad A_k=\arcsin\Big(\frac{c_k}{q_k}\Big), \quad k=1,~ 2.
\end{equation}

We use $f_{+}$ and $f_{-}$ to denote the centered simple waves issued from the points $D$ and $B$ respectively.
The two centered waves start to interact with each other inside the duct at the point $P(\cot \xi_0, 0)$.
 We draw a forward $C_{-}$ characteristic curve from the point $P$. This characteristic curve then delimits the centered simple wave $f_{+}$ region.
 We describe this characteristic curve by the following one-parametric form:
 $$x=x_{-}(\xi):=\chi(\xi)\cos\xi,\quad  y=y_{-}(\xi):=\chi(\xi)\sin\xi-1, \quad \xi_2\leq \xi\leq \xi_0.$$
Then by $$\cos\big(\alpha(x_{-}, y_{-})-2A(x_{-}, y_{-})\big) \frac{{\rm d}y_{-}}{{\rm d}\xi}-\sin\big(\alpha(x_{-},y_{-})-2A(x_{-}, y_{-})\big) \frac{{\rm d}x_{-}}{{\rm d}\xi}~=~0,$$
$\alpha(x_{-}(\xi), y_{-}(\xi))=\xi$, and $A(x_{-}(\xi), y_{-}(\xi))=\bar{A}(\xi)$,
 we have
that $\chi(\xi)$ satisfies
\begin{equation}\label{62302}
\left\{
  \begin{array}{ll}
   \chi'(\xi)=-\cot \big(2\bar{A}(\xi)\big) \chi(\xi), \quad \xi_1\leq \xi\leq \xi_2,\\[4pt]
   \chi(\xi_0)=\csc \xi_0.
  \end{array}
\right.
\end{equation}
So, by (\ref{102503}), this  characteristic line crosses the whole simple wave $f_{+}$
and ends up at a point $F(r(\xi_2)\cos\xi_2, r(\xi_2)\sin\xi_2-1)$.
We denote by $\widetilde{PF}$ this $C_{-}$ cross characteristic line.

According to $\bar{\rho}'(\xi)>0$, we also have
\begin{equation}\label{62303}
\begin{aligned}
\bar{\partial}_-\rho&~=~\bar{\rho}'(\xi)\big(\cos\beta\xi_{x}+\sin\beta\xi_{y}\big)
~=~\bar{\rho}'(\xi)\big(\cos(\xi-2\bar{A}(\xi))\xi_{x}+\sin(\xi-2\bar{A}(\xi))\xi_{y}\big)\\&~=~-\frac{\sin (2\bar{A}(\xi))}{\sqrt{x^2+(y+1)^2}}\cdot \bar{\rho}'(\xi)<0
\quad \mbox{on}~ \widetilde{{PF}}.
\end{aligned}
\end{equation}
Combining this with (\ref{52801}),  (\ref{102302}), and $\sigma(F)=\bar{\sigma}(\xi_2)=\arctan(\frac{\bar{v}(\xi_2)}{\bar{u}(\xi_2)})=\theta_{-}$, we also have
\begin{equation}\label{1202}
(r, s)\in \Pi\quad \mbox{on}~ \widetilde{{PF}}.
\end{equation}

Similarly,
we draw a forward $C_{+}$ characteristic line from the point $P$. This characteristic line then delimit the centered simple wave $f_{-}$ region
and ends up at a point $E$ on the straight $C_{-}$ characteristic line $y=1+x\tan \eta_1$. Moreover,
\begin{equation}\label{101501}
\begin{aligned}
\bar{\partial}_{+}\rho~=~\frac{\sin (2\tilde{A}(\eta))}{\sqrt{x^2+(y-1)^2}}\cdot \bar{\rho}'(\eta)<0
\quad \mbox{on}~ \widetilde{{PE}}.
\end{aligned}
\end{equation}
Combining this with (\ref{52801}),  (\ref{102302}), and $\sigma(E)=\tilde{\sigma}(\eta_1)=\arctan(\frac{\tilde{v}(\eta_1)}{\tilde{u}(\eta_1)})=\theta_{+}$, we also have
\begin{equation}\label{1203}
(r, s)\in \Pi\quad \mbox{on}~ \widetilde{{PE}}.
\end{equation}

We now consider (\ref{E1}) with the boundary condition
\begin{equation}\label{101502a}
(u, v)=\left\{
               \begin{array}{ll}
                  (\bar{u}, \bar{v})(\xi), & \hbox{$(x, y)\in\widetilde{PF}$;} \\[2pt]
               (\tilde{u}, \tilde{v})(\eta), & \hbox{$(x, y)\in\widetilde{PE}$.}
               \end{array}
             \right.
\end{equation}
The problem  (\ref{E1}), (\ref{101502a}) is a standard Goursat problem.

From (\ref{62303})--(\ref{1203}) we see that
 the boundary condition (\ref{101502a}) satisfy the assumptions in Lemma \ref{lem1}. Then the SGP (\ref{E1}), (\ref{101502a}) admits a global classical solution on a curvilinear quadrilateral domain $\overline{\Omega}_3$ bounded by $\widetilde{PF}$, $\widetilde{PE}$, $C_{+}^{F}$, and  $C_{-}^{E}$; see Fig. \ref{Fig5.1.1}. Moreover, the solution satisfies
the following estimates in $\overline{\Omega}_{3}$:
\begin{equation}\label{102504}
(r, s)\in\Pi, \quad\bar{\partial}_{\pm}\rho<0, \quad \bar{\partial}_{+}r>0,\quad \bar{\partial}_{-}s<0, \quad
\bar{\partial}_{+}\sigma>0,\quad \mbox{and}\quad \bar{\partial}_{-}\sigma<0.
\end{equation}

By (\ref{102504})  we also have
$\alpha=\sigma+A>\sigma\geq\theta_{-}$ on   $C_{+}^{F}$ and
$\beta=\sigma-A<\sigma\leq \theta_{+}$ on  $C_{-}^{E}$.
So, $C_{+}^{F}$ and $C_{-}^{E}$ do not intersect $\mathcal{W}_{+}$ and $\mathcal{W}_{-}$ respectively.
This immediately implies that the domain $\overline{\Omega}_3$ lies inside the divergent duct.
For convenience, we denote the solution of the SGP (\ref{E1}), (\ref{101502a}) by $(u, v)=(u_3, v_3)(x, y)$, $(x,y)\in\overline{\Omega}_3$.

We draw a forward straight $C_{+}$ characteristic line from $E$ with the inclination angle $\alpha=\theta_{+}+A_1>\theta_{+}$. This line intersects $\mathcal{W}_{+}$ at a point $G$. Then the flow in the triangle domain $\Omega_1$ bounded by the straight characteristic lines $\overline{BE}$, $\overline{EG}$, and the wall $\mathcal{W}_{+}$ is the constant state $(u_1, v_1)$.
We draw a forward straight $C_{-}$ characteristic line from $F$ with the inclination angle $\beta=\theta_{-}-A_2<\theta_{-}$. This line intersects $\mathcal{W}_{-}$ at a point $H$. Then the flow in the triangle domain $\Omega_2$ bounded by the straight characteristic lines $\overline{DF}$, $\overline{FH}$, and the wall $\mathcal{W}_{-}$ is the constant state $(u_2, v_2)$.

When $r(E)-s(F)<2\mathcal{R}$, the characteristic lines $C_{+}^{F}$ and  $C_{-}^{E}$ intersect at a point $I$. We next consider (\ref{E1}) with the boundary condition
\begin{equation}\label{101503}
(u, v)=\left\{
               \begin{array}{ll}
                  (u_1, v_1), & \hbox{$(x, y)\in\overline{EG}$;} \\[2pt]
               (u_3, v_3)(x,y), & \hbox{$(x, y)\in \widetilde{EI}$.}
               \end{array}
             \right.
\end{equation}
By Lemma \ref{lem3}, the BVP (\ref{E1}), (\ref{101503}) admits a simple wave solution with straight $C_{+}$ characteristic lines on a curvilinear quadrilateral domain $\overline{\Omega}_4$ bounded by $\overline{EG}$, $\widetilde{EI}$, $\widetilde{GJ}$, and $\overline{IJ}$, where $\widetilde{GJ}$ is a forward $C_{-}$ characteristic line issued from $G$ and $\overline{IJ}$ is a forward straight $C_{+}$ characteristic line issued from $I$.
Moreover, the solution satisfies $(r, s)\in \Pi$, $\bar{\partial}_{-}\rho<0$, and $\bar{\partial}_{+}\rho=0$ in $\overline{\Omega}_4$.
The domain $\Omega_4$ is inside the duct.
We denote this simple wave solution by $(u, v)=(u_4, v_4)(x, y)$, $(x,y)\in \overline{\Omega}_4$.

Similarly, we can find a simple wave solution with straight $C_{-}$ characteristic lines on a domain $\overline{\Omega}_5$ bounded by
$\overline{FH}$, $\widetilde{FI}$, $\widetilde{HK}$ and $\overline{IK}$, where $\widetilde{HK}$ is a forward $C_{+}$ characteristic line issued from $F$ and $\overline{IK}$ is a forward $C_{-}$ characteristic line issued from $I$.
The solution satisfies $(r, s)\in\Pi$, $\bar{\partial}_{-}\rho=0$ and $\bar{\partial}_{+}\rho<0$ in $\overline{\Omega}_5$. We denote this simple wave solution by $(u, v)=(u_5, v_5)(x, y)$, $(x,y)\in \overline{\Omega}_5$.

Next, we  consider (\ref{E1}) with the boundary condition
\begin{equation}\label{101601}
\left\{
               \begin{array}{ll}
                 (u, v)=(u_5, v_5)(x,y), & \hbox{$(x, y)\in \widetilde{HK}$;} \\[2pt]
               v=u\tan \theta_{-}, & \hbox{$(x, y)\in \mathcal{W}_{-}$.}
               \end{array}
             \right.
\end{equation}
By Lemma \ref{lem5}, the problem (\ref{E1}), (\ref{101601}) admits a $C^1$ solution on a curvilinear triangle domain $\overline{\Omega}_8$ bounded by $\mathcal{W}_{-}$, $\widetilde{HK}$, and $C_{-}^{K}$. Moreover, the solution satisfies $(r, s)\in\Pi$ and $\bar{\partial}_{\pm}\rho<0$ in $\overline{\Omega}_8$.
Similarly, we can find the solution on a curvilinear triangle domain $\overline{\Omega}_6$ bounded by $\widetilde{GJ}$, $\mathcal{W}_{+}$, and $C_{+}^J$. Moreover, the solution satisfies $(r, s)\in\Pi$ and $\bar{\partial}_{\pm}\rho<0$ in $\overline{\Omega}_6$.

We next consider (\ref{E1}) with the boundary condition
\begin{equation}\label{101605}
 (u, v)=\left\{
               \begin{array}{ll}
(u, v)(I), & \hbox{$(x, y)\in \overline{IJ}$;} \\[2pt]
             (u, v)(I), & \hbox{$(x, y)\in \overline{IK}$.}
               \end{array}
             \right.
\end{equation}
The problem admits a constant state solution on a parallelogram domain $\overline{\Omega}_7$ bounded by straight characteristic lines $\overline{IJ}$, $\overline{IK}$, $\overline{JN}$, and $\overline{KN}$.

Repeating the above processes and
 solving a finite number of standard Goursat problems, mixed boundary value problems, and simple wave problems, we can obtain a global continuous and piecewise smooth solution inside the semi-infinite divergent duct.

The discussion for the case of $r(E)-s(F)\geq 2\mathcal{R}$ is similar to that for $r(E)-s(F)<2\mathcal{R}$. Actually, by solving a standard Goursat problem and two mixed boundary value problems we can obtain a global continuous and piecewise smooth solution inside the semi-infinite divergent duct. We omit the details.

The discussion for the case of $\max\{\theta_{+}, -\theta_{-}\}\geq\mathcal{R}$ is similar.
If $\theta_{+}>\mathcal{R}$,  there is a cavitation between the wall $\mathcal{W}_{+}$  and a vacuum boundary $y=1+x\tan\mathcal{R}$, $x>0$.
If $-\theta_{-}>\mathcal{R}$, there is a cavitation between the wall $\mathcal{W}_{-}$ and a vacuum boundary $y=-1-x\tan\mathcal{R}$, $x>0$.
This completes the proof of Theorem \ref{thm1}.
\vskip 8pt
\subsection{Interaction of shock-fan composite waves}
We study the IBVP (\ref{E1}), (\ref{42701})
for $\tau_1^e<\tau_0<\tau_2^i$.
In this part we assume that the oblique waves at the corners $B$ and $D$ are shock-fan composite waves.
The main result of this subsection is stated as the following theorem.
\begin{thm}\label{thm2}
Assume $\tau_1^e<\tau_0<\tau_2^i$ and $u_0>c_0$.
Assume furthermore that the oblique waves at the corners $B$ and $D$ are shock-fan composite waves. Then the IBVP (\ref{E1}), (\ref{42701}) admits a global  piecewise smooth solution in $\Sigma$.
\end{thm}

Firstly, we assume that
the solution the corners  $B$ and $D$ can be represented by
$$
(u, v)=\left\{
               \begin{array}{ll}
                 (u_0, 0), & \hbox{$-\frac{\pi}{2}\leq \eta\leq -\phi_{po}$;} \\[2pt]
                 (\tilde{u}, \tilde{v})(\eta), & \hbox{$-\phi_{po}\leq \eta\leq \eta_1$;} \\[2pt]
                 (u_1, v_1), & \hbox{$\eta_1<\eta<\theta_{+}$,}
               \end{array}
             \right.
~\mbox{and}\quad
(u, v)=\left\{
               \begin{array}{ll}
                 (u_0, 0), & \hbox{$\phi_{po}<\xi<\frac{\pi}{2}$;} \\[2pt]
                 (\bar{u}, \bar{v})(\xi), & \hbox{$\xi_2<\xi<\phi_{po}$;} \\[2pt]
                 (u_{2}, v_{2}), & \hbox{$\theta_{-}<\xi<\xi_2$,}
               \end{array}
             \right.
$$
respectively; see Fig. \ref{Fig5.2.1}. Here, the function $(\bar{u}, \bar{v})(\xi)$ is determined by (\ref{102306}) and (\ref{102304}) with initial data $(\bar{q}, \bar{\tau}, \bar{\sigma})(\phi_{po})=(q_{po}, \tau_{po}, \sigma_{po})$; the function
$(\tilde{u}, \tilde{v})(\eta)$ is determined by (\ref{102306a}) and (\ref{102304a}) with initial data $(\tilde{q}, \tilde{\tau}, \tilde{\sigma})(-\phi_{po})=(q_{po}, \tau_{po}, -\sigma_{po})$;
$\eta_1$ is determined by $\tilde{v}(\eta_1)=\tilde{u}(\eta_1)\tan\theta_{+}$;
$\xi_2$ is  determined by $\bar{v}(\xi_2)=\bar{u}(\xi_2)\tan\theta_{-}$;
the constants $q_{po}$, $\sigma_{po}$, and $\phi_{po}$ are the same as that defined in Section 3.3.2. 

For convenience, we use ${\it sf}_{+}$ and ${\it sf}_{-}$ to denote the shock-fan composite waves issued from points $D$ and $B$, respectively.
The two shock-fan composite waves ${\it sf}_{\pm}$ start to interact with each other inside the duct at point $P(\cot \phi_{po}, 0)$.
From the point $P$, we draw a forward $C_{-}$ ($C_{+}$, resp.) cross characteristic line $\widetilde{PF}$ ($\widetilde{PE}$, resp.) of the centered simple wave of ${\it sf}_{+}$ (${\it sf}_{-}$, resp.).
As in (\ref{62302})--(\ref{62303}), we have
\begin{equation}\label{102604}
\bar{\partial}_{+}\rho<0
\quad \mbox{on}~ \widetilde{{PE}}; \quad \bar{\partial}_{-}\rho<0
\quad \mbox{on}~ \widetilde{{PF}}.
\end{equation}

\begin{figure}[htbp]
\begin{center}
\includegraphics[scale=0.47]{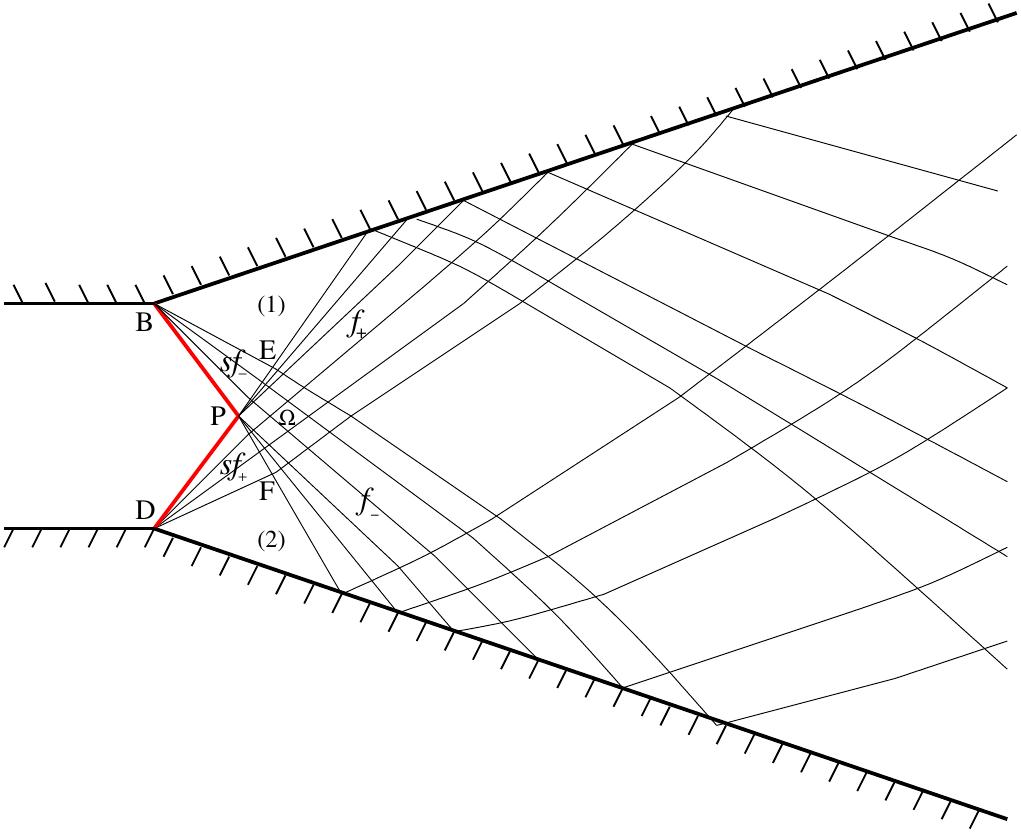}
\caption{\footnotesize Interaction of two shock-fan composite waves. }
\label{Fig5.2.1}
\end{center}
\end{figure}

We first consider (\ref{E1}) with the boundary condition
\begin{equation}\label{101502}
(u, v)=\left\{
               \begin{array}{ll}
                  (\bar{u}, \bar{v})(\xi), & \hbox{$(x, y)\in\widetilde{PF}$;} \\
               (\tilde{u}, \tilde{v})(\eta), & \hbox{$(x, y)\in\widetilde{PE}$.}
               \end{array}
             \right.
\end{equation}

From $\bar{\tau}(\phi_{po})= \tilde{\tau}(-\phi_{po})=\tau_{po}$, we have
\begin{equation}\label{102601}
(r-s)\big(\tilde{u}(-\phi_{po}), \tilde{v}(-\phi_{po})\big)=
(r-s)\big(\bar{u}(\phi_{po}), \bar{v}(\phi_{po})\big).
\end{equation}
From $(\bar{u}, \bar{v})(\phi_{po})=(u_{po}, v_{po})$ and $(\tilde{u}, \tilde{v})(-\phi_{po})=(u_{po}, -v_{po})$, we have
\begin{equation}\label{102602}
(r+s)\big(\tilde{u}(-\phi_{po}), \tilde{v}(-\phi_{po})\big)=-2\sigma_{po}, \quad
(r+s)\big(\bar{u}(\phi_{po}), \bar{v}(\phi_{po})\big)=2\sigma_{po}.
\end{equation}
Combining (\ref{102601}) with (\ref{102602}) and recalling $\sigma_{po}<0$, we have
\begin{equation}\label{102603}
r\big(\tilde{u}(-\phi_{po}), \tilde{v}(-\phi_{po})\big)>r\big(\bar{u}(\phi_{po}), \bar{v}(\phi_{po})\big),\quad s\big(\tilde{u}(-\phi_{po}), \tilde{v}(-\phi_{po})\big)>s\big(\bar{u}(\phi_{po}), \bar{v}(\phi_{po})\big).
\end{equation}
Thus the problem (\ref{E1}), (\ref{101502}) is a DGP.
Since the waves issued from the corners $B$ and $D$ are assumed to be shock-fan composite waves, by (\ref{102503}) we have $\theta_{-}<\sigma_{po}<-\sigma_{po}<\theta_{+}$. Hence, by (\ref{102604}) we also have
\begin{equation}\label{120201}
(r, s)\in \Pi\quad \mbox{on}~ \widetilde{PE}\cup \widetilde{PF}.
\end{equation}

From (\ref{102604}), (\ref{102603}), and (\ref{120201}) we see that the assumptions of Lemma \ref{lem8} hold.
Then by Lemma \ref{lem8} and Remark \ref{rem4.3} we have the following conclusion.
\begin{itemize}
  \item If $r\big(\tilde{u}(-\phi_{po}), \tilde{v}(-\phi_{po})\big)-s\big(\bar{u}(\phi_{po}), \bar{v}(\phi_{po})\big)<2\mathcal{R}$, the DGP (\ref{E1}), (\ref{101502}) admits a continuous and piecewise smooth solution on a curvilinear quadrilateral domain $\overline{\Omega}\setminus\{P\}$ bounded by characteristic curves $\widetilde{PE}$, $\widetilde{PF}$, $C_{+}^{F}$, and $C_{-}^{E}$. Moreover, the solution satisfies
$$
(r, s)\in \Pi\quad \mbox{and}\quad \bar{\partial}_{\pm}\rho<0\quad \mbox{in}~ \overline{\Omega}\setminus\{P\}.
$$
  \item If $r\big(\tilde{u}(-\phi_{po}), \tilde{v}(-\phi_{po})\big)-s\big(\bar{u}(\phi_{po}), \bar{v}(\phi_{po})\big)\geq 2\mathcal{R}$, the DGP (\ref{E1}), (\ref{101502}) admits a solution  consisting of two centered waves issued from the point $P$ and a cavitation. Moreover, the centered waves satisfy $(r, s)\in \Pi$ and $\bar{\partial}_{\pm}\rho<0$.
\end{itemize}

Next,  by solving two simple wave problems as discussed in Section 5.1, one can see that the interaction of ${\it sf}_{\pm}$ generates two simple waves, denoted by ${\it f}_{\pm}$; see Fig. \ref{Fig5.2.1}. These two simple waves will then reflect off the walls and generate two new simple waves.
The flow in the remaining part of the divergent duct consists of some constant sates, expansion simple waves, and interactions and reflections of expansion simple waves. These wave structures can be obtained by solving a finite number of SGPs, MBVPs, and simple wave problems.
Since the discussion is similar to that in Section 5.1, we omit the details.

\vskip 8pt
\subsection{Interaction of a shock wave with a shock-fan composite wave}
 We study the IBVP (\ref{E1}), (\ref{42701})
for $\tau_1^i<\tau_0<\tau_2^i$.
In this part we assume that the oblique waves at the corners $B$ and $D$ are shock wave and shock-fan composite wave, respectively. For convenience, we denote by ${\it s}_{-}$ the oblique shock at the point $B$, and by ${\it sf}_{+}$ the shock-fan composite wave at the point $D$.
We assume further that the flow near the corners  $B$ and $D$ can be represented by
$$
(u, v)=\left\{
               \begin{array}{ll}
                 (u_0, 0), & \hbox{$-\frac{\pi}{2}\leq \eta<\eta_{s}$;} \\[2pt]
                 (u_1, v_1), & \hbox{$\eta_{s}<\eta<\theta_{+}$}
               \end{array}
             \right. \mbox{and}\quad(u, v)=\left\{
               \begin{array}{ll}
                 (u_0, 0), & \hbox{$\phi_{po}\leq \xi\leq\frac{\pi}{2}$;} \\[2pt]
                 (\bar{u}, \bar{v})(\xi), & \hbox{$\xi_2\leq \xi\leq \phi_{po}$;} \\[2pt]
                 (u_2, v_2), & \hbox{$\theta_{-}<\xi<\xi_2$,}
               \end{array}
             \right.
$$
respectively.
Here, the oblique shock-fan composite wave solution near the point $D$ is the same as that defined in (\ref{8230a});
the constant state $(u_1, v_1)$ is connected on the right to $(u_0, 0)$ by an oblique 1-shock with the inclination angle $\eta_s$.
The main result of this subsection can be stated as the following theorem.
\begin{thm}\label{thm3}
Assume $\tau_1^i<\tau_0<\tau_2^i$ and $u_0>c_0$.
Assume furthermore that the oblique waves at the corners $B$ and $D$ are shock wave and shock-fan composite wave, respectively. Then the
 IBVP (\ref{E1}), (\ref{42701}) admits a global piecewise smooth solution in $\Sigma$.
\end{thm}

\begin{figure}[htbp]
\begin{center}
\includegraphics[scale=0.5]{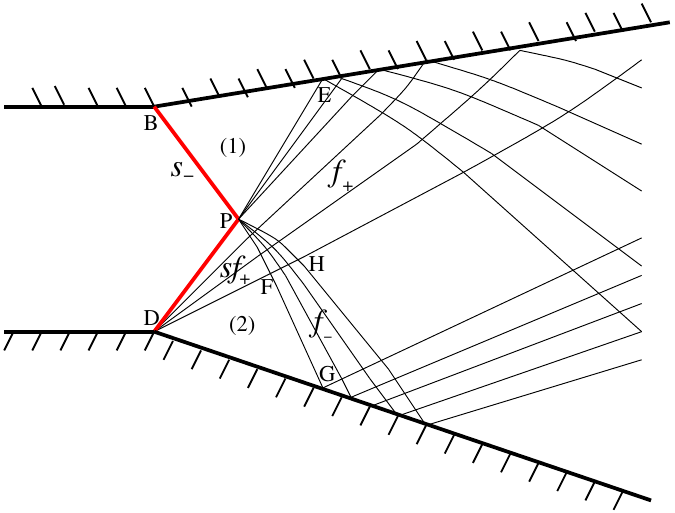} \qquad \qquad\quad\includegraphics[scale=0.5]{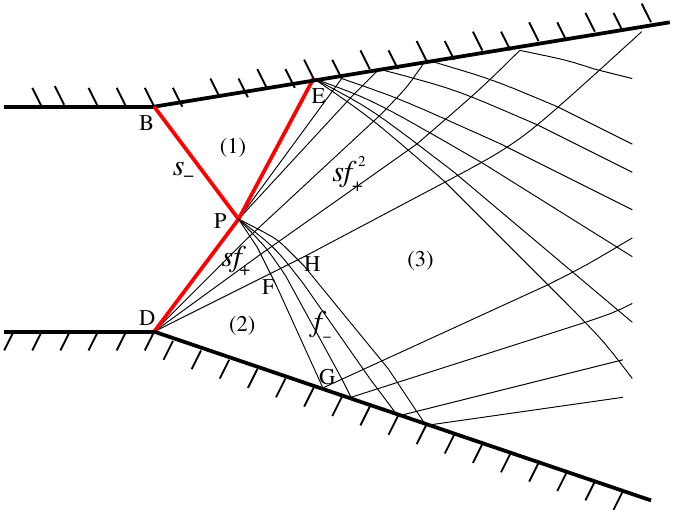}
\caption{\footnotesize Interaction of a shock wave with a shock-fan composite wave. }
\label{Fig5.3.1}
\end{center}
\end{figure}

\subsubsection{\bf Rarefaction wave curves}
The two oblique waves start to interact with each other at a point $P(x_{_P}, y_{_P})$, where $x_{_{P}}=2(\tan\phi_{po}-\tan\eta_s)^{-1}$ and $y_{_{P}}=(\tan\phi_{po}+\tan\eta_s)(\tan\phi_{po}-\tan\eta_s)^{-1}$.
In order to determine the oblique waves generated by the interaction of  ${\it s}_{-}$ with ${\it sf}_{+}$, we need to analyze the intersections of the wave curve ${\it W}_2(u_1, v_1)$ with the wave curves  ${\it W}_1(u_{po}, v_{po})$
and ${\it W}_1(u_{2}, v_{2})$, where $(u_{po}, v_{po})$ is the state on the backside of the post-sonic shock of ${\it sf}_{+}$.

\begin{lem}
For any $U\in {\it W}_{1}(u_0, 0)$, the wave curve ${\it W}_{2}(U)$ does not intersect the wave curve ${\it W}_{2}(u_0, 0)$.
\end{lem}
\begin{proof}
Firstly, by Lemma \ref{120301} we know that the wave curves ${\it W}_{1}(u_0, 0)$ and ${\it W}_{2}(u_0, 0)$ does not intersect with each other, provided that $u_0$ is large.
In order to prove this, it suffices to prove that for any $(u_*, v_*)\in  {\it W}_{2}(u_0, 0)$, there does not exist any point $U\in {\it W}_{1}(u_0, v_0)$ such that  $(u_*, v_*)\in {\it W}_{1}(U)$.
Set $\tau_*=\hat{\tau}(q_*)$ where $q_*=\sqrt{u_*^2+v_*^2}$.
We divide the discussion into the following three cases: (i) $\tau_0<\tau_*\leq \tau_2^i$;
(ii) $\tau_2^i<\tau_*\leq \tau_{po}$; (iii) $\tau_*>\tau_{po}$.
For the needs of the following discussion, we define
$$
m(\tau)=\sqrt{-\frac{2h(\tau)-2h(\tau_*)}{\tau^2-\tau_*^2}}, \quad \delta_0=\arctan\left(\frac{v_*}{u_*-u_0}\right),\quad \mbox{and}\quad \sigma_*=\arctan\Big(\frac{v_*}{u_*}\Big).
$$
\vskip 4pt

\noindent
{\bf Case (i): $\tau_0<\tau_*\leq \tau_2^i$}. We define
\begin{equation}\label{111301}
u_f=N_f\sin\phi+L_f\cos\phi, \quad v_f=-N_f\cos\phi+L_f\sin\phi, \quad \phi=\sigma_*+\arcsin\bigg(\frac{\tau_*m(\tau)}{q_*}\bigg),
\end{equation}
where
\begin{equation}\label{11701}
N_f=\tau m(\tau)\quad \mbox{and} \quad L_f=q_*\cos(\phi-\sigma_*).
\end{equation}
Then $u_f$, $v_f$, and $\phi$ can be seen as functions of $\tau\in [\tau_0, \tau_*)$.
We denote by ${\it W}^{2}(u_*, v_*)$ the one-parametric curve $(u, v)=(u_f, v_f)(\tau)$, $\tau\in [\tau_0, \tau_*)$. Then for any $U\in {\it W}^{2}(u_*, v_*)$, the state $U$ can be connected to $(u_*, v_*)$ on the upstream by an oblique 2-shock.
Since $(u_*, v_*)\in  {\it S}_{2}(u_0, 0)$, we have $u_f(\tau_0)=u_0$ and $v_{f}(\tau_0)=0$.

Let $\delta=\arctan(\frac{v_*-v_f}{u_*-u_f})$. Then by (\ref{111301}) we have $\phi=\delta+\frac{\pi}{2}$.
By (\ref{43012}), we have
$$
m'(\tau)=-\frac{\tau}{(\tau^2-\tau_*^2)m(\tau)}\big(p'(\tau)+m^2(\tau)\big)>0\quad \mbox{for}~ \tau_0\leq \tau<\tau_*.
$$
So, we have
$\frac{{\rm d}\phi}{{\rm d}\tau}>0$ for $\tau_0\leq \tau<\tau_*$.
This implies
\begin{equation}\label{111201}
\delta>\delta_0\quad \mbox{for} \quad \tau_0<\tau<\tau_*.
 \end{equation}

From the definition of 2-shock waves,
we have
\begin{equation}\label{111202}
\delta <\sigma_*\quad \mbox{for} \quad \tau_0\leq \tau<\tau_*.
\end{equation}
By (\ref{111201}) and (\ref{111202}) we know that the wave curve ${\it W}^{2}(u_*, v_*)$ lies in the region $\{(u, v)\mid \tau_0<\hat{\tau}(q)<\tau_*,~ \delta_0<\arctan\big(\frac{v-v_*}{u-u*}\big)<\sigma_*\}$; see Fig. \ref{Fig5.3.2} (left). This implies that ${\it W}^2(u_*, v_*)$ does not intersect the wave curve $W_{1}(u_0, 0)$.

\begin{figure}[htbp]
\begin{center}
\includegraphics[scale=0.48]{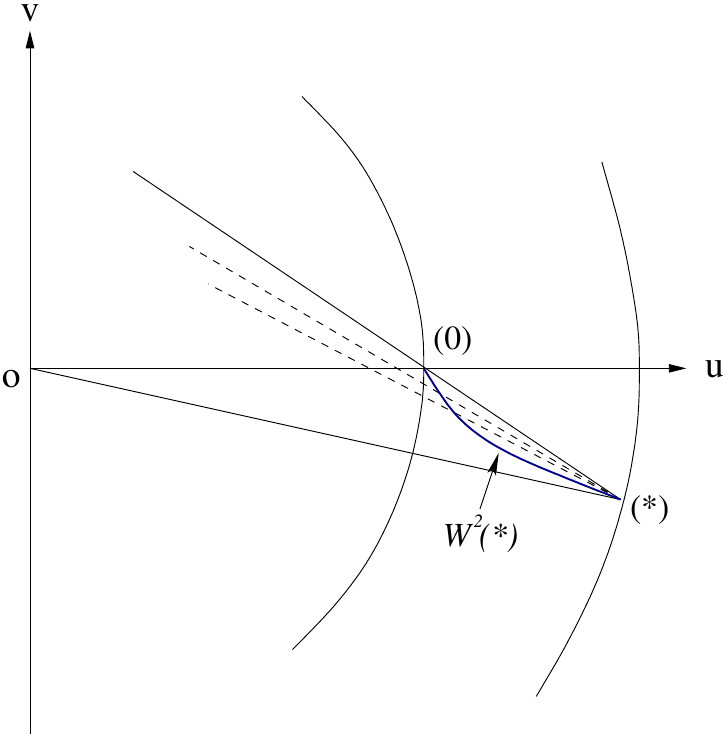}\qquad \qquad \qquad \includegraphics[scale=0.48]{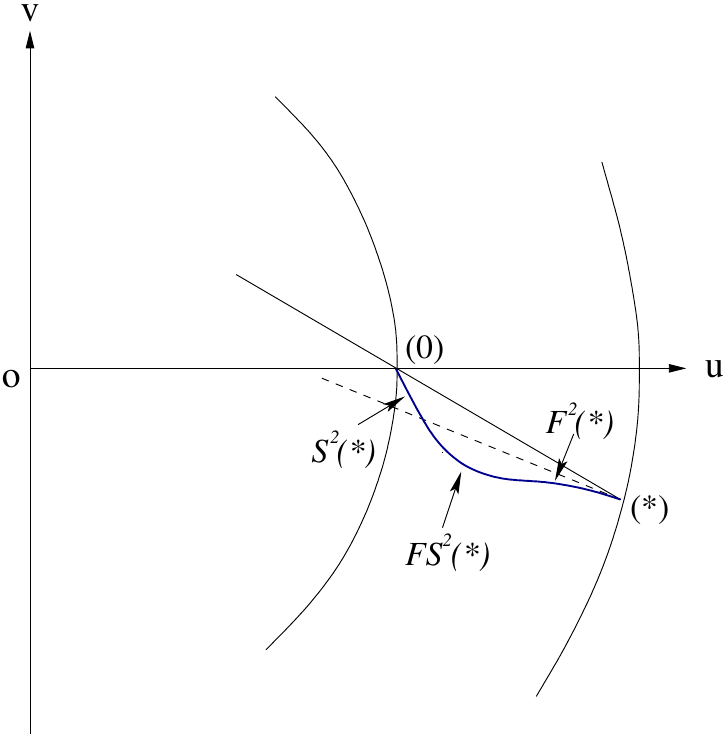}
\caption{\footnotesize Wave curve $W^2(u_*, v_*)$. Left: $\tau_0<\tau_*\leq \tau_2^i$; right: $\tau_2^i<\tau_*\leq \tau_{po}$. }
\label{Fig5.3.2}
\end{center}
\end{figure}

\vskip 4pt
\noindent
{\bf Case (ii): $\tau_2^i<\tau_*\leq \tau_{po}$.}
See Fig. \ref{Fig5.3.2} (right). Let 
$$
{\it F}^2(u_*, v_*)=\big\{(u, v)\mid  r(u, v)=r(u_*, v_*), ~\tau_2^i\leq \hat{\tau}(q)\leq \tau_*\big\}.
$$
Then, for any $U\in {\it F}^2(u_*, v_*)$, the state $U$ can be connected to $(u_*, v_*)$ on the left by a 2-fan wave.

Since $m(\tau_0)<\tau_*\hat{c}(q_*)$,
we have that along ${\it F}^2(u_*, v_*)$,
$$
\arctan\left(\frac{{\rm d}v}{{\rm d}u}\right)\bigg|_{(u_*, v_*)}=\alpha(u_*, v_*)-\frac{\pi}{2}=\sigma_*
+\hat{A}(q_*)-\frac{\pi}{2}>\delta_0.
$$
So, by (\ref{alphar}) we have
\begin{equation}\label{111303}
\alpha(u, v)>\sigma_*
+\hat{A}(q_*)\quad \mbox{and}\quad
\arctan\Big(\frac{v-v_*}{u-u_*}\Big)>\delta_0 \quad \mbox{for}~ (u, v)\in  {\it F}^2(u_*, v_*).
\end{equation}

Since $\tau_2^i<\tau_*<\tau_{po}$, we have $\tau_0<\psi_{po}^{-1}(\tau_*)<\tau_2^i$.
For $\tau\in (\psi_{po}^{-1}(\tau_*), \tau_2^i)$,
we define
$$
\tilde{u}_f=\tilde{N}_f\sin\tilde{\phi}+\tilde{L}_f\cos\tilde{\phi}, \quad \tilde{v}_f=-\tilde{N}_f\cos\tilde{\phi}+\tilde{L}_f\sin\tilde{\phi},\quad \tilde{\phi}=\sigma_b+\check{A}(\tau_b),
$$
where
$$
\quad \tilde{N}_f=\frac{\tau\check{c}(\tau_b)}{\tau_b}, \quad \tilde{L}_f=\check{q}(\tau_b)\cos\check{A}(\tau_b), \quad \tau_b=\psi_{po}(\tau), \quad \sigma_b=r(u_*, v_*)-\int_{u_0}^{\check{q}(\tau_b)}\frac{\sqrt{q^2-c^2}}{qc}{\rm d}q.
$$
Then $\tilde{u}_f$, $\tilde{v}_f$, $\tilde{\phi}$, $\tau_b$, and $\sigma_b$ can be seen as functions of $\tau\in [\psi_{po}^{-1}(\tau_*), \tau_2^i)$.
Let
$$
{\it SF}^2(u_*, v_*)=\big\{(u, v)\mid u=\tilde{u}_f(\tau), ~v=\tilde{v}_f(\tau), ~\tau\in[\psi_{po}^{-1}(\tau_*), \tau_2^i)\big\}.
$$ Then for any $U\in {\it SF}^2(u_*, v_*)$,
the state $U$ can be connected to $(u_*, v_*)$ on the upstream by an oblique 2-shock-fan composite wave.

Since the shock of the oblique shock-fan composite wave is post-sonic, by (\ref{111303}) we have
$$
\arctan\left(\frac{\tilde{v}_f(\tau)-v_b(\tau)}{\tilde{u}_f(\tau)-u_b(\tau)}\right)=\tilde{\phi}(\tau)-\frac{\pi}{2}=\sigma_b(\tau)+\check{A}(\tau_b(\tau))
-\frac{\pi}{2}>\delta_0
\quad \mbox{for}~ \tau\in [\psi^{-1}(\tau_*), \tau_2^i),
$$
where $(u_b, v_b)(\tau)=\hat{q}(\tau)(\cos\sigma_b(\tau), \sin\sigma_b(\tau))\in {\it F}^2(u_*, v_*)$.
Combining this with (\ref{111303}), we have
\begin{equation}\label{111305}
\arctan\Big(\frac{v-v_*}{u-u_*}\Big)>\delta_0 \quad \mbox{for}~ (u, v)\in  {\it SF}^2(u_*, v_*).
\end{equation}

Let ${\it S}^2(u_*, v_*)=\{(u, v)\mid (u, v)=(u_f, v_f)(\tau), \tau\in [\tau_0, \tau_*)\}$, where the functions
 $u_f(\tau)$ and $v_f(\tau)$ are the same as that defined in (\ref{111301}).
As in  (\ref{111201}) and (\ref{111202}) we have
\begin{equation}\label{111304}
\arctan\Big(\frac{v-v_*}{u-u_*}\Big)\geq \delta_0 \quad \mbox{for}~ (u, v)\in  {\it S}^2(u_*, v_*).
\end{equation}

Let ${\it W}^2(u_*, v_*)={\it F}^2(u_*, v_*)+{\it SF}^2(u_*, v_*)+{\it S}^2(u_*, v_*)$. Then  (\ref{111303})--(\ref{111304}) we have that ${\it W}^2(u_*, v_*)$ does not intersect the wave curve $W_{1}(u_0, 0)$.

\vskip 4pt

\noindent
{\bf Case (iii):  $\tau_*>\tau_{po}$.} In this case, we have
$$
{\it W}^2(u_*, v_*)={\it W}^2(u_{po}, v_{po})+ {\it W}_2(u_0, 0)\cap\{(u, v)\mid  \tau_{po}< \hat{\tau}(q)\leq \tau_*\}.
$$
From the previous step we know that ${\it W}^2(u_{po}, v_{po})$ does not  intersect the wave curve ${\it W}_{1}(u_0, 0)$.
By the result of the 2-fan wave curve, we know that the wave curves ${\it W}_2(u_0, 0)\cap\{(u, v)\mid  \tau_{po}< \hat{\tau}(q)\leq \tau_*\}$ and ${\it W}_{1}(u_0, 0)$ do not intersect with each other.


This completes the proof of the lemma.
\end{proof}

\begin{figure}[htbp]
\begin{center}
\includegraphics[scale=0.48]{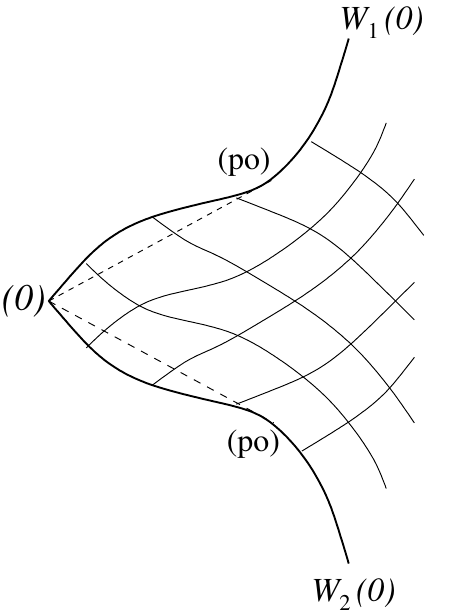}
\caption{\footnotesize Wave curves ${\it W}_1$ and ${\it W}_2$. }
\label{Fig5.3.3}
\end{center}
\end{figure}

\begin{lem}
Assume $\tau_0\in (\tau_1^i, \tau_2^i)$. Assume as well that $u_0$ is sufficiently large. Then for any $U'\in {\it W}_{1}(u_0, v_0)$ and $U''\in {\it W}_{2}(u_0, v_0)$, the wave curves ${\it W}_{1}(U'')$ and ${\it W}_{2}(U')$ have at most one intersection point before they reach the limiting circle $C_{\infty}:=\{(u, v)\mid u^2+v^2=q_{_\infty}^2\}$.
\end{lem}
\begin{proof}
As in the proof of Lemma \ref{lem3.8}, we know that when $u_0$ is sufficiently large, for any $U'\in {\it W}_{1}(u_0, v_0)$ and $U''\in {\it W}_{2}(u_0, v_0)$,
$\frac{{\rm d}\sigma}{{\rm d}\tau}>0$ along ${\it W}_1(U'')$ and  $\frac{{\rm d}\sigma}{{\rm d}\tau}<0$ along ${\it W}_2(U')$.
So, by the definitions of the wave curves ${\it W}_{1}$ and ${\it W}_{2}$ we have the desired conclusion.
This completes the proof of the lemma.
\end{proof}

If the wave curves ${\it W}_{1}(U'')$ and ${\it W}_{2}(U')$ do not intersect with each other before they reach the limiting circle, we note that they intersect at a ``vacuum" point in the following text.

Next, we divide the discussion into the following two cases: (1) $\tau_1\geq \tau_2^i$; (2) $\tau_1<\tau_2^i$.

\subsubsection{\bf $\tau_1>\tau_2^i$}

When $\tau_1>\tau_2^i$,
the wave curves ${\it W}_2(u_1, v_1)$ and ${\it W}_1(u_{po}, v_{po})$ intersect at a point or a vacuum.
From the point $P$, we draw a forward $C_{-}$ cross characteristic line of the centered simple wave of ${\it sf}_{+}$, denoted by $\widetilde{PF}$. As in (\ref{62303}), we have
$\bar{\partial}_{-}\rho <0$ on $\widetilde{PF}$. We draw a forward straight $C_{+}$ characteristic line from $P$ with the inclination angle $\alpha=\theta_{+}+A_1$, in which $A_1$ is defined as in (\ref{def12}). This stright characteristic line intersects $\mathcal{W}_{+}$ at a point $E$. We then consider
(\ref{E1}) with the boundary conditions
\begin{equation}\label{110601}
(u, v)=\left\{
               \begin{array}{ll}
                  (\bar{u}, \bar{v})(\xi), & \hbox{$(x, y)\in\widetilde{PF}$;} \\[2pt]
               (u_1, v_1), & \hbox{$(x, y)\in\overline{PE}$.}
               \end{array}
             \right.
\end{equation}

If ${\it W}_2(u_1, v_1)$ and ${\it W}_1(u_{po}, v_{po})$  intersect at a  point $(u_m, v_m)$.
Then we have
$r(u_1, v_1)=r(u_m, v_m)$, $s(u_1, v_1)>s(u_m, v_m)$, $r(u_{po}, v_{po})<r(u_m, v_m)$, $s(u_1, v_1)=s(u_m, v_m)$.
So, by Lemma \ref{lem2} and Remark \ref{rem4.1}, the DGP (\ref{E1}), (\ref{110601}) admits a classical solution on a curvilinear quadrilateral domain bounded by $\overline{PE}$, $\widetilde{PF}$, $C_{+}^{F}$, and $C_{-}^{E}$.
The solution consists of a simple wave ${\it f}_{+}$ with straight $C_{+}$ characteristic lines and a $C_{-}$ centered wave with $P$ as the center point; see Fig. \ref{Fig5.3.1} (left). The centered wave region is a curvilinear triangle region bounded by two $C_{-}$ characteristic curves issued from $P$, i.e., $\widetilde{PF}$ and $\widetilde{PH}$, and a $C_{+}$ characteristic line $\widetilde{FH}$.
The straight characteristic lines of  ${\it f}_{+}$ are issued from $\widetilde{PH}$ and the point $P$.
Moreover, if $r(u_m, v_m)-s(u_2, v_2)>2\mathcal{R}$ then the point $H$ is at infinite. For convenience, we use $(u_c, v_c)(x,y)$ to denote the $C_{-}$ centered wave issued from $P$.

From the point $F$, we draw a forward straight $C_{-}$ characteristic line with the inclination angle $\beta=\theta_{-}-A_2$, in which $A_2$ is defined as in  (\ref{def12}).  This line intersects $\mathcal{W}_{-}$ at a point $G$. We consider
(\ref{E1}) with the boundary condition
\begin{equation}\label{120303}
(u, v)=\left\{
               \begin{array}{ll}
                  (u_{c}, v_{c})(x,y), & \hbox{$(x, y)\in\widetilde{FH}$;} \\[2pt]
               (u_2, v_2), & \hbox{$(x, y)\in\overline{FG}$.}
               \end{array}
             \right.
\end{equation}
The SGP (\ref{E1}), (\ref{120303}) admits a simple wave solution, denoted by ${\it f}_{-}$, with straight $C_{-}$ characteristic lines issued from $\widetilde{FH}$.

Thus, the interaction of ${\it s}_{-}$ with ${\it sf}_{+}$ generates two  simple waves ${\it f}_{\pm}$. These two simple waves will then reflect off the walls of the divergent duct. And the flow in the remaining part of the divergent duct can be determined using the method described in Section 5.1.

If ${\it W}_2(u_1, v_1)$ and ${\it W}_1(u_{po}, v_{po})$ intersect at a vacuum, then there will be a angular cavitation region as described in Remark \ref{rem4.3}.
We can still construct a global solution inside  the divergent duct.

\begin{figure}[htbp]
\begin{center}
\includegraphics[scale=0.46]{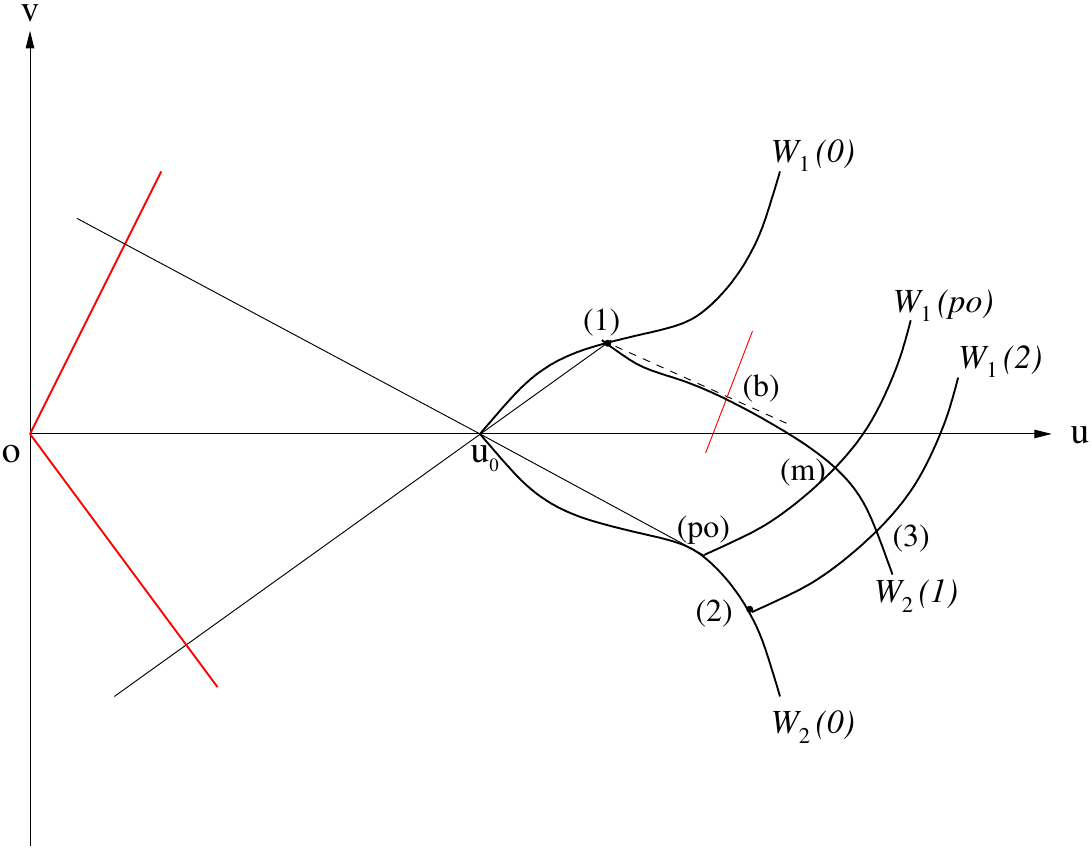}
\caption{\footnotesize Wave curves ${\it W}_2(u_1, v_1)$ and ${\it W}_1(u_{po}, v_{po})$. }
\label{Fig5.3.4}
\end{center}
\end{figure}
\subsubsection{\bf $\tau_1<\tau_2^i$}
When $\tau_1<\tau_2^i$, ${\it W}_{2}(u_1, v_1)={\it S}_{2}(u_1, v_1)+{\it SF}_{2}(u_1, v_1)$.
Since $\psi_{po}(\tau_1)<\tau_{po}=\psi_{po}(\tau_0)$, the waves curves ${\it S}_{2}(u_1, v_1)$ and ${\it W}_{1}(u_{po}, v_{po})$ do not intersect with each other.
Then,
 ${\it SF}_2(u_1, v_1)$ and ${\it W}_1(u_{po}, v_{po})$ intersect at a point or a vacuum.

There is an oblique post-sonic 2-shock issued from the point $P$ with the upstream flow state $(u_1, v_1)$.
We denote by $\phi_b$ and $(u_b, v_b)$
the inclination angle and the backside state of the post-sonic shock, respectively.
Then we have
$$
u_b=N_b\sin\phi_b+L_b\cos\phi_b, \quad v_b=-N_b\cos\phi_b+L_b\sin\phi_b,\quad \phi_b=\theta_{+}+\arctan\bigg(\frac{\tau_1N_b}{\tau_bL_b}\bigg),
$$
where
$$
N_b=\tau_b\sqrt{-p'(\tau_b)},\quad L_b=\sqrt{q_b^2-N_b^2}, \quad \tau_b=\psi_{po}(\tau_1),\quad q_b=\check{q}(\tau_b),
$$
and $\tau_1$ is defined as in  (\ref{def12}).
This post-sonic shock intersects $\mathcal{W}_{+}$ at a point $E$.
Next, we consider
(\ref{E1}) with the boundary condition
\begin{equation}\label{110601A}
(u, v)=\left\{
               \begin{array}{ll}
                  (\bar{u}, \bar{v})(\xi), & \hbox{$(x, y)\in\widetilde{PF}$;} \\[2pt]
               (u_b, v_b), & \hbox{$(x, y)\in\overline{PE}$.}
               \end{array}
             \right.
\end{equation}
If  ${\it W}_2(u_1, v_1)$ and ${\it W}_1(u_{po}, v_{po})$ intersect at a point $(u_m, v_m)$; see Fig. \ref{Fig5.3.4}.
Then by (\ref{102302d}) we have
$r(u_b, v_b)=r(u_m, v_m)$, $s(u_b, v_b)>s(u_m, v_m)$, $r(u_{po}, v_{po})<r(u_m, v_m)$, $s(u_1, v_1)=s(u_m, v_m)$.
So, by Lemma \ref{lem2} and Remark \ref{rem4.1}, the DGP (\ref{E1}), (\ref{110601A}) admits a classical solution on a curvilinear quadrilateral domain bounded by $\overline{PE}$, $\widetilde{PF}$, $C_{+}^{F}$, and $C_{-}^{E}$; see Fig. \ref{Fig5.3.1} (right).
Therefore, when  $\tau_1<\tau_2^i$ the interaction of ${\it s}_{-}$ and ${\it sf}_{+}$ generates a fan wave ${\it f}_{-}$ and another shock-fan composite wave ${\it sf}_{+}^{2}$.
The shock-fan composite wave ${\it sf}_{+}^{2}$ and the fan wave ${\it f}_{-}$ will then reflect off the walls $\mathcal{W}_{+}$ and $\mathcal{W}_{-}$, respectively. By solving two mixed boundary value problems (see Section 4.3), we see that the reflections generate two fan waves.
The flow in the remaining part of the divergent duct can be constructed using the method described in Section 5.1.

This completes the proof of Theorem \ref{thm3}.

\vskip 8pt
\subsection{Interaction of shock waves}
We consider the IBVP (\ref{E1}), (\ref{42701})
for $\tau_1^e<\tau_0<\tau_2^i$.
In this part we assume that the oblique waves at the corners $B$ and $D$ are oblique shocks with different types, and
the flow near the corners  $B$ and $D$ can be represented by
$$
(u, v)=\left\{
               \begin{array}{ll}
                 (u_0, 0), & \hbox{$-\frac{\pi}{2}\leq \eta<\eta_{s}$;} \\[2pt]
                 (u_1, v_1), & \hbox{$\eta_{s}<\eta<\theta_{+}$}
               \end{array}
             \right.
 \mbox{and}\quad(u, v)=\left\{
               \begin{array}{ll}
                 (u_0, 0), & \hbox{$\xi_s< \xi\leq\frac{\pi}{2}$;} \\[2pt]
                 (u_2, v_2), & \hbox{$\theta_{-}<\xi<\xi_s$,}
               \end{array}
             \right.
$$
respectively.
Here, the constant state $(u_1, v_1)$ satisfies $v_1=u_1\tan\theta_{+}$ and is connected to the constant state $(u_0, 0)$ on the downstream by an oblique 1-shock with the inclination angle  $\eta_{s}$;
the constant state $(u_2, v_2)$ satisfies $v_2=u_2\tan\theta_{-}$ and is connected to the constant state $(u_0, 0)$ on the downstream by an oblique 2-shock with the inclination angle  $\xi_{s}$.
The result of this subsection can be stated as the following theorem.
\begin{thm}\label{thm4}
Assume $\tau_1^i<\tau_0<\tau_2^i$ and $u_0>c_0$.
Assume furthermore that the oblique waves at the corners $B$ and $D$ are oblique shocks.  Then the
 IBVP (\ref{E1}), (\ref{42701}) admits a global piecewise smooth solution in $\Sigma$.
\end{thm}

For convenience, we denote the oblique 1-shock (2-shock, resp.) issued from $B$ ($D$, resp.) by ${\it s}_{-}$ (${\it s}_{+}$, resp.).
The two shocks interact with each other inside the duct at a point $P(x_{_P}, y_{_P})$, where $x_{_{P}}=2(\tan\xi_s-\tan\eta_s)^{-1}$ and $y_{_{P}}=(\tan\xi_s+\tan\eta_s)(\tan\xi_s-\tan\eta_s)^{-1}$.
We set $\tau_i=\hat{\tau}(\sqrt{u_i^2+v_i^2})$, $i=1, 2, \cdot\cdot\cdot$.
In order to identify the waves generated from the point $P$ after the interaction,
we divide the discussion into the following four cases:
(1) $\tau_1\geq\tau_2^i$, $\tau_2\geq\tau_2^i$;
(2) $\tau_1\geq\tau_2^i$, $\tau_2<\tau_2^i$;
(3) $\tau_1<\tau_2^i$, $\tau_2\geq\tau_2^i$;
(4) $\tau_1<\tau_2^i$, $\tau_2<\tau_2^i$.

\begin{figure}[htbp]
\begin{center}
\includegraphics[scale=0.38]{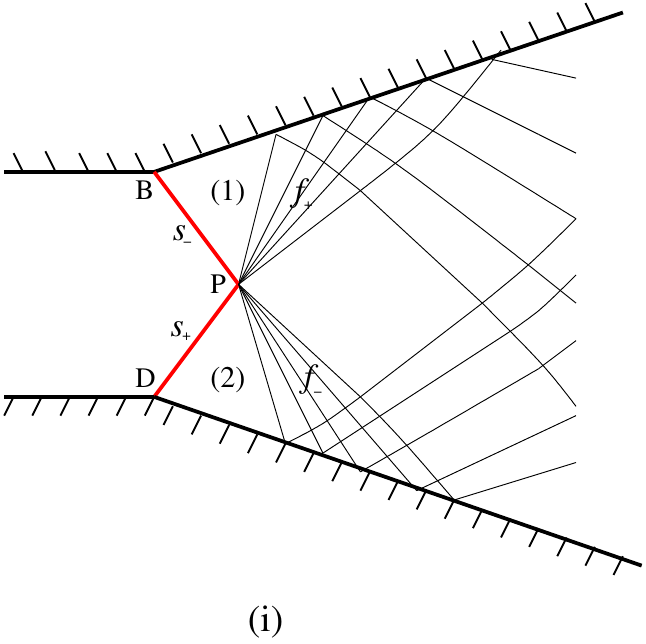} \qquad  \includegraphics[scale=0.38]{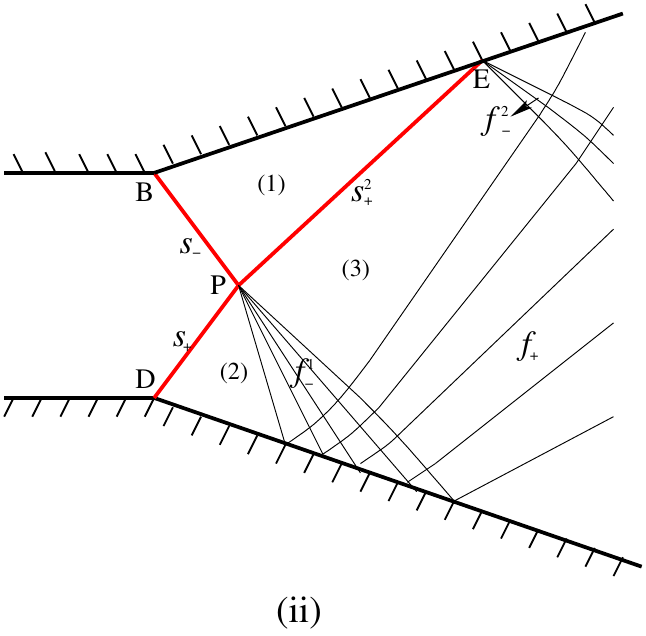}\qquad  \includegraphics[scale=0.38]{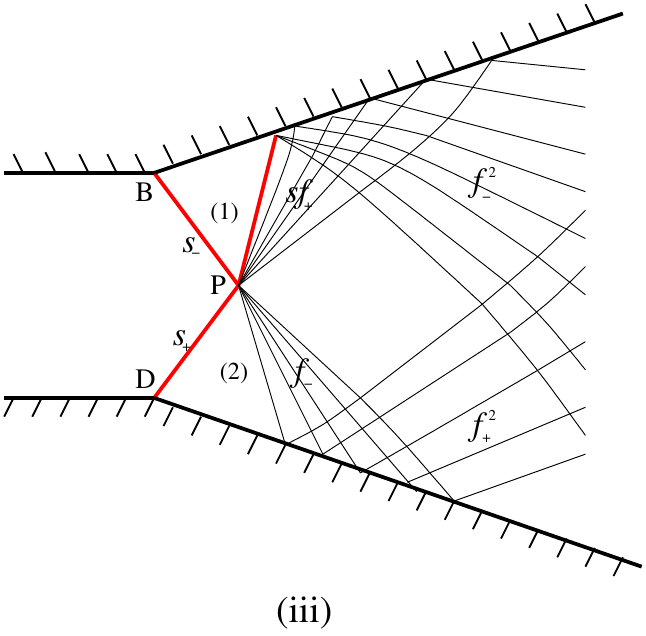}\\
\vskip 10pt
\includegraphics[scale=0.38]{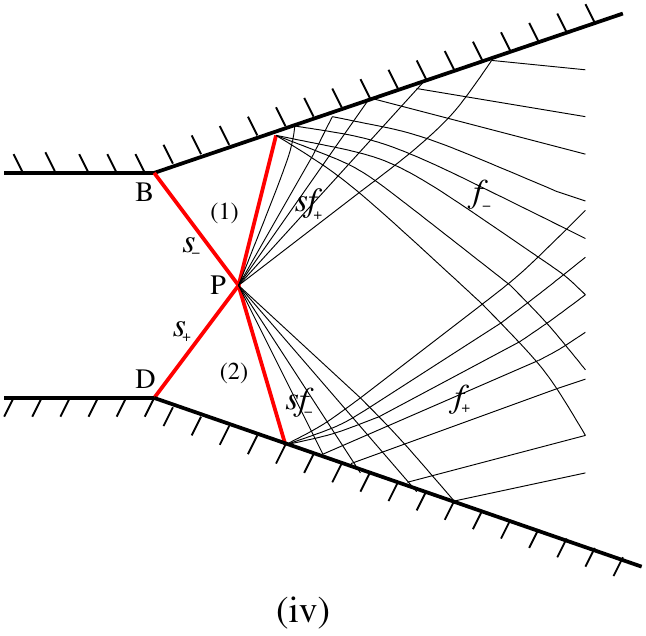}\qquad   \includegraphics[scale=0.38]{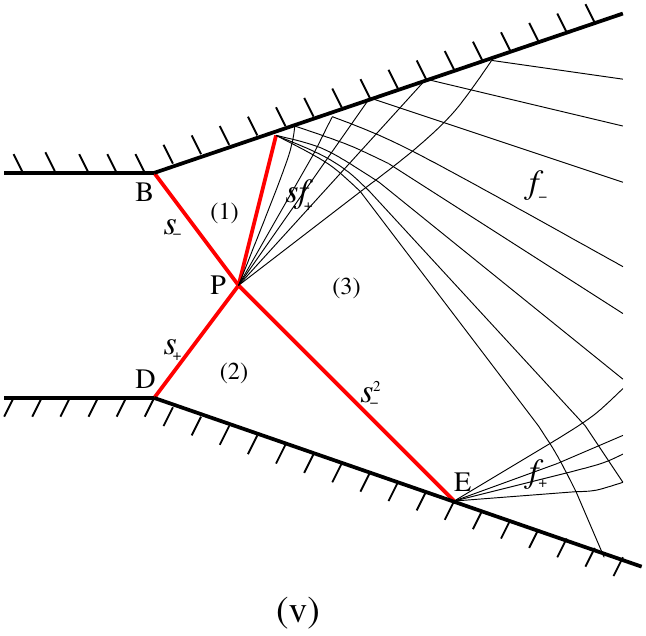}\\
\vskip 10pt
\includegraphics[scale=0.38]{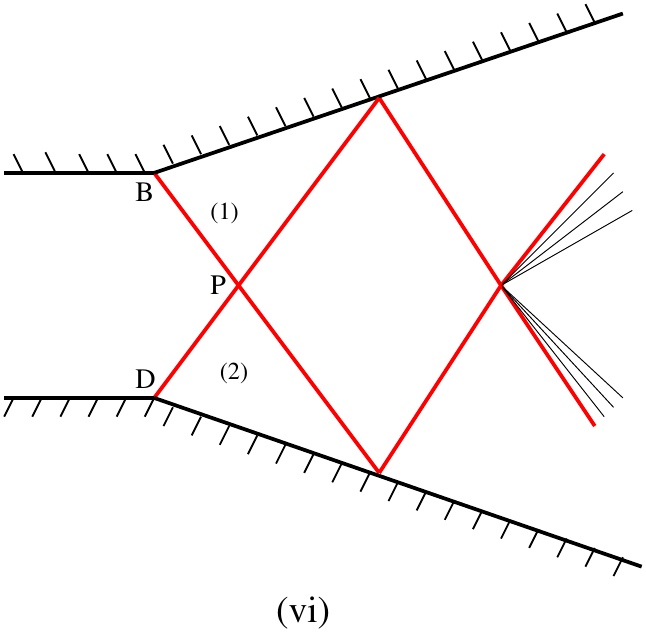}\qquad  \includegraphics[scale=0.38]{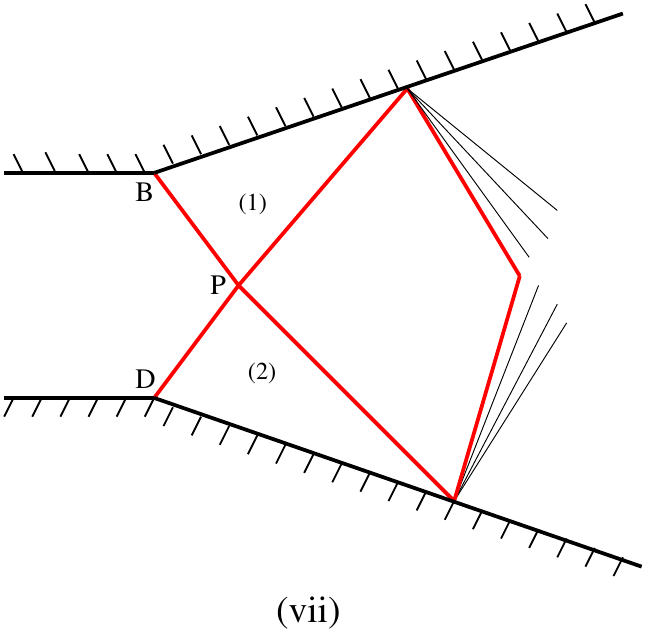}\qquad  \includegraphics[scale=0.38]{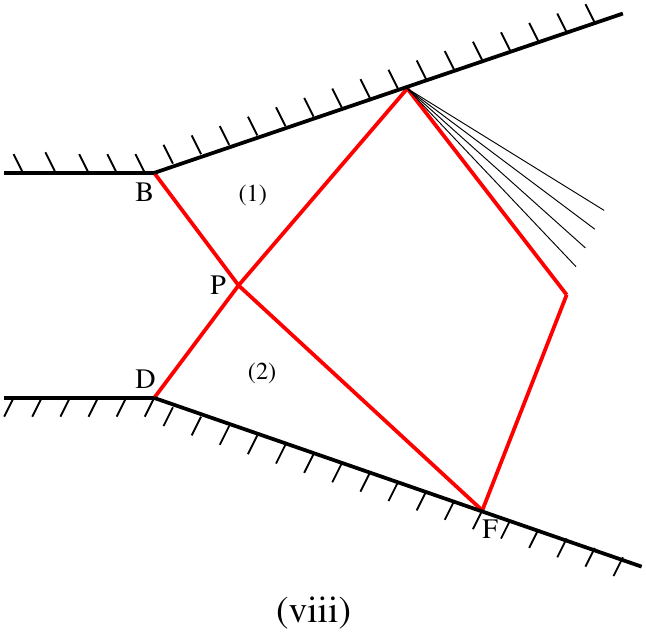}
\caption{\footnotesize Interaction of oblique shock waves. }
\label{Fig5.4.1}
\end{center}
\end{figure}

\subsubsection{$\tau_1\geq\tau_2^i$, $\tau_2\geq\tau_2^i$}
When $\tau_1\geq\tau_2^i$ and $\tau_2\geq\tau_2^i$, we have ${\it W}_1(u_2, v_2)={\it F}_1(u_2, v_2)$ and ${\it W}_2(u_1, v_1)={\it F}_2(u_1, v_1)$. The wave curves ${\it W}_1(u_2, v_2)$ and ${\it W}_2(u_1, v_1)$ intersect at a point $(u_3, v_3)$ or a vacuum.
In this case, the interaction produces two centered simple waves issued from $P$; see ${\it f}_{\pm}$ in Fig. \ref{Fig5.4.1} (i). These two simple waves will then reflect off the walls and generate two simple waves again. The flow in the remaining part of the divergent duct can be determined using the method described in Section 5.1.


\subsubsection{$\tau_1<\tau_2^i$, $\tau_2\geq\tau_2^i$}
The discussion for cases (2) and (3) are similar. Without loss of generality, we discuss case (3).
When $\tau_1<\tau_2^i$ and $\tau_2\geq\tau_2^i$, we have ${\it W}_1(u_2, v_2)={\it F}_1(u_2, v_2)$ and ${\it W}_2(u_1, v_1)={\it S}_2(u_1, v_1)+{\it SF}_2(u_1, v_1)$, respectively.
In this case, there are  the following two subcases:
\begin{itemize}
 \item the wave curves ${\it F}_1(u_2, v_2)$ and ${\it S}_2(u_1, v_1)$ intersect at a point $(u_3, v_3)$;
  \item the wave curves ${\it F}_1(u_2, v_2)$ and ${\it SF}_2(u_1, v_1)$ intersect at a point $(u_3, v_3)$ or a vacuum.
\end{itemize}

For the first subcase, the interaction of ${\it s}_{\pm}$ generates an oblique 2-shock, denoted by ${\it s}_{+}^{2}$, and a centered simple wave, denoted by ${\it f}_{-}^{1}$. The states on the front and back sides of ${\it s}_{+}^{2}$ are $(u_1, v_1)$ and $(u_3, v_3)$, respectively.
 Since $\tau_3>\tau_2>\tau_2^i$, the oblique shock ${\it s}_{+}^{2}$ will then interact with the wall $\mathcal{W}_{+}$ at a point $E$, generating a new centered simple wave ${\it f}_{+}^2$ emanating from the point $E$.
The centered simple wave ${\it f}_{-}^{1}$ will then reflect off the wall $\mathcal{W}_{-}$ and generate a new simple wave ${\it f}_{+}$. These two simple waves ${\it f}_{+}$ and ${\it f}_{-}^{2}$ will then interact with each other inside the duct;  see Fig. \ref{Fig5.4.1} (ii).

For the second subcase, the interaction of ${\it s}_{\pm}$ generates an oblique shock-fan composite wave ${\it sf}_{+}$ and a centered simple ${\it f}_{-}$ issued from the point $P$.
The shock-fan composite wave ${\it sf}_{+}$ and the centered simple ${\it f}_{-}$ will then reflect off the walls $\mathcal{W}_{+}$ and $\mathcal{W}_{-}$, respectively, and produce two new simple waves ${\it f}_{\pm}^2$ as described in Section 5.3.3;
see Fig. \ref{Fig5.4.1} (iii).

The flow in the remaining part of the divergent duct can be determined using the method described in Section 5.1;

\subsubsection{$\tau_1<\tau_2^i$, $\tau_2<\tau_2^i$}
When $\tau_1<\tau_2^i$ and $\tau_2<\tau_2^i$, we have ${\it W}_1(u_2, v_2)={\it F}_1(u_2, v_2)+{\it SF}_1(u_2, v_2)$ and ${\it W}_2(u_1, v_1)={\it S}_2(u_1, v_1)+{\it SF}_2(u_1, v_1)$.
So, we have the following four subcases:
\begin{itemize}
  \item the wave curves ${\it SF}_1(u_2, v_2)$ and ${\it SF}_2(u_1, v_1)$ intersect at a point or a vacuum;
\item the wave curves ${\it S}_1(u_2, v_2)$ and ${\it SF}_2(u_1, v_1)$ intersect at a point;
\item the wave curves ${\it SF}_1(u_2, v_2)$ and ${\it S}_2(u_1, v_1)$ intersect at a point;
\item the wave curves ${\it S}_1(u_2, v_2)$ and ${\it S}_2(u_1, v_1)$ intersect at a point.
\end{itemize}

For the first subcase, the interaction will generate two centered shock-fan composite waves ${\it sf}_{\pm}$ with different types issued from the point $P$. These two composite waves will then reflect off the walls and generate two simple waves ${\it f}_{\pm}$; see Fig. \ref{Fig5.4.1} (iv).

We now discuss the second and third subcases.
Assume that the wave curves ${\it S}_1(u_2, v_2)$ and ${\it SF}_2(u_1, v_1)$ intersect at a point $(u_3, v_3)$. Then the interaction will generate an oblique shock ${\it s}_{-}^2$ and a centered shock-fan composite wave ${\it sf}_{+}$ issued from the point $P$. The composite wave ${\it sf}_{+}$ will then reflect off the wall $\mathcal{W}_{+}$ and produce a simple wave ${\it f}_{-}$.
The oblique shock ${\it s}_{-}^2$ will then interact with the wall $\mathcal{W}_{-}$ at a point. Since $(u_3, v_3)\in {\it SF}_2(u_1, v_1)$, we have $\tau_3\geq \psi_{po}(\tau_1)>\tau_2^i$ and the interaction will generate a centered simple wave ${\it f}_{+}$ issued from this point. The two simple waves will intersect with each other inside the duct; see Fig. \ref{Fig5.4.1} (v).

For the fourth subcase, the interaction of ${\it s}_{\pm}$ generates two oblique shocks with different type issued from the point $P$. The two shocks will then interact with the walls. There are the following three possibilities.
(1). The interactions generate two centered shock-fan composite waves; see Fig. \ref{Fig5.4.1} (vii). The discussion is the same as that in Section 5.2.
(2). The interaction generate a centered shock-fan composite wave and an oblique shock; see Fig. \ref{Fig5.4.1}.
(viii). The discussion  is the same as that in Section 5.3.
(3). The interactions generate two oblique shocks; see Fig. \ref{Fig5.4.1} (vi).
We can repeat the previous steps to construct a global piecewise smooth solution inside the divergent duct.
This completes the proof of Theorem \ref{thm4}.

\subsection{Interaction of  fan-shock-fan composite waves}
We consider the IBVP (\ref{E1}), (\ref{42701})
for $\tau_0<\tau_1^e$.
In this part we assume that
the oblique waves at the corners $B$ and $D$ are fan-shock-fan composite waves with different types, and
the flow near $B$ and $D$ can be represented by
$$
(u, v)=\left\{
               \begin{array}{ll}
                 (u_0, 0), & \hbox{$-\frac{\pi}{2}<\eta<\eta_0$;} \\[2pt]
                 (\tilde{u}_l, \tilde{v}_l)(\eta), & \hbox{$\eta_0<\eta\leq -\xi_{e}$;} \\[2pt]
                  (\tilde{u}_r, \tilde{v}_r)(\eta), & \hbox{$-\xi_{e}\leq \eta<{\eta}_1$;} \\[2pt]
                 (u_1, v_1), & \hbox{$\eta_1<\xi<\theta_{+}$;}
               \end{array}
             \right.
\quad \mbox{and}
\quad
(u, v)=\left\{
               \begin{array}{ll}
                 (u_0, 0), & \hbox{$\xi_0<\xi<\frac{\pi}{2}$;} \\[2pt]
                 (\bar{u}_l, \bar{v}_l)(\xi), & \hbox{${\xi}_e<\xi\leq \xi_0$;} \\[2pt]
                  (\bar{u}_r, \bar{v}_r)(\xi), & \hbox{${\xi}_2\leq \xi<\xi_{e}$;} \\[2pt]
                 (u_2, v_2), & \hbox{$\theta_{-}<\xi<{\xi}_2$,}
               \end{array}
             \right.
$$
respectively. Here, the solution near the corner $D$ is the same as that given in (\ref{120401}); $\xi_0=-\eta_0=\arcsin\frac{c_0}{u_0}$;
$(\tilde{u}_l, \tilde{v}_l)(\eta)=(\bar{u}_l, \bar{v}_l)(-\eta)$, $\eta_0<\eta\leq -\xi_{e}$; $(\tilde{u}_r, \tilde{v}_r)(\eta)$ is determined by (\ref{102306a}) and (\ref{102304a}) with initial data $(\tilde{q}_r, \tilde{\tau}_r, \tilde{\sigma}_r)(-\xi_{e})=(q_2^e, \tau_2^e, -\sigma_2^e)$;
$\eta_1$ is determined by $\tilde{v}_{r}(\eta_1)=\tilde{u}_{r}(\eta_1)\tan\theta_{-}$; the constants $q_2^e$, $\tau_2^e$, $\xi_2$ and $\sigma_2^e$ are the same as that defined in Section 3.3.4.
The main result of this part can be stated as the following theorem.

\begin{thm}\label{thm5}
Assume $\tau_0<\tau_1^e$ and $u_0>c_0$.
Assume furthermore that the oblique waves at the corners $B$ and $D$ are fan-shock-fan composite waves. Then when $\tau_1^e-\tau_0$ is sufficiently small, the IBVP (\ref{E1}), (\ref{42701}) admits a global  piecewise smooth solution in $\Sigma$.
\end{thm}
\begin{figure}[htbp]
\begin{center}
\includegraphics[scale=0.4]{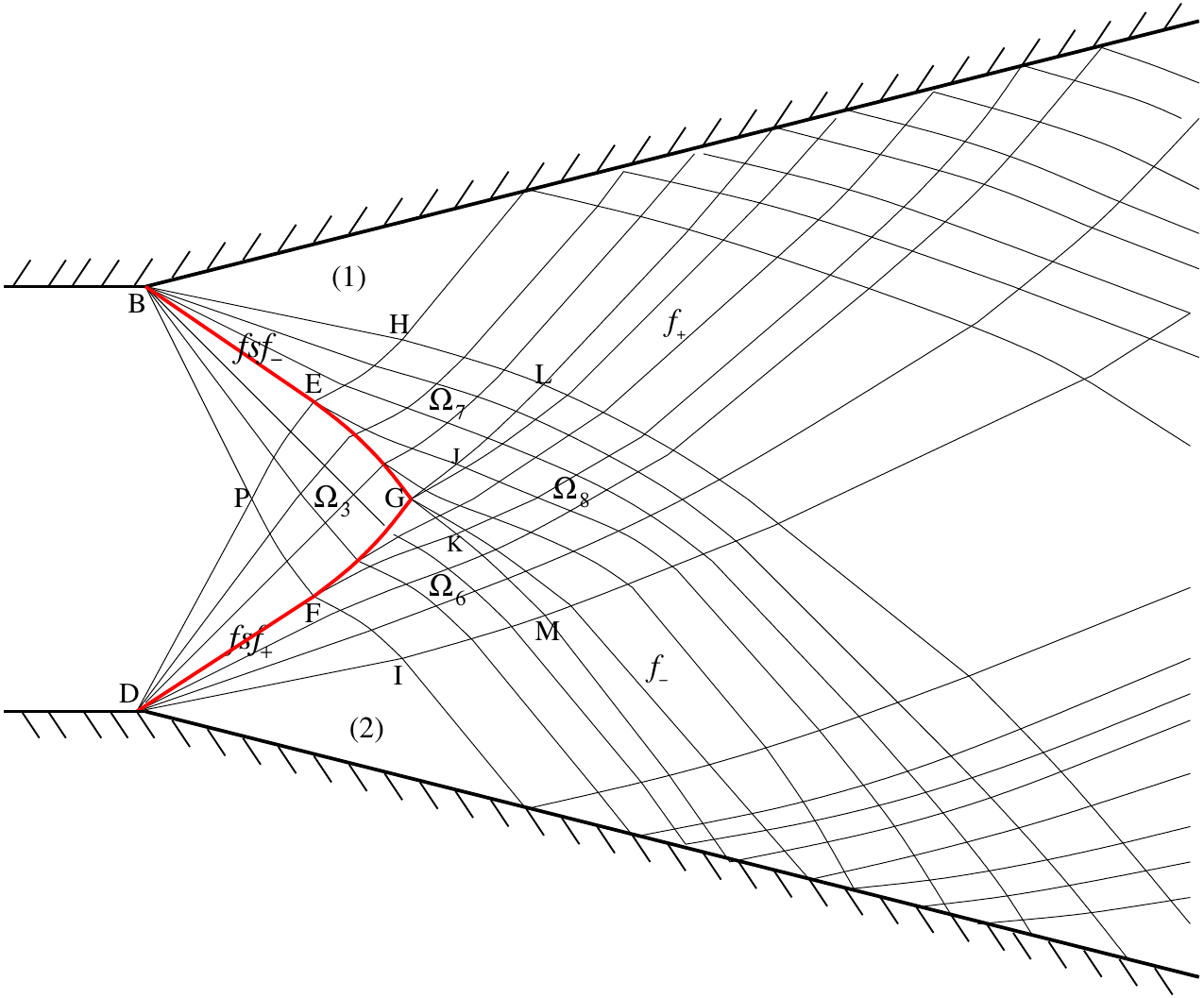}
\caption{\footnotesize Interaction of two oblique fan-shock-fan composite waves }
\label{Fig5.5.1}
\end{center}
\end{figure}

\subsubsection{\bf Characteristic boundary value problem}
For convenience, we use ${\it fsf}_{+}$ and ${\it fsf}_{-}$ to denote the oblique fan-shock-fan composite waves issued from $D$ and $B$, respectively. The two fan-shock-fan composite waves start to interact with each other at point ${P}(\cot\xi_0,0)$.
From the point $P$, we draw a forward $C_{-}$  cross characteristic line of the fan-shock-fan composite wave ${\it fsf}_{+}$. This characteristic line intersects the double-sonic shock of ${\it fsf}_{+}$ at a point $F$ and ends up at a point $I$ on the straight $C_{+}$ characteristic line $y=-1+x\tan\xi_2$.
From the point $P$, we draw a forward $C_{+}$  cross characteristic line of the fan-shock-fan composite wave ${\it fsf}_{-}$. This characteristic line intersects the double-sonic shock of ${\it fsf}_{-}$ at a point $E$ and ends up at a point $H$ on  the straight $C_{-}$ characteristic line $y=1+x\tan\eta_1$.
We shall show that the two cross characteristic lines delimit the two fan-shock-fan composite wave regions.
As in Section 5.1, we can
obtain the cross characteristic curves $\widetilde{{PE}}$, $\widetilde{{EH}}$, $\widetilde{{PF}}$, and $\widetilde{{FI}}$, and have
the following estimates:
\begin{equation}\label{120602}
\bar{\partial}_{+}\rho<0\quad on~ \widetilde{{PE}}~\mbox{and}~\widetilde{{EH}};
\quad \bar{\partial}_{-}\rho<0\quad \mbox{on}~\widetilde{{PF}}~\mbox{and}~ \widetilde{{FI}}.
\end{equation}
As in (\ref{120201}) we also have
\begin{equation}\label{120602a}
~(r, s)\in\Pi\quad \mbox{on}~ \widetilde{{PE}}\cup\widetilde{{EH}}\cup\widetilde{{PF}}\cup\widetilde{{FI}}.
\end{equation}

In order to determine the flow in the interaction region of ${\it fsf}_{\pm}$,  we consider (\ref{E1}) with the following boundary condition:
\begin{equation}\label{102807}
(u, v)=\left\{
               \begin{array}{ll}
                  (\tilde{u}_l, \tilde{v}_l)(\eta), & \hbox{$(x,y)\in \widetilde{{PE}}$;} \\[2pt]
                  (\tilde{u}_r, \tilde{v}_r)(\eta), & \hbox{$(x,y)\in \widetilde{{EH}}$;} \\[2pt]
               (\bar{u}_l, \bar{v}_l)(\xi), & \hbox{$(x,y)\in \widetilde{{PF}}$;} \\[2pt]
                  (\bar{u}_r, \bar{v}_r)(\xi), & \hbox{$(x,y)\in \widetilde{{FI}}$.}
               \end{array}
             \right.
\end{equation}
We shall show that when $\tau_1^e-\tau_0$ is sufficiently small, the BVP (\ref{E1}), (\ref{102807}) admits a global piecewise smooth solution on a  curvilinear quadrilateral domain $\overline{\Omega}$ bounded by $\widetilde{{PE}}\cup \widetilde{{EH}}$, $\widetilde{{PF}}\cup \widetilde{{FI}}$, $C_{-}^{H}$, and $C_{+}^{I}$.
Moreover, the solution satisfies
$$
\bar{\partial}_{+}\rho<0\quad \mbox{on}~C_{+}^{I}, \quad \bar{\partial}_{-}\rho<0\quad \mbox{on}~C_{-}^{H}.
$$

\subsubsection{\bf Interaction of fan waves}
We first consider (\ref{E1}) with the boundary condition
\begin{equation}\label{102701}
(u, v)=\left\{
               \begin{array}{ll}
                  (\tilde{u}_l, \tilde{v}_l)(\eta) & \hbox{$(x,y)\in \widetilde{{PE}}$;} \\[2pt]             (\bar{u}_l, \bar{v}_l)(\xi) & \hbox{$(x,y)\in \widetilde{{PF}}$.}
              \end{array}
             \right.
\end{equation}

\begin{figure}[htbp]
\begin{center}
\includegraphics[scale=0.36]{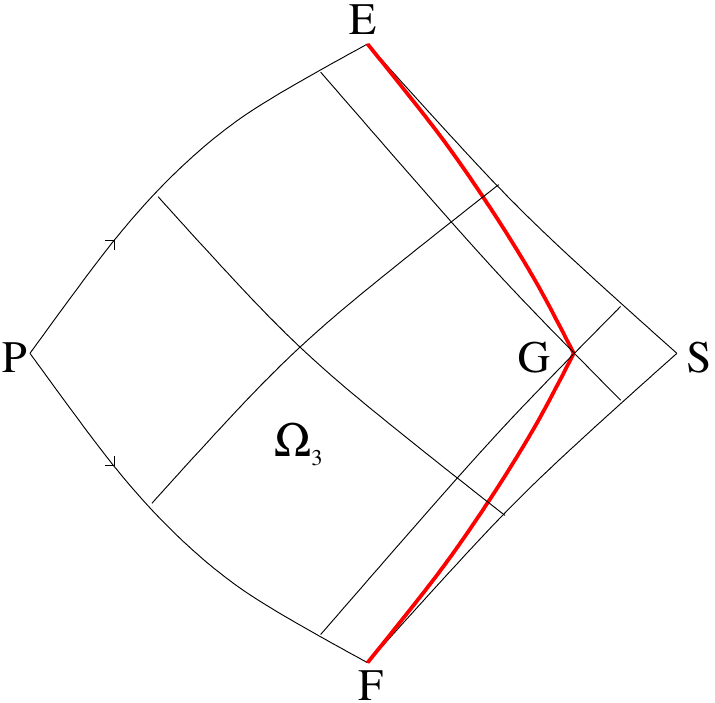}\qquad\qquad \qquad\qquad\includegraphics[scale=0.36]{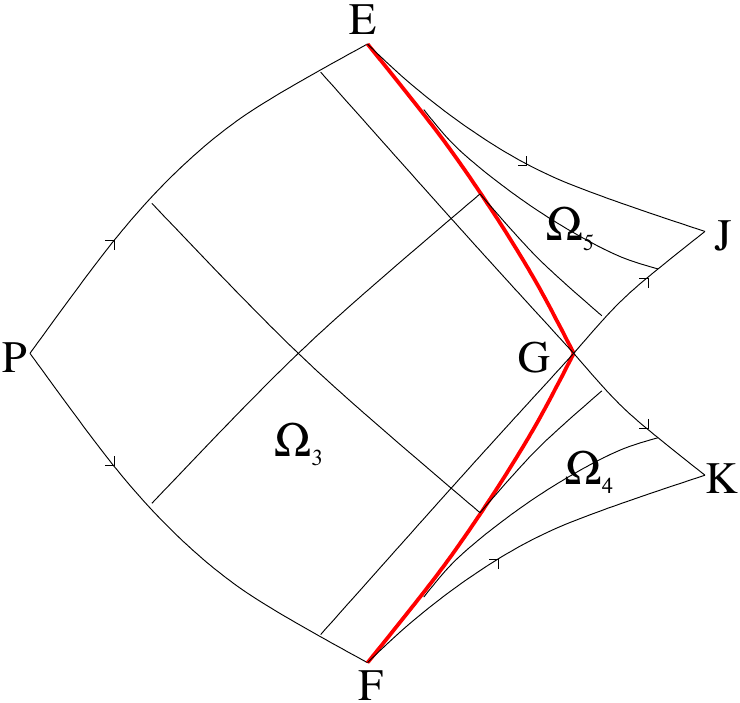}
\caption{\footnotesize Post-sonic shocks through points $E$ and $F$ and the flows upstream and downstream of them. }
\label{Fig5.5.2}
\end{center}
\end{figure}

\begin{lem}\label{lem110101}
Assume $\tau_1^e-\tau_0$ to be sufficiently small. Then the Goursat problem (\ref{E1}), (\ref{102701}) admits a classical solution on a curvilinear quadrilateral domain $\overline{\Omega^3}$ bounded by $\widetilde{{PE}}$, $\widetilde{{PF}}$, $\widetilde{{ES}}$, and $\widetilde{{FS}}$, where $\widetilde{{ES}}$ is a forward $C_{+}$ characteristic line issued from ${E}$ and $\widetilde{{FS}}$ a forward $C_{-}$ characteristic line issued from ${F}$; see Fig. \ref{Fig5.5.2} (left). Moreover, the solution satisfies
\begin{equation}\label{6811}
\big\|\big(\tau-\tau_1^e, ~\alpha-A_1^e, ~\beta+A_1^e, ~\bar{\partial}_{+}\rho+\mathcal{L},~  \bar{\partial}_{-}\rho+\mathcal{L}\big)\big\|_{{0; \overline{\Omega^3}}}\rightarrow 0\quad \mbox{as} ~ \tau_1^e-\tau_0\rightarrow 0,
\end{equation}
where
\begin{equation}\label{102901}
\mathcal{L}=\frac{c_1^eu_0\sin^2(2A_1^e)}{(\tau_1^e)^4p''(\tau_1^e)}, \quad c_1^e=\tau_1^e\sqrt{-p'(\tau_1^e)},\quad \mbox{and} \quad A_1^e=\arcsin\Big(\frac{c_1^e}{u_0}\Big).
\end{equation}
\end{lem}
\begin{proof}
When $\tau_1^e-\tau_0$ is sufficiently small, the existence of $C^1$ solution on $\overline{\Omega^3}$ follows routinely from Lemma \ref{lem1}.
Moreover, by the hodograph transformation we know that the solution satisfies
\begin{equation}\label{102702}
\big\|(\tau-\tau_1^e, u-u_0, v)\big\|_{{0; \overline{\Omega^3}}}\rightarrow 0\quad \mbox{as}~ \tau_1^e-\tau_0\rightarrow 0.
\end{equation}

As in (\ref{62303}) and (\ref{101501}), we have
$$
\bar{\partial}_-\rho=-\frac{\sin (2\bar{A}(\xi))\bar{\rho}'(\xi)}{\sqrt{x^2+(y+1)^2}}<0
\quad \mbox{along}~\widetilde{{PF}}; \quad \bar{\partial}_{+}\rho~=~-\frac{\sin (2\tilde{A}(\eta)) \tilde{\rho}'(\eta)}{\sqrt{x^2+(y-1)^2}}<0
\quad \mbox{along}~ \widetilde{{PE}}.
$$
Combining this with (\ref{102304}), (\ref{102304a}), (\ref{81104}), (\ref{102702}) we can get the desired estimates for $\bar{\partial}_{\pm}\rho$.


This completes the proof.
\end{proof}

For convenience, we denote the solution of the Goursat problem (\ref{E1}), (\ref{102701}) by $(u, v)=(u_3, v_3)(x,y)$.
We also let $$\tau_3(x,y)=\hat{\tau}\big(\sqrt{u_3^2+v_3^2}\big)\quad \mbox{and}
\quad(\alpha_3, \beta_3, A_3, \sigma_3)(x,y)=(\alpha, \beta, A, \sigma)(u_3(x,y), v_3(x,y)).$$
We shall show that the double-sonic shocks will propagate into the domain $\Omega^3$ from ${E}$ and ${F}$ and become post-sonic shocks.

\subsubsection{\bf Existence of post-sonic shocks propagated from $E$ and $F$}


 For convenience, we denote the shocks propagating from ${E}$ and ${F}$ by ${\it S}_{_{E}}$ and ${\it S}_{_{F}}$, respectively.

We first discuss the shock wave propagating from ${F}$. 
We use subscripts `$f$' and `$b$' to denote the states on the front  and back sides of the shock ${\it S}_{_{F}}$, respectively.
We use $\phi$ to denote the inclination angle of the shock curve ${\it S}_{_{F}}$. 
Since ${\it S}_{_{F}}$ is a $2$-shock, we set
\begin{equation}\label{62604}
L=u\cos \phi+v\sin \phi\quad\mbox{and} \quad N=u\sin \phi- v\cos \phi.
\end{equation}
Then by the Rankine-Hugoniot relations we have that on ${\it S}_{_{F}}$,
\begin{equation}\label{RH1}
\left\{
  \begin{array}{ll}
    \rho_f N_f=\rho_b N_b=m,  \\[4pt]
L_f=L_b,\\[4pt]
   N_f^2+2h(\tau_f)=N_b^2+2h(\tau_b).
  \end{array}
\right.
\end{equation}
By the first and third relations of (\ref{RH1}), one has
\begin{equation}\label{250101}
m^2=-\frac{2h(\tau_f)-2h(\tau_b)}{\tau_f^2-\tau_b^2}.
\end{equation}

From (\ref{62604}) we have
\begin{equation}\label{61112}
\phi=\sigma+\arcsin\Big(\frac{N}{q}\Big)\quad \mbox{along}~{\it S}_{_{F}}.
\end{equation}
Since $\tau_f({F})=\tau_1^e$ and $\tau_b({F})=\tau_2^e$,
we have
\begin{equation}\label{250401}
N_b=c_b=q_b\sin A_b\quad \mbox{and}\quad \alpha_b=\phi=\xi_{e}\quad  \mbox{at}~ {F};
\end{equation}
\begin{equation}\label{102002}
N_f=c_f=q_f\sin A_f\quad \mbox{and}\quad \alpha_f=\phi=\xi_{e}\quad \mbox{at}~ {F}.
\end{equation}


From (\ref{250101}) we have
\begin{equation}\label{250104}
\begin{aligned}
2m \bar{\partial}_{_\Gamma}m ~=~&\frac{2\tau_f}{\tau_f^2-\tau_b^2}\left(-p'(\tau_f)+\frac{2h(\tau_f)-2h(\tau_b)}{\tau_f^2-\tau_b^2}\right)
\bar{\partial}_{_\Gamma}\tau_f\\&
-\frac{2\tau_b}{\tau_f^2-\tau_b^2}\left(-p'(\tau_b)+\frac{2h(\tau_f)-2h(\tau_b)}{\tau_f^2-\tau_b^2}\right)
\bar{\partial}_{_\Gamma}\tau_b\\~=~&\frac{2\rho_f}{\tau_f^2-\tau_b^2}(c_f^2-N_f^2)\bar{\partial}_{_\Gamma}\tau_f
-\frac{2\rho_b}{\tau_f^2-\tau_b^2}(c_b^2-N_b^2)\bar{\partial}_{_\Gamma}\tau_b\quad \mbox{along}~{\it S}_{_{F}}.
\end{aligned}
\end{equation}
Combining this with $\tau_f({F})=\tau_1^e$ and $\tau_b({F})=\tau_2^e$,  we have
\begin{equation}\label{250402}
\bar{\partial}_{_\Gamma}m=0\quad \mbox{at}~ {F}.
\end{equation}

From (\ref{U}) and (\ref{62604}) we have
$
m=\rho  N =\rho  q  (\cos\sigma \sin\phi-\sin\sigma \cos\phi)
$  on $S_{_F}$.
Thus,
\begin{equation}\label{4250103a}
\bar{\partial}_{_\Gamma} m=\frac{N }{q }\partial_{s}(\rho q )-\rho L \bar{\partial}_{_\Gamma}\sigma +\rho L \bar{\partial}_{_\Gamma} \phi\quad \mbox{along}~ {\it S}_{_{F}}.
\end{equation}

From the Bernoulli law (\ref{E2}) we have
\begin{equation}\label{102801a}
 \bar{\partial}_{_\Gamma} q =\frac{\tau ^3 p'(\tau )\bar{\partial}_{_\Gamma} \rho }{q }\quad \mbox{along}~{\it S}_{_{F}}.
\end{equation}
This yields
\begin{equation}\label{110403}
\bar{\partial}_{_\Gamma} (\rho q )=\frac{1}{q }(q ^2+\tau ^2 p'(\tau ))\bar{\partial}_{_\Gamma}\rho\quad \mbox{along}~{\it S}_{_{F}}.
\end{equation}
Hence, by (\ref{4250103a}) we get
\begin{equation}\label{250102}
\bar{\partial}_{_\Gamma} \phi=\frac{\bar{\partial}_{_\Gamma} m}{\rho L }-\frac{N }{\rho L q ^2}(q ^2+\tau ^2 p'(\tau ))\bar{\partial}_{_\Gamma}\rho +\bar{\partial}_{_\Gamma}\sigma
\end{equation}
and
\begin{equation}\label{42804a}
\bar{\partial}_{_\Gamma} \sigma =\bar{\partial}_{_\Gamma} \phi-\frac{\bar{\partial}_{_\Gamma} m}{\rho L }+\frac{N }{\rho L q ^2}(q ^2+\tau ^2 p'(\tau ))\bar{\partial}_{_\Gamma}\rho \quad \mbox{along}~ {\it S}_{_{F}}.
\end{equation}

By a direct computation,
we also have
\begin{equation}\label{50701}
\bar{\partial}_{_\Gamma}u=\cos\sigma\bar{\partial}_{_\Gamma}q-q\sin\sigma\bar{\partial}_{_\Gamma}\sigma\quad \mbox{and}\quad
\bar{\partial}_{_\Gamma}v=\sin\sigma\bar{\partial}_{_\Gamma}q
+q\cos\sigma\bar{\partial}_{_\Gamma}\sigma.
\end{equation}



In what follows, we are going to construct a post-sonic shock propagated from the point $F$.
We first assume a priori that $ {\it S}_{_{F}}$ states in the interior of $\Omega^3$.
Then
\begin{equation}\label{101101}
(u_f, v_f, \tau_f)=(u_{3}, v_{3}, \tau_{3})(x, y)
\end{equation}
on ${\it S}_{_{F}}$.
We also assume a priori that
\begin{equation}\label{72202}
\tau_1^e\leq\tau_f<\tau_2^i\quad \mbox{on}~ {\it S}_{_{F}}\setminus {F}.
\end{equation}
Then by Proposition \ref{51003} we have
\begin{equation}\label{62603}
\tau_b=\psi_{po}(\tau_f)\quad \mbox{and}\quad
m^2=-p'(\psi_{po}(\tau_f))>-p'(\tau_f)\quad \mbox{on}~ {\it S}_{_{F}}\setminus {F}.
\end{equation}
This implies
\begin{equation}\label{6810}
N_b=c_b=\sqrt{-\tau_b^2p'(\tau_b)}\quad \mbox{and}\quad N_f>c_f=\sqrt{-\tau_f^2p'(\tau_f)}\quad \mbox{on}~{\it S}_{_{F}}\setminus {F}.
\end{equation}

By (\ref{6810}) and (\ref{61112}) we have
\begin{equation}\label{61113}
\phi=\sigma_b+A_b=\alpha_b\quad\mbox{and}\quad \phi>\sigma_f+A_f=\alpha_f  \quad \mbox{on}~ {\it S}_{_{F}}\setminus {F}.
\end{equation}


By (\ref{6810}) we have
\begin{equation}\label{61105}
\bar{\partial}_{_\Gamma} m=\frac{-p''(\psi_{po}(\tau_f))\psi_{po}'(\tau_f)}{2m }\bar{\partial}_{_\Gamma} \tau_f\quad \mbox{along}~{\it S}_{_{F}}.
\end{equation}


We denote by $y=y_{+}(x)$, $x_{_{F}}\leq x\leq x_{_{S}}$ the $C_{+}$ characteristic line $\widetilde{{FS}}$ and by
$y=y_{s}(x)$, $x\geq x_{_{F}}$ the post-sonic shock curve ${\it S}_{_{F}}$.
Then we have
\begin{equation}\label{61102}
\begin{aligned}
&y_{s}'(x)=\tan(\phi(x, y_s(x))), \quad y_{+}'(x)=\tan(\alpha_f(x, y_+(x))), \quad\\&\qquad
y_{s}(x_{_{F}})=y_{+}(x_{_{F}}), \quad  y_{s}'(x_{_{F}})=y_{+}'(x_{_{F}})=\tan\xi_{e}.
\end{aligned}
\end{equation}

Actually, in order to determine $y_s(x)$ we can consider
\begin{equation}\label{61115}
\bar{\partial}_{_\Gamma} \phi=\frac{-p''(\psi_{po}(\tau_f))\psi_{po}'(\tau_f)\bar{\partial}_{_\Gamma} \tau_f}{2\rho_f^2 L_fN_f}-\frac{N_f}{\rho_fL_fq_f^2}\big(q_f^2+\tau_f^2 p'(\tau_f)\big)\bar{\partial}_{_\Gamma}\rho_f+\bar{\partial}_{_\Gamma}\sigma_f
\end{equation}
 with the initial data
\begin{equation}\label{61114}
\phi({F})=\alpha_f({F})=\xi_{e}.
\end{equation}
The problem (\ref{61115}), (\ref{61114}) can be seen as an initial value problem for an ordinary differential equation.

\begin{lem}\label{91011}
Assume $\tau_1^e-\tau_0$ to be sufficiently small. Then
the post-sonic shock front  ${\it S}_{_{F}}$ exists and stays in $\Omega^3$ until it intersects the characteristic line $\widetilde{ES}$ at some point.
\end{lem}

\begin{proof}


From $\tan(\alpha_f(y_{+}(x),x))=y_{+}'(x)$ we have
\begin{equation}\label{61103}
y_{+}''(x)=\frac{\bar{\partial}_{+}\alpha_{3}(x, y_{+}(x))}{\cos^3(\alpha_{3}(x, x_{+}(y)))}.
\end{equation}
From $\tan(\phi(x, y_{s}(x)))=y_{s}'(x)$ we have
\begin{equation}\label{61104}
y_{s}''(x)=\frac{\bar{\partial}_{_\Gamma}\phi(x, y_{s}(x))}{\cos^3(\phi(x, y_{s}(x)))}.
\end{equation}


From (\ref{61115}), (\ref{7}), (\ref{10}), (\ref{61114}), (\ref{102002}), and (\ref{250402}) we have that at the point ${F}$,
\begin{equation}\label{61108}
\begin{aligned}
\bar{\partial}_{_\Gamma} \phi&=\frac{\bar{\partial}_{_\Gamma} m}{\rho_fL_f}-\frac{N_f}{\rho_fL_fq_f^2}(q_f^2+\tau_f^2 p'(\tau_f))\bar{\partial}_{_\Gamma}\rho_f+\bar{\partial}_{_\Gamma}\sigma_f\\&=
-\frac{N_f}{\rho_fL_fq_f^2}(q_f^2+\tau_f^2 p'(\tau_f))\bar{\partial}_{+}\rho_{3}+\bar{\partial}_{+}\Big(\frac{\alpha_{3}+\beta_{3}}{2}\Big)
\\&=-\frac{\sin 2A_f}{2\rho_f}\bar{\partial}_{+}\rho_{3}+\frac{\bar{\partial}_{+}\alpha_{3}}{2}-\frac{p''(\tau_f)\tan A_f}{4c_f^2\rho_f^4}\bar{\partial}_{+}\rho_{3}
\\&=-\frac{\sin 2A_f}{2\rho_f}\bar{\partial}_{+}\rho_{3}+\frac{\bar{\partial}_{+}\alpha_{3}}{2}-\frac{p''(\tau_f)\tan A_f}{4c_f^2\rho_f^4}\bar{\partial}_{+}\rho_{3}\\&>-\frac{\sin 2A_f}{2\rho_f}\bar{\partial}_{+}\rho_{3}+\frac{\bar{\partial}_{+}\alpha_{3}}{2}
\end{aligned}
\end{equation}
and
\begin{equation}\label{103006}
\begin{aligned}
\bar{\partial}_+ \alpha_{3}&=-\frac{p''(\tau_f)}{4c_f^2\rho_f^4}\left(-\frac{4p'(\tau_f)+\tau_f p''(\tau_f)}{\tau_f p''(\tau_f)}-\tan^2A_f\right)\sin 2A_f\bar{\partial}_{+}\rho_{3}\\&<
-\frac{p''(\tau_f)}{4c_f^2\rho_f^4}\left(-\frac{4p'(\tau_f)}{\tau_f p''(\tau_f)}\right)\sin 2A_f\bar{\partial}_{+}\rho_{3}\\&=-\frac{\sin 2A_f}{\rho_f}\bar{\partial}_{+}\rho_{3}.
\end{aligned}
\end{equation}
This implies
\begin{equation}\label{61106}
\bar{\partial}_{_\Gamma} \phi(x_{_{F}}, y_{s}(x_{_{F}}))>\bar{\partial}_+ \alpha_{3}(x_{_{F}}, y_{+}(x_{_{F}})).
\end{equation}

Combining (\ref{61102}) and (\ref{61106}), there exists a sufficiently small $\varepsilon>0$ such that 
\begin{equation}\label{61101}
y_{s}(x)>y_{+}(x)\quad \mbox{for}~ x_{_{F}}<x<x_{_{F}}+\varepsilon.
\end{equation}

A direct computation yields
$$
(\cos\phi, \sin\phi)=\frac{\sin (\beta_f-\phi)}{\sin (\beta_f-\alpha_f)}(\cos\alpha_f, \sin\alpha_f)-
\frac{\sin (\alpha_f-\phi)}{\sin (\beta_f-\alpha_f)}(\cos\beta_f, \sin\beta_f).
$$
Thus, we have
\begin{equation}\label{82003}
\bar{\partial}_{_\Gamma}\tau_f=\frac{\sin (\beta_{3}-\phi)}{\sin (\beta_{3}-\alpha_{3})}\bar{\partial}_{+}\tau_{3}-\frac{\sin (\alpha_{3}-\phi)}{\sin (\beta_{3}-\alpha_{3})}\bar{\partial}_{-}\tau_{3}.
\end{equation}

By (\ref{61115}) we know that
\begin{equation}\label{82004}
\|\phi-\xi_{e}\|_{0;  {\it S}_{_{F}}}\rightarrow 0 \quad \mbox{as}~\tau_1^e-\tau_0\rightarrow 0,
\end{equation}
since the diameter of the region $\Omega^3$ approaches $0$ as $\tau_1^e-\tau_0\rightarrow 0$.
Therefore, by (\ref{6811}), (\ref{82003}) and (\ref{82004}) we have
\begin{equation}\label{82008}
\bar{\partial}_{_\Gamma}\tau_f>0\quad \mbox{along}~{\it S}_{_{F}},
\end{equation}
as $\tau_1^e-\tau_0$ is sufficiently small.
Combining this with $\tau_f({F})=\tau_f^e$ and (\ref{6811}), we immediately obtain
(\ref{72202}) for sufficiently small $\tau_1^e-\tau_0$.
This also implies that ${\it S}_{_{F}}$ does not intersect the characteristic lines $\widetilde{PE}$ and $\widetilde{PF}$.

Furthermore, by (\ref{61113}) we  obtain
\begin{equation}\label{62603a}
\phi(x, y_{s}(x))>\alpha_f(x, y_{s}(x))
\end{equation}
if $(x, y_{s}(x))$ lies in $\Omega^3$.
This implies that the post-sonic shock front does not intersect the $C_{+}$ characteristic line $\widetilde{FS}$.
This completes the proof of the lemma.
\end{proof}

The states on the backside of ${\it S}_{_{F}}$ can then be determined by
\begin{equation}\label{110405}
u_b=L_b\cos\phi+N_b\sin\phi,\quad  v_b=L_b\sin\phi-N_b\cos\phi,
\end{equation}
where
$L_b=L_f$ and  $N_b=\frac{\psi_{po}(\tau_f)N_f}{\tau_f}$.

By symmetry we can obtain the post-sonic shock propagated from $E$ and it's back side  states.
We still use $\phi$ to denote the inclination angle of ${\it S}_{_{E}}$.
Since ${\it S}_{_{E}}$ is an oblique $1$-shock, we let
\begin{equation}\label{110402}
L=u\cos \phi+v\sin \phi\quad \mbox{and} \quad N=-u \sin \phi+v\cos \phi.
\end{equation}
We use subscripts `$f$' and `$d$' to denote the states on the front and back sides of $S_{_E}$, respectively.
Then (\ref{RH1}) and (\ref{101101}) still hold hon $S_{_E}$.



By (\ref{110402}) we have
$
N=q\sin(\sigma-\phi)
$ on $S_{_E}$.
Thus, we have
\begin{equation}
\phi=\sigma-\arcsin\Big(\frac{N}{q}\Big)\quad \mbox{on}~{\it S}_{_{E}}.
\end{equation}
From (\ref{U}) and (\ref{110402}) we have
$
m=\rho  N =\rho  q  (\sin\sigma \cos\phi-\cos\sigma \sin\phi)
$  on $S_{_E}$.
Thus,
\begin{equation}\label{4250103}
\bar{\partial}_{_\Gamma} m=\frac{N }{q }\partial_{s}(\rho q )+\rho L \bar{\partial}_{_\Gamma}\sigma -\rho L \bar{\partial}_{_\Gamma} \phi\quad \mbox{along}~ {\it S}_{_{E}}.
\end{equation}
Inserting (\ref{110403}) into this, we get
\begin{equation}\label{110408}
\bar{\partial}_{_\Gamma} \phi=\frac{p''(\psi_{po}(\tau_f))\psi_{po}'(\tau_f)\bar{\partial}_{_\Gamma} \tau_f}{2\rho_f^2 L_fN_f}+\frac{N_{f}}{\rho_{f}L_{f}q_{f}^2}(q_{f}^2+\tau_{f}^2 p'(\tau_{f}))\bar{\partial}_{_\Gamma}\rho_{f}+\bar{\partial}_{_\Gamma}\sigma_{f}\quad\mbox{along}~{\it S}_{_{E}}.
\end{equation}

The states on the backside of the post-sonic shock ${\it S}_{_{E}}$ can then be determined by
\begin{equation}\label{110406}
u_d=L_d\cos\phi-N_d\sin\phi,\quad  v_d=L_d\sin\phi+N_d\cos\phi,
\end{equation}
where
$L_d=L_f$ and  $N_d=\frac{\psi_{po}(\tau_f)N_f}{\tau_f}$.

From Lemma \ref{91011} we can see that the two shocks intersect at some point $G$ inside the domain $\Omega^3$. From now on, we use  $\wideparen{EG}$ the $\wideparen{FG}$ to denote the post-sonic shocks propagated from $E$ and $F$, respectively; see Fig. \ref{Fig5.5.2} (left).
We now let $\Omega_3$ be the domain bounded by $\widetilde{PE}$, $\widetilde{PF}$, $\wideparen{EG}$, and $\wideparen{FG}$.
Then the flow in the domain $\Omega_3$ is $(u, v)=(u_3, v_3)(x, y)$.

\begin{lem}\label{110411}
(Structural conditions)
Assume $\tau_1^e-\tau_0$ to be sufficiently small. Then we have the following estimates:
\begin{equation}\label{72204a}
\bar{\partial}_{_\Gamma}r_{b}>0\quad \mbox{and}\quad \bar{\partial}_{_\Gamma}s_{b}>0\quad \mbox{along}~ \wideparen{FG};
\end{equation}
\begin{equation}\label{72204b}
\bar{\partial}_{_\Gamma}r_{d}<0\quad \mbox{and}\quad \bar{\partial}_{_\Gamma}s_{d}<0\quad \mbox{along}~ \wideparen{EG};
\end{equation}
where $(r_{b}, s_b)=(r(u_b, v_b), s(u_b, v_b))$ and  $(r_{d}, s_d)=(r(u_d, v_d), s(u_d, v_d))$.
\end{lem}

\begin{proof}
We first prove (\ref{72204a}).
The proof proceeds in three steps.

\vskip 4pt

\noindent
{\it Step 1.}
Using (\ref{7}), (\ref{10}), (\ref{102002}),  and (\ref{250402}) we have
\begin{equation}\label{250407}
\begin{aligned}
\bar{\partial}_{_\Gamma} \phi&=\frac{\bar{\partial}_{_\Gamma} m}{\rho_fL_f}-\frac{N_f}{\rho_fL_fq_f^2}(q_f^2+\tau_f^2 p'(\tau_f))\bar{\partial}_{_\Gamma}\rho_f+\bar{\partial}_{_\Gamma}\sigma_f\\&=
-\frac{N_f}{\rho_fL_fq_f^2}(q_f^2-c_f^2)\bar{\partial}_{+}\rho_3
+\bar{\partial}_{+}\Big(\frac{\alpha_3+\beta_3}{2}\Big)
\\&=-\frac{\sin 2A_{3}}{2\rho_{3}}\bar{\partial}_{+}\rho_{3}-\frac{p''(\tau_{3})}{4c_{3}^2\rho_{3}^4}
\big(\varpi(\tau_{3})+1\big)\sin A_{3}\cos A_{3}\bar{\partial}_{+}\rho_{3}\\&>-\frac{\sin 2A_{3}}{\rho_{3}}\bar{\partial}_{+}\rho_{3}\quad\mbox{at}~ {F}.
\end{aligned}
\end{equation}
Using (\ref{6811}), (\ref{250407}), and (\ref{250402}) we know that when $\tau_1^e-\tau_0$ is sufficiently small,
\begin{equation}\label{61109}
\bar{\partial}_{_\Gamma} \phi>\frac{\tau_1^e\sin(2A_1^e)}{4}\mathcal{L}\quad \mbox{along}~ \wideparen{FG}.
\end{equation}

\noindent
{\it Step 2.}
In view of (\ref{50701}) we have
\begin{equation}\label{42801}
\cos\alpha_b \bar{\partial}_{_\Gamma} u_b+\sin\alpha_b \bar{\partial}_{_\Gamma} v_b~=~
\cos A_b \bar{\partial}_{_\Gamma} q_b+q_b\sin A_b \bar{\partial}_{_\Gamma} \sigma_b\quad \mbox{along}~ \wideparen{FG}.
\end{equation}

Inserting (\ref{102801a}) and  (\ref{250102}) into  (\ref{42801}) and recalling the first equality in (\ref{6810}), we get
\begin{equation}\label{110412}
\begin{aligned}
&\cos\alpha_b \bar{\partial}_{_\Gamma} u_b+\sin\alpha_b \bar{\partial}_{_\Gamma} v_b\\=~&
\cos A_b\left(-\tau_bq_b\sin^2 A_b\bar{\partial}_{_\Gamma} \rho_b\right)
+q_b\sin A_b\left(\bar{\partial}_{_\Gamma} \phi-\frac{\bar{\partial}_{_\Gamma} m}{\rho_bL_b}+\frac{\sin A_b\cos A_b}{\rho_b}\bar{\partial}_{_\Gamma}\rho_b\right)\\=~&
q_b\sin A_b \left(\bar{\partial}_{_\Gamma} \phi-\frac{\bar{\partial}_{_\Gamma} m }{\rho_bL_b}\right) \quad \mbox{along}~ \wideparen{FG}.
\end{aligned}
\end{equation}
Therefore, by (\ref{250402}) and (\ref{61109}) we know that when $\tau_1^e-\tau_0$ is sufficiently small, $$\cos\alpha_b \bar{\partial}_{_\Gamma} u_b+\sin\alpha_b \bar{\partial}_{_\Gamma} v_b>0\quad \mbox{along} ~\wideparen{FG}.$$
Hence, by (\ref{102803a}) we have
\begin{equation}\label{123104}
\bar{\partial}_{_\Gamma}r_{b}=\frac{\cos\alpha_b \bar{\partial}_{_\Gamma} u_b+\sin\alpha_b \bar{\partial}_{_\Gamma} v_b}{c_b}>0\quad \mbox{along}~ \wideparen{FG}.
\end{equation}
\noindent
{\it Step 3.}
In view of (\ref{50701}), we have
\begin{equation}\label{82006}
\cos\beta_b \bar{\partial}_{_\Gamma} u_b+\sin\beta_b \bar{\partial}_{_\Gamma} v_b~=~
\cos A_b \bar{\partial}_{_\Gamma} q_b-q_b\sin A_b \bar{\partial}_{_\Gamma} \sigma_b\quad \mbox{along}~\wideparen{FG}.
\end{equation}
Inserting (\ref{102801a}) and  (\ref{250102}) into (\ref{82006}), we get
\begin{equation}\label{110413}
\begin{aligned}
&\cos\beta_b \bar{\partial}_{_\Gamma} u_b+\sin\beta_b \bar{\partial}_{_\Gamma} v_b\\=~&
\cos A_b\left(-\tau_bq_b\sin^2 A_b\bar{\partial}_{_\Gamma} \rho_b\right)
-q_b\sin A_b\left(\bar{\partial}_{_\Gamma} \phi-\frac{\bar{\partial}_{_\Gamma} m}{\rho_bL_b}+\frac{\sin A_b\cos A_b}{\rho_b}\bar{\partial}_{_\Gamma}\rho_b\right)\\=~&
-q_b\sin A_b \left(\bar{\partial}_{_\Gamma} \phi-\frac{\bar{\partial}_{_\Gamma} m }{\rho_bL_b}\right)-
2\tau_bq_b\cos A_b\sin^2 A_b\bar{\partial}_{_\Gamma} \rho_b\quad \mbox{along}~ \wideparen{FG}.
\end{aligned}
\end{equation}

Recalling $\tau_b=\psi_{po}(\tau_f)$ and (\ref{82008}), we know that
when $\tau_1^e-\tau_0$ is sufficiently small,
\begin{equation}\label{61308}
\bar{\partial}_{_\Gamma} \rho_b=-\frac{1}{\tau_b^2}\psi_{po}'(\tau_f)\bar{\partial}_{_\Gamma} \tau_f>0 \quad \mbox{along}~\wideparen{{FG}}.
\end{equation}
Therefore, when $\tau_1^e-\tau_0$ is sufficiently small, $ \cos\beta_b \bar{\partial}_{_\Gamma} u_b+\sin\beta_b \bar{\partial}_{_\Gamma} v_b<0$ along $\wideparen{{FG}}$.
Hence, by (\ref{102803a}) we have
\begin{equation}\label{123105}
\bar{\partial}_{_\Gamma}s_{b}=\frac{-(\cos\beta_b \bar{\partial}_{_\Gamma} u_b+\sin\beta_b \bar{\partial}_{_\Gamma} v_b)}{c_b}>0\quad \mbox{along}~ \wideparen{FG}.
\end{equation}
This completes the proof of (\ref{72204a}). The proof for (\ref{72204b}) is similar; we omit the details.
\end{proof}

\begin{lem}\label{2501a}
Assume $\tau_1^e-\tau_0$ to be sufficiently small. Then we have
$$
(r_d(x, y), s_d(x, y))\in \Pi\quad \mbox{for}~ (x,y)\in\wideparen{EG}; \quad (r_b(x, y), s_b(x, y))\in \Pi\quad \mbox{for}~ (x,y)\in\wideparen{FG}.
$$
\end{lem}
\begin{proof}
By the previous discussion we know have
$\tau_2^i\leq\tau_b\leq\tau_2^e$ on $\wideparen{FG}$, and $\tau_2^i\leq\tau_d\leq\tau_2^e$  on $\wideparen{EG}$.

By Proposition \ref{pro3.1.1} and $v_{f}(G)=0$, we have
\begin{equation}\label{120601}
\frac{r_b(G)+s_{b}(G)}{2}=\sigma_b(G)<0<\sigma_d(G)=\frac{r_d(G)+s_{d}(G)}{2}.
\end{equation}
Therefore by Lemma \ref{110411} we have
$$
0<\frac{r_d+s_{d}}{2}<\frac{r_d(E)+s_{d}(E)}{2}<\theta_{+}\quad \mbox{on} ~\wideparen{EG}
$$
and
$$
\theta_{-}<\frac{r_b(F)+s_{b}(F)}{2}<\frac{r_b+s_{b}}{2}<0\quad \mbox{on} ~\wideparen{FG}.
$$
This completes the proof.
\end{proof}

We next consider (\ref{E1}) with the data
\begin{equation}\label{102804}
(u, v)=(u_b, v_b)(x, y),\quad (x,y)\in \wideparen{FG}.
\end{equation}
By (\ref{61113}) and the first inequality of (\ref{72204a}) we see that the problem (\ref{E1}), (\ref{102804}) is a singular Cauchy problem. By (\ref{72204a})  we know that when $\tau_1^e-\tau_0$ is sufficiently small, the assumptions of Lemma \ref{lem7} hold.
So, by Lemma \ref{lem7} we know that
the singular Cauchy problem  (\ref{E1}), (\ref{102804}) admits a solution on a triangle domain $\overline{\Omega}_4$ bounded by $\wideparen{FG}$, $\widetilde{GK}$, and $\widetilde{FK}$, where $\widetilde{GK}$ is a forward $C_{-}$
characteristic line issued from $G$ and $\widetilde{FK}$ is a forward $C_{+}$ characteristic line issued from $F$; see Fig. \ref{Fig5.5.2} (right). Moreover, the solution satisfies
\begin{equation}\label{102805}
\bar{\partial}_{+}\rho<0 \quad \mbox{in}~ \overline{\Omega}_4; \quad \bar{\partial}_{-}\rho<0 \quad \mbox{in}~ \overline{\Omega}_4\setminus\wideparen{FG}.
\end{equation}
For convenience, we use subscript `$4$' to denote the flow in the region  $\overline{\Omega}_4$, i.e., $(u, v)=(u_4, v_4)(x,y)$, $(x,y)\in \overline{\Omega}_4$.

Similarly,
by solving (\ref{E1}) with the data
$$
(u, v)=(u_b, v_b)(x, y),\quad (x,y)\in \wideparen{EG},
$$
 we obtain
a solution on a triangle domain $\overline{\Omega}_5$ bounded by $\wideparen{EG}$, $\widetilde{EJ}$, and $\widetilde{GJ}$, where $\widetilde{EJ}$ is a forward $C_{-}$
characteristic line issued from $E$ and $\widetilde{GJ}$ is a forward $C_{+}$ characteristic line issued from $G$.
Moreover, the solution satisfies
\begin{equation}\label{102806}
\bar{\partial}_{-}\rho<0 \quad \mbox{in}~ \overline{\Omega}_4; \quad \bar{\partial}_{+}\rho<0 \quad \mbox{in}~ \overline{\Omega}_5\setminus\widetilde{EG}.
\end{equation}
We use subscript `$5$' to denote the flow in the region  $\overline{\Omega}_5$, i.e., $(u, v)=(u_5, v_5)(x,y)$, $(x,y)\in \overline{\Omega}_5$.
Actually, we have $(u_5, v_5)(x,y)=(u_4, -v_4)(x,-y)$, $(x,y)\in \overline{\Omega}_5$.

\begin{rem}
If $r_b(G)-s_{b}(F)<2\mathcal{R}$ then the points $K$ and $J$ exist. However, if $r_b(G)-s_{b}(F)\geq2\mathcal{R}$ then the points $K$ and $J$ are at infinity. When $K$ is at infinity, we still use $\widetilde{FK}$ ($\widetilde{GK}$, resp.) to denote the forward $C_{+}$ ($C_{-}$, resp.) characteristic curve issued from $F$ ($G$, resp.) of the singular Cauchy problem.
\end{rem}

\subsubsection{\bf Global solution to the problem (\ref{E1}), (\ref{102807})}
We now construct the solution of the problem (\ref{E1}), (\ref{102807}) in the remaining part.
We consider (\ref{E1}) with the boundary condition
\begin{equation}\label{102811}
(u, v)=\left\{
               \begin{array}{ll}
                  (u_5, v_5)(x,y), & \hbox{$(x, y)\in\widetilde{EJ}$;} \\
               (\tilde{u}_r, \tilde{v}_r)(\eta), & \hbox{$(x, y)\in\widetilde{EH}$.}
               \end{array}
             \right.
\end{equation}
The problem (\ref{E1}), (\ref{102811}) is a standard Goursat problem.
From (\ref{120602}) and (\ref{102806}) we see that the assumptions of Lemma \ref{lem1} hold.
Then by Lemma \ref{lem1}
we know that the problem (\ref{E1}), (\ref{102811}) admits a classical solution on a  curvilinear quadrilateral domain $\overline{\Omega}_7$ bounded by $\widetilde{EJ}$, $\widetilde{EH}$, $\widetilde{JL}$, and $\widetilde{HL}$, where $\widetilde{HL}$ is a forward $C_{-}$ characteristic curve issued from $H$ and $\widetilde{JL}$ is a forward $C_{+}$ characteristic curve issued from $J$. Moreover, the solution satisfies $\bar{\partial}_{\pm}\rho<0$ and $(r, s)\in\Pi$ in $\overline{\Omega}_{7}$.
We denote the solution in $\overline{\Omega}_7$ by $(u, v)=(u_7, v_7)(x,y)$.

Similarly, by solving (\ref{E1}) with the boundary condition
\begin{equation}\label{102812}
(u, v)=\left\{
               \begin{array}{ll}
                  (u_4, v_4)(x,y), & \hbox{$(x, y)\in\widetilde{FK}$;} \\
               (\hat{u}_r, \hat{v}_r)(\xi), & \hbox{$(x, y)\in\widetilde{FI}$,}
               \end{array}
             \right.
\end{equation}
we can obtain the flow on a  curvilinear quadrilateral domain $\overline{\Omega}_6$ bounded by $\widetilde{FK}$, $\widetilde{FI}$, $\widetilde{KN}$, and $\widetilde{IM}$, where $\widetilde{KM}$ is a forward $C_{-}$ characteristic curve issued from $K$ and $\widetilde{IM}$ is a forward $C_{+}$ characteristic curve issued from $I$. Moreover, the solution satisfies $\bar{\partial}_{\pm}\rho<0$ and $(r, s)\in\Pi$ in $\overline{\Omega}_{6}$.
We denote the solution in $\overline{\Omega}_6$ by $(u, v)=(u_6, v_6)(x,y)$.

Finally, we consider (\ref{E1}) with the boundary condition
\begin{equation}\label{102813}
(u, v)=\left\{
               \begin{array}{ll}
                  (u_5, v_5)(x,y), & \hbox{$(x,y)\in \widetilde{{GJ}}$;} \\
                   (u_7, v_7)(x,y), & \hbox{$(x,y)\in \widetilde{{JL}}$;} \\
               (u_4, v_4)(x,y), & \hbox{$(x,y)\in \widetilde{{GK}}$;} \\
                   (u_6, v_6)(x,y), & \hbox{$(x,y)\in \widetilde{{KM}}$.}
               \end{array}
             \right.
\end{equation}
By  Lemma \ref{lem8} we know that the problem (\ref{E1}), (\ref{102813}) admits a classical solution on a  curvilinear quadrilateral domain $\overline{\Omega}_8$ bounded by characteristic curves $\widetilde{{GJ}}\cup \widetilde{{JL}}$, $ \widetilde{{GK}}\cup  \widetilde{{KM}}$, $C_{+}^{M}$, and $C_{-}^{L}$. Moreover, the solution satisfies $\bar{\partial}_{\pm}\rho<0$ in $\overline{\Omega}_{8}\setminus\{G\}$. 

\begin{rem}
If the points $K$ and $J$ are at infinity, we only need to replace (\ref{102813}) with
$$
(u, v)=\left\{
               \begin{array}{ll}
               (u_4, v_4)(x,y), & \hbox{$(x,y)\in \widetilde{{GK}}$;} \\
                  (u_5, v_5)(x,y), & \hbox{$(x,y)\in \widetilde{{GJ}}$.}
               \end{array}
             \right.
$$
\end{rem}

Let
\begin{equation}\label{123101}
(u, v)=(u_i, v_i)(x,y), \quad (x, y)\in \Omega_i,\quad i=3, 4, 5, 6, 7, 8.
\end{equation}
Then the $(u, v)$ defined in (\ref{123101}) is a piecewise smooth solution
 to the boundary value problem (\ref{E1}), (\ref{102807}).
This solve the interaction of the fan-shock-fan composite waves ${\it fsf}_{\pm}$.

\subsubsection{\bf Global existence of a piecewise smooth solution in the divergent duct}
By solving two simple wave problems as described in Section 5.1, we see that
there are two simple waves, denoted by ${\it f}_{\pm}$, issued from the fan-shock-fan composite wave interaction region; see Fig. \ref{Fig5.5.1}.
The two simple waves will then reflect off the walls and generate two new simple waves, and the solution in the remaining part of the duct can be determined using the method described in Section 5.1.
This completes the proof of Theorem \ref{thm5}.

\begin{rem}\label{rem5.3}
In Theorem \ref{thm5}, we make an assumption that  $\tau_1^e-\tau_0$ is sufficiently small.
This assumption is crucial. Firstly, if  $\tau_1^e-\tau_0$ is not small,  the solution of the SGP problem (\ref{E1}), (\ref{102701}) does not satisfy $\tau<\tau_1^i$ in the whole determinate region and the monotonicity condition (\ref{120602}) might cause a blow-up of the solution. Secondly, we use the smallness to ensure that the structural conditions (\ref{72204a}) and (\ref{72204a}) hold. Actually, in view of the characteristic decomposition (\ref{cdr}) we can see that these structural conditions are necessary to obtain the flow downstream of the post-sonic shocks $\wideparen{EG}$ and $\wideparen{FG}$.
\end{rem}

\subsubsection{\bf Non-existence of transonic shocks and pre-sonic shocks}
One might ask whether the shock propagating from the point $F$
could be a transonic shock or a pre-sonic shock.
In what follows, we are going to show that the shock porpagating from the point  $F$ could not be a transonic shock or a pre-sonic shock.

\vskip 8pt
\noindent
{\bf (1) Nonexistence of transonic shocks.}
Suppose there exists
is a smooth transonic shock ${\it S}_{_F}$ propagated from $F$.
Then ${\it S}_{_F}\setminus\{F\}\in \Omega^3$ and for any point on ${\it S}_{_F}\setminus\{F\}$, the backward $C_{+}$ characteristic line issued from this point stays in the sectorial domain between ${\it S}_{_F}$ and $\widetilde{FI}$ until it intersects $\widetilde{FI}$ at a point.
We use subscripts `$f$' and `$b$' to denote the flow upstream and downstream of ${\it S}_{_F}$, respectively,  and use $\phi$ to denote the inclination angle of ${\it S}_{_F}$.

As in (\ref{61109}), we have that when $\tau_1^e-\tau_0$ is sufficiently small,
$$
\bar{\partial}_{_\Gamma} \phi>\frac{\tau_1^e\sin(2A_1^e)}{4}\mathcal{L}\quad \mbox{at}~ {F}.
$$

For a given small $\varepsilon>0$, we let $F_{\varepsilon}$ be the point on  ${\it S}_{_F}$ such that $\mbox{dist}(F, F_\varepsilon)=\varepsilon$.
When $\varepsilon$ is sufficiently small, we have $\tau_1^e<\tau_3<\tau_1^i$ on $\wideparen{FF_{\varepsilon}}$. Since
$\wideparen{FF_{\varepsilon}}$ is a transonic shock, we have $\tau_b<\tau_2^e$ on $\wideparen{FF_{\varepsilon}}$. This implies
$$
\bar{\partial}_{_\Gamma}\rho_b\geq 0 \quad \mbox{at}~  {F}.
$$

As in (\ref{110413}),
we have
$$
\bar{\partial}_{_\Gamma}s_{b}=-\frac{\cos\beta_b \bar{\partial}_{_\Gamma} u_b+\sin\beta_b \bar{\partial}_{_\Gamma} v_b}{c_b}=\bar{\partial}_{_\Gamma} \phi+
\tau_b\sin(2A_b)\bar{\partial}_{_\Gamma} \rho_b>\frac{\tau_1^e\sin(2A_1^e)}{4}\mathcal{L}\quad \mbox{at}~{F}.
$$

Integrating the first equation of (\ref{cdr}) along the forward $C_{+}$ characteristic lines issued from $\widetilde{FI}$, we see that when $\varepsilon$ is sufficiently small, $\bar{\partial}_{-}s_b$ is bounded on $\widetilde{FF_{\varepsilon}}$. So, when $\varepsilon$ is sufficiently small,
\begin{equation}\label{121701a}
\bar{\partial}_{+}s_{b}=\frac{\sin (2A_b)}{\sin(\phi-\beta_b)}\bar{\partial}_{_\Gamma}s_{b}-
\frac{\sin (\alpha_b-\phi)}{\sin(\phi-\beta_b)}\bar{\partial}_{-}s_{b}>0 \quad\mbox{on}~ \wideparen{FF_{\varepsilon}}.
\end{equation}

Since $\wideparen{FF_{\varepsilon}}$ is not a characteristic of the flow downstream of it. The flow downstream of it must satisfy the equations (\ref{52801}) on $\wideparen{FF_{\varepsilon}}$. While, by (\ref{121701a}) we know that the characteristic equation $\bar{\partial}_{+}s=0$ does not hold on  $\wideparen{FF_{\varepsilon}}$. This leads to a contradiction. So, the shock propagated from the point $F$ could not be a transonic shock.

\begin{figure}[htbp]
\begin{center}
\includegraphics[scale=0.35]{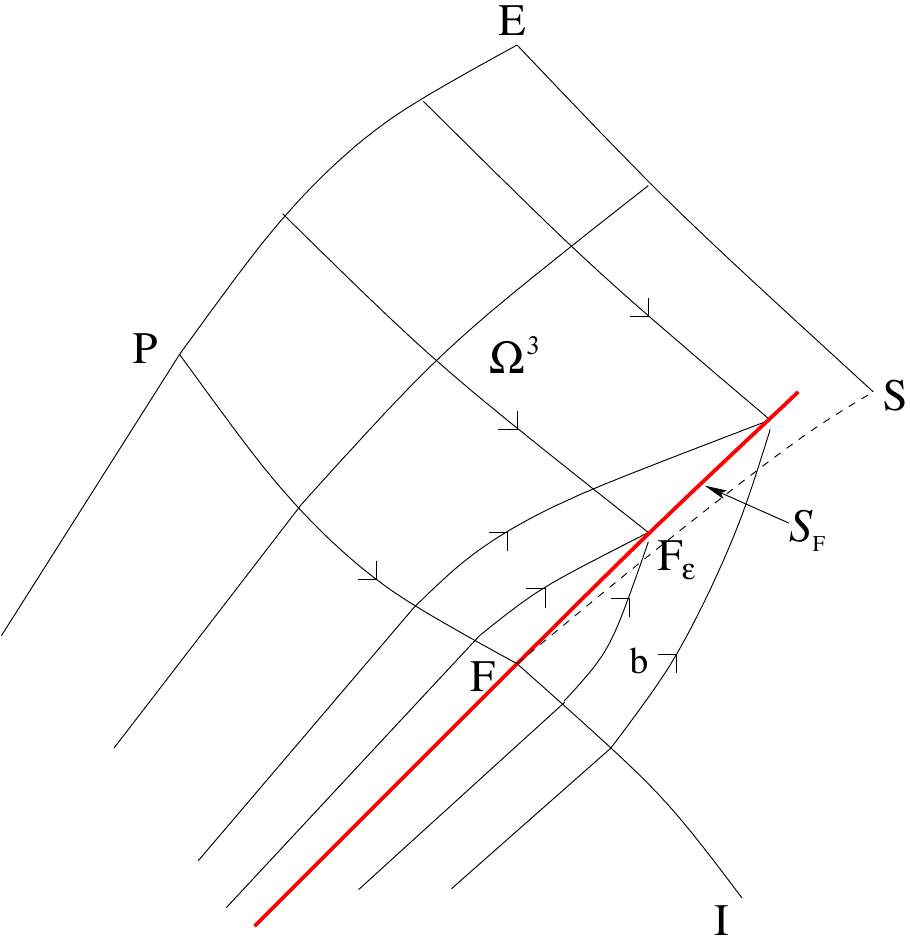}\qquad\qquad \qquad\includegraphics[scale=0.35]{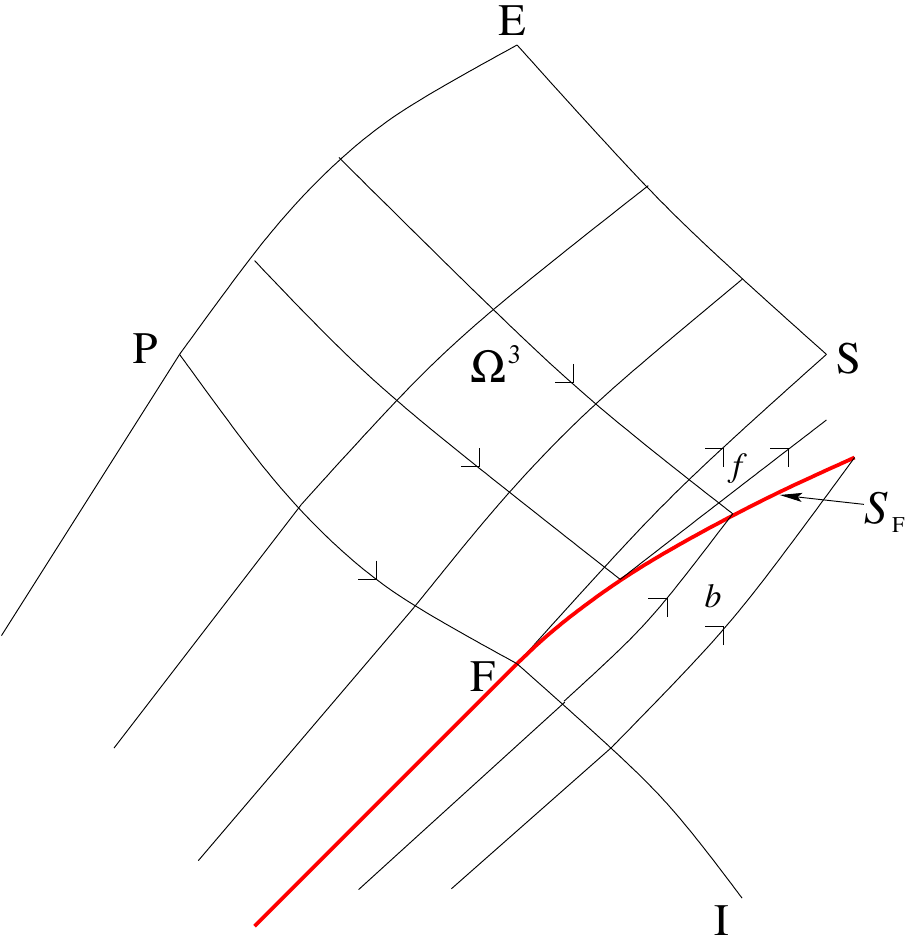}
\caption{\footnotesize Left: an imaginary transonic shock propagated from $F$; right:  an imaginary pre-sonic shock propagated from $F$.}
\label{Fig5.5.3}
\end{center}
\end{figure}

\vskip 8pt
\noindent
{\bf (2) Nonexistence of  pre-sonic shocks.}
Suppose that there is a smooth pre-sonic shock ${\it S}_{_F}$ propagated from $F$; see Fig. \ref{Fig5.5.3} (right).
Since the shock ${\it S}_{_F}$ is assumed to be pre-sonic, it must lie outside the domain ${\Omega}^3$.
For any point in the sectorial domain between $\widetilde{FS}$ and ${\it S}_{_F}$, the backward $C_{-}$ ($C_{+}$, resp.) characteristic curve issued from this point stays in this sectorial domain until it intersects $\widetilde{FS}$ (${\it S}_{_F}$, resp.) at a point.
For any point on ${\it S}_{_F}$, the backward $C_{+}$ characteristic curve stays in the sectorial domain between ${\it S}_{_F}$ and $\widetilde{FI}$ until it intersects $\widetilde{FI}$ at a point.
For convenience, we use subscripts `$f$' and `$b$' to denote the flow upstream and downstream of ${\it S}_{_F}$, respectively, and use $\phi$ to denote the inclination angle of ${\it S}_{_F}$.
Actually, by the method of characteristics we can see that $r_f$ is $C^1$ smooth up to the pre-sonic shock ${\it S}_{_F}$.


Integrating the second equation of (\ref{cdr}) along $\widetilde{PF}$ and recalling $(\bar{\partial}_{+}r)(P)>0$ and $\phi(F)=\alpha_f(F)$, we have
\begin{equation}\label{10705}
\bar{\partial}_{_\Gamma}r_f=\bar{\partial}_{+}r_f=\bar{\partial}_{+}r_3>0\quad \mbox{at}~ F.
\end{equation}
From (\ref{250104}) we have $\bar{\partial}_{_\Gamma} m=0$ at $F$.
So, as in (\ref{110412}) and (\ref{123104}) we have
$$
\bar{\partial}_{_\Gamma}\phi=\frac{\bar{\partial}_{_\Gamma} m}{\rho_fL_f}+\bar{\partial}_{_\Gamma}r_f
=\bar{\partial}_{_\Gamma}r_f>0\quad \mbox{at}~ F.
$$

Since the shock is assumed to be pre-sonic, by the Liu's extended entropy condition we have $\tau_b\leq \tau_2^e$ along ${\it S}_{_F}$. Hence, we have $\bar{\partial}_{_\Gamma}\rho_b\geq 0$ at $F$.
Then, as in (\ref{250407}) we have
$$
\begin{aligned}
\bar{\partial}_{_\Gamma} \phi&=\frac{\bar{\partial}_{_\Gamma} m}{\rho_bL_b}-\frac{N_b}{\rho_bL_bq_b^2}(q_b^2+\tau_b^2 p'(\tau_b))\bar{\partial}_{_\Gamma}\rho_b+\bar{\partial}_{_\Gamma}\sigma_b\\&=
-\frac{N_b}{\rho_bL_bq_b^2}(q_b^2-c_b^2)\bar{\partial}_{+}\rho_b
+\bar{\partial}_{+}\Big(\frac{\alpha_b+\beta_b}{2}\Big)
\\&=-\frac{\sin 2A_{b}}{2\rho_{b}}\bar{\partial}_{_\Gamma}\rho_{b}-\frac{p''(\tau_{b})}{4c_{b}^2\rho_{b}^4}
\Big[\big(\varpi(\tau_{b})-\tan^2A_{b}\big)\sin A_{b}\cos A_{b}+\tan A_{b}\Big]\bar{\partial}_{_\Gamma}\rho_{b}\\&=-\frac{\sin 2A_{b}}{\rho_{b}}\bar{\partial}_{_\Gamma}\rho_{b}<0\quad\mbox{at}~ {F}.
\end{aligned}
$$
This leads to a contradiction. So, the shock propagated from the point $F$ could not be a pre-sonic shock.




\subsection{Interaction of a fan wave with a shock-fan composite wave}
We consider the IBVP (\ref{E1}), (\ref{42701}) for $\tau_1^e<\tau_0<\tau_1^i$.
In this part, we assume that the oblique waves at the corners $B$ and $D$ are centered simple wave and shock-fan composite wave, respectively, and the flow near the corners $B$ and $D$ can be represented by
$$
(u, v)=\left\{
               \begin{array}{ll}
                 (u_0, 0), & \hbox{$-\frac{\pi}{2}\leq \eta\leq \eta_0$;} \\[2pt]
                 (\tilde{u}, \tilde{v})(\eta), & \hbox{$\eta_0\leq \eta\leq \eta_1$;} \\[2pt]
                 (u_1, v_1), & \hbox{$\eta_{1}<\eta<\theta_{+}$}
               \end{array}
             \right.
\quad \mbox{and}\quad
(u, v)=\left\{
               \begin{array}{ll}
                 (u_0, 0), & \hbox{$\phi_{po}\leq \xi\leq \frac{\pi}{2}$;} \\[2pt]
                 (\bar{u}, \bar{v})(\xi), & \hbox{$\xi_2\leq \eta\leq \phi_{po}$;} \\[2pt]
                 (u_2, v_2), & \hbox{$\theta_{-}<\xi<\xi_2$,}
               \end{array}
             \right.
$$
respectively.
Here, the flow near the point $D$ is defined the same as that defined in Section 5.2; the function
$(\tilde{u}, \tilde{v})(\eta)$ is determined by (\ref{102306a}) and (\ref{102304a}) with initial data $(\bar{q}, \bar{\tau}, \bar{\sigma})(\eta_0)=(u_{0}, \tau_{0}, 0)$;
$\eta_1$ is determined by $\tilde{v}(\eta_1)=\tilde{u}(\eta_1)\tan\theta_{+}$. The main result of this subsection can be stated as the following theorem.
\begin{thm}\label{thm6}
Assume $\tau_1^e<\tau_0<\tau_1^i$ and $u_0>c_0$.
Assume furthermore that the oblique waves at the corners $B$ and $D$ are centered simple wave and shock-fan composite wave, respectively. 
Then the IBVP (\ref{E1}), (\ref{42701}) admits a global  piecewise smooth solution in $\Sigma$.
\end{thm}

\begin{figure}[htbp]
\begin{center}
\includegraphics[scale=0.5]{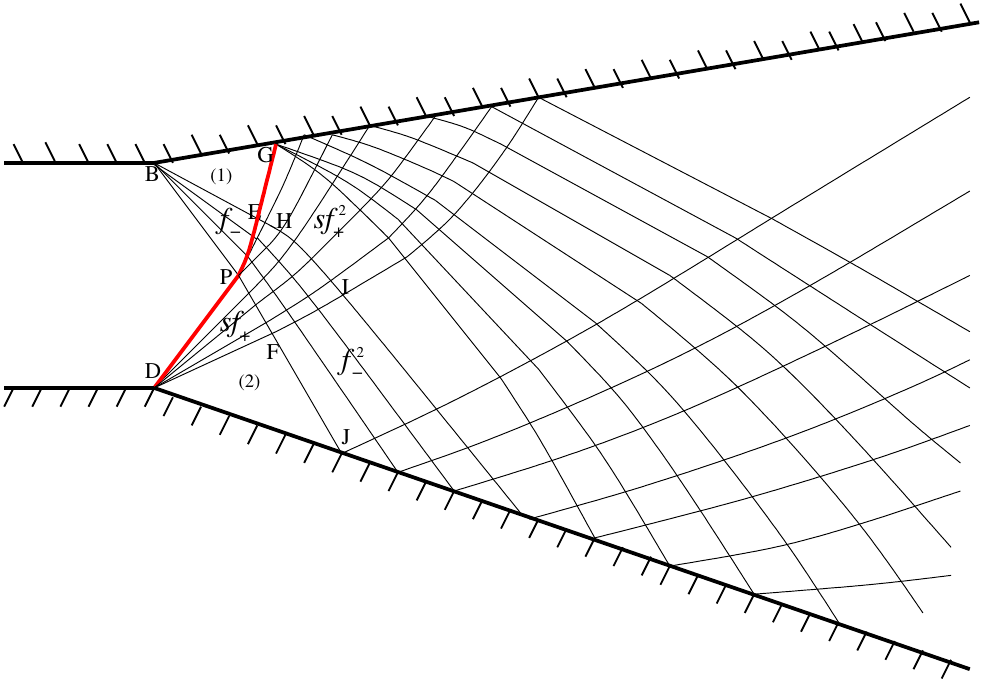}
\caption{\footnotesize Interaction of a centered simple wave with a shock-fan composite wave. }
\label{Fig5.6.1}
\end{center}
\end{figure}

We denote the fan wave and shock-fan composite wave by ${\it f}_{-}$ and ${\it sf}_{+}$, respectively; see Fig. \ref{Fig5.6.1}.
The two waves start to interact with each other at a point $P(x_{_P}, y_{_P})$, where
$x_{_P}=2(\tan\phi_{po}-\tan\eta_0)^{-1}$ and $y_{_P}=1+2(\tan\phi_{po}-\tan\eta_0)^{-1}\tan\eta_0$.
The post-sonic shock of the shock-fan composite wave ${\it sf}_{+}$ will propagate into the centered wave ${\it f}_{-}$ region from the point $P$ and remain as a post-sonic shock.

For convenience, we use ${\it S}_{_{P}}$ to denote the post-sonic shock front propagating from the point $P$.
We use subscripts `$f$' and `$b$' to denote the flow upstream and downstream of ${\it S}_{_{P}}$, respectively, and use $\phi$ to denote the inclination angle of ${\it S}_{_{P}}$.
We also set
$$
L=u\cos \phi+v\sin \phi\quad \mbox{and} \quad N=u\sin \phi-v\cos \phi\quad \mbox{along}~ S_{_P}.
$$

When ${\it S}_{_{P}}$ stays in the centered simple wave ${\it f}_{-}$ region,
the states on the front side of ${\it S}_{_{P}}$ are known by $(u_f, v_f, \tau_f)=(\tilde{u}, \tilde{v}, \tilde{\tau})(\eta)$.
The specific volume of the flow on the backside of ${\it S}_{_{P}}$ can be represented by $\tau_b=\psi_{po}(\tau_f)$.

As in (\ref{250102}), we have
\begin{equation}\label{123103}
\bar{\partial}_{_\Gamma} \phi=\frac{\bar{\partial}_{_\Gamma} m}{\rho_fL_f}-\frac{N_{f}}{\rho_{f}L_{f}q_{f}^2}(q_{f}^2+\tau_{f}^2 p'(\tau_{f}))\bar{\partial}_{_\Gamma}\rho_{f}+\bar{\partial}_{_\Gamma}\sigma_{f}\quad\mbox{along}~{\it S}_{_{P}}.
\end{equation}
From (\ref{250104}) we have
$$
\bar{\partial}_{_\Gamma}m
~=~\frac{\tau_f(N_f^2-c_f^2)}{\rho_fN_f(\tau_f^2-\tau_b^2)}
\bar{\partial}_{_\Gamma}\rho_f\quad \mbox{along}~ {\it S}_{_{P}}.
$$
So, we have
\begin{equation}\label{100903}
\bar{\partial}_{_\Gamma} \phi=\frac{\tau_f(N_f^2-c_f^2)}{\rho_f^2N_fL_f(\tau_f^2-\tau_b^2)}
\bar{\partial}_{_\Gamma}\rho_f-\frac{N_{f}\cos^2 A_f}{\rho_{f}L_{f}}\bar{\partial}_{_\Gamma}\rho_{f}+\bar{\partial}_{_\Gamma}\sigma_{f}\quad\mbox{along}~{\it S}_{_{P}}.
\end{equation}
Therefore, the post-sonic shock front ${\it S}_{_{P}}$ can be obtained by solving (\ref{100903}) with initial data $\phi(P)=\phi_{po}$. The states on the backside of ${\it S}_{_{P}}$ can then be determined by
$$
u_b=L_b\cos\phi+N_b\sin\phi\quad \mbox{and} \quad  v_b=L_b\sin\phi-N_b\cos\phi,
$$
where
$L_b=L_f$ and  $N_b=\frac{\tau_bN_f}{\tau_f}$.

The post-sonic shock front ${\it S}_{_{P}}$ will then intersects the straight $C_{-}$ characteristic line $y=1+x\tan\eta_1$ of the centered simple wave ${\it f}_{+}$  at a point $E$.  For convenience, we denote the post-sonic shock connecting $P$ and $E$ by $\wideparen{{PE}}$.

\begin{lem}\label{120801a}
There hold the following estimates:
\begin{equation}\label{72204}
(r_b, s_b)\in \Pi, \quad \bar{\partial}_{_\Gamma}r_{b}>0,\quad \mbox{and}\quad \bar{\partial}_{_\Gamma}s_{b}>0\quad \mbox{along}~\wideparen{PE};
\end{equation}
where $r_{b}=r(u_b, v_b)$ and $s_b=s(u_b, v_b)$.
\end{lem}
\begin{proof}
Firstly,
we point out that the proof of this lemma is different from that of Lemma \ref{110411}, since the shock at the point $P$ is not double-sonic and the amplitude of the centered simple wave ${\it f}_{+}$ is not required to be small.

We decompose $\bar{\partial}_{_\Gamma}$ in two characteristic directions
\begin{equation}\label{123102}
\bar{\partial}_{_\Gamma}=t_{+}\bar{\partial}_{+}+t_{-}\bar{\partial}_{-},
\end{equation}
where
$$
t_{+}=\frac{\sin(\phi-\beta_f)}{\sin(2A_f)}\quad \mbox{and}\quad t_{-}=\frac{\sin(\alpha_f -\phi)}{\sin(2A_f)}.
$$

As in (\ref{101501}), we have $\bar{\partial}_{+}\rho<0$ in ${\it f}_{-}$.
By the definition of 2-shock we have $0<\phi-\beta_f<\pi$. Then we have
$$
\bar{\partial}_{_\Gamma}\rho_{f}<0\quad \mbox{along}~  \wideparen{{PE}}.
$$
Combining this with (\ref{5906}) we have
$$
\bar{\partial}_{_\Gamma}\rho_{b}>0\quad \mbox{along}~  \wideparen{{PE}}.
$$

Then by (\ref{7}), (\ref{10}), and (\ref{123103}) we have
$$
\begin{aligned}
\bar{\partial}_{_\Gamma} \phi-\frac{\bar{\partial}_{_\Gamma} m}{\rho_fL_f}&=-\frac{N_{f}}{\rho_{f}L_{f}q_{f}^2}(q_{f}^2+\tau_{f}^2 p'(\tau_{f}))\bar{\partial}_{_\Gamma}\rho_{f}+\bar{\partial}_{_\Gamma}\sigma_{f}\\&=
-\frac{N_{f}\cos^2 A}{\rho_{f}L_{f}}t_{+}\bar{\partial}_{+}\rho_{f}+t_{+}\bar{\partial}_{+}\Big(\frac{\alpha_{f}+\beta_{f}}{2}\Big)
\\&=
-\Big(\frac{\sin(2A_f)}{2\rho_f}+\frac{N_{f}\cos^2 A}{\rho_{f}L_{f}}\Big)t_{+}\bar{\partial}_{+}\rho_{f}
\\&=
-\Big(\frac{\sin(2A_f)}{2\rho_f}+\frac{N_{f}\cos^2 A}{\rho_{f}L_{f}}\Big)\bar{\partial}_{_\Gamma}\rho_{f}>0\quad \mbox{along}~ \wideparen{{PE}}.
\end{aligned}
$$
 Hence,
\begin{equation}
\begin{aligned}
\bar{\partial}_{_\Gamma} \phi-\frac{\bar{\partial}_{_\Gamma} m}{\rho_bL_b}&=
\bar{\partial}_{_\Gamma} \phi-\frac{\bar{\partial}_{_\Gamma} m}{\rho_fL_f}+\left(\frac{1}{\rho_fL_f}-\frac{1}{\rho_bL_b}\right)\bar{\partial}_{_\Gamma} m\\
~&=~-\Big(\frac{\sin(2A_f)}{2\rho_f}+\frac{N_{f}\cos^2 A}{\rho_{f}L_{f}}\Big)\bar{\partial}_{_\Gamma}\rho_{f}+\frac{\tau_f(N_f^2-c_f^2)}{\rho_fN_fL_f(\tau_f+\tau_b)}
\bar{\partial}_{_\Gamma}\rho_f\\
~&>~-\Big(\frac{\sin(2A_f)}{2\rho_f}+\frac{N_{f}\cos^2 A}{\rho_{f}L_{f}}\Big)\bar{\partial}_{_\Gamma}\rho_{f}+\frac{N_f^2-c_f^2}{\rho_fN_fL_f}
\bar{\partial}_{_\Gamma}\rho_f\\&=
-\left(\frac{\sin(2A_f)}{2\rho_f}+\frac{(q_f^2-N_f^2)\sin^2 A_f}{\rho_fL_fN_f}\right)\bar{\partial}_{_\Gamma}\rho_{f}>0\quad \mbox{along}~ \wideparen{{PE}}.
\end{aligned}
\end{equation}

Therefore,
as in (\ref{110412}), (\ref{123104}), (\ref{110413}), and (\ref{123105}) we have
$$
\bar{\partial}_{_\Gamma}r_{b}= \bar{\partial}_{_\Gamma} \phi-\frac{\bar{\partial}_{_\Gamma} m }{\rho_bL_b}>0
$$
and
$$
\bar{\partial}_{_\Gamma}s_{b}= \left(\bar{\partial}_{_\Gamma} \phi-\frac{\bar{\partial}_{_\Gamma} m }{\rho_bL_b}\right)+
\tau_b\sin(2 A_b)\bar{\partial}_{_\Gamma} \rho_b>0
$$
along $\wideparen{PE}$.

The proof for $(r_b, s_b)\in \Pi$ on $\wideparen{PE}$ is similar to that of Lemma \ref{2501a}.
This completes the proof of the lemma.
\end{proof}

We then consider (\ref{E1}) with the data
\begin{equation}\label{100703a}
(u, v)=(u_b, v_b)(x, y), \quad (x,y)\in \wideparen{{PE}}.
\end{equation}
By Lemma \ref{lem7} we know that
the singular Cauchy problem (\ref{E1}), (\ref{100703a}) admits a solution on a curvilinear triangle domain $\overline{\Omega}_3$ bounded by $\wideparen{{PE}}$, $\widetilde{{PH}}$, and $\widetilde{{EH}}$, where $\widetilde{{PH}}$ is a forward $C_{+}$ characteristic curve issued from $P$ and $\widetilde{{EH}}$ is a forward $C_{-}$ characteristic curve issued from $E$. Moreover, the solutions satisfies $\tau>\tau_2^i$,  $(r, s)\in \Pi$, and $\bar{\partial}_{+}\rho<0$ in $\overline{\Omega}_3$, and $\bar{\partial}_{-}\rho<0$ in $\overline{\Omega}_3\backslash\wideparen{PE}$.

From the point $P$, we draw a cross $C_{-}$ characteristic line of the centered simple wave of ${\it sf}_{+}$. This curve intersects the straight $C_{+}$ characteristic line $y=-1+x\tan\xi_2$ at a point $F$. We then solve a standard Goursat problem for (\ref{E1}) on a curvilinear quadrilateral domain $\overline{\Omega}_4$ bounded by characteristic lines $\widetilde{PH}$,
$\widetilde{PF}$, $C_{-}^{H}$, and $C_{+}^{F}$.
Moreover, the solution satisfies $\tau>\tau_2^i$, $(r, s)\in \Pi$ and $\bar{\partial}_{\pm}\rho<0$.

Since the state $(u_1, v_1)$ is connected to $(u_0, 0)$ on the downstream by a centered simple wave, we have $\tau_1<\tau_1^i$ where $\tau_1$ is defined as in  (\ref{def12}).
Then as in Section 5.3.3, there is a straight post-sonic 2-shock propagating from the point $E$ with the front side state $(u_1, v_1)$. This straight shock will then intersects $\mathcal{W}_{+}$ at a point $G$.
From the point $F$, we draw a straight $C_{-}$ characteristic line with the characteristic angle $\beta=\theta_{-}-A_2$, where $A_2$ is defined as in  (\ref{def12}). This straight characteristic line intersects $\mathcal{W}_{-}$ at a point $J$.
As in Section 5.3.3,
by solving two simple wave problems we see that the interaction of ${\it f}_{-}$ and ${\it sf}_{+}$ generates a shock-fan composite wave ${\it sf}_{+}^2$ with straight $C_{+}$ characteristic lines issued from $\widetilde{FH}\cup C_{-}^{H}$ and a fan wave ${\it f}_{-}^2$ with straight $C_{-}$ characteristic lines issued from $C_{+}^{F}$; see Fig. \ref{Fig5.6.1}. The oblique waves ${\it sf}_{+}^2$ and ${\it f}_{-}^2$ will then reflect off the walls $\mathcal{W}_{+}$ and $\mathcal{W}_{-}$, respectively, and generate two simple waves.
The flow in the remaining part of the divergent duct can then be determined using the method described in Sections 5.1 and 5.3.
This completes the proof of Theorem \ref{thm6}.

\subsection{Interaction of a fan-shock composite wave with a shock-fan composite wave}
We consider the IBVP for $\tau_1^e<\tau_0<\tau_1^i$.
In this part, we assume that the oblique waves at the corners $B$ and $D$ are fan-shock composite wave and shock-fan composite wave, respectively, and the flow near the corners  $B$ and $D$ can be represented by
$$
(u, v)=\left\{
               \begin{array}{ll}
                 (u_0, 0), & \hbox{$-\frac{\pi}{2}\leq \eta\leq \eta_0$;} \\[2pt]
                 (\tilde{u}, \tilde{v})(\eta), & \hbox{$\eta_0\leq \eta\leq \eta_1$;} \\[2pt]
                 (u_1, v_1), & \hbox{$\eta_{1}<\eta<\theta_{+}$}
               \end{array}
             \right.
\quad \mbox{and}\quad
(u, v)=\left\{
               \begin{array}{ll}
                 (u_0, 0), & \hbox{$\phi_{po}\leq \xi\leq \frac{\pi}{2}$;} \\[2pt]
                 (\bar{u}, \bar{v})(\xi), & \hbox{$\xi_2\leq \eta\leq \phi_{po}$;} \\[2pt]
                 (u_2, v_2), & \hbox{$\theta_{-}<\xi<\xi_2$,}
               \end{array}
             \right.
$$
respectively. Here, the flow near the point $D$ is the same as that defined in Section 5.2; the function
$(\tilde{u}, \tilde{v})(\eta)$ is determined by (\ref{102306a}) and (\ref{102304a}) with initial data $(\tilde{q}, \tilde{\tau}, \tilde{\sigma})(\eta_0)=(u_0, \tau_0, 0)$; the ray $y=1+x\tan\eta_1$, $x>0$ represents a pre-sonic shock with the front side state $(\tilde{u}, \tilde{v})(\eta_1)$ and the back side state $(u_1, v_1)$; $v_1=u_1\tan\theta_{+}$.
The main result of this part can be stated as the following theorem.
\begin{thm}\label{thm7}
Assume $\tau_1^e<\tau_0<\tau_1^i$ and $u_0>c_0$.
Assume furthermore that the oblique waves at the corners $B$ and $D$ are fan-shock composite wave and shock-fan composite wave, respectively.  Then the IBVP (\ref{E1}), (\ref{42701}) admits a global  piecewise smooth solution in $\Sigma$.
\end{thm}

\begin{figure}[htbp]
\begin{center}
\includegraphics[scale=0.35]{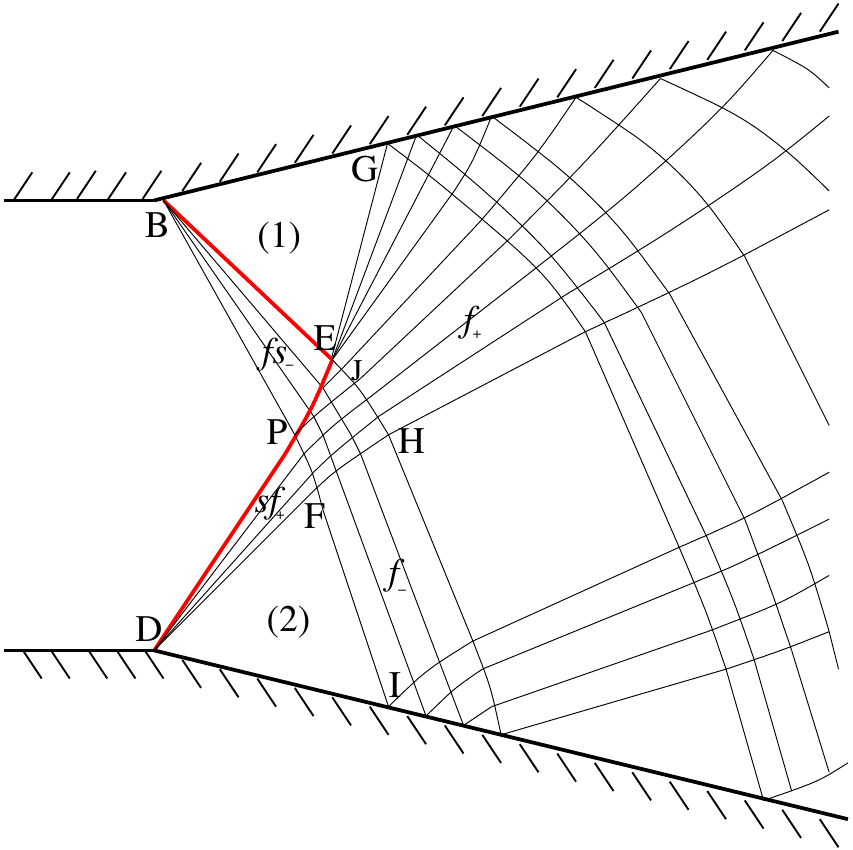}\qquad \qquad\qquad\includegraphics[scale=0.35]{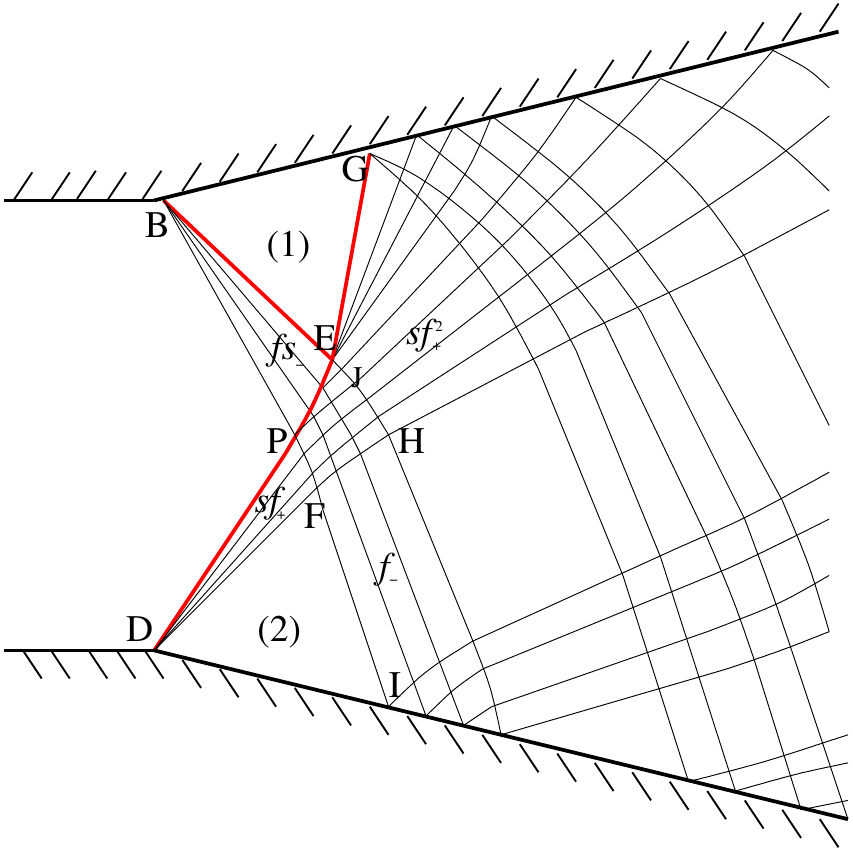}
\caption{\footnotesize Interaction of a fan-shock composite wave with a shock-fan composite wave. Left:  $\tau_1\geq \tau_2^i$; right: $\tau_1<\tau_2^i$.}
\label{Fig5.7.1}
\end{center}
\end{figure}

For convenience, we denote the fan-shock composite wave and the shock fan composite wave by ${\it fs}_{-}$ and ${\it sf}_{+}$, respectively; see Fig. \ref{Fig5.7.1}.
The two composite waves start to interact with each other at a point $P(x_{_P}, y_{_P})$, where
$x_{_P}=2(\tan\phi_{po}-\tan\eta_0)^{-1}$ and $y_{_P}=1+2(\tan\phi_{po}-\tan\eta_0)^{-1}\tan\eta_0$.
Moreover, the post-sonic shock of ${\it sf}_{+}$ propagates into the centered wave region of ${\it fs}_{-}$ from the point $P$ and remains as a post-sonic shock. This post-sonic shock will then intersect the pre-sonic shock of ${\it fs}_{-}$ at a point $E$. As in Section 5.6, we can obtain the post-sonic shock $\wideparen{PE}$ and the flow downstream of it, i.e., the flow on a triangle domain bounded by $\wideparen{PE}$, $\widetilde{PJ}$, and $\widetilde{EJ}$, where $\widetilde{PJ}$ is a forward $C_{+}$ characteristic line issued from $P$ and $\widetilde{EJ}$
is a forward $C_{+}$ characteristic line issued from $E$.

From the point $P$, we draw a cross $C_{-}$ characteristic line of the centered simple wave of ${\it sf}_{+}$. This curve intersects the straight $C_{+}$ characteristic line $y=-1+x\tan\xi_2$ of ${\it sf}_{+}$ at a point $F$. We then solve a standard Goursat problem for (\ref{E1}) on a curvilinear quadrilateral domain bounded by characteristic lines $\widetilde{PJ}$,
$\widetilde{PF}$, $\widetilde{JH}$, and $\widetilde{FH}$, where $\widetilde{JH}$ is a forward $C_{-}$ characteristic curve issued from $J$ and $\widetilde{FH}$ is a forward $C_{+}$ characteristic curve issued from $F$.

From the point $F$, we draw a straight $C_{-}$ characteristic line with the characteristic angle $\beta=\theta_{-}-A_2$ where $A_2$ is the same as that defined in  (\ref{def12}). This characteristic line intersects the wall $\mathcal{W}_{-}$ at a point $I$. We then solve a simple wave problem for (\ref{E1}) to obtain a fan wave ${\it f}_{-}$ with straight $C_{-}$ characteristic lines issued from  $\widetilde{EH}$.

We next construct the flow connected to the constant state $(u_1, v_1)$. Let $\tau_1$ be the same as that defined in (\ref{def12}).
When $\tau_1\geq\tau_2^i$,
as in Section 5.3.2, there is
a simple wave ${\it f}_{+}$ with straight $C_{+}$ characteristic lines issued from the point $E$ and $\widetilde{EJ}\cup \widetilde{JH}$. In this case, the interaction of ${\it fs}_{-}$ and ${\it sf}_{+}$ generates two fan waves with different types.
When $\tau_1<\tau_2^i$, as in Section 5.3.3, there is
an oblique shock-fan composite wave ${\it sf}_{+}^2$ with a straight post-sonic shock propagated from $E$ and straight $C_{+}$ characteristic lines issued from the point $E$ and $\widetilde{EJ}\cup \widetilde{JH}$.
In this case, the interaction of the composite waves ${\it fs}_{-}$ and ${\it sf}_{+}$ generates a shock-fan composite wave and a fan wave.
The flow after the interaction can be determined using the method described in Sections 5.1 and 5.3. This completes the proof of Theorem \ref{thm7}.

\subsection{Interaction of a fan-shock-fan composite wave with a fan-shock composite wave}
We consider the IBVP (\ref{E1}), (\ref{42701})
for $\tau_0<\tau_1^e$.
In this part we assume that the oblique waves at the corners $B$ and $D$ are fan-shock-fan and fan-shock composite waves, respectively, and the flow near the corners  $B$ and $D$ can be represented by
$$
(u, v)=\left\{
               \begin{array}{ll}
                 (u_0, 0), & \hbox{$-\frac{\pi}{2}<\eta<\eta_0$;} \\[2pt]
                 (\tilde{u}_l, \tilde{v}_l)(\eta), & \hbox{$\eta_0<\eta\leq -\xi_{e}$;} \\[2pt]
                  (\tilde{u}_r, \tilde{v}_r)(\eta), & \hbox{$-\xi_{e}\leq \eta<{\eta}_1$;} \\[2pt]
                 (u_1, v_1), & \hbox{$\eta_1<\xi<\theta_{+}$;}
               \end{array}
             \right.
\quad \mbox{and}
\quad
(u, v)=\left\{
               \begin{array}{ll}
                 (u_0, 0), & \hbox{$\xi_0\leq \xi\leq \frac{\pi}{2}$;} \\[2pt]
                 (\bar{u}, \bar{v})(\xi), & \hbox{$\phi_{pr}<\xi<\xi_0$;} \\[2pt]
                (u_2, v_2), & \hbox{$\theta_{-}<\xi<{\phi}_{pr}$,}
               \end{array}
             \right.
$$
respectively.
Here, the flow near the point $B$ is the same as that defined in  Section 5.5; the flow near the point $D$ is the same as that defined in  (\ref{120702a}). The main result of this part can be stated as the following theorem.
\begin{thm}\label{thm8}
Assume $\tau_0<\tau_1^e$ and $u_0>c_0$.
Assume furthermore that the oblique waves at the corners $B$ and $D$ are fan-shock-fan composite wave and fan-shock composite wave, respectively.  Then when $\tau_1^e-\tau_0$ is sufficiently small,
the IBVP (\ref{E1}), (\ref{42701}) admits a global  piecewise smooth solution in $\Sigma$.
\end{thm}


\begin{figure}[htbp]
\begin{center}
\includegraphics[scale=0.38]{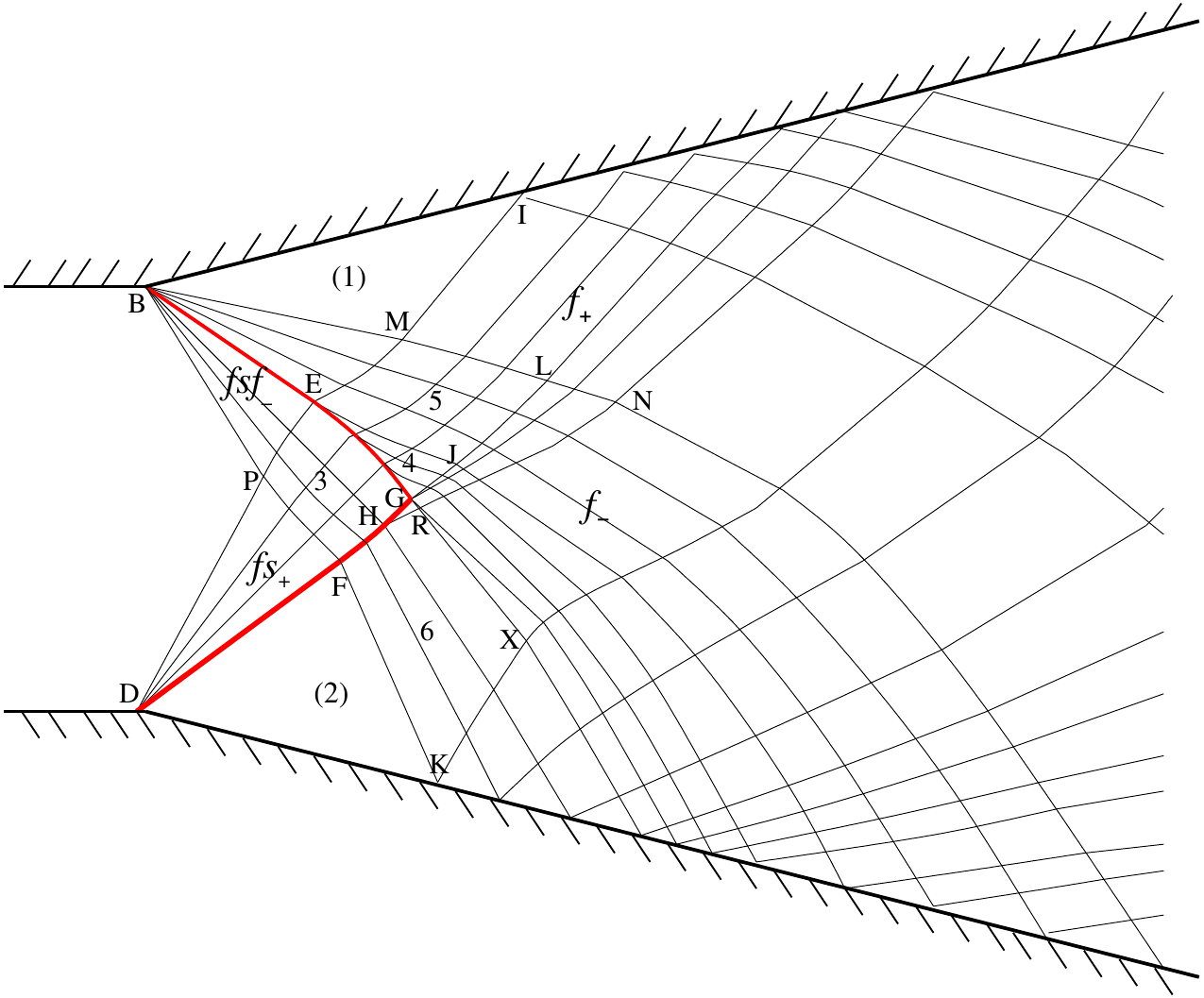}
\caption{\footnotesize Interaction of a fan-shock-fan composite wave with a fan-shock composite wave.}
\label{Fig5.8.1}
\end{center}
\end{figure}

\subsubsection{\bf Interaction of fan waves}
For convenience, we use ${\it fsf}_{-}$ and ${\it fs}_{+}$ to denote the fan-shock-fan composite wave issued from $B$
and the fan-shock composite wave issued from $D$, respectively.
The two composite waves start to interact with each other at point ${P}(\cot\xi_0,0)$.
From the point $P$, we draw a forward $C_{-}$  cross characteristic curve of the oblique fan-shock composite wave ${\it fs}_{+}$. This characteristic line intersects the pre-sonic shock of ${\it fs}_{+}$ at a point $F$.
From the point $P$, we draw a forward $C_{+}$  cross characteristic curve of the oblique fan-shock-fan composite wave ${\it fsf}_{-}$. This characteristic line intersects the double-sonic shock of ${\it fsf}_{-}$ at a point $E$ and ends up at a point $M$ on the straight $C_{-}$ characteristic line $y=1+x\tan\eta_1$ of ${\it fsf}_{-}$.
We will show that in some cases the fan-shock composite wave ${\it fs}_{+}$ region  is not totally delimited by the cross characteristic curve $\widetilde{PF}$.

We first consider (\ref{E1}) with the boundary condition
\begin{equation}\label{102903a}
(u, v)=\left\{
               \begin{array}{ll}
                  (\tilde{u}_l, \tilde{v}_l)(\eta), & \hbox{$(x,y)\in \widetilde{{PE}}$;} \\[2pt]             (\bar{u}, \bar{v})(\xi), & \hbox{$(x,y)\in \widetilde{{PF}}$.} \\[2pt]
              \end{array}
             \right.
\end{equation}

\begin{lem}
Assume $\tau_1^e-\tau_0$ to be sufficiently small. Then the Goursat problem (\ref{E1}), (\ref{102903a}) admits a classical solution on a curvilinear quadrilateral domain $\overline{\Omega^3}$ closed by $\widetilde{{PE}}$, $\widetilde{{PF}}$, $\widetilde{{ES}}$, and $\widetilde{{FS}}$, where $\widetilde{{ES}}$ is a forward $C_{+}$ characteristic line issued from ${E}$ and $\widetilde{{FS}}$ a forward $C_{-}$ characteristic line issued from ${F}$; see Fig. \ref{Fig5.8.2} (left).
Moreover, the solution satisfies $\bar{\partial}_{\pm}\rho<0$ in $\overline{\Omega^3}$.
\end{lem}
\begin{proof}
The existence can be obtained by Lemma \ref{lem1}.
\end{proof}

For convenience, we use subscript `$3$' to denote the flow states in $\overline{\Omega^3}$, i.e.,
$$(u, v, \tau, \alpha, \beta, A, \sigma)=(u_3, v_3, \tau_3, \alpha_3, \beta_3, A_3, \sigma_3)(x,y), \quad (x,y)\in \overline{\Omega^3}.$$

\subsubsection{\bf Post-sonic shock from the point $E$}
The simple wave interaction region $\Omega^3$ will not exist as a whole in the duct. Actually, it will be delimited by shocks propagating from the points $E$ and $F$. We next discuss the shocks propagating from $E$ and $F$.

As in Section 5.5.3, the double-sonic shock of ${\it fsf}_{-}$ will propagate into the region $\Omega^3$ from the point $E$ and become a post-sonic shock.
For convenience, we denote this shock by $S_{_E}$. The shock front $S_{_E}$ can be constructed using the method described in Section 5.5.3; see (\ref{110408}). There are the following two possibilities:
\begin{enumerate}
  \item $S_{_E}$ intersects $\widetilde{FS}$ at a point;
  \item $S_{_E}$ intersects $\widetilde{PF}$ at a point.
\end{enumerate}

Let $\tau_2$ be the same as that defined in (\ref{def12}).
Then we have $\tau_2=\psi_{pr}(\bar{\tau}(\phi_{pr}))<\tau_2^e$.
From Lemma \ref{91011} we can see that for a fixed $\tau_0<\tau_1^e$,  when $\tau_2^e-\tau_2$ is sufficiently small, $S_{_E}$ intersects the characteristic boundary  $\widetilde{FS}$ of the SGP (\ref{E1}), (\ref{102903a}) at a point. We next show that the second case is also possible in some cases.

We use $\phi$ to denote the inclination angle of ${\it S}_{_{E}}$, and use subscripts `$f$' and `$d$' to denote the states on the front and back sides of $S_{_E}$, respectively.
We set
$$
L=u\cos \phi+v\sin \phi\quad \mbox{and} \quad N=-u \sin \phi+v\cos \phi.
$$

\begin{lem}\label{111503}
Let the value of $\tau_2$ be fixed. Then  when $\tau_1^e-\tau_0$ is sufficiently small, the shock curve $S_{_E}$ intersects the cross characteristic curve  $\widetilde{PF}$ at a point $H$; see Fig. \ref{Fig5.8.2} (left). Moreover, the point $H$ approaches the point $P$ as $\tau_1^e-\tau_0\rightarrow 0$.
\end{lem}
\begin{proof}
We use $\phi$ to denote the inclination angle of $S_{_E}$. The shock curve $S_{_E}$ can be represented by a function
$y=y_{s}(x)$ which satisfies $y_{s}'(x)=\tan(\phi(x, y_{s}(x)))$.
The characteristic line $\widetilde{PF}$ can be represented by a function $y=y_{-}(x)$ which satisfies
$y_{-}'(x)=\tan(\beta_3(x, y_-(x)))$.
It is obvious that $x_{_E}-x_{_P}\rightarrow 0$ and $y_{s}(x_{_E})-y_{-}(x_{_P})\rightarrow 0$ as $\tau_1^e-\tau_0\rightarrow 0$.

For a given $\tau_2'\in (\tau_2, \tau_2^e)$, we let $F'$ be the point on $\widetilde{PF}$ such that $\tau(F')=\psi_{pr}^{-1}(\tau_2')$.
From $F'$,
 we draw a forward $C_{+}$ characteristic curve of the Goursat problem (\ref{E1}), (\ref{102903a}). This characteristic curve intersects $\widetilde{ES}$ at a point $S'$.
We denote by $\Omega_3'$ the curvilinear quadrilateral domain bounded by $\widetilde{{PE}}$, $\widetilde{{PF'}}$, $\widetilde{{ES'}}$, and $\widetilde{{F'S'}}$.

As in Lemma \ref{lem110101}, we have
\begin{equation}\label{110407}
\big\|(\alpha_3-A_1^e, \beta_3+A_1^e, \tau_3-\tau_1^e, \bar{\partial}_{\pm}\rho_3+\mathcal{L})\big\|_{0; \overline{\Omega_3'}}~\rightarrow ~0\quad \mbox{as} ~ \tau_1^e-\tau_0\rightarrow 0~ \mbox{and} ~ \tau_2^e-\tau_2'\rightarrow 0,
\end{equation}
where the constants $\mathcal{L}$  and $A_1^e$ are the same as that defined in (\ref{102901}).

Using (\ref{82301}), (\ref{8}), (\ref{110408}), and (\ref{110407}), there is a small $\delta>0$  such that when $\tau_1^e-\tau_0<\delta$ and $\tau_2^e-\tau_2'=\delta$,
$$
\bar{\partial}_{-}\beta>-\frac{\sin (2A_1^e)}{\rho_1^e}\mathcal{L}+\frac{\tan A_1^e p''(\tau_1^e)}{4(c_1^e)^2(\rho_1^e)^4}\mathcal{L}\quad \mbox{on}~ \widetilde{PF'},
$$
and
$$
\begin{aligned}
\bar{\partial}_{_\Gamma}\phi&= -\frac{\bar{\partial}_{_\Gamma} m}{\rho_fL_f}+\frac{N_{f}}{\rho_{f}L_{f}q_{f}^2}(q_{f}^2+\tau_{f}^2 p'(\tau_{f}))\bar{\partial}_{_\Gamma}\rho_{f}+\bar{\partial}_{_\Gamma}\sigma_{f}
\\&<-\frac{\sin (2A_1^e)}{\rho_1^e}\mathcal{L}+\frac{\tan A_1^e p''(\tau_1^e)}{8(c_1^e)^2(\rho_1^e)^4}\mathcal{L}\quad \mbox{for}~ (x, y_{s}(x))\in \Omega_3',
\end{aligned}
$$
where $(u_f, v_f)(x, y)=(u_3, v_3)(x,y)$.

Therefore, for any fixed $\tau_2\in (\tau_2^i, \tau_2^e)$, when $\tau_1^e-\tau_0$ is sufficiently small there exists an $x_{_H}\in (x_{_E}, x_{_{F'}})$ such that
$$
y_{s}(x_{_H})=y_{-}(x_{_H}), \quad \mbox{and}\quad y_{s}(x)>y_{-}(x)\quad \mbox{for}~  x_{_E}<x<x_{_H}.
$$
Moreover, $x_{_H}-x_{_E}\rightarrow 0$ as $\tau_1^e-\tau_0\rightarrow 0$.
This completes the proof of the lemma.
\end{proof}

\begin{figure}[htbp]
\begin{center}
\includegraphics[scale=0.52]{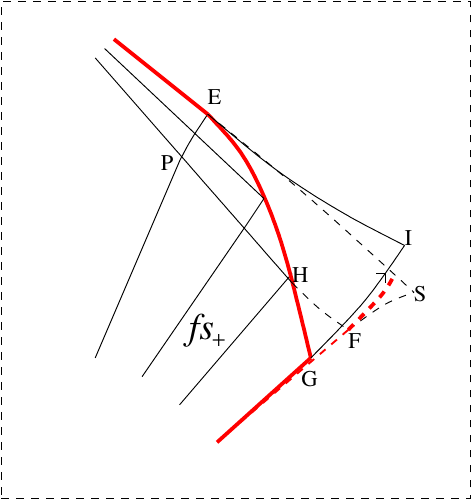}\qquad\qquad \includegraphics[scale=0.4]{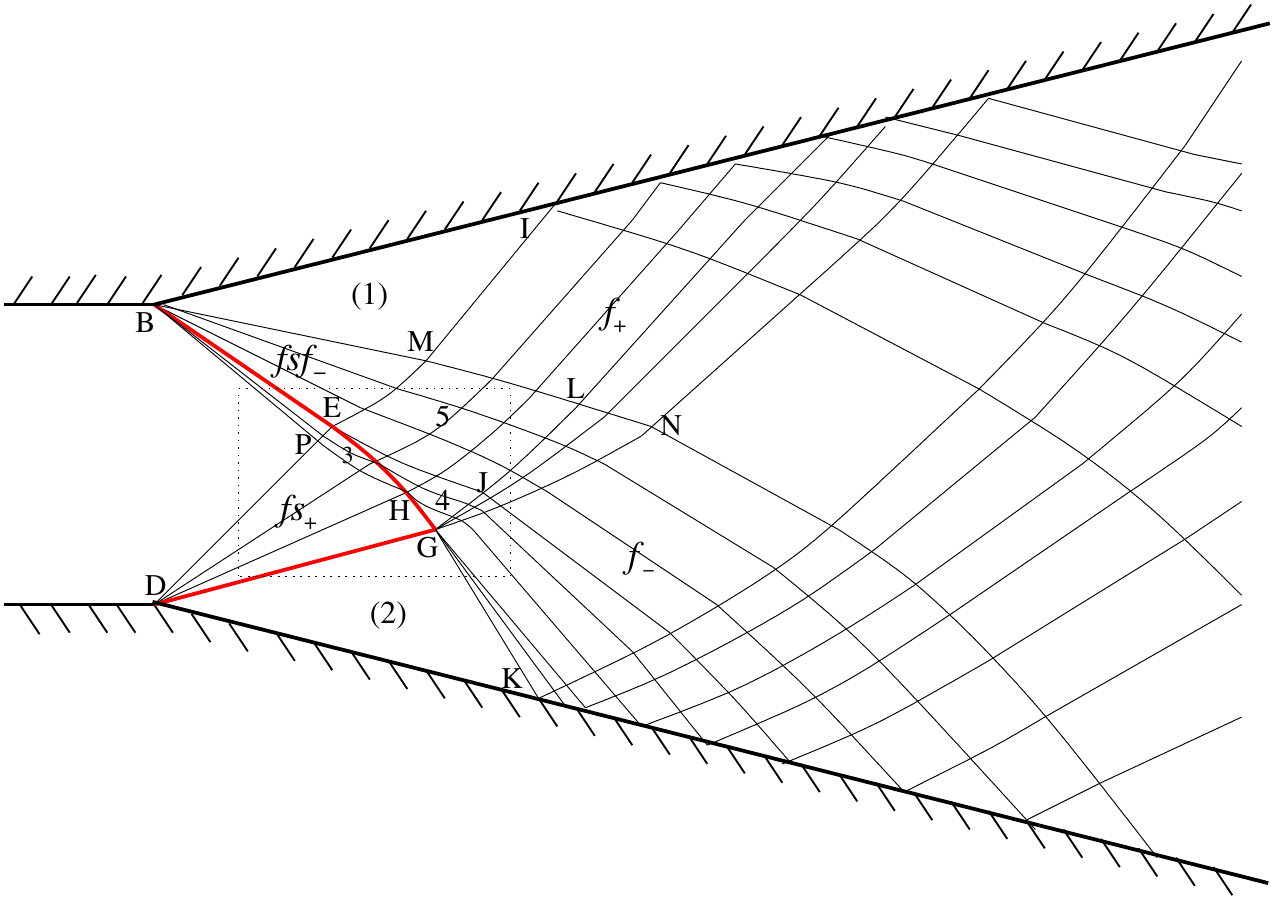}\quad
\caption{\footnotesize Interaction of a fan-shock-fan composite wave with a fan-shock composite wave.}
\label{Fig5.8.2}
\end{center}
\end{figure}


If $S_{_E}$ intersects $\widetilde{PF}$ at a point $H$,
 then the post-sonic shock $S_{_E}$ will propagate into the simple wave region of ${\it fs}_{+}$  from $H$ and intersects the pre-sonic shock of ${\it fs}_{+}$ at a point $G$. For convenience, we denote this post-sonic shock by $\wideparen{HG}$.
The shock front $\wideparen{HG}$ can be determined by
\begin{equation}
\bar{\partial}_{_\Gamma} \phi=\frac{-p''(\psi_{po}(\tau_f))\psi_{po}'(\tau_f)\bar{\partial}_{_\Gamma} \tau_f}{2 \rho_f^2L_f^2}+\frac{N_{f}\cos^2 A_f}{\rho_{f}L_{f}}\bar{\partial}_{_\Gamma}\rho_{f}+\bar{\partial}_{_\Gamma}\sigma_{f}
\end{equation}
where $(u_f, v_f)(x,y)=(\bar{u}, \bar{v})(\xi)$.
The backside states of the flow on $\widetilde{HG}$ can be determined by (\ref{110406}).
 For convenience, we denote the shock propagated from $E$ by $\wideparen{EG}$, i.e., $\wideparen{EG}=\wideparen{EH}\cup \wideparen{HG}$. We denote by $\Omega_3$ the domain bounded by $\widetilde{PE}$, $\widetilde{PH}$, and $\wideparen{EH}$.
Then the flow in $\Omega_3$ is $(u, v)=(u_3, v_3)(x,y)$.
We next construct the flow downstream of the shock $\wideparen{EG}$.

\begin{lem}\label{121101}
(Structural conditions)
Assume $\tau_1^e-\tau_0$ to be sufficiently small. Then
\begin{equation}\label{120901a}
(r, s)\in\Pi,\quad
\bar{\partial}_{_\Gamma}r_{d}<0,\quad \mbox{and}\quad \bar{\partial}_{_\Gamma}s_{d}<0\quad \mbox{along}~ \widetilde{EG};
\end{equation}
where $r_{d}=r(u_d, v_d)$, $s_d=s(u_d, v_d)$, and $(u_d, v_d)(x,y)$ represent the backside  states on $\widetilde{EG}$.
\end{lem}
\begin{proof}
As in Lemma \ref{110411}, when  $\tau_1^e-\tau_0$ is sufficiently small the desired estimates hold along the section $\wideparen{FH}$. As in Lemma \ref{120801a}, the desired estimates hold along $\wideparen{HG}$.
\end{proof}

Based on Lemma \ref{121101},
we can solve a
singular Cauchy problem for (\ref{E1})
to obtain the flow downstream of the post-sonic shock $\wideparen{EG}$, i.e, the flow
on a triangle domain $\overline{\Omega}_4$ bounded by $\widetilde{EJ}$, $\wideparen{EG}$, and $\widetilde{GJ}$, where $\widetilde{EJ}$ is a forward $C_{-}$
characteristic line issued from $E$ and $\widetilde{GJ}$ is a forward $C_{+}$ characteristic line issued from $G$; see Fig. \ref{Fig5.8.2} (right). Moreover, the solution satisfies
$$
(r, s)\in\Pi\quad \mbox{and}\quad\bar{\partial}_{-}\rho<0 \quad \mbox{in}~ \overline{\Omega}_4; \quad \bar{\partial}_{+}\rho<0 \quad \mbox{in}~ \overline{\Omega}_4\setminus\widetilde{EG}.
$$
We use subscript `$4$' to denote the flow in $\overline{\Omega}_4$, i.e., $(u, v)=(u_4, v_4)(x,y)$, $(x,y)\in \overline{\Omega}_4$.

Next, we solve a standard Gorusat problem for (\ref{E1}) to obtain the flow on a curvilinear quadrilateral domain $\overline{\Omega}_5$ bounded by $\widetilde{EM}$,  $\widetilde{EJ}$, $\widetilde{ML}$ and $\widetilde{JL}$, where $\widetilde{ML}$ is a forward $C_{-}$ characteristic line issued from $M$, and $\widetilde{JL}$ is a forward $C_{+}$ characteristic line issued from $J$. Moreover, the solution satisfies $\bar{\partial}_{\pm}\rho<0$ and $(r, s)\in \Pi$ in $\overline{\Omega}_5$.
We use subscript `$5$' to denote the flow in domain $\overline{\Omega}_5$, i.e., $(u, v)=(u_5, v_5)(x,y)$, $(x,y)\in \overline{\Omega}_5$.

See Fig. \ref{Fig5.8.2} (right).
When $\tau_2\geq\tau_2^i$,
we draw a straight $C_{-}$ characteristic line with the characteristic angle $\beta=\theta_{-}-A_2$ from the point $G$, where $A_2$ is the same as that defined in (\ref{def12}). This line intersects $\mathcal{W}_{-}$ at a point $K$.
We then consider (\ref{E1}) with the boundary conditions
\begin{equation}\label{121102}
(u, v)=\left\{
               \begin{array}{ll}
                  (u_4, v_4)(x,y), & \hbox{$(x, y)\in\widetilde{GJ}$;} \\
(u_5, v_5)(x,y), & \hbox{$(x, y)\in\widetilde{JL}$;} \\
               (u_2, v_2), & \hbox{$(x, y)\in\overline{GK}$,}
               \end{array}
             \right.
\end{equation}
By Lemma  \ref{lem8} and Remark \ref{rem4.1}, the problem (\ref{E1}), (\ref{121102})
admits a solution on a curved quadrilateral domain bounded by $\widetilde{GJ}\cup\widetilde{JL}$, $\overline{GK}$, $C_{-}^{L}$, and $C_{+}^{K}$.
This solution consists of a centered wave centered at $G$ and a simple wave $f_{-}$ with straight $C_{-}$ characteristic lines.
The centered wave flow region is bounded by $\widetilde{GJ}\cup\widetilde{JL}$, $\widetilde{GN}$, and $\widetilde{LN}$, where $\widetilde{LN}$ is a $C_{-}$ characteristic line issued from $L$, and $\widetilde{GN}$ is a $C_{+}$ characteristic line issued from $G$. The straight $C_{-}$ characteristic lines of the simple wave $f_{-}$ are issued from $G$ and $\widetilde{GN}$.

From the point $M$ we draw a straight $C_{+}$ characteristic line with the characteristic angle $\alpha=\theta_{+}+A_1$. This line intersects $\mathcal{W}_{+}$ at a point $I$. When then solve a simple wave problem for
(\ref{E1})  to obtain a simple wave ${\it f}_{+}$ on a curved quadrilateral domain bounded by $\overline{MI}$, $\widetilde{ML}\cup \widetilde{LN}$, $C_{-}^{I}$, and $C_{+}^{N}$.

Therefore, for Case (ii), when $\tau_2\geq\tau_2^i$ the interaction of ${\it fsf}_{-}$ with ${\it fs}_{+}$ generates two simple waves ${\it f}_{\pm}$.
The simple waves will then reflect off the walls $\mathcal{W}_{\pm}$ and generate two new simple waves.
The solution in the remaining part of the divergent duct can be obtained using the method described in Section 5.1.

\vskip 4pt

There is a slight difference in the discussion between  $\tau_2<\tau_2^i$ and $\tau_2\geq\tau_2^i$.
When $\tau_2<\tau_2^i$, we only need to replace the straight $C_{-}$ characteristic line issued from $G$ by a straight post-sonic 2-shock.
The state on the front side of the post-shock is $(u_2, v_2)$. The back side state and the inclination angle of the post-sonic shock can be determined by the Rankine-Hugoniot conditions.
Therefore, for Case (ii), when $\tau_2<\tau_2^i$ the interaction of ${\it fsf}_{-}$ with ${\it fs}_{+}$ generates a simple wave and a shock-fan composite wave. The simple wave and shock-fan composite wave will then reflect off the walls $\mathcal{W}_{+}$ and $\mathcal{W}_{-}$, and generate two new simple waves.
The solution in the remaining part of the divergent duct can be obtained using the method described in Section 5.1.

This completes the proof of Theorem \ref{thm8} for Case (ii).

\subsubsection{\bf Shock wave propagated from the point $F$}
Assume $\tau_1^e-\tau_0$ to be small. Then by Lemmas \ref{91011} and \ref{111503} we know that
there is an $\epsilon>0$ depending only on  $\tau_1^e-\tau_0$  such that
 $S_{_E}$ intersects the characteristic curve $\widetilde{FS}$ of the SGP (\ref{E1}), (\ref{102903a}) if and only if $0<\tau_2^e-\tau_2<\epsilon$.
Moreover, $\epsilon\rightarrow0$ as $\tau_1^e-\tau_0\rightarrow 0$.
We shall show that when $0<\tau_2^e-\tau_2<\epsilon$ the pre-sonic shock of ${\it fs}_{+}$ will propagate into the domain $\Omega^3$ from the point $F$ and become a transonic shock.

From the point $F$ we draw a straight $C_{-}$ characteristic line with the characteristic angle $\beta=\theta_{-}-A_2$. This characteristic line intersects $\mathcal{W}_{-}$ at a point $K$.
In order to construct the transonic shock propagating from  $F$ and the flow downstream of it, we consider (\ref{E1}) with the following boundary conditions:
\begin{equation}\label{102904}
(u, v)=(u_2, v_2)\quad \mbox{on}~ \overline{FK};
\end{equation}
\begin{equation}\label{102905}
(\rho u-\rho_3 u_3)(u-u_3)+(\rho v-\rho_3 v_3)(v-v_3)=0\quad \mbox{on}~ y=y_{t}(x),
\end{equation}
\begin{equation}\label{102906}
\left\{
  \begin{array}{ll}
    \displaystyle\frac{{\rm d}y_{t}}{{\rm d}x}=-\frac{u(x, y_t)-u_3(x, y_t)}{v(x, y_t)-v_3(x, y_t)}, & \hbox{$x>x_{_F}$;} \\[6pt]
    y_t(x_{_F})=y_{_F}.
  \end{array}
\right.
\end{equation}
Here, $y=y_t(x)$, $x>x_{_F}$ represents the shock propagating from $F$.
The problem (\ref{E1}), (\ref{102904})--(\ref{102906}) is a shock free boundary problem.

\begin{figure}[htbp]
\begin{center}
\qquad\includegraphics[scale=0.58]{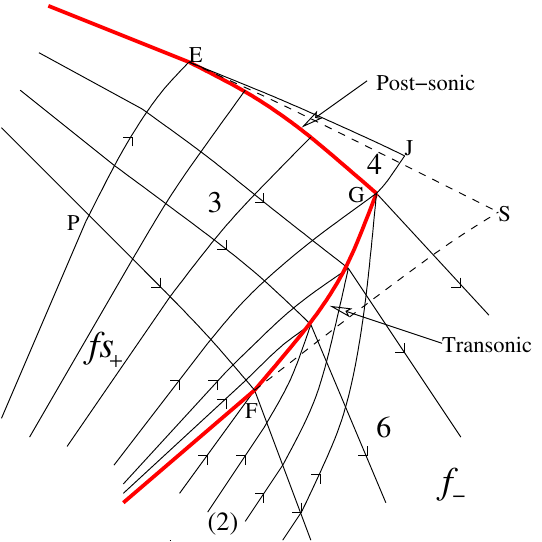}\qquad\qquad\qquad\includegraphics[scale=0.58]{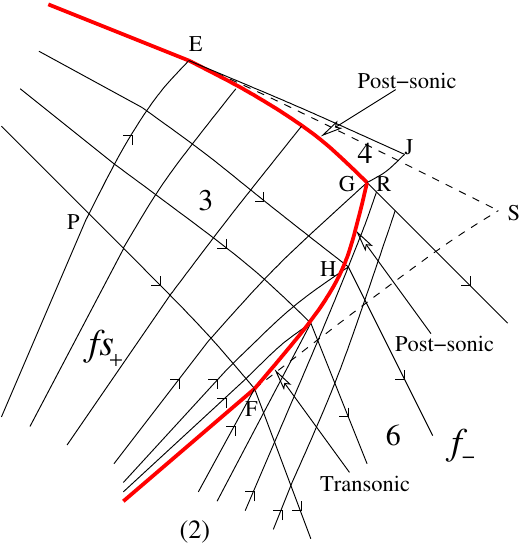}\quad  \qquad
\caption{\footnotesize Shock waves propagated from $E$ and $F$, and the flows upstream and downstream of them. }
\label{Fig5.8.3}
\end{center}
\end{figure}

\begin{lem}\label{120903a}
 The shock free boundary problem (\ref{E1}), (\ref{102904})--(\ref{102906}) admits a (local) simple wave solution.
Moreover, the simple wave solution satisfies $\bar{\partial}_{+}\rho<0$ and $\bar{\partial}_{-}\rho=0$.
\end{lem}
\begin{proof}
We assume a priori that the free boundary problem admits a transonic shock solution. Then for any point on the transonic shock front, the backward $C_{+}$ characteristic line issued from this point goes back to $\overline{FK}$.
By Theorem \ref{thm0}, the solution of the shock free boundary problem must be a simple wave if it exists a transonic shock solution.

For convenience, we denote by $S_{_F}$ the transonic shock propagating from $F$.
Let $\phi$ be the inclination angle of $S_{_F}$.
We also set
$L=u\cos \phi+v\sin \phi$, $N=u\sin \phi- v\cos \phi$,
$L_3=u_3\cos \phi+v_3\sin \phi$, $N_3=u_3\sin \phi- v_3\cos \phi$, $m=\rho_3 N_3$ on $S_{_F}$.

As in (\ref{250104}), we have
\begin{equation}\label{103001}
\begin{aligned}
2m \bar{\partial}_{_\Gamma}m ~=~\frac{2\rho_3}{\tau_3^2-\tau^2}(c_3^2-N_3^2)\bar{\partial}_{_\Gamma}\tau_3
-\frac{2\rho}{\tau_3^2-\tau^2}(c^2-N^2)\bar{\partial}_{_\Gamma}\tau\quad \mbox{along}~{\it S}_{_{F}}.
\end{aligned}
\end{equation}
As in (\ref{250102}), we have
\begin{equation}\label{103002}
\bar{\partial}_{_\Gamma} \phi=\frac{\bar{\partial}_{_\Gamma} m}{\rho_3L_3}-\frac{N_3\cos^2 A_3}{\rho_3L_3}\bar{\partial}_{_\Gamma}\rho_3+\bar{\partial}_{_\Gamma}\sigma_3
\end{equation}
and
\begin{equation}\label{103003}
\bar{\partial}_{_\Gamma} \phi=\frac{\bar{\partial}_{_\Gamma} m}{\rho L}-\frac{N\cos^2 A}{\rho L}\bar{\partial}_{_\Gamma}\rho+\bar{\partial}_{_\Gamma}\sigma
\end{equation}
along ${\it S}_{_{F}}$.
Hence, we have
$$
\begin{aligned}
(\tau_3-\tau)\frac{\bar{\partial}_{_\Gamma} m}{L_3}-\frac{N_3\cos^2 A_3}{\rho_3L_3}\bar{\partial}_{_\Gamma}\rho_3+\bar{\partial}_{_\Gamma}\big(\frac{\alpha_3+\beta_3}{2}\big)
=-\frac{N\cos^2 A}{\rho L}\bar{\partial}_{_\Gamma}\rho+\bar{\partial}_{_\Gamma}\big(\frac{\alpha+\beta}{2}\big)\quad \mbox{along}~ {\it S}_{_{F}}.
\end{aligned}
$$
Inserting (\ref{103001}) into this, we get
$$
\begin{aligned}
&\frac{\tau_3(N_3^2-c_3^2)}{\rho_3L_3N_3(\tau_3+\tau)}\bar{\partial}_{_\Gamma}\rho_3-\frac{N_3\cos^2 A_3}{\rho_3L_3}\bar{\partial}_{_\Gamma}\rho_3+\bar{\partial}_{_\Gamma}\Big(\frac{\alpha_3+\beta_3}{2}\Big)
\\&\qquad\qquad ~=~\frac{\tau(N^2-c^2)}{\rho LN(\tau_3+\tau)}\bar{\partial}_{_\Gamma}\rho-\frac{N\cos^2 A}{\rho L}\bar{\partial}_{_\Gamma}\rho+\bar{\partial}_{_\Gamma}\Big(\frac{\alpha+\beta}{2}\Big)\quad \mbox{along}~ {\it S}_{_{F}}.
\end{aligned}
$$
This implies
\begin{equation}\label{103004}
\begin{aligned}
&\left(\frac{\tau(N^2-c^2)}{\rho LN(\tau_3+\tau)}-\frac{N\cos^2 A}{\rho L}-\frac{\sin (2A)}{2\rho}\right)\bar{\partial}_{_\Gamma}\rho\\ =~&
t_{+}^{f}\left(\frac{\tau_3(N_3^2-c_3^2)}{\rho_3L_3N_3(\tau_3+\tau)}-\frac{N_3\cos^2 A_3}{\rho_3L_3}-\frac{\sin (2A_3)}{2\rho_3}\right)\bar{\partial}_+\rho_3\\&~+~t_{-}^{f}\left(\frac{\tau_3(N_3^2-c_3^2)}
{\rho_3L_3N_3(\tau_3+\tau)}-\frac{N_3\cos^2 A_3}{\rho_3L_3}+\frac{\sin (2A_3)}{2\rho_3}\right)\bar{\partial}_{-}\rho_3\quad \mbox{along}~{\it S}_{_{F}},
\end{aligned}
\end{equation}
where
$$
t_{+}^{f}=\frac{\sin(\phi-\beta_3)}{\sin(2A_3)}\quad \mbox{and}\quad t_{-}^{f}=\frac{\sin(\alpha_3 -\phi)}{\sin(2A_3)}.
$$

Since the flow downstream of $S_{_F}$ is a simple wave,  $S_{_F}$ can be determined by the density of the flow on the backside of it. So, in order to construct the transonic shock $S_{_F}$ and the states on the backside of it, we consider (\ref{103002}) and (\ref{103004}) with initial data
\begin{equation}\label{103005}
(\phi, \rho)(x_{_F}, y_{_F})=(\phi_{pr}, \rho_2).
\end{equation}
The problem (\ref{103002}), (\ref{103004}), and (\ref{103005}) is actually an initial value problem for a system of  ``ordinary differential equations". Since this initial value problem also use the flow states in the domain $\Omega^3$, we need to establish an a priori estimate about the position relation between $S_{_F}$ and $\widetilde{FS}$.
We denote by $y=y_{+}(x)$, $x_{_{F}}\leq x\leq x_{_{S}}$ the $C_{+}$ characteristic curve $\widetilde{{FS}}$.

We first establish an a priori estimate about the transonic shock front $S_{_F}$.
By (\ref{103004}) we have
\begin{equation}
\bar{\partial}_{_\Gamma}\rho<0\quad \mbox{at}~ F.
\end{equation}
Then by (\ref{103001})
we have
\begin{equation}
m\bar{\partial}_{_\Gamma}m =\frac{\tau_2^3}{\tau_*^2-\tau_2^2}\left(-p'(\tau_2)+\frac{2h(\tau_*)-2h(\tau_2)}{\tau_*^2-\tau_2^2}\right)
\bar{\partial}_{_\Gamma}\rho>0 \quad \mbox{at}~ F,
\end{equation}
where $\tau_*=\psi_{pr}^{-1}(\tau_2)$.

As in (\ref{61108}) and (\ref{103006}), we have
\begin{equation}
\begin{aligned}
\bar{\partial}_{_\Gamma} \phi&=\frac{\bar{\partial}_{_\Gamma} m}{\rho_3L_3}-\frac{N_3}{\rho_3L_3q_3^2}(q_3^2+\tau_3^2 p'(\tau_3))\bar{\partial}_{_\Gamma}\rho_3+\bar{\partial}_{_\Gamma}\sigma_3\\&=\frac{\bar{\partial}_{_\Gamma} m}{\rho_3L_3}-\frac{\sin 2A_3}{2\rho_3}\bar{\partial}_{+}\rho_{3}+\frac{\bar{\partial}_{+}\alpha_{3}}{2}-\frac{p''(\tau_3)\tan A_3}{4c_3^2\rho_3^4}\bar{\partial}_{+}\rho_{3}\\&>-\frac{\sin 2A_3}{2\rho_3}\bar{\partial}_{+}\rho_{3}+\frac{\bar{\partial}_{+}\alpha_{3}}{2}
\end{aligned}
\end{equation}
and
\begin{equation}
\bar{\partial}_+ \alpha_{3}<-\frac{\sin 2A_3}{\rho_3}\bar{\partial}_{+}\rho_{3}
\end{equation}
at $F$. So, we have
\begin{equation}\label{103009}
\bar{\partial}_{_\Gamma} \phi(x_{_F}, y_{_F})>\bar{\partial}_{+}\alpha_{3}(x_{_F}, y_{_F}).
\end{equation}
Then there exists a sufficiently small $\varepsilon>0$ such that
$\phi(x, y_t(x))>\alpha_3(x, y_t(x))$ for $x_{_{F}}<x<x_{_{F}}+\varepsilon$.
This implies
$$
y_{t}(x)>y_{+}(x), \quad x_{_{F}}<x<x_{_{F}}+\varepsilon.
$$
This implies that the transonic shock $S_{_F}$ exists.

For convenience we denote the solution of (\ref{103002}), (\ref{103004}), and (\ref{103005}) by
$(\phi, \rho)=(\phi, \rho_b)(x, y)$, $(x, y)\in S_{_F}$.
The states on the backside of $S_{_F}$ can  be determined by
$$
u_b=L_b\cos\phi-N_b\sin\phi,\quad  v_b=L_b\sin\phi+N_b\cos\phi,
$$
where
$L_b=L_3$ and  $N_b=\frac{N_3}{\rho_b\tau_f}$.

The simple wave and its straight characteristic lines can then be determined by
the backside states of the transonic shock $S_{_F}$.
This completes the proof of the lemma.
\end{proof}

Since $\epsilon\rightarrow0$ as $\tau_1^e-\tau_0\rightarrow 0$,
the radius of the domain $\Omega^3$ approaches $0$ as $\tau_1^e-\tau_0\rightarrow 0$ and $0<\tau_2^e-\tau_2<\epsilon$. We have
\begin{equation}\label{103011}
\parallel\alpha_3-A_1^e, ~\beta_3+A_1^e,~ \bar{\partial}_{-}\rho_3+\mathcal{L}, \bar{\partial}_{-}\rho_3+\mathcal{L}\parallel_{0;\overline{\Omega^3}}~\rightarrow 0 \quad \mbox{as}~0<\tau_2^e-\tau_2<\epsilon~\mbox{and}~ \tau_1^e-\tau_0\rightarrow 0.
\end{equation}
Then when  $\tau_1^e-\tau_0$ is sufficiently small and $0<\tau_2^e-\tau_2<\epsilon$, $\phi-A_1^e$ is small along $S_{_F}$, and hence
\begin{equation}\label{2601}
\bar{\partial}_{_\Gamma}\rho_3=t_{+}^f\bar{\partial}_{+}\rho_3+t_{-}^f\bar{\partial}_{-}\rho_3<0\quad \mbox{along}~ S_{_F}.
\end{equation}
Combining this with $\tau_3(F)=\tau_*>\tau_1^e$, we know that $S_{_F}$ does not intersect the characteristic curve $\widetilde{PE}$.

Moreover, by (\ref{103004}) and (\ref{103011}) we also have that when $\tau_1^e-\tau_0$ is sufficiently small and  $0<\tau_2^e-\tau_2<\epsilon$, the solution of the free boundary problem (\ref{E1}), (\ref{102904})--(\ref{102906}) satisfies
\begin{equation}\label{12201}
\bar{\partial}_{_\Gamma}\rho<0\quad \mbox{along}~ S_{_F}.
\end{equation}

We next assert that if $\phi-\alpha_b<0$ along  $S_{_F}$, then $S_{_F}$  does not intersect $\widetilde{FS}$.
Suppose that $S_{_F}$ intersects $\widetilde{FS}$ at a point. Then there must exists a point $x_{0}>x_{_F}$ such that
$$
\phi(x_0, y_t(x_0))=\alpha_3(x_0, y_t(x_0)),\quad \mbox{and} \quad \phi(x, y_t(x))>\alpha_3(x, y_t(x))\quad \mbox{for}~ x_{_F}<x<x_0.
$$
While, by (\ref{2601}) and (\ref{12201}) we have
$$
\tau(x_0, y_t(x_0))>\tau_2=\psi_{pr}(\tau_3(F))>\psi_{pr}(\tau_3(x_0, y_t(x_0))),
$$
and hence $\phi(x_0, y_t(x_0))>\alpha_3(x_0, y_t(x_0))$.
This leads to a contradiction.
This completes the proof of the assertion. 

Therefore,
there are the following two subcases about the transonic shock $S_{_F}$ of the free boundary problem (\ref{E1}), (\ref{102904})--(\ref{102906}):

\begin{description}
  \item[(i1)] $S_{_F}$ stays in $\Omega^3$ and remains transonic until it intersects $S_{_E}$ at a point $G$, see Fig. \ref{Fig5.8.3} (left);
  \item[(i2)] $S_{_F}$ stays in  $\Omega^3$ and remains transonic until there is a point $H(x_{_H}, y_t(x_{_H}))$ on $S_{_F}$ such that $\phi(H)-\alpha_b(H)=0$. Moreover, the transonic shock does not intersect $S_{_E}$, see Fig. \ref{Fig5.8.3} (right).
\end{description}

We next show that the second subcase can happen in some cases.
\begin{lem}
(Transition from a transonic shock to a post-sonic shock)
For a fixed small $\tau_1^e-\tau_0$,
when $\tau_2^e-\tau_2$ is sufficiently small, there is an $x_{_H}>x_{_F}$  such that $(\phi-\alpha_b)(x_{_H}, y_t(x_{_H}))=0$, and $(\phi-\alpha_b)(x, y_t(x))<0$ for $x_{_F}<x<x_{_H}$.
\end{lem}
\begin{proof}
By (\ref{103004}) we know that
when $\tau_1^e-\tau_0$ is sufficiently small and $0<\tau_2^e-\tau_2<\epsilon$,
\begin{equation}
\bar{\partial}_{_\Gamma}\rho_b<-\frac{\tau_1^e\sin (2A_1^e)}{2\tau_2\sin (2A_2)}\mathcal{L}\quad \mbox{along}~ S_{_F}.
\end{equation}

$$
\bar{\partial}_{_\Gamma}\big(\rho_b-\psi_{po}(\tau_3)\big)
=\bar{\partial}_{_\Gamma}\rho_b+\psi_{po}'(\tau_3)\bar{\partial}_{_\Gamma}\tau_3<
\bar{\partial}_{_\Gamma}\rho<-\frac{\tau_1^e\sin (2A_1^e)}{2\tau_2\sin (2A_2)}\mathcal{L}\quad \mbox{along}~ S_{_F}.
$$
We also have
$$\tau_b(F)-\psi_{po}(\tau_3(F))=\tau_2-\psi_{po}(\tau_*)\rightarrow 0\quad \mbox{as}~ \tau_2^e-\tau_2\rightarrow 0.$$
So,
 for a fixed $\tau_0<\tau_1^e$, when $\tau_2^e-\tau_2$ is sufficiently small, there is an $x_{_H}>x_{_F}$ such that
$\tau_b(x_{_H}, y_t(x_{_H}))-\psi_{po}(\tau_3(x_{_H}, y_t(x_{_H})))=0$. Moreover, $x_{_H}\rightarrow x_{_F}$ as $\tau_2^e-\tau_2\rightarrow 0$.

This completes the proof of the lemma.
\end{proof}

If the second subcase happens, we do not use (\ref{103002}) and (\ref{103004}) to look for the shock propagating from $H$, since it will violate the Liu's extended entropy condition. Instead, we look for a post-sonic shock.
As in Section 5.5, we can find a post-sonic shock propagated from $H$, and this post-sonic shock stays in $\Omega^3$ until it intersects ${\it S}_{_E}$ at a point $G$.

For convenience, we denote by $\wideparen{EG}$ and $\wideparen{FG}$ the shocks propagated from $E$ and $F$, respectively. The shock $\wideparen{EG}$ is post-sonic. For the first subcase, the shock $\wideparen{FG}$ is transonic; see Fig. \ref{Fig5.8.3} (left).   For the second subcase, the shock $\wideparen{FG}$ consists of a transonic shock $\wideparen{FH}$ and a post-sonic shock $\wideparen{HG}$; see Fig. \ref{Fig5.8.3} (right).
In what follows we are going to study the second subcase; see Fig. \ref{Fig5.8.1}.

We first construct the flow downstream of the post-sonic shock $\wideparen{EG}$.
We solve a singular Cauchy problem to find the flow behind the post-sonic shock $\wideparen{EG}$, i.e., the flow on a triangle domain $\overline{\Omega}_4$ bounded by $\wideparen{EG}$, $\widetilde{EJ}$, and $\widetilde{GJ}$, where
$\widetilde{GJ}$ is a forward $C_{+}$ characteristic line issued from $G$ and
$\widetilde{EJ}$ is a forward $C_{-}$ characteristic line issued from $E$. We solve a standard Goursat problem to obtain the flow  on a curvilinear quadrilateral domain $\overline{\Omega}_5$ bounded by $\widetilde{EM}$,  $\widetilde{EJ}$, $\widetilde{ML}$ and $\widetilde{JL}$, where
$\widetilde{ML}$ is a forward $C_{-}$ characteristic curve issued from $M$ and $\widetilde{JL}$ is a forward $C_{+}$ characteristic curve issued from $J$.

We next construct the flow downstream of the post-sonic shock $\widetilde{FG}$.
By Lemma \ref{120903a} we know that there is a simple wave with straight $C_{-}$ characteristic lines issued from the transonic shock $\wideparen{FH}$.
We solve a singular Cauchy problem to find the flow behind the post-sonic shock $\widetilde{HG}$, i.e., the flow on a triangle domain bounded by $\wideparen{HG}$, $\widetilde{HR}$, and $\widetilde{GR}$, where
$\widetilde{HR}$ is a forward $C_{+}$ characteristic line issued from $H$ and
$\widetilde{GR}$ is a forward $C_{-}$ characteristic line issued from $G$.
We then solve a simple wave problem to obtain a simple wave with straight $C_{-}$ characteristic lines issued from $\widetilde{HR}$.
When $\tau_1^e-\tau_0$ is sufficiently small and  $0<\tau_2^e-\tau_2<\epsilon$, we have $\tau_2>\tau_2^i$.
Then the straight characteristic lines issued from $\wideparen{FH}$ and $\widetilde{HR}$ do not intersect with other.
Therefore, we obtain a flow on a domain $\overline{\Omega}_6$ bounded by $\wideparen{FG}$, $\overline{FK}$, $\widetilde{GR}\cup \overline{RX}$, and $\widetilde{KX}$, where $\overline{RX}$ is a forward straight $C_{-}$ characteristic line issued from $R$ and $\widetilde{KX}$ is a forward $C_{+}$ characteristic line issued from $K$.

For convenience, we use subscript `$i$' to denote the flow states in domains $\overline{\Omega}_{i}$, i.e., $(u, v)=(u_i, v_i)(x, y)$, $(x, y)\in \overline{\Omega}_i$, $i=4,~5,~6$.
Finally, we consider (\ref{E1}) with the boundary condition
\begin{equation}\label{121105}
(u, v)=\left\{
               \begin{array}{ll}
                  (u_4, v_4)(x,y), & \hbox{$(x, y)\in\widetilde{GJ}$;} \\
(u_5, v_5)(x,y), & \hbox{$(x, y)\in\widetilde{JL}$;} \\
               (u_6, v_6), & \hbox{$(x, y)\in\widetilde{GR}\cup\overline{RX}$,}
               \end{array}
             \right.
\end{equation}
The problem (\ref{E1}), (\ref{121105})
admits a solution on a curved quadrilateral domain bounded by $\widetilde{GJ}\cup\widetilde{JL}$, $\widetilde{GR}\cup\overline{RX}$, $C_{-}^{L}$, and $C_{+}^{X}$.

 From the point $R$,
we draw a forward $C_{+}$ characteristic curve, and this characteristic line intersects $C_{-}^{L}$ at a point $N$. Then, the interaction of ${\it fsf}_{-}$ with ${\it fs}_{+}$ generates a simple wave with straight $C_{+}$ characteristic lines issued from $\widetilde{MN}$ and a simple wave
 with straight $C_{-}$ characteristic lines issued from $\wideparen{FH}$, $\widetilde{HR}$, and $\widetilde{RN}$.
The two simple waves will then reflect off the walls $\mathcal{W}_{\pm}$, and the flow in the remaining part of the divergent duct can be constructed using the method described in Section 5.1.

The discussion for the first subcase is similar; we omit the details. This completes the proof of Theorem \ref{thm8}.


\subsubsection{\bf Non-existence of a pre-sonic shock propagated from $F$}
Suppose that there is a smooth pre-sonic shock propagated from $F$, denoted by ${\it S}_{_F}$ for convenience; see Fig. \ref{Fig5.8.4}. The shock ${\it S}_{_F}$ lies outside the domain ${\Omega}^3$.
For convenience, we still use subscripts `$f$' and `$b$' to denote the flow upstream and downstream of ${\it S}_{_F}$, respectively, and use $\phi$ to denote the inclination angle of ${\it S}_{_F}$.


As in (\ref{10705}) we have
\begin{equation}\label{1802}
\bar{\partial}_{_\Gamma}r_f=\bar{\partial}_{+}r_f=\bar{\partial}_{+}r_3>0\quad \mbox{at}~ F.
\end{equation}
So, as in (\ref{110412}) and (\ref{123104}) we have
$$
\bar{\partial}_{_\Gamma}\phi-\frac{\bar{\partial}_{_\Gamma} m}{\rho_fL_f}=\bar{\partial}_{_\Gamma}r_f
>0\quad \mbox{at}~ F.
$$

\begin{figure}[htbp]
\begin{center}
\includegraphics[scale=0.35]{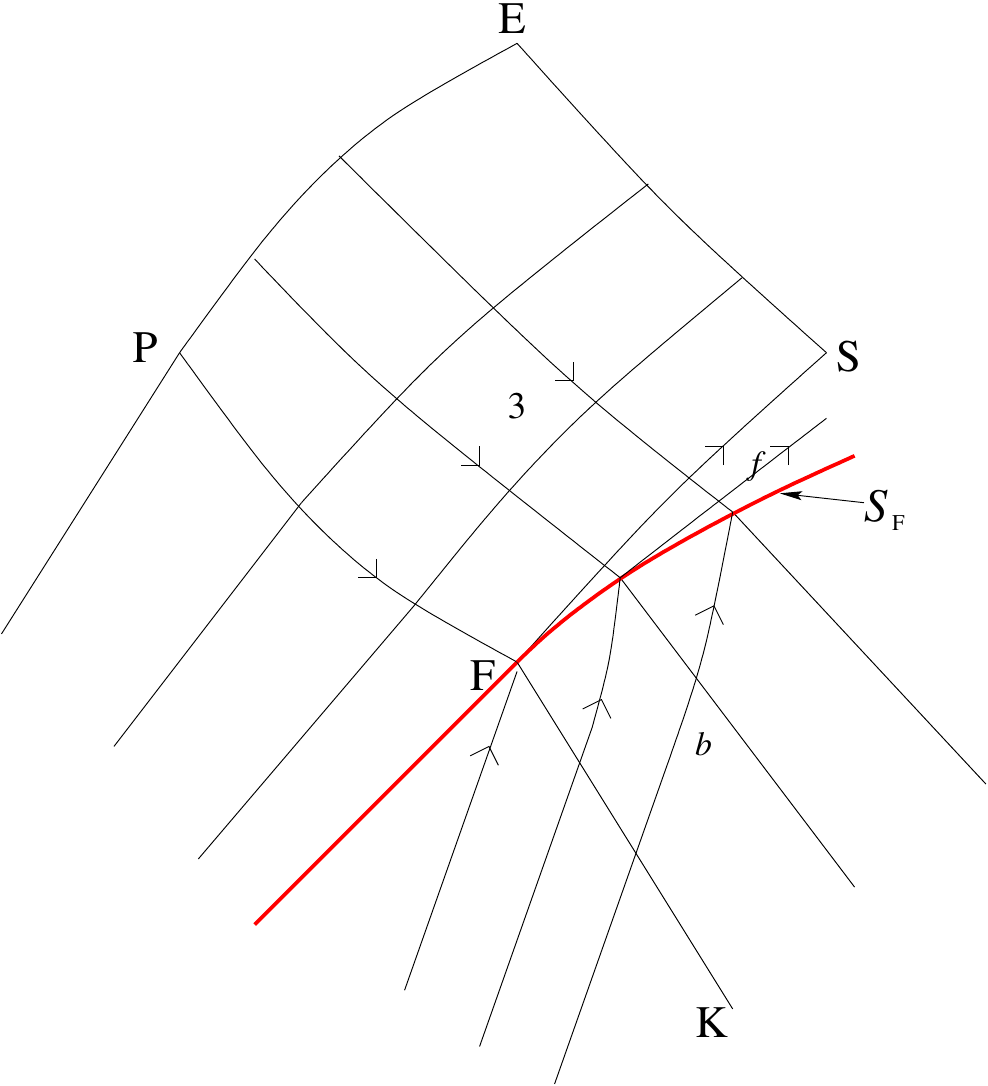}
\caption{\footnotesize An impossible pre-sonic shock from the point $F$. }
\label{Fig5.8.4}
\end{center}
\end{figure}

Since ${\it S}_{_F}$ lies outside the domain ${\Omega}^3$, by (\ref{alphar}) we have
\begin{equation}\label{1801}
\bar{\partial}_{_\Gamma}s_f\leq 0 \quad \mbox{at}~~ F.
\end{equation}
Combining (\ref{1801}) and (\ref{1802}) we have
$\bar{\partial}_{_\Gamma}\rho_f<0$ at $F$.
Moreover, since $\tau_b=\psi_{pr}(\tau_f)$ on ${\it S}_{_F}$,
we have
\begin{equation}
\bar{\partial}_{_\Gamma}\rho_b>0 \quad \mbox{at}~~ F.
\end{equation}

As in (\ref{250407}), we have
$$
\begin{aligned}
\bar{\partial}_{_\Gamma} \phi-\frac{\bar{\partial}_{_\Gamma} m}{\rho_fL_f}=~&\frac{(\tau_b-\tau_f)\bar{\partial}_{_\Gamma} m}{L_b}-\frac{N_b}{\rho_bL_bq_b^2}(q_b^2+\tau_b^2 p'(\tau_b))\bar{\partial}_{_\Gamma}\rho_b+\bar{\partial}_{_\Gamma}\sigma_b\\<~&\frac{(\tau_b-\tau_f)\bar{\partial}_{_\Gamma} m}{L_b}
+t_{+}^b\bar{\partial}_{+}\Big(\frac{\alpha_b+\beta_b}{2}\Big)
\\=~&\left\{\frac{N_b^2-c_b^2}{mL_b\rho_b(\tau_b+\tau_f)}-\frac{p''(\tau_{b})}{8c_{b}^2\rho_{b}^4}
\big(\varpi(\tau_{b})+1\big)\sin (2A_{b})\right\}\bar{\partial}_{_\Gamma}\rho_{b}<0\quad \mbox{at}~~ F,
\end{aligned}
$$
where $t_{+}^b=\frac{\sin(\phi-\beta_b)}{\sin (2A_b)}$.
This leads to a contradiction. So, the shock propagated from the point $F$ could not be a pre-sonic shock.

\subsection{Interaction of a fan-shock-fan composite wave with a fan wave}
We consider the IBVP (\ref{E1}), (\ref{42701})
for $\tau_0<\tau_1^e$.
In this part we assume the oblique waves at the corners $B$ and $D$ are fan-shock-fan composite wave and centered simple wave, respectively.
We assume further that the flow near  $B$ and $D$ can be represented by
\begin{equation}
(u, v)=\left\{
               \begin{array}{ll}
                 (u_0, 0), & \hbox{$-\frac{\pi}{2}<\eta<\eta_0$;} \\[2pt]
                 (\tilde{u}_l, \tilde{v}_l)(\eta), & \hbox{$\eta_0<\eta\leq -\xi_{e}$;} \\[2pt]
                  (\tilde{u}_r, \tilde{v}_r)(\eta), & \hbox{$-\xi_{e}\leq \eta<{\eta}_1$;} \\[2pt]
                 (u_1, v_1), & \hbox{$\eta_1<\xi<\theta_{+}$,}
               \end{array}
             \right.
\quad \mbox{and}
\quad
(u, v)=\left\{
               \begin{array}{ll}
                 (u_0, 0), & \hbox{$\xi_0\leq \xi\leq \frac{\pi}{2}$;} \\[2pt]
                 (\bar{u}, \bar{v})(\xi), & \hbox{$\xi_2<\xi<\xi_0$;} \\[2pt]
                (u_2, v_2), & \hbox{$\theta_{-}<\xi<\xi_2$,}
               \end{array}
             \right.
\end{equation}
respectively.
Here, the flow near the point $B$ is the same as that defined in  Section 5.5; $(\bar{u}, \bar{v})(\xi)$ is determined by (\ref{102306}) and (\ref{102304}) with initial data $(\bar{q}, \bar{\tau}, \bar{\sigma})(\xi_0)=(u_0, \tau_0, 0)$; $\xi_2$ is determined by $\bar{v}(\xi_2)=\bar{u}(\xi_2)\tan\theta_{-}$;
$(u_2, v_2)=(\bar{u}, \bar{v})(\xi_2)$. The main result of this subsection can be stated as the following theorem.

\begin{figure}[htbp]
\begin{center}
\includegraphics[scale=0.29]{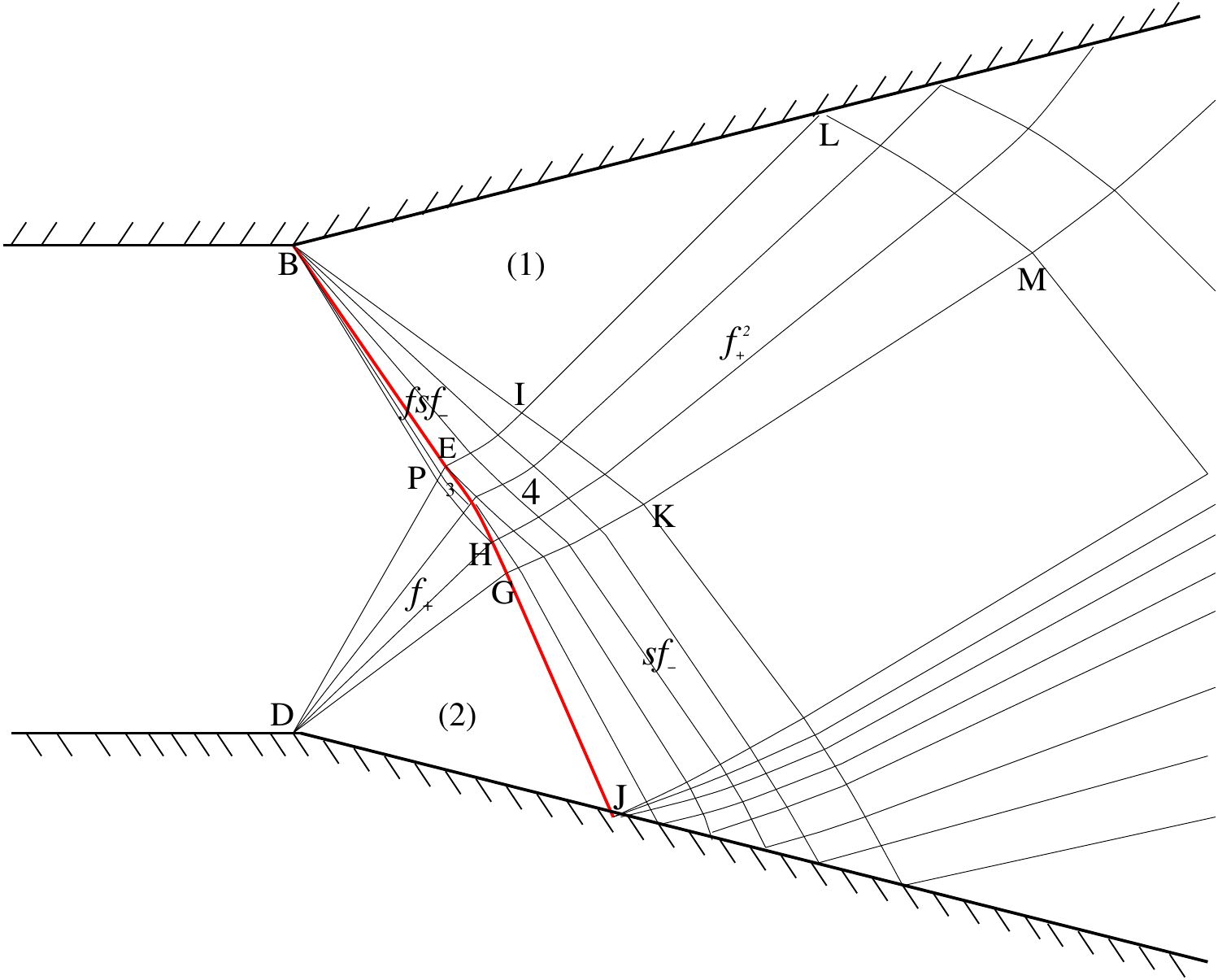}
\caption{\footnotesize Interaction of a fan-shock-fan composite wave with a centered simple wave. }
\label{Fig5.9.1}
\end{center}
\end{figure}

\begin{thm}\label{thm9}
Assume $\tau_0<\tau_1^e$ and $u_0>c_0$.
Assume furthermore that the oblique waves at the corners $B$ and $D$ are fan-shock-fan composite wave and centered simple wave, respectively.  Then when $(\tau_1^e-\tau_0)/(\tau_2-\tau_0)$ is sufficiently small, the IBVP (\ref{E1}), (\ref{42701}) admits a global  piecewise smooth solution in $\Sigma$.
\end{thm}

We use ${\it fsf}_{-}$ and ${\it f}_{+}$ to denote the fan-shock-fan composite wave issued from point $B$ and the centered simple wave issued from point $D$, respectively; see Fig. \ref{Fig5.9.1}.
These two composite waves start to interact with each other at point ${P}(\cot\xi_0,0)$.
From the point $P$, we draw a forward $C_{+}$  cross characteristic line in ${\it fsf}_{-}$.
This characteristic line intersects the double-sonic shock of ${\it fsf}_{-}$ at a point $E$ and ends up at a point $I$ on the straight $C_{-}$ characteristic line $y=1+x\tan\eta_1$ of ${\it fsf}_{-}$.
From the point $P$, we draw a forward $C_{-}$  cross characteristic line in ${\it f}_{+}$.
This characteristic line intersects the straight $C_{+}$ characteristic  line $y=-1+x\tan\xi_2$  of ${\it f}_{+}$ at a point $F$.

We first consider (\ref{E1}) with the boundary condition
\begin{equation}\label{102903}
(u, v)=\left\{
               \begin{array}{ll}
                  (\tilde{u}_l, \tilde{v}_l)(\eta), & \hbox{$(x,y)\in \widetilde{{PE}}$;} \\[2pt]
            (\bar{u}, \bar{v})(\xi), & \hbox{$(x,y)\in \widetilde{{PF}}$.} \\[2pt]
              \end{array}
             \right.
\end{equation}
When $\tau_1^e-\tau_0$ is sufficiently small,
the Goursat problem (\ref{E1}), (\ref{102903}) admits a classical solution on a curvilinear quadrilateral domain $\overline{\Omega^3}$ bounded by $\widetilde{{PE}}$, $\widetilde{{PF}}$, $\widetilde{{ES}}$, and $\widetilde{{FS}}$, where $\widetilde{{ES}}$ is a forward $C_{+}$ characteristic line issued from ${E}$ and $\widetilde{{FS}}$ a forward $C_{-}$ characteristic line issued from ${F}$. Moreover the solution satisfies $\tau<\tau_1^i$ in $\overline{\Omega^3}$.

As described in Section 5.8.2, the double-sonic shock of ${\it fsf}_{-}$ will propagate into the domain $\Omega^3$ from the point $E$ and become a post-sonic shock.
Moreover, when  $(\tau_1^e-\tau_0)/(\tau_2-\tau_0)$ is sufficiently small, this shock will intersect $\widetilde{{PF}}$ at a point $H$. Then the shock propagates into the fan wave ${\it f}_{+}$ region from $H$ until it intersects the straight $C_{+}$ characteristic  line $y=-1+x\tan\xi_2$ at a point $G$.
We use $\wideparen{EG}$ to denote the post-sonic shock from $E$ to $G$, and
 let $\overline{\Omega}_3$ be the domain bounded by $\widetilde{PE}$, $\widetilde{PH}$, and $\wideparen{EH}$.
The flow upstream of $\wideparen{EG}$ is thus known.

The states on the backside of the post-sonic shock $\wideparen{EG}$ can be determined using the method described in Section 5.8.2.
We next solve a singular Cauchy problem and a standard Goursat problem of (\ref{E1}) to obtain the flow downstream of $\wideparen{EG}$, i.e., the flow on a curvilinear quadrilateral $\overline{\Omega}_4$ bounded by  $\widetilde{EI}$, $\wideparen{EG}$, $\widetilde{IK}$, and $\widetilde{GK}$, where $\widetilde{IK}$ is a forward $C_{-}$
characteristic line issued from $I$ and $\widetilde{GK}$ is a forward $C_{+}$ characteristic line issued from $G$.
For convenience, we use subscript `$4$' to denote the flow in $\overline{\Omega}_4$, i.e., $(u, v)=(u_4, v_4)(x,y)$, $(x,y)\in \overline{\Omega}_4$.

The post-sonic shock will continue to propagate at a constant speed from the point $G$ until it intersects $\mathcal{W}_{-}$ at a point $J$. The states on the front and  back sides of the straight post-sonic shock $\overline{GJ}$ are $(u_2, v_2)$ and $(u_4, v_4)(G)$, respectively. We then consider (\ref{E1}) with the boundary condition
\begin{equation}\label{121106}
(u, v)=\left\{
               \begin{array}{ll}
                  (u_4, v_4)(x,y), & \hbox{$(x, y)\in\widetilde{GK}$;} \\[2pt]
(u_4, v_4)(G), & \hbox{$(x, y)\in\widetilde{GJ}$.}
               \end{array}
             \right.
\end{equation}
The problem (\ref{E1}), (\ref{121106}) admits a simple wave solution with straight $C_{-}$ characteristic line issued from $\widetilde{GK}$. This simple wave and the post-sonic shock $\overline{GJ}$ form a shock-fan composite wave, denoted by  ${\it sf}_{-}$ for convenience.

From the point $I$ we draw a straight $C_{+}$ characteristic line with the characteristic angle $\alpha=\theta_{+}+A_1$. This line intersects $\mathcal{W}_{+}$ at a point $L$. When then solve a simple wave problem for
(\ref{E1})  to obtain a simple wave, denoted by ${\it f}_{+}^2$,  on a curvilinear quadrilateral domain bounded by characteristic lines $\overline{IL}$, $\widetilde{IK}$, $\overline{KM}$, and $\widetilde{LM}$, where
$\overline{KM}$ is a straight $C_{-}$ characteristic line issued from $K$ and $\widetilde{LM}$ is a forward $C_{-}$ characteristic curve issued from $L$.

Therefore, when $\tau_1^e-\tau_0$ is sufficiently small, the interaction of ${\it fsf}_{-}$ with ${\it f}_{+}$ generates a simple wave ${\it f}_{+}^2$ and a shock-fan composite wave ${\it sf}_{-}$.
The shock-fan composite wave ${\it sf}_{-}$ and the simple wave ${\it f}_{+}^2$
will then reflect off the walls $\mathcal{W}_{-}$ and $\mathcal{W}_{+}$, respectively, and generate two new simple waves.
The solution in the remaining part of the divergent duct can be construct using the method described in Sections 5.1 and 5.3.
This completes the proof of Theorem \ref{thm9}.

\subsection{Interaction of a fan-shock composite wave with a fan wave}
We consider the IBVP (\ref{E1}), (\ref{42701})
for $\tau_0<\tau_1^i$.
In this part we assume that the oblique waves at the corners $B$ and $D$ are fan-shock composite wave and centered simple wave, respectively, and the flow near the corners  $B$ and $D$ can be represented by
$$
(u, v)=\left\{
               \begin{array}{ll}
                 (u_0, 0), & \hbox{$-\frac{\pi}{2}<\eta<\eta_0$;} \\[2pt]
                 (\tilde{u}, \tilde{v})(\eta), & \hbox{$\eta_0<\eta\leq \eta_{pr}$;} \\[2pt]
                 (u_1, v_1), & \hbox{$\eta_{pr}<\xi<\theta_{+}$;}
               \end{array}
             \right.
\quad \mbox{and}
\quad
(u, v)=\left\{
               \begin{array}{ll}
                 (u_0, 0), & \hbox{$\xi_0\leq \xi\leq \frac{\pi}{2}$;} \\[2pt]
                 (\bar{u}, \bar{v})(\xi), & \hbox{$\xi_2<\xi<\xi_0$;} \\[2pt]
                (u_2, v_2), & \hbox{$\theta_{-}<\xi<\xi_2$,}
               \end{array}
             \right.
$$
respectively.
Here, the flow near the point $B$ is the same as that defined in  Section 5.7; the flow near the point $D$ is the same as that defined in  Section 5.9.
The main result of this subsection can be stated as the following theorem.
\begin{thm}\label{thm10}
Assume $\tau_0<\tau_1^i$ and $u_0>c_0$.
Assume furthermore that the oblique waves at the corners $B$ and $D$ are fan-shock composite wave and centered simple wave, respectively. Then when $\tau_1>\tau_2^i$ and $(\psi_{pr}^{-1}(\tau_1)-\tau_0)/(\tau_2-\tau_0)$ is sufficiently small, the IBVP (\ref{E1}), (\ref{42701}) admits a global  piecewise smooth solution in $\Sigma$.
\end{thm}

\begin{figure}[htbp]
\begin{center}
\includegraphics[scale=0.35]{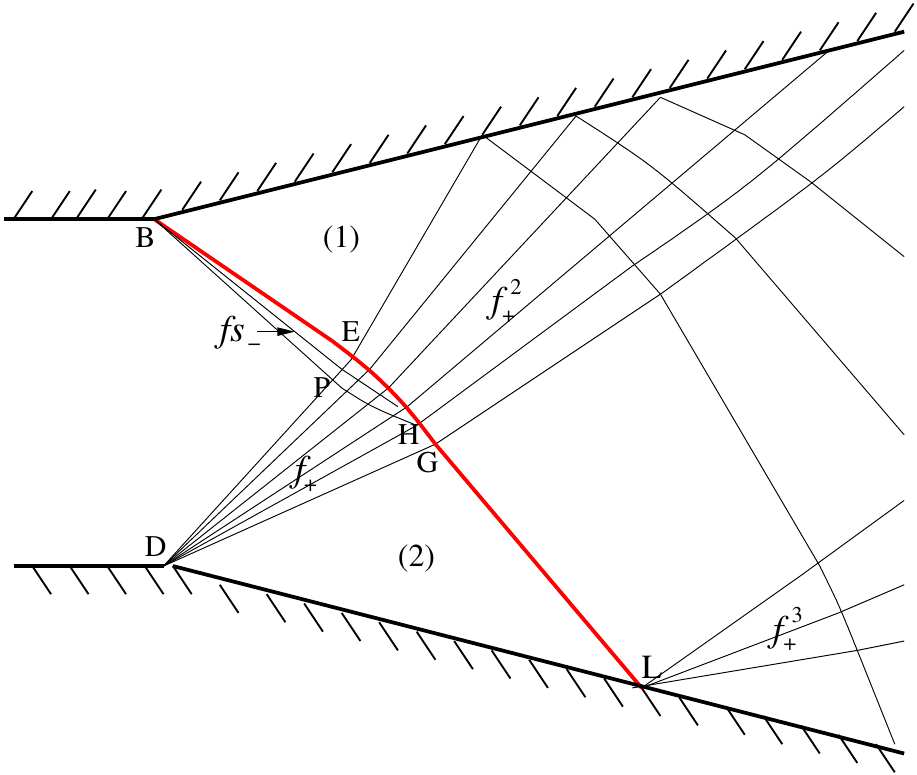}\qquad\qquad\qquad\includegraphics[scale=0.35]{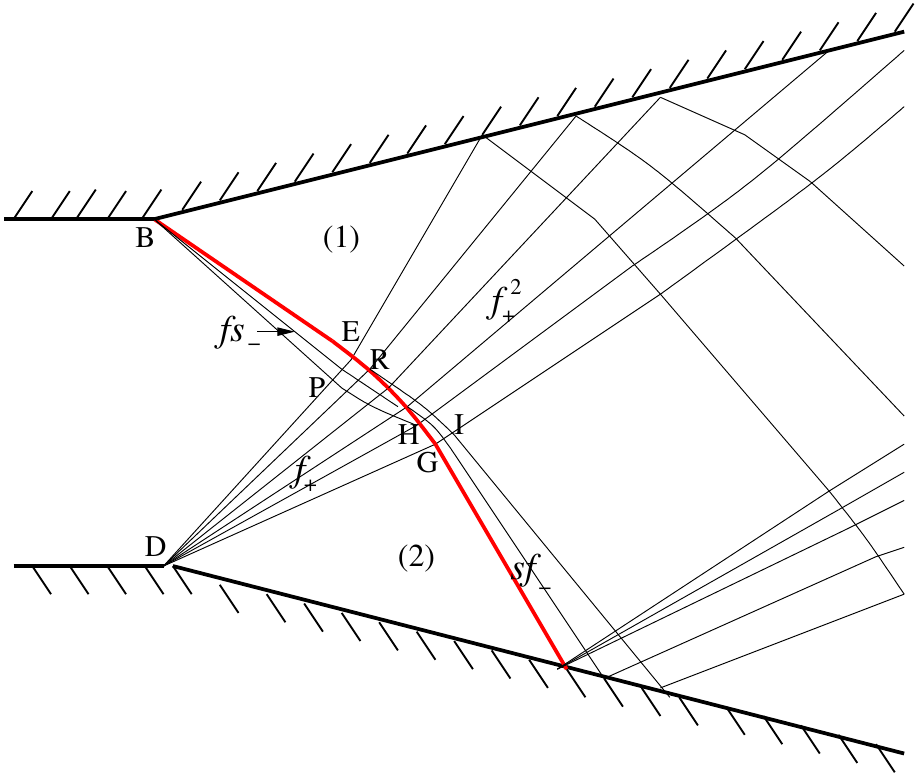}
\caption{\footnotesize Interaction of a fan-shock composite wave with a centered simple wave. }
\label{Fig5.10.1}
\end{center}
\end{figure}

For convenience, we use ${\it fs}_{-}$ and ${\it f}_{+}$ to denote the fan-shock composite wave issued from $B$
and the centered simple wave issued from $D$, respectively.
The two oblique waves start to interact with each other at the point ${P}(\cot\xi_0,0)$.
From the point $P$, we draw a forward $C_{-}$  cross characteristic curve of the centered simple wave ${\it f}_{+}$. This characteristic curve intersects the straight characteristic line $y=-1+x\tan\xi_2$ at a point $F$.
From the point $P$, we draw a forward $C_{+}$  cross characteristic curve of the  fan-shock composite wave ${\it fs}_{-}$. This characteristic curve intersects the pre-sonic shock of ${\it fs}_{-}$ at a point $E$.

We first consider (\ref{E1}) with the boundary condition
\begin{equation}\label{121201a}
(u, v)=\left\{
               \begin{array}{ll}
                  (\tilde{u}, \tilde{v})(\eta), & \hbox{$(x,y)\in \widetilde{{PE}}$;} \\[2pt]             (\bar{u}, \bar{v})(\xi), & \hbox{$(x,y)\in \widetilde{{PF}}$.} \\[2pt]
              \end{array}
             \right.
\end{equation}
When  $\psi_{pr}^{-1}(\tau_1)-\tau_0$ is sufficiently small,  the SGP (\ref{E1}), (\ref{121201a}) admits a classical solution on a curvilinear quadrilateral domain $\overline{\Omega^3}$ bounded by $\widetilde{{PE}}$, $\widetilde{{PF}}$, $\widetilde{{ES}}$, and $\widetilde{{FS}}$, where $\widetilde{{ES}}$ is a forward $C_{+}$ characteristic line issued from ${E}$ and $\widetilde{{FS}}$ a forward $C_{-}$ characteristic line issued from ${F}$.

See Fig. \ref{Fig5.10.1}. By solving a shock free boundary problem like (\ref{E1}), (\ref{102904})--(\ref{102906}) we see that the pre-sonic shock of ${\it fs}_{-}$ will propagate into the region $\Omega^3$ from the point $E$. For convenience, we denote this shock by ${\it S}_{_E}$.
As in Lemma \ref{111503}, when $\psi_{pr}^{-1}(\tau_1)-\tau_0$ is sufficiently small, ${\it S}_{_E}$ intersects the characteristic curve $\widetilde{PF}$ at a point $H$.
The shock ${\it S}_{_E}$ will then propagate into the centered simple wave ${\it f}_{+}$ region from the point $H$ and intersect the straight $C_{+}$ characteristic line $y=-1+x\tan\xi_2$ at a point $G$. We use $\wideparen{EG}$ to denote the shock ${\it S}_{_E}$ connecting $E$ and $G$.

Like the shock  $\wideparen{FG}$ discussed in Section 5.8.3, we have the following two cases.
\begin{itemize}
  \item $\wideparen{EG}$ is a transonic shock; see Fig. \ref{Fig5.10.1} (left).
  \item There is a point $R$ on $\wideparen{EG}$ such that the segment $\wideparen{ER}$ is transonic and the segment $\wideparen{RG}$ is post-sonic;  see Fig. \ref{Fig5.10.1} (left).
\end{itemize}

For the first case, there is a simple wave (denoted by ${\it f}_{+}^2$) with straight $C_{+}$ characteristic lines  issued from the shock $\wideparen{EG}$, and the shock will continue to propagate at a constant from the point $G$.
Then the interaction of ${\it fs}_{-}$ with ${\it f}_{+}$ generates a simple wave and a shock wave.
The shock propagating from $G$ will impinge on the wall $\mathcal{W}_{-}$ at a point $L$, generating
a centered simple wave (denoted by ${\it f}_{+}^3$) emanating from that point.

For the second case, by solving a singular Cauchy problem we obtain the flow downstream of the post-sonic shock $\wideparen{RG}$, i.e., the flow on a triangle region bounded by  $\widetilde{RI}$, $\widetilde{RG}$, and $\widetilde{GI}$, where $\widetilde{RI}$ is a forward $C_{-}$
characteristic line issued from $R$ and $\widetilde{GI}$ is a forward $C_{+}$ characteristic line issued from $G$.
As in Section 5.8.3, there is a simple wave (denoted by ${\it f}_{+}^2$) with straight $C_{+}$ characteristic lines
issued from $\widetilde{ER}$ and $\widetilde{RI}$, and a shock-fan composite wave (denoted by ${\it sf}_{-}$) issued from $\widetilde{GI}$. Then the interaction of ${\it fs}_{-}$ with ${\it f}_{+}$ generates a simple wave and a shock-fan composite wave.

The flow in the remaining part of the duct can then be determined using the method described in Sections 5.1 and 5.3.
This completes the proof of Theorem \ref{thm10}.

\begin{rem}
 The assumption of $\tau_1>\tau_2^i$
is made to ensure that the straight $C_{-}$ characteristic lines of the simple wave adjacent to the constant state $(u_1, v_1)$ do not intersect with each other.
\end{rem}

\subsection{Interaction of a fan wave with a shock wave}
We consider the IBVP (\ref{E1}), (\ref{42701})
for $\tau_1^e<\tau_0<\tau_1^i$.
In this part we assume that the oblique waves at the corners $B$ and $D$ are centered simple wave and oblique shock wave, respectively, and the flow near the corners  $B$ and $D$ can be represented by
$$
(u, v)=\left\{
               \begin{array}{ll}
                 (u_0, 0), & \hbox{$-\frac{\pi}{2}\leq \eta\leq \eta_0$;} \\[2pt]
                 (\tilde{u}, \tilde{v})(\eta), & \hbox{$\eta_1\leq \eta\leq \eta_0$;} \\[2pt]
                 (u_1, v_1), & \hbox{$\eta_{1}<\eta<\theta_{+}$,}
               \end{array}
             \right.
\quad \mbox{and}\quad
(u, v)=\left\{
               \begin{array}{ll}
                 (u_0, 0), & \hbox{$\xi_s< \xi\leq\frac{\pi}{2}$;} \\[2pt]
                 (u_2, v_2), & \hbox{$\theta_{-}<\xi<\xi_s$,}
               \end{array}
             \right.
$$
respectively.
Here, the flow near the corner $B$ is the same as that defined in  Section 5.6;
the flow near the corner $D$ is the same as that defined in  Section 5.4.
The main result of this subsection can be stated as the following theorem.
\begin{thm}\label{thm11}
Assume $\tau_1^e<\tau_0<\tau_1^i$ and $u_0>c_0$.
Assume as well that the oblique waves at the corners $B$ and $D$ are centered simple wave and oblique shock wave, respectively. Assume furthermore $\tau_2>\tau_2^i$, Then the IBVP (\ref{E1}), (\ref{42701}) admits a global  piecewise smooth solution in $\Sigma$.
\end{thm}

We denote the centered simple wave issued from $B$ and the oblique shock issued from $D$ by ${\it f}_{-}$ and ${\it s}_{+}$, respectively.
The two oblique waves start to interact with each other at the point $P(x_{_P}, y_{_P})$, where
$x_{_P}=2(\tan\xi_s-\tan\eta_0)^{-1}$ and $y_{_P}=1+2(\tan\xi_s-\tan\eta_0)^{-1}\tan\eta_0$.
As in Section 5.8.3, the shock ${\it s}_{+}$ propagates into the fan wave ${\it f}_{-}$ region and intersects the straight $C_{-}$ characteristic line $y=1+x\tan\eta_1$ of ${\it f}_{-}$ at a point $E$.
We denote the shock segment connecting points $P$ and $E$ by $\wideparen{PE}$.
As the shock wave $\wideparen{FG}$ discussed in Section 5.8.3, there are the following two cases about the shock wave $\wideparen{PE}$.
\begin{itemize}
\item $\wideparen{PE}$ is a transonic shock; see Fig. \ref{Fig5.11.1} (left).
\item There is a point $H$ on $\wideparen{PE}$ such that the segment $\wideparen{PH}$ is transonic and the segment $\wideparen{HE}$ is post-sonic;  see Fig. \ref{Fig5.11.1} (right).
 \end{itemize}

\begin{figure}[htbp]
\begin{center}
\includegraphics[scale=0.35]{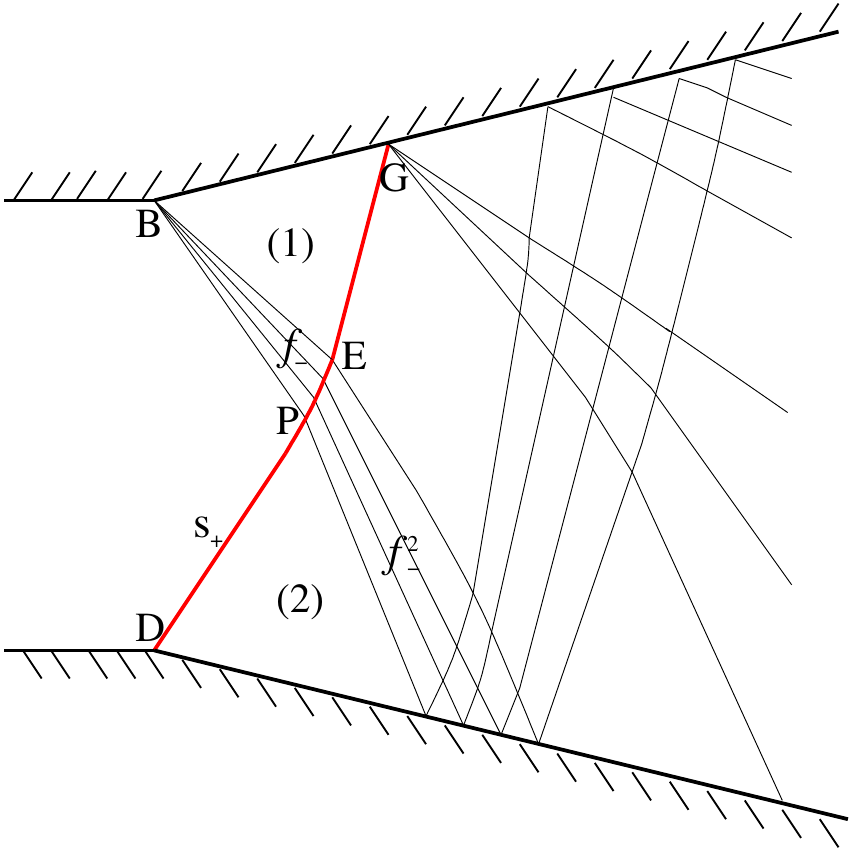}\qquad\qquad\qquad\includegraphics[scale=0.35]{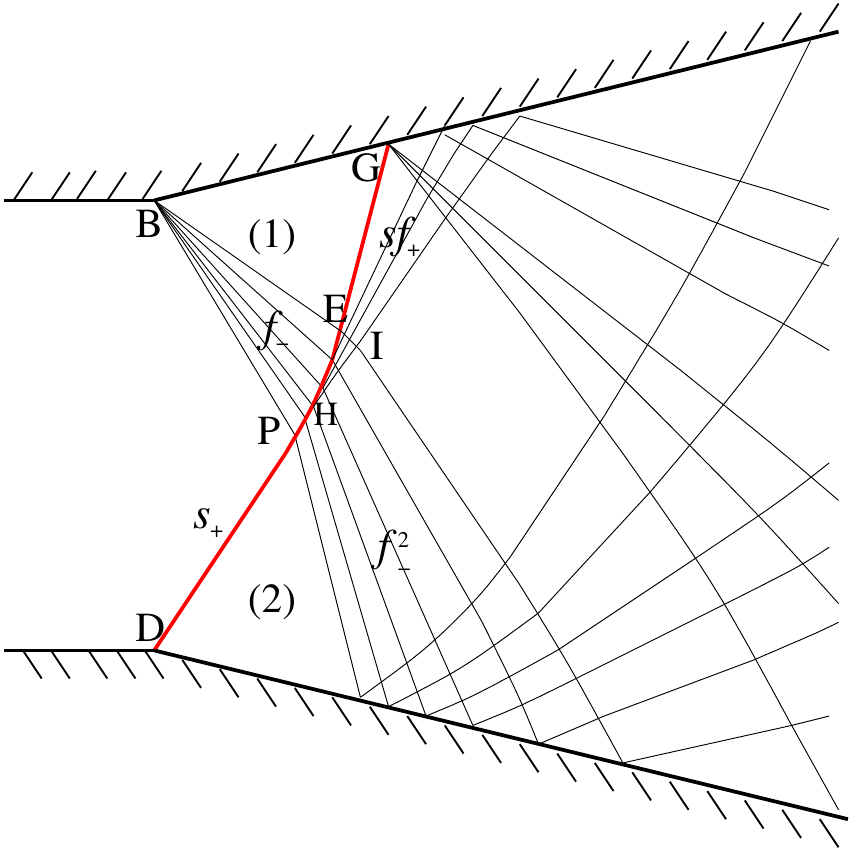}
\caption{\footnotesize Interaction of a centered simple wave with a shock wave. }
\label{Fig5.11.1}
\end{center}
\end{figure}

For the first case, there is a simple wave (denoted by ${\it f}_{-}^2$) with straight $C_{-}$ characteristic lines  issued from the shock $\wideparen{PE}$. The shock propagates at a constant speed from the point $E$ and impinges on the wall $\mathcal{W}_{+}$ at a point $G$, generating a simple wave emanating from the point $G$.

For the second case, we solve a singular Cauchy problem to obtain the flow downstream of the post-sonic shock $\wideparen{HE}$, i.e., the flow on a triangle region bounded by  $\widetilde{HI}$, $\wideparen{HE}$, and $\widetilde{EI}$, where $\widetilde{EI}$ is a forward $C_{-}$
characteristic line issued from $E$ and $\widetilde{HI}$ is a forward $C_{+}$ characteristic line issued from $H$.
As in the previous discussion, there is a simple wave (denoted by ${\it f}_{+}^2$) with straight $C_{+}$ characteristic lines
issued from  $\wideparen{PH}$ and $\widetilde{HI}$, and a shock-fan composite wave (denoted by ${\it sf}_{+}$) with a post-sonic shock issued from the point $E$ and straight $C_{+}$ characteristic lines issued from $\widetilde{EI}$. Then the interaction of ${\it f}_{-}$ with ${\it s}_{+}$ generates a simple wave and a shock-fan composite wave.

The flow after the interaction of ${\it f}_{-}$ with ${\it s}_{+}$ can be constructed using the method described in the previous subsections.
We omit the details. This completes the proof of Theorem \ref{thm11}.

\begin{rem}
The assumption of $\tau_2>\tau_2^i$ is made to ensure that the straight $C_{-}$ characteristic lines of the simple wave adjacent to the constant state $(u_2, v_2)$ do not intersect with each other.
\end{rem}

\subsection{Interaction of a fan-shock composite wave with a shock wave}
We consider the IBVP (\ref{E1}), (\ref{42701})
for $\tau_1^e<\tau_0<\tau_1^i$.
In this part we assume that the oblique waves at the corners $B$ and $D$ are fan-shock composite wave and oblique shock wave, respectively, and the flow near the corners  $B$ and $D$ can be represented by
$$
(u, v)=\left\{
               \begin{array}{ll}
                 (u_0, 0), & \hbox{$-\frac{\pi}{2}<\eta<\eta_0$;} \\[2pt]
                 (\tilde{u}, \tilde{v})(\eta), & \hbox{$\eta_0<\eta\leq \eta_{pr}$;} \\[2pt]
                 (u_1, v_1), & \hbox{$\eta_{pr}<\xi<\theta_{+}$,}
               \end{array}
             \right.
\quad \mbox{and}\quad
(u, v)=\left\{
               \begin{array}{ll}
                 (u_0, 0), & \hbox{$\xi_s< \xi\leq \frac{\pi}{2}$;} \\[2pt]
                 (u_2, v_2), & \hbox{$\theta_{-}<\xi<\xi_s$,}
               \end{array}
             \right.
$$
respectively.
Here, the flow near the corner $B$ is the same as that defined in  Section 5.7;
the flow near the corner $D$ is the same as that defined in  Section 5.4. The main result of this subsection can be stated as the following theorem.
\begin{thm}\label{thm12}
Assume $\tau_1^e<\tau_0<\tau_1^i$ and $u_0>c_0$.
Assume as well that the oblique waves at the corners $B$ and $D$ are are fan-shock simple wave and oblique shock wave, respectively. Assume furthermore that $\tau_2>\tau_2^i$. Then the IBVP (\ref{E1}), (\ref{42701}) admits a global  piecewise smooth solution in $\Sigma$.
\end{thm}

We denote the oblique fan-shock composite wave issued from $B$ and the oblique shock wave issued from $D$ by ${\it fs}_{-}$ and ${\it s}_{+}$, respectively.
The two oblique waves start to interact with each at the point $P(x_{_P}, y_{_P})$, where
$x_{_P}=2(\tan\xi_s-\tan\eta_0)^{-1}$ and $y_{_P}=1+2(\tan\xi_s-\tan\eta_0)^{-1}\tan\eta_0$.
 We divide the discussion into the following two cases: (i) $\tau_1\geq\tau_2^i$; (ii) $\tau_1<\tau_2^i$.

\noindent
{\bf (i) $\tau_1\geq \tau_2^i$.}
As in Section 5.8.3, the shock ${\it s}_{+}$ will propagate into the fan wave ${\it fs}_{-}$ region and intersect the pre-sonic shock of ${\it fs}_{-}$ at a point $E$.
From the point $P$, we draw a straight $C_{-}$ characteristic with the characteristic angle $\beta=\theta_{-}-A_2$. This characteristic line intersects the wall $\mathcal{W}_{-}$ at a point $F$.
Due to $\tau_1>\tau_2^i$, from the point $E$ we draw a straight $C_{+}$ characteristic with the characteristic angle $\alpha=\theta_{+}+A_1$. This line intersects the wall $\mathcal{W}_{+}$ at a point $G$.

As the shock wave $\wideparen{FG}$ discussed in Section 5.8.3, there are the following two subcases about the shock wave $\wideparen{PE}$.
\begin{description}
  \item[(i1)] $\wideparen{PE}$ is a transonic shock; see Fig. \ref{Fig5.12.1} (left).
  \item[(i2)] There is a point $H$ on $\wideparen{PE}$ such that the segment $\wideparen{PH}$ is transonic and the segment $\wideparen{HE}$ is post-sonic;  see Fig. \ref{Fig5.12.1} (right).
\end{description}

For subcase (i1), there is a simple wave  with straight $C_{-}$ characteristic lines  issued from the shock $\wideparen{PE}$  on a curvilinear quadrilateral domain bounded by $\overline{PF}$, $\wideparen{PE}$, $\overline{EK}$, and $\widetilde{FK}$, where
 $\overline{EK}$ is a straight characteristic line issued from $E$ and $\widetilde{FK}$ is a $C_{+}$ characteristic line issued from $F$. We next solve a Riemann-type problem for (\ref{E1}) with $\overline{EG}$ and $\overline{EK}$ as the characteristic boundaries to get two centered simple waves with different types issued from the point $E$.
Then the interaction of ${\it fs}_{-}$ and ${\it s}_{+}$ generates two simple waves; see Fig. \ref{Fig5.12.1} (left).

For subcase (i2), we first solve a singular Cauchy problem for (\ref{E1}) to obtain the flow downstream of the post-sonic shock $\wideparen{HE}$, i.e., the flow on a triangle region bounded by  $\widetilde{HI}$, $\wideparen{HE}$, and $\widetilde{EI}$, where $\widetilde{EI}$ is a forward $C_{-}$
characteristic line issued from $E$ and $\widetilde{HI}$ is a forward $C_{+}$ characteristic line issued from $H$.
As in the previous discussion, there is a simple wave (denoted by ${\it f}_{-}$) with straight $C_{-}$ characteristic lines on a curvilinear quadrilateral domain bounded by  $\wideparen{PH}\cup\widetilde{HI}$, $\overline{PF}$, $\widetilde{FK}$, and $\overline{IK}$, where $\widetilde{FK}$ is a forward $C_{+}$ characteristic curve issued from $F$ and $\overline{IK}$ is a forward straight $C_{-}$ characteristic line issued from $I$.
Next, we solve a DGP for (\ref{E1}) with $\overline{EG}$ and $\widetilde{EI}\cup \overline{IK}$ as the characteristic boundaries.
This DGP admits a solution on a curvilinear  quadrilateral domain bounded by $\overline{EG}$, $\widetilde{EI}\cup \overline{IK}$, $C_{+}^{K}$ and $C_{-}^{G}$. The solution contains a centered wave on a curvilinear triangle domain bounded by $\widetilde{EI}$, $\widetilde{EJ}$, and $\widetilde{IJ}$ where $\widetilde{EJ}$ is a $C_{-}$ characteristic curve issued from $E$ and $\widetilde{IJ}$ is a $C_{+}$ characteristic curve issued from $I$.
The solution also contains a simple wave with straight $C_{+}$ characteristic lines issued from the point $E$ and
$\widetilde{EJ}$ and a simple wave with straight $C_{-}$ characteristic lines issued from $\widetilde{IJ}$.
In this case, the interaction of ${\it fs}_{-}$ with ${\it s}_{+}$ also generates two simple waves.

\begin{figure}[htbp]
\begin{center}
\includegraphics[scale=0.32]{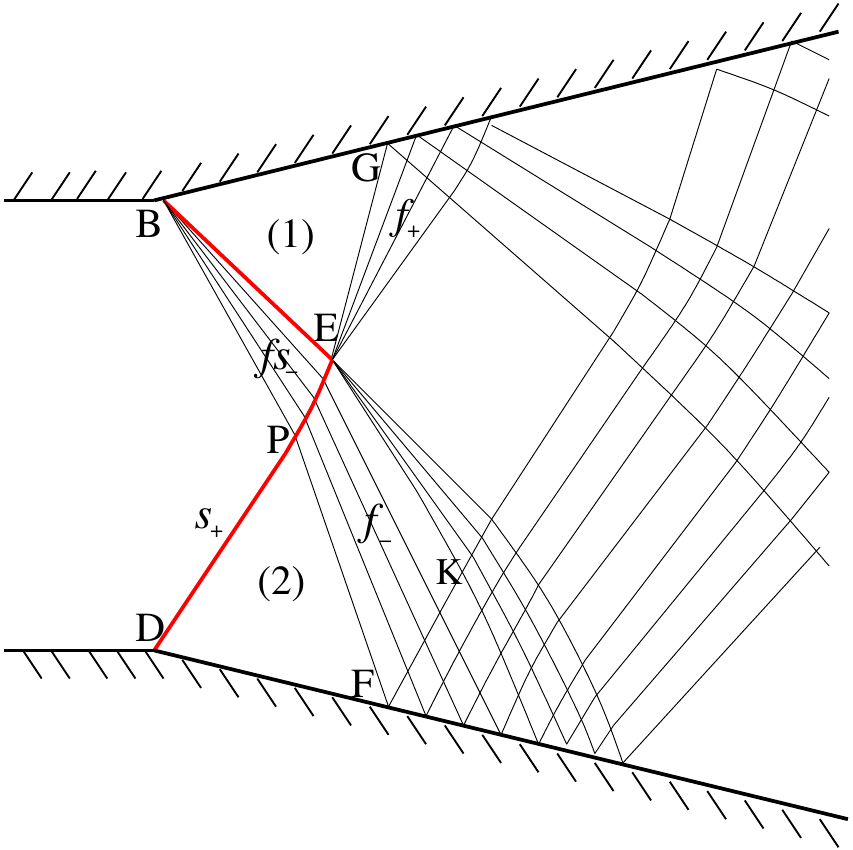}\qquad\qquad \qquad \includegraphics[scale=0.32]{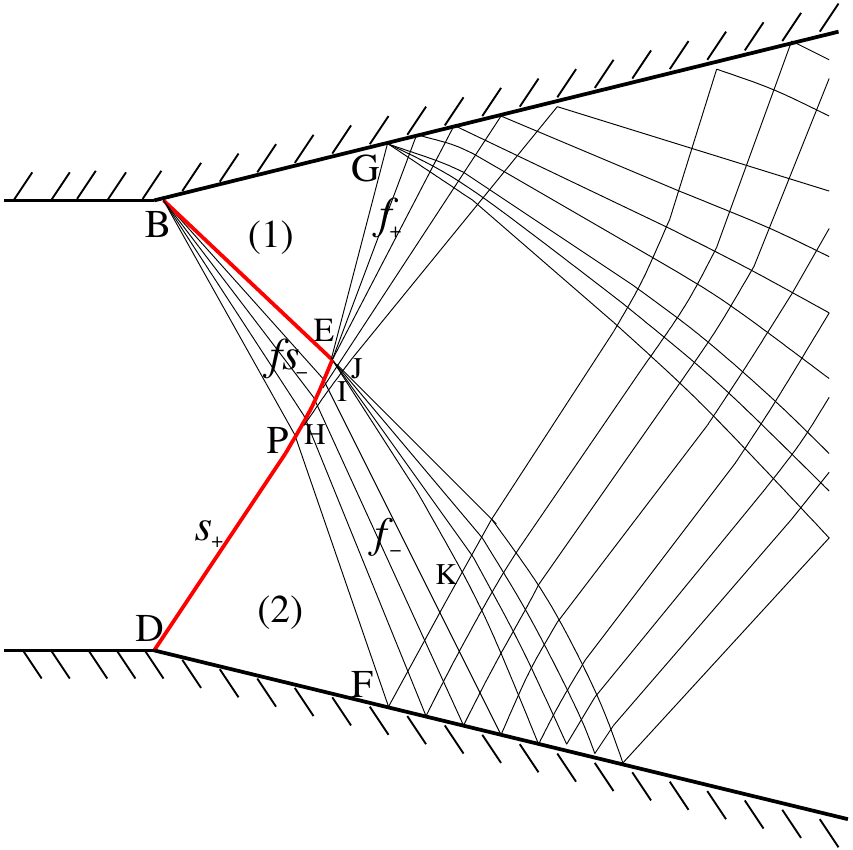}
\caption{\footnotesize Interaction of a fan-shock composite wave with a shock: $\tau_1\geq\tau_2^i$.}
\label{Fig5.12.1}
\end{center}
\end{figure}
\begin{figure}[htbp]
\begin{center}
\includegraphics[scale=0.32]{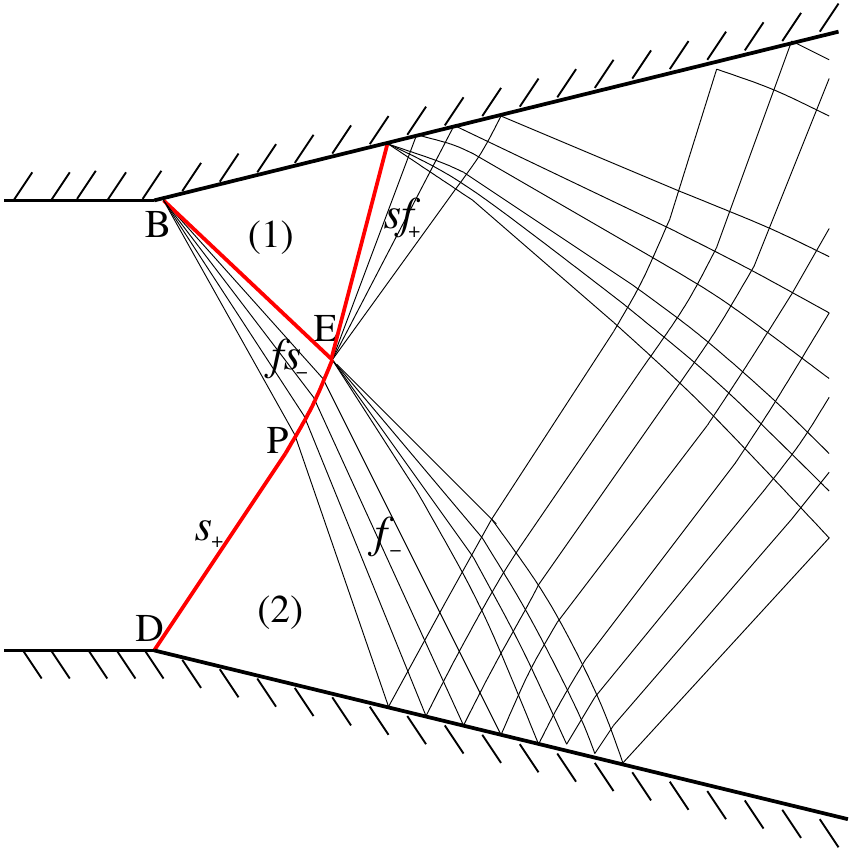}\qquad\qquad \qquad \includegraphics[scale=0.32]{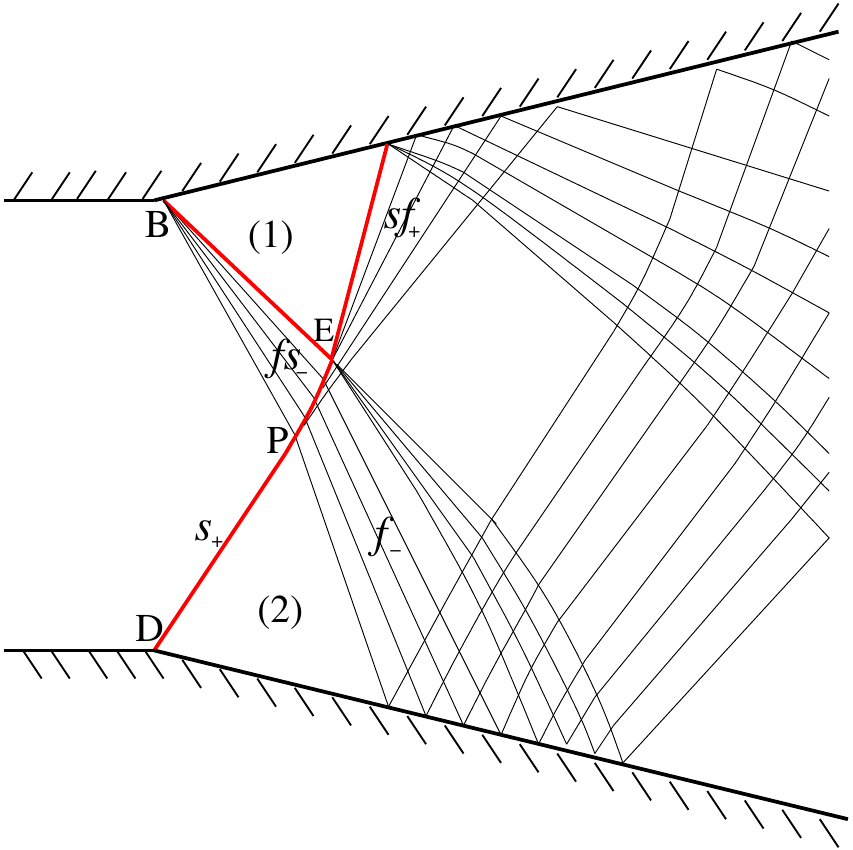}
\caption{\footnotesize Interaction of a fan-shock composite wave with a shock: $\tau_1<\tau_2^i$. }
\label{Fig5.12.2}
\end{center}
\end{figure}

\noindent
{\bf (ii) $\tau_1<\tau_2^i$.}
There is only a slight difference in the discussion between cases (i) and (ii).
 Due to $\tau_1<\tau_2^i$, we construct a straight post-sonic shock propagating from the point $E$. The front side state of the post-sonic shock is $(u_1, v_1)$. The backside state and the inclination angle of the shock can then be determined by the shock relations. 
The discussion in the other part of the interaction is the same as that for case (i). Then for $\tau_1<\tau_2^i$, the interaction of ${\it fs}_{-}$ and ${\it s}_{+}$ generates a shock-fan composite wave and a simple wave; see Fig. \ref{Fig5.12.2}.

The flow after the interaction  of ${\it fs}_{-}$ and ${\it s}_{+}$ can be determined using the method described in the previous subsections.
We omit the details. This completes the proof of Theorem \ref{thm12}.

\subsection{Interaction of fan-shock composite waves}
We consider the IBVP for $\tau_1^e<\tau_0<\tau_1^i$.
In this part we assume that the oblique waves at the corners $B$ and $D$ are fan-shock composite waves, and the flow near the corners  $B$ and $D$ can be represented by
$$
(u, v)=\left\{
               \begin{array}{ll}
                 (u_0, 0), & \hbox{$-\frac{\pi}{2}<\eta<\eta_0$;} \\[2pt]
                 (\tilde{u}, \tilde{v})(\eta), & \hbox{$\eta_0<\eta\leq \eta_{1}$;} \\[2pt]
                 (u_1, v_1), & \hbox{$\eta_{1}<\xi<\theta_{+}$;}
               \end{array}
             \right.
\quad \mbox{and}\quad
(u, v)=\left\{
               \begin{array}{ll}
                 (u_0, 0), & \hbox{$\xi_0\leq \xi\leq \frac{\pi}{2}$;} \\[2pt]
                 (\bar{u}, \bar{v})(\xi), & \hbox{$\phi_{pr}<\xi<\xi_0$;} \\[2pt]
                (u_2, v_2), & \hbox{$\theta_{-}<\xi<{\phi}_{pr}$,}
               \end{array}
             \right.
$$
respectively.
Here, the flow near the corner $B$ is the same as that defined in  Section 5.7;
the flow near the corner $D$ is the same as that defined in  Section 5.8.
The main result of this subsection can be stated as the following theorem.
\begin{thm}\label{thm13}
Assume $\tau_0<\tau_1^i$ and $u_0>c_0$.
Assume furthermore that the oblique waves at the corners $B$ and $D$ are fan-shock composite waves and $\min\{\tau_1, \tau_2\}>\tau_2^i$, where $\tau_1$ and $\tau_2$ are defined the same as that defined in (\ref{def12}).
Then when $\psi_{pr}^{-1}(\tau_1)-\tau_0$ is sufficiently small, the IBVP (\ref{E1}), (\ref{42701}) admits a global  piecewise smooth solution in $\Sigma$.
\end{thm}

We use ${\it fs}_{-}$ and ${\it fs}_{+}$ to denote the fan-shock composite waves issued from point $B$ and $D$, respectively.
These two composite waves start to interact with each other at point ${P}(\cot\xi_0,0)$.
From the point $P$, we draw a forward $C_{+}$  cross characteristic curve in ${\it fs}_{-}$.
This characteristic curve intersects the pre-sonic shock of ${\it fs}_{-}$ at a point $E$.
From the point $P$, we draw a forward $C_{-}$  cross characteristic curve in ${\it fs}_{+}$.
This characteristic curve intersects the pre-sonic shock  of ${\it fs}_{+}$ at a point $F$.

We first consider (\ref{E1}) with the boundary condition
\begin{equation}\label{10601}
(u, v)=\left\{
               \begin{array}{ll}
                  (\tilde{u}, \tilde{v})(\eta), & \hbox{$(x,y)\in \widetilde{{PE}}$;} \\[2pt]             (\bar{u}, \bar{v})(\xi), & \hbox{$(x,y)\in \widetilde{{PF}}$.} \\[2pt]
              \end{array}
             \right.
\end{equation}
When $\psi_{pr}^{-1}(\tau_1)-\tau_0$ is sufficiently small,
the SGP (\ref{E1}), (\ref{10601}) admits a classical solution on a curvilinear quadrilateral domain $\overline{\Omega^3}$ bounded by $\widetilde{{PE}}$, $\widetilde{{PF}}$, $\widetilde{{ES}}$, and $\widetilde{{FS}}$, where $\widetilde{{ES}}$ is a forward $C_{+}$ characteristic line issued from ${E}$ and $\widetilde{{FS}}$ a forward $C_{-}$ characteristic line issued from ${F}$. Moreover the solution satisfies $\tau<\tau_1^i$ in $\overline{\Omega^3}$.

As discussed in the previous subsections, there are the following three cases:
\begin{itemize}
  \item The shock of ${\it fs}_{-}$ propagates into the domain $\Omega^3$ from the point $E$, and stays in $\Omega^3$ until it intersects $\widetilde{PF}$ at a point $H$. This shock will then propagate in the fan wave of ${\it fs}_{+}$ until it intersects the pre-sonic shock of ${\it fs}_{+}$ at a point $G$; see Fig. \ref{Fig5.13.2}.
  \item When $(\psi_{pr}^{-1}(\tau_2)-\tau_0)/(\psi_{pr}^{-1}(\tau_1)-\tau_0)$ is small, the shock of ${\it fs}_{+}$ propagates into the domain $\Omega^3$ from the point $F$ and stays in $\Omega^3$ until it intersects $\widetilde{PE}$ at a point. This shock will then propagate in the fan wave of ${\it fs}_{-}$ from that point until it intersects the pre-sonic shock of ${\it fs}_{-}$ at some point.
  \item The two shocks of ${\it fs}_{\pm}$ propagate into the domain $\Omega^3$ from the points $E$ and $F$, and intersect with each other at a point $G$ inside the domain $\Omega^3$; see Fig. \ref{Fig5.13.1}.
\end{itemize}
In what follows, we are going to discuss the first and third cases. The discussion for the second case is similar to that of the first case.

\begin{figure}[htbp]
\begin{center}
\includegraphics[scale=0.33]{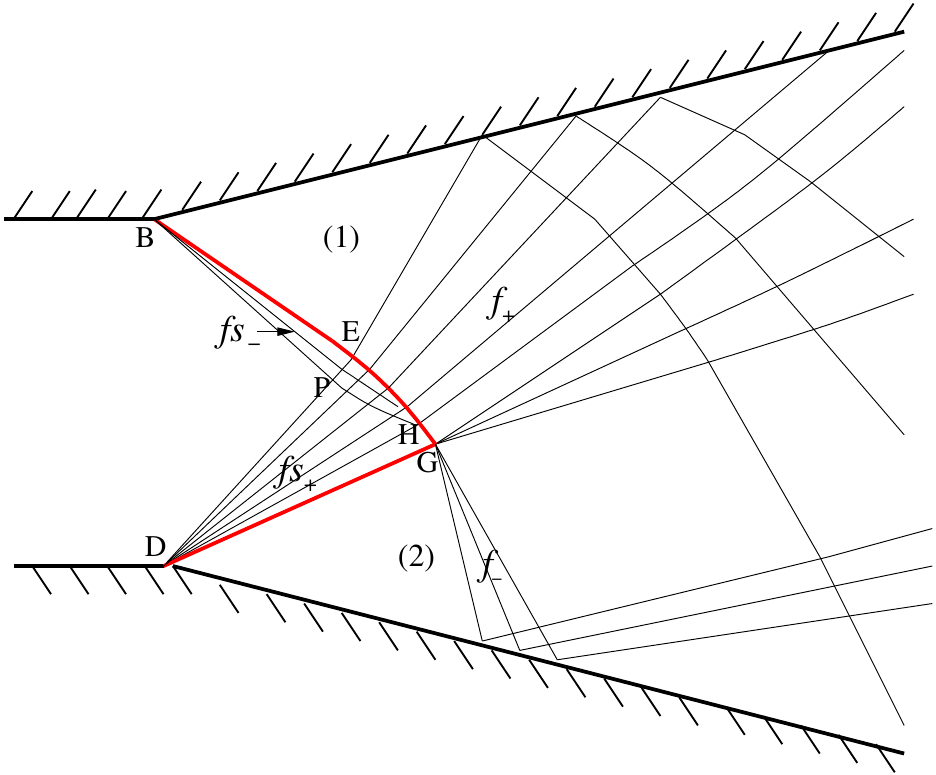}\qquad\qquad\qquad\includegraphics[scale=0.33]{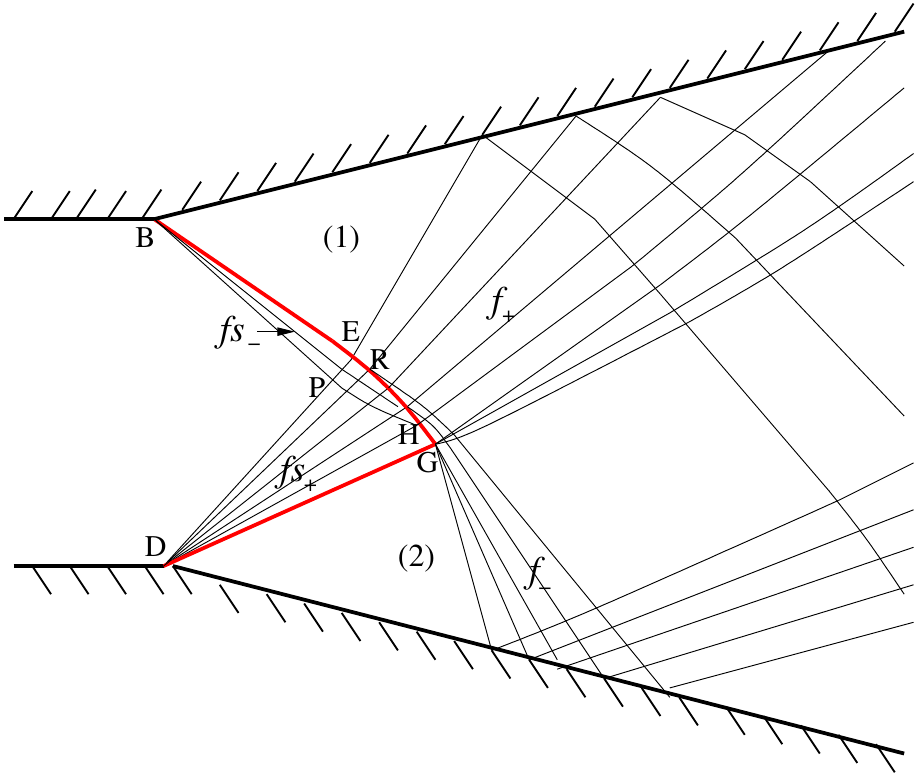}
\caption{\footnotesize Interaction of fan-shock composite waves for the first case. }
\label{Fig5.13.2}
\end{center}
\end{figure}
\begin{figure}[htbp]
\begin{center}
\includegraphics[scale=0.33]{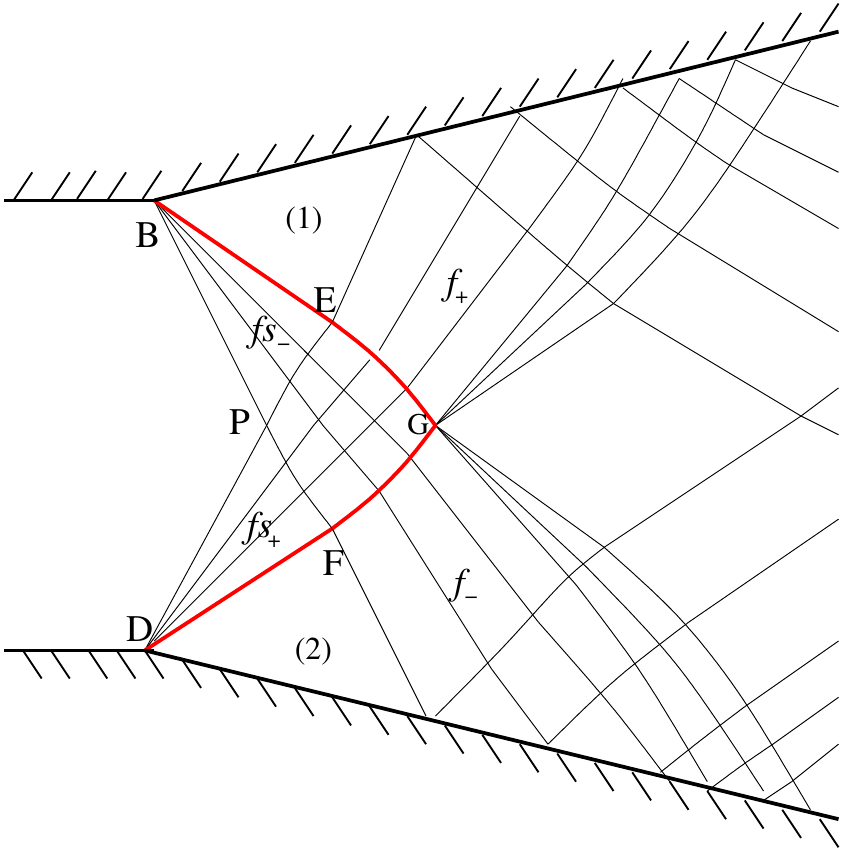}\qquad\qquad\qquad\includegraphics[scale=0.33]{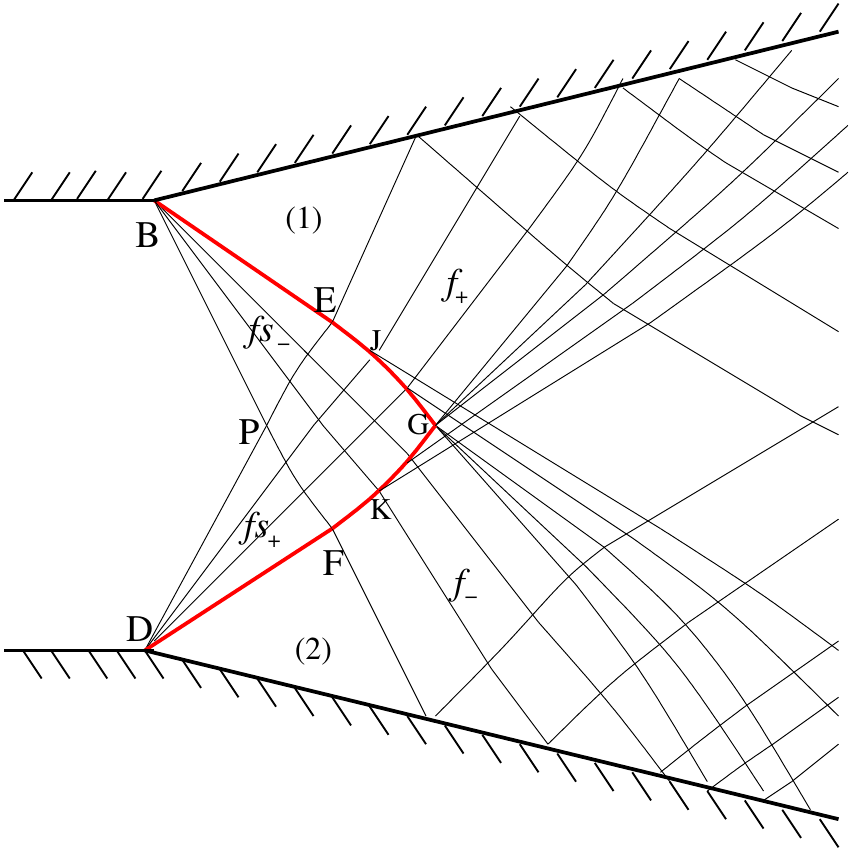}
\caption{\footnotesize Interaction of fan-shock composite waves for the third case.}
\label{Fig5.13.1}
\end{center}
\end{figure}

For the first case, we use $\wideparen{EG}$ to denote the shock wave from the point $E$. The shock  $\wideparen{EG}$ could be a pure transonic shock (see Fig. \ref{Fig5.13.2} (left)) or a combination of a transonic shock and a post-sonic shock (see $\wideparen{ER}$ and $\wideparen{RG}$ in Fig. \ref{Fig5.13.2} (right)).
If the shock  $\wideparen{EG}$ is a pure transonic shock,
the flow downstream of it is a simple wave and can be determined by solving a free boundary value problem as in Section 5.8.
If the shock  $\wideparen{EG}$ is not a pure transonic shock, the flow downstream of it
 can be obtained by solving a free boundary problem and a singular Cauchy problem.
Finally, by solving a simple wave problem and a discontinuous Goursat problem we can obtain the oblique waves generated from the interaction of ${\it fs}_{\pm}$. As in the previous subsections, we see that the interaction generates two simple waves with different types.

For the third case, we use $\wideparen{EG}$ and  $\wideparen{FG}$ to denote the shock waves propagated from the points $E$ and $F$, respectively. The shock  $\wideparen{EG}$ ($\wideparen{FG}$, resp.) may be a pure transonic shock or a combination of a transonic shock and a post-sonic shock. The flow downstream of $\wideparen{EG}$ and $\wideparen{FG}$ can be obtained by solving some simple wave problems and  singular Cauchy problems.
Moreover, by solving a discontinuous Goursat problem we see that the interaction of ${\it fs}_{\pm}$ generates two simple waves with different types.

The flow after the interaction  of ${\it fs}_{\pm}$  can then be determined using the method described in Section 5.1; we omit the details. This completes the proof of Theorem \ref{thm13}.

\section{Relation to 2D Riemann problems}
2D Riemann problems refer to Cauchy problems with special initial data that are
constant along each ray from the origin.
It is well known that the 2D Riemann problem is an important problem in nonlinear hyperbolic conservation laws. Solving  2D Riemann problems would illuminate the physical structure and interaction of waves, and would serve to improve numerical algorithms; see, e.g., Menikoff and Plohr
 \cite{MP}. 
There have been many important works on the existence of global solutions to 2D Riemann problems for convex gases; see \cite{BCF1,BCF2,CCHLQ,CX,CF1,CF2,Chen1,Dai,ELiu2,Hu,Li1,Li3,Li4,Serre,Zheng1,Zheng2}.
However, to the best of our knowledge, much less is known about the global existence of solutions to 2D Riemann problems for nonconvex gases. In \cite{Lai0,Lai1,Lai3,Lai5}, the author studied several special 2D Riemann problems for a van der Waals gas. These Riemann problems involve the interactions of planar fan waves, shock-fan composite waves, and fan-shock-fan composite waves.

As the first step in solving 2D Riemann problems, one needs to study various types of wave interactions.
The results and methods of this paper are also applicable to some 2D Riemann problems for gases with nonconvex equations of state.
For example, we consider the 2D  Euler equations for potential flow
\begin{equation}\label{2DRiemann1}
\left\{
  \begin{array}{ll}
    \rho_t+(\rho\Phi_x)_x+(\rho\Phi_y)_y=0,  \\[4pt]
    \displaystyle\Phi_t+\frac{1}{2}(\Phi_x^2+\Phi_y^2)+h(\tau)=\mathcal{B},
  \end{array}
\right.
\end{equation}
with the Riemann initial data
\begin{equation}\label{2DRiemann2}
(\tau, \Phi)\mid_{t=0}~=~(i):=(\tau_i, u_ix+v_iy), \quad (x,y)\in i^{th}~\mbox{quadrant}, ~~i=1, 2, 3, 4.
\end{equation}
Here, $\Phi$ is the velocity potential, i.e., $(\Phi_x, \Phi_y)=(u, v)$; $\mathcal{B}$ is a constant; $(\tau_i, u_i, v_i)$ ($i=1, 2, 3, 4$) are four distinct constant initial states.

We assume that
the initial data is chosen so that only a shock wave, fan wave, or composite wave  connects two neighboring constant initial states
 at the initial time. 
 We choose the following relation:
\begin{equation}\label{13001}
\begin{array}{rll}
\tau_2 & > \tau_1 \\
\vee & \quad \wedge \\
\tau_3 & < \tau_4
\end{array}
\end{equation}
for the specific volume;
$$
u_1=u_4>u_2=u_3
\quad and \quad
v_1=v_2>v_3=v_4
$$
for the velocities.
 If the gas is a polytropic gas, then the constant states (1) and (2) is connected by a forward fan wave $\overrightarrow{F}_{21}$, (1) and (4)  is connected by a forward fan wave $\overrightarrow{F}_{41}$, (3) and (2)  is connected by a backward fan wave $\overrightarrow{F}_{32}$, and
  (3) and (4)  is connected by a backward fan wave $\overrightarrow{F}_{34}$; see \cite{Schulz-Rienne,ZhangZheng}.
However, if the gas is a nonconvex gas, the initial configuration is of the form
\begin{equation}\label{13002}
\begin{tikzpicture}[baseline=(current bounding box.center)]
\node at (0,0.7)    {$\overrightarrow{W}_{21}$};  
\node at (-0.7,0)   {$\overleftarrow{W}_{32}$};   
\node at (0.7,0)    {$\overrightarrow{W}_{41}$};   
\node at (0,-0.7)   {$\overleftarrow{W}_{34}$};  
\end{tikzpicture}
\end{equation}
where $W_{ij}$ might be a shock wave, fan wave, fan-shock composite wave, shock-fan composite wave, or fan-shock-fan composite wave.
Then,
in order to solve the above 2D Riemann problem we need to study the interactions of $\overrightarrow{W}_{21}$ with $\overrightarrow{W}_{41}$ and the interaction of $\overrightarrow{W}_{32}$ with $\overrightarrow{W}_{34}$.
 Under assumptions (\ref{13001}) and (\ref{13002}) there also exist thirteen distinct types
of wave interactions as discussed in this paper. 







\vskip 32pt


\end{document}